\documentclass[leqno]{aadmbook}
\usepackage{amsthm}
\usepackage{amsfonts}
\usepackage{amsmath}
\usepackage{latexsym,amsfonts,mathptm}
\usepackage{epsfig,float,graphicx}
\usepackage{graphics}
\restylefloat{figure} 
\def\dj{d\kern-0.4em\char"16\kern-0.1em}
\def\Dj{\mbox{\raise 0.3ex\hbox{-}\kern-0.38em D}}

\newtheorem{thm}{Theorem}
\newtheorem{lem}{Lemma}
\newtheorem{Proposition}{Proposition}

\newtheorem{df}{Definition}
\newtheorem{conj}{Conjecture}
\newtheorem{Example}{Example}

\def\ds{\displaystyle}
\def\dzn{,\kern-0.1em,}

\input cyracc.def
 
\def\be{\begin{equation} }
\def\ee{\end{equation} }
\def\bfl{\begin{flushleft} }
\def\efl{\end{flushleft} }
\def\bfr{\begin{flushright} }
\def\efr{\end{flushright} }
\def\bc{\begin{center}}
\def\vs*{\vspace*}
\def\hs*{\hspace*}
\def\ec{\end{center}}
\def\beq{\begin{eqnarray}}
\def\eeq{\end{eqnarray}}

\def\ben{\begin{enumerate}}
\def\een{\end{enumerate}}
\def\bit{\begin{itemize}}
\def\eit{\end{itemize}}

\begin{document}

% Your \newcommands below (if there are any):

\oddsidemargin 16.5mm
\evensidemargin 16.5mm

\thispagestyle{plain}

%\begin{center}
%{\large \sc  Applicable Analysis and Discrete Mathematics}
%{\small available online at  http:/$\!$/pefmath.etf.rs }
%\end{center}

%\noindent{\small{\sc  Appl.\ Anal.\ Discrete Math.\ }{\bf x} (xxxx),
%xxx--xxx.} \hfill{\scriptsize doi:10.2298/AADMxxxxxxxx}

\vspace{5cc}
\begin{center}
{\Large\bf   A spanning union of cycles  in thin cylinder, torus and Klein bottle grid graphs
\rule{0mm}{6mm}\renewcommand{\thefootnote}{}%Enter at least one, but not more than 3 MSCs.
% First entered MSC will be a primary one, others (at most 2) will be secondary.
\footnotetext{

\rule{2.4mm}{0mm}Keywords and Phrases:  enumeration, algorithm,  grid graphs, transfer matrix method}}

\vspace{1cc} {\large\it   Jelena \Dj oki\' c, Ksenija Doroslova\v
cki\footnote{corresponding author},  Olga Bodro\v{z}a-Panti\'{c}}

\vspace{1cc}
\parbox{24cc}{{\small
 We  propose an algorithm for obtaining the common     transfer digraph
${\cal D}^*_m$ for  enumeration  of 2-factors  in graphs from the title
all of which with $m \cdot n$ vertices ($m, n \in N, m \geq 2$).
 The numerical data  gathered for $m \leq 18$ reveal some matchings of 
the numbers of 2-factors for different types of torus  or Klein bottle.
In latter case we  conjecture that these numbers are  invariant under  twisting.  
}}
\end{center}

\vspace{1cc}

% The paper should have at least two sections

\vspace{1.5cc}
\begin{center}
{\bf  1.  INTRODUCTION}
\end{center}
\label{sec:intro}

\vspace*{4mm}

After conducted research concerning  enumeration of 2-factors on the  rectangular $RG_{m}(n)$, thick  cylinder $TkC_{m}(n)$ and Moebius strip $MS_{m}(n)$ grid graphs, \cite{JOK1} and \cite{JKO2},
 we now continue on new three classes grid  graphs: the thin cylinder $TnC_{m}(n)$, torus $TG^{(p)}_{m}(n)$ and Klein bottle $KB^{(p)}_m(n)$ all of which with $m \cdot n$ vertices.
 The  motivation for these explorations is exposed in  \cite{JOK1}.

 A spanning union of (disjoint) cycles for a  graph is 2-regular spanning subgraph and is therefore   called a 2-factor.
  The structure of (square) grid graphs under consideration  allows the grouping of vertices in  ``columns''
  and  considering  their edges as ``vertical'' and ``horizontal''.
For  an arbitrary 2-factor  of a such graph we code its vertices. To each its  column we join a word over  some alphabet of length $m$ (abbr. m-word), where  $m \in N$ is called the {\em width} of the grid graph. These words represent the vertices of the  so-called \emph{Transfer digraph}.
The process of enumeration of 2-factors is reduced to enumeration of  some oriented walks in it.
 In contrary to the case of rectangular, thick cylinder and Moebious strip  grid graphs where
 the words assigned to columns of vertices are linear,
 for the graphs from the title: $TnC_{m}(n)$, $TG^{(p)}_{m}(n)$ and  $KB^{(p)}_m(n)$ ($ 0 \leq p \leq m-1; m,n \in N $)  we treat those words as circular ones.

  \begin{figure}[htb]
\begin{center}
\includegraphics[width=4in]{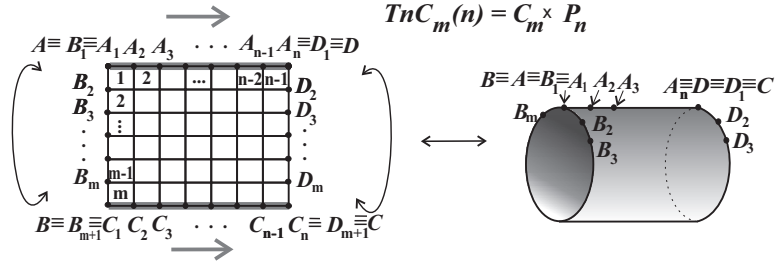}
\\ \ \vspace*{-18pt}
\end{center}
\caption{ Construction of the thin grid cylinder $TnC_m(n)= C_{m} \times  P_n$.}
\label{Tankicilindar}
\end{figure}

\begin{figure}[htb]
\begin{center}
\includegraphics[width=4in]{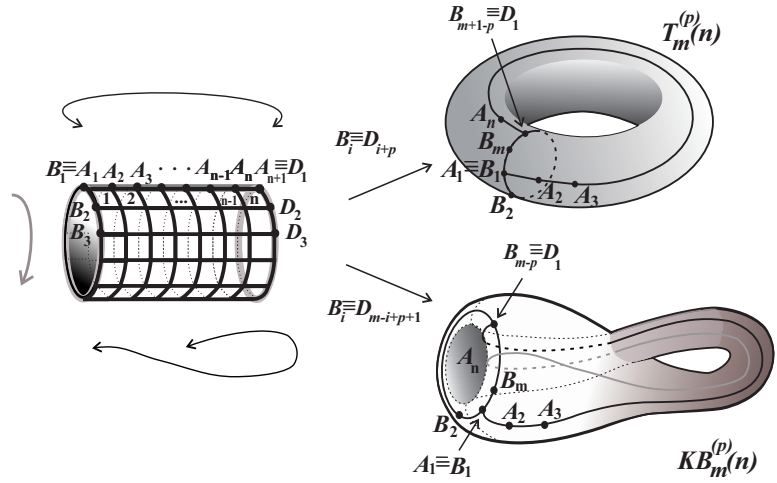}
\\ \ \vspace*{-18pt}
\end{center}
\caption{Construction of the torus grid  $TG^{(p)}_m(n)= C_{m} \times  C_n$ and Klein bottle $KB^{(p)}_m(n)$ }
\label{torus-Klein}
\end{figure}

 The thin grid cylinder $TnC_m(n)= C_{m} \times  P_n$, where  $P_n$ and $C_n$ denote the path and  cycle with $n$ vertices, respectively, can be obtained from the  rectangular grid graph  $RG_{m+1}(n)$ $ = P_{m+1} \times P_{n}$
  by identification   of  corresponding vertices from the first and  last  rows (see Figure~\ref{Tankicilindar}, the vertex $A_i$ coincides with $C_i$ for $i=1, \ldots , n$) without   producing (horizontal) multiplying edges. The other two graphs which we considered  can be obtained from  the thin grid cylinder $TnC_m(n+1)= C_{m} \times  P_{n+1}$
   by gluing   corresponding vertices  from the first  and  last  columns, without   producing vertical multiplying edges.
    The identification  of corresponding vertices $B_i $, $1 \leq i \leq m$ from the first    column   and $D_j $, $1 \leq j \leq m$ from the  last  column
    in $TnC_m(n+1)$ (see Figure~\ref{torus-Klein})  can be done in several ways resulting the following grid graphs on a torus or a Klein bottle surface:
\begin{itemize}
\item
   If $B_i \equiv D_{i+p}$  for $i=1, \ldots , m$ where $0 \leq p \leq m-1$ and the sign $+$ in subscript is addition  modulo $m$, then we obtain the torus $TG^{(p)}_m(n)$.  Note that $TG^{(0)}_m(n) = C_m \times  C_n$ is just a special case.
\item If $B_i \equiv D_{m-i+p+1}$  for $i=1, \ldots , m$ ($0 \leq p \leq m-1$), then resulting grid graph is on  a Klein bottle surface and we denote it by $KB^{(p)}_m(n)$. (Informally, we elongate the cylinder, then bend it so that the two ends of the cylinder, i.e. the cycles with vertices $ B_1, B_2, \ldots , B_m$ and $ D_1, D_2, \ldots , D_m$ come close to each other. Prior to the  identification, one end of the cylinder is twisted until the corresponding vertices match.)
\end{itemize}

To be able to  `` walk to the left, right, upper or down'' trough the cycles in a 2-factor of one of our grid graphs imagine that we  ``cut and develop  in the plane" the  surfaces of the considered  graph.  At first,  in the case of  the torus $TG^{(p)}_m(n)$  or Klein bottle $KB^{(p)}_m(n)$  we ``cut'' the edges connecting the vertices of the  last and first column. Then,  in all  the  three cases,  we do the same with   the edges connecting the vertices of the  last and first row of the cylinder and develop it in the plane.
 The obtained picture in which the lines corresponding to the  edges of the considered 2-factor are emphasized (bolded)  is called {\bf flat-representation (FR)} of this 2-factor for the considered graph  (see Figure~ \ref{maliTankib} below).

 \begin{figure}[htb]
\begin{center}
\includegraphics[width=3in]{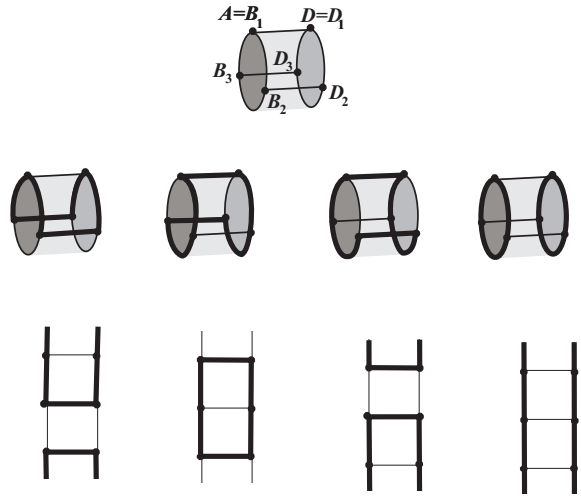}
\\ \ \vspace*{-18pt}
\end{center}
\caption{
 The graph  $TnC_3(2)= C_3 \times P_2$ with
all the possible  2-factors (bold edges) and
their  flat representations (below).}
\label{maliTankib}
\end{figure}

In case of thin cylinder,  using translation of infinite number of copies of FR
 and  linking them to each other in a sequence, as shown in   Figure~ \ref{maliTankia},
 we obtain the infinite  strip grid graph of width  $n$. In other two cases, at first, we produce the similar infinite  strip. Then,  we make the copies of it and using   translation (in case of torus) and glade symmetry (in case of  Klein-bottle) we tile the whole plane
taking into account that $B_i \equiv D_{i+p}$, $i=1, \ldots , m$ for $TG^{(p)}_m(n)$, or   $B_i \equiv D_{m+p+1-i}$, $i=1, \ldots , m$  for $KB^{(p)}_m(n)$.
Figure~\ref{noviTK} illustrates this  tiling technique  for $TG^{(0)}_m(n)$, $TG^{(2)}_m(n)$, $KB^{(0)}_m(n)$ and  $KB^{(4)}_m(n)$.

\begin{figure}[htb]
\begin{center}
\includegraphics[width=3.5in]{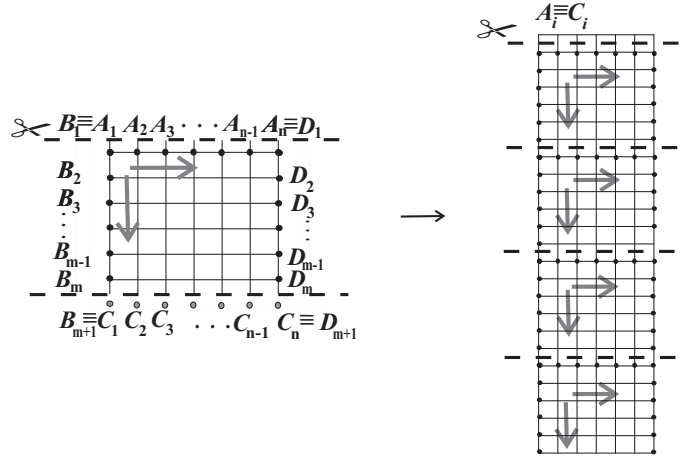}
\\ \ \vspace*{-18pt}
\end{center}
\caption{
The technique of obtaining rolling imprints using  flat representation for $TnC_m(n)$;}
 \label{maliTankia}
\end{figure}

The picture of the infinite grid graph  obtained from copies of FR (see  Figure~\ref{Flat}) is called  {\bf  Rolling Imprints (RI)}.
(The concept of  rolling imprints was introduced in \cite{BKDP1}.)
Observe that  the bold lines in RI  determine a 2-factor of the  infinity grid graph
for  all three cases.

We use matrix labeling. Thus,  each vertex of $G \equiv G_{m}(n)$ is represented with an ordered pair \ $(i,j) \in  \{1,2, \ldots, m \} \times \{1,2, \ldots, n \}$ \  where \ $i$ \ represents
the ordinal number of the row viewed from  up to down in FR, while \ $j$
\ represents the ordinal number of the column in FR, viewed from left to
right (the vertices $A_1, A_2, \ldots A_n$  belong to the first row, and the vertices  $B_1, B_2, \ldots B_m$ belong to the first column).
In  Figure~\ref{maliTankib}  the grid  graph $TnC_{3}(2) = C_3 \times P_2$ and its  all  possible 2-factors are shown.
  With the exception of  the last case all other 2-factors   are connected, i.e.  Hamiltonian cycles.

 We associate each  2-factor of the observed graph $G$ with the so-called {\bf Code
  matrix } \ $[ \alpha_{i,j} ]_{m\times n }$ \
   where  $\alpha_{i,j}$ is the one of six possible labels shown in Figure~\ref{cvkod1}
   (called  {\bf alpha-letters}) that is corresponding to the vertex \ $(i,j)$ \
($ 1 \leq i \leq m $, \ $ 1 \leq j \leq n $).
By reading the columns of these code matrices, from up to down,
we obtain so-called  {\bf alpha words}.
For instance, three alpha words: $facd$, $cedb$ and $dfac$
(following column order) determine 2-factors   shown in  Figure~\ref{Flat}.

In Section~2,   we give  a characterization of these code matrices, hence 2-factors,
for each  of the grid graphs from the title.
 Using this characterization and  the  transfer matrix method we propose   algorithms for obtaining the numbers of 2-factors, labeled by $f^{TnC}_m(n)$, $f^{TG}_{m,p}(n)$  and  $f^{KB}_{m,p}(n)$, where the integer $p$ refers to a type of the grid graph. Additionally, we give some properties of the transfer digraph ${\cal D}_m$ whose adjacency matrix is the transfer matrix used in these  algorithms. Thus, we prove that the number of vertices of  ${\cal D}_m$ is $3^m + (-1)^m$, that  ${\cal D}_m$ is disconnected digraph with strongly connected components.

 In Section~3, we reduce   ${\cal D}_m$ to the transfer digraph ${\cal D}^{*}_m$ of order $2^m$.  Thus, we improve the algorithms  for  obtaining $f^{TnC}_m(n)$, $f^{TG}_{m,p}(n)$  and  $f^{KB}_{m,p}(n)$. In case of thin grid cylinder we use just
 one of  the components  of ${\cal D}^{*}_m$, labeled by ${\cal N}^{*}_m$ which can be additionally reduced resulting in the digraph ${\cal N}^{**}_m$.
 We prove that ${\cal D}^{*}_m$ is also disconnected digraph with strongly connected components, its adjacency matrix is symmetric and that the number of its edges is $3^m + (-1)^m$.

  The numerical date about   ${\cal D}^*_m$ we gathered for $m \leq 18$
with the implementation of the  algorithms described in Section 3 reveal some interesting properties of ${\cal D}^{*}_m$
concerning  the order of  components of  ${\cal D}^*_m$ and the generating functions for the considered sequences $f^{TnC}_m(n)$, $f^{TG}_{m,p}(n)$  and  $f^{KB}_{m,p}(n)$. In Section~4, we   pose   these assertions and prove some of them  for an arbitrary $m \geq 2$.
Namely, some of the  functions  $f^{TG}_{m,p}(n)$  coincide  for different values of $p$. We prove the related assertion.
 In case of Klein bottle, such matchings are more numerous. For example, when $m$ is even the number  $f^{KB}_{m,p}(n)$
is independent of the value of  $p$. We propose a  conjecture concerning these observations.

\begin{figure}[htb]
\begin{center}
\includegraphics[width=3.4in]{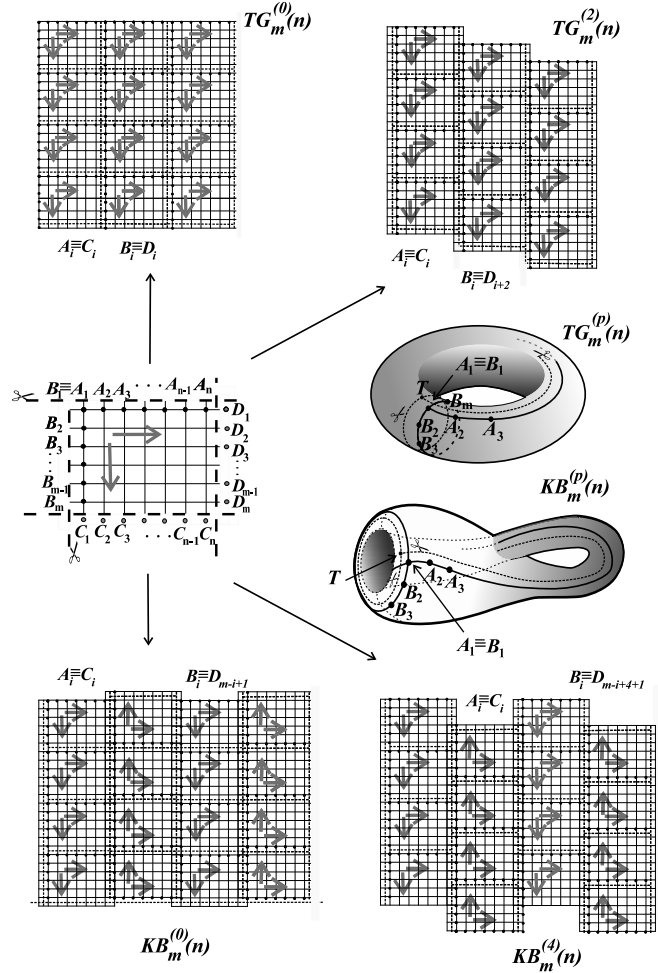}
%\\ \ \vspace*{-18pt}
\end{center}
\caption{The technique of obtaining rolling imprints using  flat representation for $TG^{(p)}_m(n)$ and $KB^{(p)}_m(n)$.}
\label{noviTK}
\end{figure}

 \begin{figure}[htb]
\begin{center}
\includegraphics[width=8in]{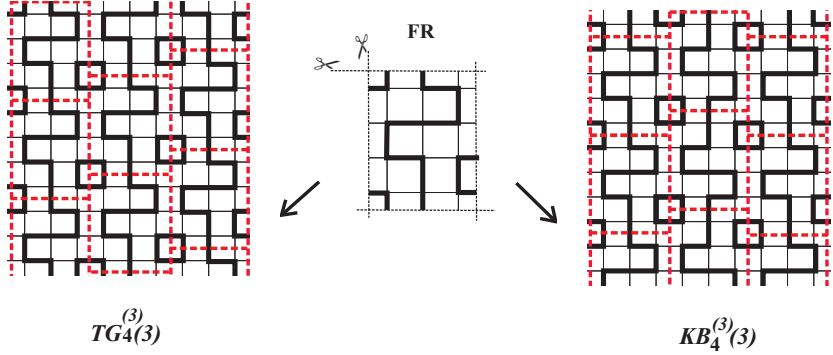}
\\ \ \vspace*{-18pt}
\end{center}
\caption{RI of a 2-factor of $TG^{(3)}_4(3)$ (left) and RI of a 2-factor of $KB^{(3)}_4(3)$ (right) with their common flat representation (in the middle).}
\label{Flat}
\end{figure}

Computational results we gathered for $m \leq 18$  are partially given in Section~4, the rest of them in  Appendix (Section~5)

% =========
% Section 2
% =========

\vspace{1.5cc}
\begin{center}
{\bf  2.  Enumeration of 2-factors using  ${\cal D}_m$}
\end{center}
\label{sec:intro}

\vspace*{4mm}

For each alpha-letter $\alpha$, denote by $\overline{\alpha}$  ($\alpha'$) the alpha-letter
of the situation from  Figure~\ref{cvkod1}  obtained
 by reflecting  the situation  of $\alpha$ over  horizontal (vertical) line. Precisely,
 $\overline{\alpha}$ and $\alpha'$ are images of $\alpha$ under mapping

 $$ \left( \begin{array}{cccccc}   a \; & \;  b  \; & \;  c \; & \;  d \; & \;  e \; & \;  f \\   c \; & \;  b  \; & \;  a \; & \;  f \; & \;  e \; & \;  d  \end{array} \right) \mbox{  and  }
 \left( \begin{array}{cccccc}   a \; & \;  b  \; & \;  c \; & \;  d \; & \;  e \; & \;  f \\   d \; & \;  b  \; & \;  f \; & \;  a \; & \;  e \; & \;  c  \end{array} \right), \mbox{respectively.} $$

\begin{figure}[htb]
\begin{center}
\includegraphics[width=4in]{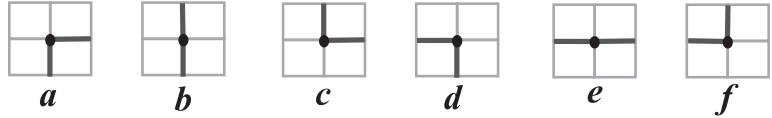}
\\ \ \vspace*{-18pt}
\end{center}
\caption{Six possible
situations in any vertex for given 2-factor }
\label{cvkod1}
\end{figure}

 The digraphs \ ${\cal D}_{lr}$ \ and \ ${\cal D}_{ud}$,  shown in  Figure~\ref{cvkod2}, \
describe the possibility that two alpha letters are  corresponding to two  adjacent vertices of $G$ in a  2-factor.
 The digraph \ ${\cal D}_{ud}$ is responsible for building alpha-words, and the other one  for the adjacency of columns in code matrices in the following sense:

 \begin{figure}[htb]
 \begin{center}
\includegraphics[width=3.5in]{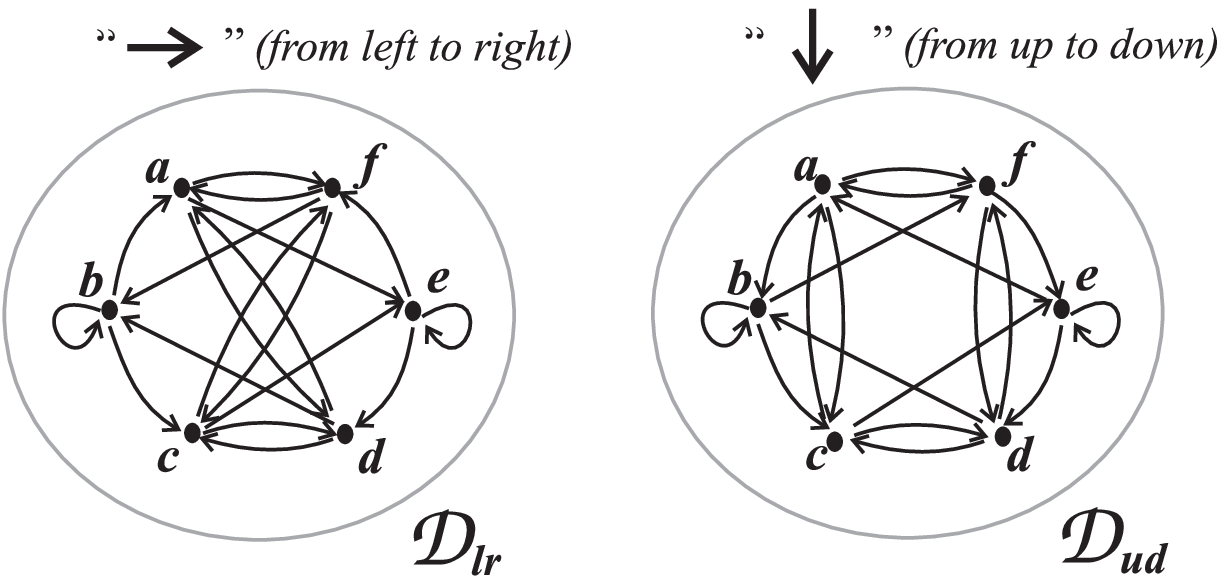}
\caption{Digraphs ${\cal D}_{ud}$ and  ${\cal D}_{lr}$}
\label{cvkod2}
\end{center}
\end{figure}

\begin{lem}  \label{lem:1} (Characterisation of 2-factors)

The code matrix $  [ \alpha_{i,j} ]_{m\times n }$ for any 2-factor of   $G \in  \{TnC_{m}(n), TG^{(p)}_{m}(n), KB^{(p)}_{m}(n) \}$ satisfies  the following properties ($ \alpha_{m+1,j} \stackrel{\rm def}{=}  \alpha_{1,j}$):
\begin{enumerate}
\item \textbf{Column condition:} \ For every fixed \ $j$  \ ($1 \leq j \leq n$),

 the ordered pairs \ $ (\alpha_{i,j}, \alpha_{i+1,j})$, \ where \ $1 \leq i
\leq m$, \ must be arcs in the digraph \ ${\cal D}_{ud}$.

\item \textbf{Adjacency of column condition:} \  For every  fixed  $j$, where  \ $1 \leq j \leq n-1 $,
 the ordered pairs \ $ (\alpha_{i,j}, \alpha_{i,j+1})$, \ where \ $1
\leq i \leq m$, \ must be arcs in the digraph \ ${\cal D}_{lr}$.

\item \textbf{First and Last Column conditions:}
\begin{enumerate}
 \item
If $G= TnC_{m}(n)$, then
the alpha-word of the first  column consists of the letters from the
set \ $\{ a, b, c \}$
 and of  the last column of the letters
  from the set \ $\{  b, d, f \}$.

\item
If $G= TG^{(p)}_{m}(n)$,  then
 the ordered pairs \ $ (\alpha_{i+p,n}, \alpha_{i,1})$, \ where \ $1
\leq i \leq m$, \ must be arcs in the digraph \ ${\cal D}_{lr}$.

\item
If $G= KB^{(p)}_{m}(n)$, then
 the ordered pairs \ $ (\alpha_{m+p+1-i,n}, \overline{\alpha}_{i,1})$, \ where \ $1
\leq i \leq m $, \ must be arcs in the digraph \ ${\cal D}_{lr}$.
 \end{enumerate}
\end{enumerate}
The converse, for every matrix $[\alpha_{i,j}]_{m \times n }$ with entries from $\{a,b,c,d,e,f\}$
that satisfies conditions 1--3 there is  a  unique 2-factor  on the grid graph $G$.
\end{lem}

\begin{proof}
The properties 1-3  can  be easily proved
by  checking  compatibility for   all the possible  edge arrangements for adjacent vertices of $G$ in the considered 2-factor.
Vice versa, the validity of the conditions 1-3 for the  matrix $[\alpha_{i,j}]_{m \times n }$ insures
that the obtained  subgraph of  $G$ is  a  unique spanning 2-regular graph, i.e.  2-factor. $\Box$
\end{proof}

\begin{df}
For any alpha-word $\alpha \equiv \alpha_{1}\alpha_{2} \ldots \alpha_{m-1}\alpha_{m}$,
the alpha words $\overline{\alpha} \equiv \overline{\alpha}_{m}\overline{\alpha}_{m-1} \ldots $ $ \overline{\alpha}_{2}\overline{\alpha}_{1}$ and $\alpha' \equiv \alpha'_{1}\alpha'_{2} \ldots\alpha'_{m-1}\alpha'_{m}$ are called {\bf horizontal } and {\bf vertical conversion}  of $\alpha$, respectively.
The circular word $\rho(\alpha) \equiv
\alpha_{2}\alpha_{3} \ldots \alpha_{m-1}\alpha_{m}\alpha_{1}$  is called
{\bf rotation} of $\alpha$.
\end{df}

Clearly, $\overline{ \overline{\alpha}} = \alpha$, $ (\alpha')' = \alpha$,
$ \rho^{p}(\alpha) \equiv \alpha_{p+1}\alpha_{p+2} \ldots \alpha_{m}\alpha_{1}\alpha_{2} \ldots \alpha_{p}$ $(1 \leq p \leq m-1)$ and
$\rho^{0}(\alpha) \equiv \alpha$. Note that
$ \rho^{p}(\overline{\alpha}) = \overline{\rho^{m-p}(\alpha)} \mbox{ for } p =0,1,  \ldots ,m-1.$

 Now, we can create  for each integer
  \ $ m $ \ $ (m \in N) $  a  digraph  \ $ {\cal D}_{m}
\stackrel{\rm def}{=} (V({\cal D}_{m}), E({\cal D}_{m}))$  \ (common  for all three types of grid graphs) in the following way:
  \ the set of vertices  \ $ V({\cal D}_{m}) $  \ consists of all possible circular words $\alpha_{1,j}\alpha_{2,j} \ldots \alpha_{m,j}$ over  alphabet $ \{ a,b,c,d,e,f \}$ (alpha-words) which fulfill   Condition 1 (Column condition)  from Lemma~\ref{lem:1};  an arc joins
   \ $ v = \alpha_{1,j}\alpha_{2,j} \ldots \alpha_{m,j}$ \ to  \ $ u= \alpha_{1,j+1}\alpha_{2,j+1} \ldots \alpha_{m,j+1}$, i.e. $(v,u) \in E({\cal D}_{m})$, \
or \ $ v \rightarrow u $ \ \  iff
  Condition 2 (Adjacency of column condition) from Lemma~\ref{lem:1} is satisfied  for the  vertices  \  $ v$ \
and  $ u $.

The subset of $V({\cal D}_{m})$ which consists of all the  possible first (last) columns in code matrices $[\alpha_{i,j}]_{m \times n}$  for $TnC_{m}(n)$
(Condition 3a, Lemma~\ref{lem:1}) is denoted by ${\cal F}_m$  (${\cal L}_m$).
From Condition 1, the first column of $[\alpha_{i,j}]_{m \times n}$ is
an alpha-word  from  \ $\{ a, b, c \}^m$. Similarly, the  last  column of $[\alpha_{i,j}]_{m \times n}$  is an alpha-word from  \ $\{  b, d, f \}^m$.
Consequently, we have

\begin{lem} \label{lem:2}
 The cardinalities of $ {\cal F}_{m}  $ and ${\cal L}_{m} $ $(m \in N)$  are equal to the $m$-th  member of the  Lucas sequence $L_m$ ($ L_1=1, L_2=3; L_k = L_{k-1} + L_{k-2}$, where $k \geq 3$).
\end{lem}
\begin{proof}
It can be carried out  analogically as it has been  done in linear case (see Lemma 3.1 in  \cite{JOK1}).
Namely,
  characteristic polynomial of the adjacency matrix of  the subdigraph of ${\cal D}_{ud}$ induced by the set $\{ a,b,c \}$ or $\{ b,d,f \}$
is $\lambda (\lambda^2 - \lambda -1)$. The polynomial in brackets   determines the recurrence relation for the numbers of all oriented closed walks  of length $m$ in the considered subdigraph
starting and ending with  a same  vertex. Consequently, $\mid {\cal F}_{m}\mid = \mid {\cal F}_{m-1}\mid + \mid {\cal F}_{m-2}\mid $ and $\mid {\cal L}_{m}\mid = \mid {\cal L}_{m-1}\mid + \mid {\cal L}_{m-2}\mid $, for $m \geq 3$. \\
Since   $ \mid {\cal F}_{1} \mid =\mid \{ b \} \mid = \mid {\cal L}_{1} \mid  =1= L_1$ and  $ \mid {\cal F}_{2} \mid  =  \mid \{
bb, ac, ca \} \mid =  \mid {\cal L}_{2} \mid  = \mid \{ bb, df, fd \} \mid = 3 = L_2$,
 we conclude that $ \mid {\cal F}_{m} \mid  =  \mid {\cal L}_{m} \mid  = L_m$, for all $m \in N$. $\Box$
\end{proof}

\begin{df}
A  square symmetric binary  matrix ${\cal H}_m = [h_{ij}]$ of order $ \mid V({\cal D}_{m}) \mid $ such that $h_{i,j} =1$ iff the $i$-th and $j$-th vertex of the digraph
 ${\cal D}_{m}$ can be obtained  from each other by horizontal conversion is called the {\bf H-conversion  matrix} for ${\cal D}_{m}$.
A permutation binary  matrix ${\cal R}_m = [r_{ij}]$ of order $ \mid V({\cal D}_{m}) \mid $ such that $r_{i,j} =1$ iff $ v_j = \rho(v_i)$ where $ v_i$ and $ v_j$
are the $i$-th and $j$-th vertex of the digraph  ${\cal D}_{m}$ is called the {\bf rotation matrix} for ${\cal D}_{m}$.
\end{df}

\begin{lem}  \label{lem:3}
If  $f_m^{TnC}(n)$,  $f_{m,p}^{TG}(n)$ and  $f_{m,p}^{KB}(n)$ ($ m \geq 2$) denote the number of 2-factors of $TnC_{m}(n)$, $TG^{(p)}_{m}(n)$ and  $KB^{(p)}_{m}(n)$, respectively, then
\be \label{glavna1}
f_m^{TnC}(n)
=   \sum_{ \begin{array}{c}v_i \in {\cal F}_m ; \\  v_j \in {\cal L}_m \end{array} } a_{i,j}^{(n-1)} =
    \sum_{\begin{array}{c}v_i \in {\cal F}_m \end{array}} a_{i,i}^{(n)} =a_{1,1}^{(n+1)}  \ee

\be \label{glavna2}
f_{m,p}^{TG}(n) =  tr(  {\cal T}_m^n    \cdot  {\cal R}^{p}_m) = tr({\cal R}^{p}_m \cdot {\cal T}_m^n) = \sum_{\begin{array}{c} v_i, v_j \in   V({\cal D}_m); \\
 v_i= \rho^{p}(v_j) \end{array} } a_{i,j}^{(n)}\ee

\be \label{glavna2}
f_{m,p}^{KB}(n) =
 tr( {\cal T}_m^n   \cdot {\cal R}^{p}_m \cdot {\cal H}_m) = tr( {\cal R}^{p}_m \cdot {\cal H}_m  \cdot  {\cal T}_m^n) = \sum_{\begin{array}{c} v_i, v_j  \in   V({\cal D}_m); \\
\overline{v_i}= \rho^{p}(v_{j}) \end{array} } a_{i,j}^{(n)}\ee
   where
    ${\cal T}_m = [a_{ij}]$, ${\cal H}_m = [h_{ij}]$ and  ${\cal R}_m = [r_{ij}]$ are
     the  adjacency (transfer), H-conversion and rotation matrix  for  the digraph  \ $ {\cal D}_{m} $, respectively   and $v_1=b^{m}$.
\end{lem}
\begin{proof}
The number  of all the possible code matrices  is equal to the number of all oriented walks of the length \ $ n-1 $ in the digraph  \ $ {\cal D}_{m} $ for which the initial and  final vertex satisfy The First and  Last Column conditions from Lemma~\ref{lem:1}.

In case of $ TnC_m(n)$, this number is equal to  the sum of  $(i,j)$-entries  of $(n-1)$-th degree of ${\cal T}_m $ where $v_i \in {\cal F}_m $ and $v_j\in {\cal L}_m $. Since  every vertex from   ${\cal F}_m$ is  direct successor of any   vertex from ${\cal L}_m$, we have  that
$\ds  f_m^{TnC}(n) =  \sum_{v_i \in {\cal F}_m } a_{i,i}^{(n)}$. As the vertex $b^{m}$ belongs to both  $ {\cal F}_m $ and ${\cal L}_m$, further we obtain  that the requested number is equal to $ a_{1,1}^{(n+1)} $.

In case of $  TG^{(p)}_m(n)$
the number  of all the possible code matrices (i.e.  rolling imprints) with  the first  column $v_i$ where $v_i  \in V({\cal D}_m)$ is equal to
the number  of all oriented walks of  length  \ $ n-1$ \ in the digraph \ $ {\cal D}_{m} $ \
 with  initial vertex    $v_i $ and one of predecessors of $v_j= \rho^{m-p}(v_i)$ (equivalent to $v_i= \rho^{p}(v_j)$) as  a  final vertex (condition 3(b) from Lemma~\ref{lem:1}) or all oriented walks  of  length $n$ which start with   $v_i$ and end with $v_j= \rho^{m-p}(v_i)$ (the number $a_{i,j}^{(n)}$).
Consequently, the total number  2-factors of $  TG^{(p)}_m(n)$   equals
 the trace of the matrix ${\cal T}_m^n  \cdot  {\cal R}^{p}_m$.
The trace of a product of square matrices does  not depend on the order of multiplication, so
 $ \ds f_{m,p}^{TG}(n) =    tr({\cal T}_m^n  \cdot  {\cal R}^{p}_m) =tr({\cal R}^{p}_m \cdot {\cal T}_m^n)$.

Case $ KB^{(p)}_m(n)$ is similar to the previous  one. Namely,  the enumeration  of all oriented walks of  length
\ $ n-1$ \ in $ {\cal D}_{m} $
starting from   $v_i \in V({\cal D}_m)$   and ending at one of predecessors of $v_j= \rho^{m-p}(\overline{v_{i}})= \overline{ \rho^{p}(v_{i})}$ (condition 3(c) from Lemma~\ref{lem:1})
comes down  to finding the trace of  the matrix $  {\cal R}^{p}_m \cdot {\cal H}_m  \cdot {\cal T}_m^n$ or $ {\cal T}_m^n    \cdot {\cal R}^{p}_m  \cdot {\cal H}_m $. $\Box$
\end{proof}

 \begin{Proposition} \label{Primedba2}
(Properties of vertical conversion)
\\
For every $v,w \in V({\cal D}_m) ,$

  \be  \label{c1}  v  \rightarrow v' ; \ee \noindent
 \be  \label{c2} \mbox{ If } v \rightarrow w, \; \; \mbox{ \ \  then  \ \  } w' \rightarrow v' .  \ee
\end{Proposition}

\begin{lem}  \label{lem:4}
  The number of vertices in  ${\cal D}_m$  is $\mid V({\cal D}_m) \mid = \ds 3^{m} + (-1)^{m}.$
\end{lem}
\begin{proof}
Let $d_m$ $(m \in N)$ be
the number of all  circular  words  of length $m$ over
alphabet $\{a,b,c,d,e,f \}$ that satisfy Column condition of  Lemma~\ref{lem:1}. It is  equal to the number of all  closed
oriented walks (the first vertex is labeled) of length $m$ in ${\cal D}_{ud}$, i.e. the trace of the $m$th degree of adjacency matrix of ${\cal D}_{ud}$.
Since the characteristic polynomial of this matrix is $P(\lambda) =\lambda^4(1+\lambda)(\lambda -3)$, we have
 $d_m = 2d_{m-1}+3d_{m-2}$, with initial conditions   $d_1=2$ (words from $\{ e, b \}$), $d_2 = 10$ (words from $\{ ac, af, bb, ca, cd, dc, df, ee, fa, fd \}$).
 By solving the above  recurrence relation, we obtain  $\ds d_m    = (-1)^{m}+3^{m}$, ($m \in N$),
 which implies  $ \mid V({\cal D}_{m}) \mid \leq \ds (-1)^{m}+3^{m}$.
 In order to prove that every  circular  word  $w$ accounted in $d_m$ is a vertex of ${\cal D}_{m}$, note that $w \rightarrow w'$ and $w' \rightarrow w'' = w$ (see (\ref{c1})).
This means  that the word $w$ appears as a column  in the    code matrix for $TG^{(0)}_m(2)$. $\Box$
\end{proof}

The following definition and lemma are  taken from \cite{JOK1}  with the only  difference that now they refer to circular words.

\begin{df}
The  \emph{ outlet
(inlet)
word}   of the word  $\alpha \equiv \alpha_1 \alpha_2 \ldots\alpha_m \in V({\cal D}_{m})$ is the binary
  word \ $o(\alpha) \equiv o_1o_2 \ldots o_m$  ($i(\alpha) \equiv i_1i_2 \ldots i_m$) \ where \
  $$ \ds o_j
\stackrel{\rm def}{=} \left \{
\begin{array}{cc}{}
0, & \; \; if \; \; \alpha_j \in \{ b, d, f \} \\
1, & \; \; if \; \; \alpha_j \in \{ a, c, e \}
\end{array}
 \right.  \; \;  and  \; \;  i_j
\stackrel{\rm def}{=} \left \{
\begin{array}{cc}{}
0, & \; \; if \; \; \alpha_j \in \{a,b,c \} \\
1, & \; \; if \; \; \alpha_j \in \{ d,e,f \}
\end{array}
 \right. ,  \; \; \; 1 \leq j \leq m .
$$
\end{df}
Note that  $o_j = 1$ ($i_j = 1$) iff  the situation shown in Figure~\ref{cvkod1} matched  to  $\alpha_j$ has an edge ``on the right''  (``on the left'').

\begin{lem}  \label{lem:5}
Digraph ${\cal D}_{m}$ for $m \geq 2$ is disconnected. Each of its components is a strongly connected digraph.
\end{lem}
\begin{proof}
This proof is analogical  as in the linear case. Namely, if two words $w,v $ are adjacent vertices in ${\cal D}_{m}$, then the numbers of 1s in their outlet words must have
 the same  parity. Consequently, there exist  two vertices  of ${\cal D}_{m}$ (for example the words $ab^{m-2}f$ and $b^{m}$) belonging to different components.

Note that if there exists a directed walk $w_0w_1 \ldots  w_{k-1} w_k$ (of length  $k \in N$), then the existence of the directed walk $w_k w'_k  w'_{k-1} \ldots$   $  w'_1 w'_0 w_0 $ (see  Proposition~\ref{Primedba2}) confirms  the second part  of this lemma. $\Box$
 \end{proof}

Let ${\cal D}_m = {\cal A}_m \cup {\cal B}_m$, $m \geq 1$
where ${\cal A}_m$ is the component of ${\cal D}_{m}$  which contains the vertex  $e^{m}$ ($m \geq 1$).
Also, we introduce ${\cal N}_m$ as the component of   ${\cal D}_{m}$ which contains the vertex $b^{m}$ ($m \geq 1$)
(this component is sufficient for enumerating 2-factors  for the thin grid cylinder  $TnC_m(n)$).

\begin{lem}  \label{lem:6}
 ${\cal A}_m \equiv {\cal N}_{m}$ iff   $m $ is even.
 \end{lem}
\begin{proof}
If $m$ is even, then $e^{m}  $  and $b^{m}$ are connected (belong to the same component) because  $e^{m}  \rightarrow (df)^{m/2}    \rightarrow b^{m}$, i.e.
 ${\cal A}_m \equiv {\cal N}_{m}$.  If $m$ is odd,
 the numbers of $1$s in outlet words of   vertices  in $V({\cal N}_m)$ and  $V({\cal A}_m)$  are of different parities and  therefore
 ${\cal A}_m \not\equiv {\cal N}_{m}$. $\Box$
 \end{proof}

\vspace{1.5cc}
\begin{center}
{\bf  3.  Enumeration of 2-factors using  ${\cal D}^{*}_m$ and ${\cal N}^{**}_m$}
\end{center}
\label{sec:intro}

\vspace*{4mm}

If two vertices  from \ ${\cal D}_m$ \ have
the same outlet word, then they  have the same set of direct successors.
We group  all vertices from  \ $ V({\cal D}_{m})$ \  \ having the same corresponding outlet word and replace
 them with just one vertex,   labeled by their common  outlet word.
Next, we substitute all the arcs in   $ {\cal D}_{m}$ starting with vertices with the same outlet word and ending
in  same vertex with only one arc.
In this way, we  reduce the  transfer digraph  ${\cal D}_m$ onto
$ {\cal D}^*_{m}$ with  adjacency (transfer) matrix  ${\cal T}^*_m$.

 For example, vertices $abbc, abfe, afdc, cdfa, edbc, edfe \in  V({\cal D}_{4})$ with the common outlet word $1001$ are all replaced with   $1001 \in  V({\cal D}^*_{4})$.
All edges starting from these vertices and ending in $dcaf$ are changed by the one arc: $1001 \rightarrow 0110 $.
But, the same is valid for edges  starting from these vertices and ending in $facd$.  This results in
 the existence of two parallel  edges ($1001 \rightarrow  0110$) in ${\cal D}^*_{4}$.

 Note that in linear case  \cite{JOK1} the digraph  ${\cal D}^*_m$ is simple, i.e. without multiple arcs, because
 two different  vertices from  \ $ V({\cal D}_{m})$  with the same outlet word can not  have the same direct predecessor.
The situation  with circular words is different.
Namely, for two binary words $u$ and $w$ of length $m$ there exist at most two  vertices from $V({\cal D}_m)$ whose inlet and outlet word are $u$ and $w$, respectively.
Precisely, for given horizontal edges of a 2-factor  corresponding to a column of vertices of the considered grid graph, if such 2-factor exists at all, then  there exist at most two possible choices  of vertical edges for that column and that 2-factor.
 This implies that the  entries of ${\cal T}^*_m$ are from the set $\{ 0,1,2 \}$.

 \begin{thm}  \label{thm:TkC3}
Adjacency matrix  ${\cal T}^*_m$ of the digraph ${\cal D}^*_m$ is a symmetric  matrix, i.e. ${\cal T}^*_m = ({\cal T}^*_m)^{T} $.
\end{thm}
\begin{proof}
Suppose that  $v \rightarrow w $, where $v,w \in V({\cal D}^*_m)$.
Then  there exist vertices  $x, y \in V({\cal D}_m)$  such that
$x \rightarrow y$, where $ o(x) = v$ and $ o(y) = w$.
Since $v= o(x) =i(y) =  o(y')$ and  $ y \rightarrow y' $, we have that
$ w \rightarrow v$. $\Box$   \end{proof}

\begin{thm}  \label{thm:SC}
Digraph ${\cal D}^*_{m}$ is disconnected. Each of its components is a strongly connected digraph.
\end{thm}
\begin{proof}
Note that all the  vertices   of ${\cal D}_m$ which are  glued in one vertex   belong to the same component, so such  gluing does not reduce the number of components.
Using this the statement  of the theorem follows  directly  from Lemma~\ref{lem:5}. $\Box$
 \end{proof}

The component of ${\cal D}^*_{m}$ which contains the circular word  $0^m \equiv 00 \ldots 0$ is denoted by ${\cal N}^*_m$. The component of ${\cal D}^*_{m}$ containing the circular word $1^m$ is labeled by ${\cal A}^*_m$ and the union of the remaining components by ${\cal B}^*_m$.
Their adjacency matrices  for $m=2$,   $m=3$ and  $m=4$ are  exposed below.
 \begin{center}
$ \ds  {\cal T}^*_{{\cal A}_2} = {\cal T}^*_{{\cal N}_2} = \left[
\begin{array}{cc}
 1 & 2 \\
 2 & 1
\end{array} \right]
$  $  \ds \begin{array}{c}
v_1= 00 \\
 v_2= 11
\end{array}$ and $ \ds   {\cal T}^*_{{\cal B}_2} = \left[
\begin{array}{cc}
0 & 2 \\
 2 & 0
\end{array} \right]
$  $  \ds \begin{array}{c}
v_1=10 \\
 v_2=01
\end{array}$ ;
\end{center}

 \begin{center}
$ \ds {\cal T}^*_{{\cal N}_3} = {\cal T}^*_{{\cal B}_3} = \left[
\begin{array}{cccccc}
 1 & 1 &  1 & 1 \\
 1 & 0 &  1 & 1 \\
 1 & 1 &  0 & 1 \\
 1 & 1 &  1 & 0
\end{array} \right]
$ $  \ds \begin{array}{c}
v_1=000  \\
 v_2=101 \\
 v_3 = 011 \\
 v_4 = 110
\end{array}$  and
$ \ds  {\cal T}^*_{{\cal A}_3} =  \left[
\begin{array}{cccccc}
 1 & 1 &  1 & 1 \\
 1 & 0 &  1 & 1 \\
 1 & 1 &  0 & 1 \\
 1 & 1 &  1 & 0
\end{array} \right]
$  $  \ds \begin{array}{c}
v_1=111 \\
 v_2=001 \\
 v_3 =100  \\
 v_4 = 010
\end{array}$;
\end{center}

\begin{center}
$ \ds  {\cal T}^*_{{\cal A}_4} = {\cal T}^*_{{\cal N}_4} =\left[
\begin{array}{cccccc}
 1 & 2 &  1 & 1 & 1 & 1 \\
 2 & 1 &  1 & 1 & 1 & 1 \\
 1 & 1 &  0 & 1 & 2 & 1 \\
 1 & 1 &  1 & 0 & 1 & 2 \\
 1 & 1 &  2 & 1 & 0 & 1 \\
 1 & 1 &  1 & 2 & 1 & 0
\end{array} \right]
$  $  \ds \begin{array}{c}
v_1= 0000 \\
 v_2= 1111 \\
 v_3 = 0011 \\
 v_4 = 0110 \\
 v_5 = 1100 \\
 v_6 = 1001
\end{array}; $ \\ $
 \ds {\cal T}^*_{{\cal B}^{(1)}_4} = \left[
\begin{array}{cccccccc}

 0 &  0 & 0 & 0 & 1 & 2 & 1  &   1     \\
0 &  0 & 0 & 0 & 1 & 1 & 1   & 2   \\
 0 &  0 & 0 & 0 & 1 & 1 & 2  &  1       \\
 0 &  0 & 0 & 0 & 2 & 1 & 1  &  1       \\
 1 &  1 & 1 & 2 & 0 & 0 & 0  &  0       \\
 2 &  1 & 1 & 1 & 0 & 0 & 0  &  0       \\
 1 &  1 & 2 & 1 & 0 & 0 & 0  &  0   \\
 1 &  2 & 1 & 1 & 0 & 0 & 0  &  0   \\
\end{array} \right]
$ $  \ds \begin{array}{c}
 v_1=0001 \\
 v_2 = 0111 \\
 v_3 = 1101 \\
 v_4 = 0100 \\
 v_5 = 1011 \\
  v_6 = 1110 \\
 v_7 = 0010 \\
 v_8=1000, \end{array}$ and
 $ \ds  {\cal T}^*_{{\cal B}^{(2)}_4} = \left[
\begin{array}{cc}
0 & 2 \\
 2 & 0
\end{array} \right]
$   $  \ds \begin{array}{c}
v_1=0101 \\
 v_2=1010
\end{array}$.
\end{center}

\begin{thm}  \label{thm:TkC5}
The number of vertices in digraph ${\cal D}^*_m$  is
$ \ds  \mid V({\cal D}^*_m) \mid =  2^{m} $.
 \end{thm}
\begin{proof}
We should  prove that there exists
 at least one  vertex  $x \in V({\cal D}_m)$ with $o(x) = v$,
 where   $v$ is an arbitrary circular   binary word of length  $m$.
   In \cite{JOK1} we proved that every binary (linear) m-word except the word $0(10)^{(m-1)/2}$ (one of so-called {\em queens})
  for odd $m$ is a vertex of  the  transfer digraph for linear case (thick  cylinder and Moebius strip  grid graphs).
   Note that this digraph  is the subdigraph of ${\cal D}_m^*$.
 So, it remains to find $x \in V({\cal D}_m)$ with $o(x) = 010101...010$ (for  case  $m$-odd).
  The  example $x= (fa)^{(m-1)/2}b $  confirms  the statement of this theorem. $\Box$
  \end{proof}

\begin{thm}  \label{thm:TkC6}
  The number of edges in  ${\cal D}^{*}_m$  is $\mid E({\cal D}^*_m) \mid = \ds 3^{m} + (-1)^{m}.$
\end{thm}
\begin{proof}
We establish
 the  bijection between the set $E({\cal D}^*_m)$ and $V({\cal D}_m) $ in the following way.

  Each single arc $vw$ ($v \rightarrow w$) from ${\cal D}^*_m$  is paired with exactly one vertex $x  \in V({\cal D}_m)$ whose
 all incoming arcs  are replaced with  the  arc $vw$.  Clearly,  $o(x) = w$, $(\forall z \in V({\cal D}_m))(z \rightarrow x  \Rightarrow o(z)=v)$ and
  there is no vertex  $y \in V({\cal D}_m)$ with $o(y) = w$ whose incoming  arcs have origin with outlet word $v$.

  Double arc $vw$ (two parallel arcs) in ${\cal D}^*_m$  is paired with exactly two  vertices  $x $ and $y$ from $ V({\cal D}_m)$
  where
 all incoming arcs of $x$  are replaced with one of the two parallel  arcs $vw$ and all incoming arcs of $y$   with the another one.  Now,  $o(x) = o(y) = w$, $(\forall z \in V({\cal D}_m))((z \rightarrow x  \Rightarrow o(z)=v) \wedge (z \rightarrow y  \Rightarrow o(z)=v))$ and
  $x$ and $y$ are the only two  vertices in ${\cal D}_m$ with these properties.
 The above described bijection  and
  Lemma~\ref{lem:4} prove the theorem. $\Box$
 \end{proof}

For a  binary word $v \equiv b_1b_2 \ldots b_{m-1}b_{m} \in \{ 0,1\}^m$, we introduce the label $\overline{v} \stackrel{\rm def}{=}  b_mb_{m-1}$ $ \ldots b_{2}b_{1}$ and
$\rho (v) \stackrel{\rm def}{=} b_2 \ldots b_{m-1} b_{m}b_1$.  Note that $o(\rho(x)) = \rho(o(x))$, for  any $ x \in V({\cal D}_{m})$.
 \begin{Proposition} \label{Primedba3} (Properties of horizontal conversion)
\\
 For  arbitrary  $ x, y \in V({\cal D}_{m})$,
 \be \label{prim3a}  \  o(\overline{x}) = \overline {o(x) }; \ee
 \be \label{prim3b} \mbox{  If } o(x) = o(y), \mbox{ then } o(\overline{x}) = o(\overline{y}).  \; \; \Box \ee
\end{Proposition}

Let ${\cal H}_m^* = [h_{ij}]$ be the square binary  matrix of order $ \mid V({\cal D}^*_{m}) \mid $ whose entry $h_{i,j} =1$ iff the $i$-th and $j$-th vertex of the digraph
 ${\cal D}^*_{m}$ satisfy $v_i = \overline{v_j}$. We call it {\bf H-conversion  matrix}  for   \ $ {\cal D}^*_{m} $.
  Note that the matrix ${\cal H}_m^*$ is a symmetric  one, i.e. $(\forall i,j)( v_i = \overline{v_j} \Leftrightarrow v_j = \overline{v_i})$.
Let ${\cal R}_m^* = [r_{ij}]$ be the  permutation  matrix of order $ \mid V({\cal D}^*_{m}) \mid $
   whose entry $r_{i,j} =1$ iff the $i$-th and $j$-th vertex of the digraph
 ${\cal D}^*_{m}$ satisfy $\rho(v_i) = v_j$. We call it {\bf rotation  matrix}  for   \ $ {\cal D}^*_{m} $.

The  process of enumeration of 2-factors can be  improved using the  new  transfer-matrix  ${\cal T}_m^*$.
\begin{thm}  \label{thm:f2}
\be \label{gl1} f_m^{TnC}(n)=  a_{1,1}^{(n)}\ee
\be \label{gl2} f_{m,p}^{TG}(n)=  tr(({\cal R}_m^*)^p \cdot ({\cal T}_m^*)^n) =
 tr( ({\cal T}_m^*)^n  \cdot   ({\cal R}_m^*)^p) =
 \ds \sum_{\begin{array}{c} v_i,v_j \in   V({\cal D}_m^*)\\
 v_i= \rho^{p}(v_j) \end{array}  } a_{i,j}^{(n)}\ee
 \be \label{gl3}
  f_{m,p}^{KB}(n)=
    tr( ({\cal R}_m^{*} )^p  \cdot  {\cal H}_m^{*}  \cdot  ({\cal T}_m^*)^n)
     =  tr(  ({\cal T}_m^*)^n  \cdot ({\cal R}_m^{*} )^p  \cdot {\cal H}_m^{*} )
     =\ds \sum_{\begin{array}{c} v_i,v_j \in   V({\cal D}_m^*)\\
 \overline{v_{i}}= \rho^{p}(v_j) \end{array}  } a_{i,j}^{(n)} , \ee
 where  ${\cal T}^*_m = [a_{ij}]$ is   the adjacency matrix of the digraph  \ $ {\cal D}^{*}_{m} $ and       $v_1 \equiv 00 \ldots 0$.
\end{thm}
\begin{proof}

For the observed digraph (${\cal D}_{m}$ or ${\cal D}^*_{m}$)  the number of oriented walks of length  $n$  which start with  vertex $x$  and finish with vertex $y$
is denoted by  ${\cal W}^{y}_{x}(n)$. Recall that
the number  ${\cal W}^{v_j}_{v_i}(n)$  where $v_i,v_j \in V({\cal D}^{*}_{m})$         is equal to
the $(i,j)$-entry  $ a_{i,j}^{(n)}$ of $n$-th degree of the matrix ${\cal T}_m^*$.
Note that
\begin{equation}\label{o2}
 {\cal W}^{v_j}_{v_i}(n) = \ds \sum_{\ds  \begin{array}{c}  y \in  V({\cal D}_{m}) \\ o(y) =v_j \end{array}} {\cal W}^{y}_{x} (n), \mbox{ where } x \in V({\cal D}_{m}) \mbox{ and } o(x) =v_i. \end{equation}
Using this and  Lemma~\ref{lem:3}, we have
 $$\ds  f_m^{TnC}(n)
=  {\cal W}^{b^{m}}_{b^{m}} (n+1) = \ds \sum_{\ds y \in   {\cal L}_m }  {\cal W}^{y}_{b^{m}} (n) = {\cal W}^{0^m}_{0^m} (n)=a_{1,1}^{(n)},
$$ which completes the case  $G = TnC$.

In order to obtain the expressions for the other two  cases observe  that
\begin{equation}\label{o1}
 {\cal W}^{y}_{x_1}(n) = {\cal W}^{y}_{x_2}(n), \mbox{ for any } x_1,x_2,y \in V({\cal D}_{m}) \mbox{ where } o(x_1) = o(x_2). \end{equation}

By using  (\ref{o2}) and  (\ref{o1}), for any integer $k$ and $v_j \in   V({\cal D}_m^*)$, we get
\begin{equation}\label{o3}
 {\cal W}_{\rho^{k}(v_j)}^{v_j}(n) = \ds \sum_{\ds  \begin{array}{c}  y \in  V({\cal D}_{m}) \\ o(y) = o(y_1) =  v_j \end{array}} {\cal W}^{y}_{\rho^{k}(y_1)} (n) = \ds \sum_{\ds  \begin{array}{c}  y \in  V({\cal D}_{m}) \\ o(y) = v_j \end{array}} {\cal W}^{y}_{\rho^{k}(y)} (n).
\end{equation}

 Applying  Lemma~\ref{lem:3} again and (\ref{o3})  to  $TG^{(p)}_m(n)$ we obtain

$$ \ds  f_{m,p}^{TG}(n) = \ds \sum_{\ds \begin{array}{c} x,y \in   V({\cal D}_m)  \\ x= \rho^{p}(y) \end{array}} {\cal W}^{y}_{x} (n) =
  \sum_{\ds v_j \in   V({\cal D}_m^{*}) } \sum_{\ds \begin{array}{c}  y \in  V({\cal D}_{m})  \\ o(y) =v_j  \end{array}}   {\cal W}^{y}_{ \rho^{p}(y)} (n) = $$
   $$   \sum_{v_j \in   V({\cal D}_m^{*})}   {\cal W}^{v_j}_{\rho^{p}(v_j) } (n) =
 \ds \sum_{\ds \begin{array}{c} v_i,v_j \in   V({\cal D}_m^*)\\ v_i = \rho^{p}(v_j) \end{array} } a_{i,j}^{(n)} =  tr(({\cal R}_m^*)^p \cdot ({\cal T}_m^*)^n)=  tr(({\cal T}_m^*)^n \cdot ({\cal R}_m^*)^p ). $$

For $KB^{(p)}_m(n)$ the procedure is similar.
From  (\ref{o2}), (\ref{o1}) and  Proposition~\ref{Primedba3}, for any integer $k$ and $v_j \in   V({\cal D}_m^*)$, we get
\begin{equation} \label{o4}
 {\cal W}_{\overline{\rho^{k}}(v_j)}^{v_j}(n) = \ds \sum_{\ds  \begin{array}{c}  y \in  V({\cal D}_{m}) \\ o(y) = o(y_1) =  v_j \end{array}} {\cal W}^{y}_{\overline{\rho^{k}}(y_1)} (n) = \ds \sum_{\ds  \begin{array}{c}  y \in  V({\cal D}_{m}) \\ o(y) = v_j \end{array}} {\cal W}^{y}_{\overline{\rho^{k}}(y)} (n).
\end{equation}

Now, Lemma~\ref{lem:3} and (\ref{o4})  yield that

$$ f_{m,p}^{KB}(n) = \ds \sum_{\ds \begin{array}{c} x,y \in   V({\cal D}_m)  \\ \overline{x}= \rho^{p}(y) \end{array}} {\cal W}^{y}_{x} (n) =
 \sum_{\ds y \in   V({\cal D}_m)} {\cal W}^{y}_{\overline{\rho^{p}}(y)} (n) =$$ $$
  \sum_{\ds v_j \in   V({\cal D}_m^{*}) } \sum_{\ds \begin{array}{c}  y \in  V({\cal D}_{m}) \\ o(y) =v_j   \end{array}}
    {\cal W}^{y}_{\overline{\rho^{p}}(y)} (n) =
  \sum_{\ds v_j \in   V({\cal D}_m^{*}) }   {\cal W}_{\overline{\rho^{p}}(v_j)}^{v_j}(n)
         = $$ $$
    \ds \sum_{\ds \begin{array}{c} v_i,v_j \in   V({\cal D}_m^*) \\ v_j = \rho^{m-p}(\overline{v_i})
    \end{array} }    a_{i,j}^{(n)} =  tr(({\cal R}_m^*)^p  \cdot {\cal H}_m^* \cdot ({\cal T}_m^*)^n)=
     tr(({\cal T}_m^*)^n \cdot ({\cal R}_m^*)^p  \cdot {\cal H}_m^*) \; .   \; \; \Box$$
\end{proof}

Further reduction  of the transfer  matrices is possible just  in case $G=TnC_m(n)$ using the following
\begin{thm}  \label{thm:simmetry}
If   $v \in V({\cal N}^*_{m})$,  then  $ \overline{v}  \in V({\cal N}^*_{m})$ and $ \rho^k(v)  \in V({\cal N}^*_{m}) \;  (0 \leq k \leq m-1)$.
\end{thm}
\begin{proof}
If  $o(x) = v$ ($x \in V({\cal N}_{m})$),  then   there exists an integer $n \geq 1$ for which ${\cal W}^{b^{m}}_x (n) \neq 0$.  By  using the  property of reflection symmetry ($\overline{b^{m}} = b^{m}$),
we obtain that ${\cal W}^{b^{m}}_{\overline{x}}  (n) = {\cal W}^{b^{m}}_x (n) \neq 0$, which implies $\overline{x} \in V({\cal N}_{m})$.
Since $\overline{v} = o(\overline{x})$, we conclude that  $ \overline{v}  \in V({\cal N}^*_{m})$.
\\ By  using the  property of rotational   symmetry (if $ x \rightarrow y$, then $ \rho(x) \rightarrow \rho(y)$) and $\rho(b^m) = b^m$, we have
$\rho(v)  \in V({\cal N}^*_{m})$, and therefore $\rho^k(v)  \in V({\cal N}^*_{m}) $ for arbitrary integer $k$. \ $\Box$
 \end{proof}

Now, we can glue  vertices $v$, $\overline{v}$, $\rho^k(v)$ and $\rho^k(\overline{v})$ ($1 \leq k \leq m-1$)  into one vertex for  $v \in V({\cal R}^*)$ resulting in a  new digraph ${\cal N}_m^{**}$.
During gluing  we retain arcs starting from just one of these vertices and delete the ones starting from another vertices.
Multiple  arcs appear  when glued vertices  have a  common direct predecessor.

\begin{Example}
For example, for  $m=4$ in the transition from
 ${\cal N}^*_{4}$ to  ${\cal N}^{**}_{4}$ the  vertices $v_3$, $v_4$, $v_5$ and $v_6$ are glued. The adjacency matrix of ${\cal N}^{**}_{4}$ is obtained by adding the $4$-th, $5$-th and $6$-th columns to the $3$-rd one, at first. Then, the added columns and corresponding rows are removed. The resulting matrix is
$ \ds  {\cal T}^{**}_{{\cal N}_4} =\left[
\begin{array}{ccc}
 1 & 2 &  4  \\
 2 & 1 &  4  \\
 1 & 1 &  4
\end{array} \right]
$  where  $
v_1= 0000,  \;
 v_2= 1111, \;
 v_3 = 0011.$
\end{Example}

The symmetry of $TnC_m(n)$ leads us to the conclusion that ${\cal W}^{0^m}_{\overline{v}}  (n) = {\cal W}^{0^m}_v(n)= {\cal W}^{0^m}_{\rho^k(v)}(n) = {\cal W}^{0^m}_{\rho^k(\overline{v})}(n)$ for any vertex $v \in V({\cal N}^*)$ ($1 \leq k \leq m-1$).
Consequently, the number ${\cal W}^{0^m}_{0^m}(n)$ remains the same  in both digraph ${\cal N}^*$ and ${\cal N}^{**}$, i.e.

\begin{thm}  \label{thm:glue} The number
$f_m^{TnC}(n)$ is equal to entry $a_{1,1}^{(n)}$ of the $n$-th degree of the adjacency matrix for  ${\cal N}^{**}$ where   $v_1 \equiv 0^m$.
\end{thm}

\vspace{1.5cc}
\begin{center}
{\bf  4.  Computational results}
\end{center}
\label{sec:intro}

\vspace*{4mm}

We wrote computer programs for the computation of the  adjacency
matrices of  these   digraphs ${\cal D}_m$,  \ ${\cal D}^*_m$, ${\cal N}_m$,  \ ${\cal N}^*_m$, \  ${\cal N}^{**}_m$ and initial members of required sequences
$f_m^{TnC}(n)$,  $f_{m,p}^{TG}(n)$ and  $f_{m,p}^{KB}(n)$ ($ m \geq 2$).  The data pertaining to the size of these digraphs are summarized in Table 1 and Table2.

\noindent
\hspace*{-1cm}
\begin{table}
\begin{center}
%\small{
\begin{tabular}{||c||r|r|r|r|r|r|r|r|r|r|r|r|r|||}
\hline \hline  m   & 2 & 3 & 4 & 5& 6& 7 & 8 & 9 & 10   \\
\hline \hline
 $ \mid {\cal F}_{m} \mid  = \mid {\cal L}_{m}\mid =L_m $   &  3 &4
&7 & 11 & 18& 29 & 47 & 76 & 123
\\ \hline
 $ \mid V({\cal N}_{m}) \mid $  &   6& 13
& 38 &121 &282 & 1093 & 2214 & 9841   & 17906
\\ \hline
 $ \mid V({\cal N}^{*}_{m}) \mid $  & 2 & 4 & 6&
16 & 20 & 64 &70 & 256 & 252
\\ \hline
 $ \mid V({\cal N}^{**}_{m}) \mid $ & 2 &2 &3 & 4&
6 & 9 & 11 & 23 & 26
\\  \hline
 order  & 2 &2&3 & 4 &
 6 & 8 & 10 & 16 & 21
\\
 \hline
\hline
\end{tabular}

\vspace*{5mm}

\begin{tabular}{||c||r|r|r|r|r|r|r|r||}
\hline \hline  m  & 11 & 12 & 13 & 14 & 15 & 16   & 17 & 18 \\
\hline \hline
 $ \mid {\cal F}_{m} \mid  = \mid {\cal L}_{m} \mid = L_m$  & 199 & 322 & 521 & 843  & 1364
& 2207&     3571 & 5778
\\ \hline
 $ \mid V({\cal N}^{*}_{m}) \mid $  & 1024  & 924 & 4096 & $\ll$ & $\ll$  &
$\ll$ & $\ll$ & $\ll$
\\ \hline
 $ \mid V({\cal N}^{**}_{m}) \mid $ & 63  & 62 & 190 & 170 & 612 & 487 &2056& 1530
 \\  \hline
 order   &32 & 39 & 64 & 73&$\ll$ &$\ll$&$\ll$ & $\ll$
 \\
 \hline
\hline
\end{tabular}

\caption{ The number of  vertices for \ ${\cal
N}_m$, \ ${\cal N}^{*}_m$ \ and  ${\cal N}^{**}_m$  \ and
order od recurrence relations for thin cylinder graphs \ $C_{m}
\times P_{n}$;  The number $ L_{m} $ is the $m$th member of the  Lucas sequence.}
\label{tab1}
\end{center}
\end{table}

\begin{table}[htb]
\begin{center}
\scalebox{0.85}{
\begin{tabular}{||c||r|r|r|r|r|r|r|r|r|r|r||}
\hline \hline  m   &  2 & 3 & 4 & 5& 6& 7 & 8 & 9 & 10 & 11 & 12    \\
\hline \hline
 $ \mid V({\cal D}_{m}) \mid $  &$10$&$26$
&$ 82$&$242$&$730$&$2186$&$6562$&$19682$&$59050$&$177146$&$531442$
\\ \hline
 $ \mid V({\cal D}^*_{m}) \mid $  &4&8&16&32&64&128&256&512&1024&2048 &4096
\\ \hline  \hline
 $ \mid V({\cal A}^*_{m}) \mid $ & 2 & 4 &  6& 16& 20
 &64 & 70 & 256 &252  & 1024 & 924
\\  \hline
 $ \mid V({\cal B}^{*(1)}_{m}) \mid $ &2  & 4 &8&16 &30
 &64&112&256&420&1024&1584
 \\  \hline
 $ \mid V({\cal B}^{*(2)}_{m}) \mid $ &-  & - & 2 & -
 & 12 & - & 56 &-  &240  &-  &990
 \\  \hline
 $ \mid V({\cal B}^{*(3)}_{m}) \mid $ &-  &-  &-  &-
 &2  &-  &16  &-  &90  &-  &440
 \\  \hline
 $ \mid V({\cal B}^{*(4)}_{m}) \mid $ &- &-  &-  &-
 &-  &-  & 2 & - & 20 & - &132
 \\  \hline
 $ \mid V({\cal B}^{*(5)}_{m}) \mid $ &-  & - & - &
- & - & - & - & - & 2 & - & 24
 \\  \hline
 $ \mid V({\cal B}^{*(6)}_{m}) \mid $ &- & - & - &
- & - & - & - & - & - & - &2
\\  \hline \hline
\end{tabular}}
\caption{The number of  vertices of  \ ${\cal D}_m$, ${\cal D}^{*}_m$ and  components of ${\cal D}^{*}_m$.}
\label{tab2}
\end{center}
\end{table}

\normalsize

By analyzing the computational data for $m \leq 10$ (in case of $TnC_m(n)$ for $m \leq 18 $) we  spotted
some  properties of the  digraphs ${\cal D}_m$, ${\cal D}^*_m$ and ${\cal N}_m^{*}$ which   have been  discussed and  proved in the previous section for arbitrary  $m \in N$.
The proofs of another ones,
concerning the  structure of   ${\cal D}^*_m$
will be exposed in another place (see \cite{JKO4}). Now, we just
express them in the following theorem.

\begin{thm}  \label{thm:hip}
For each even  $m\geq 2 $,
 the digraph  ${\cal D}^{*}_m$ has exactly $\ds \left\lfloor    \ds \frac{m}{2} \right\rfloor + 1$ components, i.e.
 \\
 $\ds {\cal D}^{*}_m = {\cal A}^*_m    \cup $ $\ds  (\bigcup_{s=1}^{\left\lfloor    \frac{m}{2} \right\rfloor }{\cal B}^{*(s)}_m)$, where
  ${\cal A}^*_m $  contains both  $1^m$ and $0^m$, all the components  ${\cal B}^{*(s)}_{m}$ ($ 1 \leq s \leq \ds \left\lfloor    \ds \frac{m}{2} \right\rfloor $) are bipartite digraphs,
  $\ds \mid V({\cal B}^{*(s)}_{m}) \mid  = \ds 2 {m \choose   m/2 -s} \mbox{ \  and \  } \ds \mid V({\cal A}^{*}_{m}) \mid =  \ds \ds {m \choose m/2}.$ \\ \\
For each odd  $m\geq 1 $,  the digraph  ${\cal D}^{*}_m$ has exactly two  components, i.e.
 ${\cal D}^{*}_m = {\cal A}^*_m    \cup {\cal N}^*_m $,  which are mutually  isomorphic  and
  with $2^{m-1}$ vertices.
 \end{thm}

Similar to the linear case, computational results revealed that the maximum eigenvalues of the
adjacency matrix of ${\cal D}^{*}_{m}$ and of its component ${\cal A}^{*}_{m}$ (the one  which contain  the vertex $1^m$) are matched.
It remains the open question. \\
Since the digraph ${\cal N}^{*}_{m}$ is isomorphic (or the same) to ${\cal A}^{*}_{m}$
(in accordance to Lemma~\ref{lem:6} and Theorem~\ref{thm:hip}) their  maximum eigenvalues  coincide, too.
  This eigenvalue is  denoted by $\theta_m$.
Some of the values $\theta_m$ and corresponding  coefficients $a^{TnC}_m$, $a^{TG}_m$ and $a^{KB}_m$ for
 $TnC_{m}(n)$, $TG^{p}_{m}(n)$ and $KB^{p}_m(n)$, respectively,
 are given in Table~\ref{tab5}).

\begin{table}[htb]
\begin{center}
\scalebox{0.82}{
\begin{tabular}{||c||r|r|r|r||} \hline\hline
$m$ & \multicolumn{1}{c| }{$\theta_m$}
 & \multicolumn{1}{c|}{$a^{TnC}_m$}
  & \multicolumn{1}{c|}{$a^{TG}_m$}
    & \multicolumn{1}{c||}{$a^{KB}_m$}
\\ \hline\hline
2 & 3
  & 0.5  & 1 & 1   \\ \hline
3 & 3.30277563773199464655961
  & 0.36132495094369271949541 &  2 & 2   \\ \hline
4 & 6.37228132326901432992531
  & 0.20648058601107554045568 & 1 & 1 \\ \hline
5 & 8.18892699556896444855102
  & 0.14742465083948628541204 & 2 & 2 \\ \hline
6 & 14.50643149404807519214675
  & 0.083803607557435033268322 & 1 & 1   \\ \hline
7 & 19.73524639197846469681942
  & 0.060124577862634389701985 & 2 & 2   \\ \hline
8 & 33.67678695772204595105959
  & 0.033993311467178925450426 & 1 & 1   \\ \hline
9  & 47.19198108116434356177681
   &  0.024494324930673487693541 &  2 & 2   \\ \hline
10 & 78.81886459182770309510614
   &  0.013795808626212712909327 &  1 & 1   \\ \hline
11 & 112.47596764975684359653883
   & 0.0099731885141630148539155 &  2 & 2   \\ \hline
12 & 185.23985780663511179029955
   &   0.0056015242668523217503199 &  1 & 1  \\ \hline
13 & 267.61048630198595550870131
   & 0.0040594014830894625219171 & 2 & 2   \\ \hline
14 & 436.39957118470795413801428
   & 0.0022751952682333538270877  & 1 & 1  \\ \hline
15 & $\approx_{(400)}$ 636.07041461464648869768059
   & $\approx_{(400)}$ 0.0016519675074264452373550 &  2 & 2   \\ \hline
16 & $\approx_{(300)}$ 1029.6591497906007164409665
   & $\approx_{(300)}$ 0.00092436404208619807770845 &  1 & 1  \\ \hline
17 & $\approx_{(300)}$ 1510.8614061966116128903621
   & $\approx_{(300)}$ 0.00067217371233189645627886 &  2 & 2   \\ \hline
18 & $\approx_{(300)}$ 2431.89478407418442105820
   & $\approx_{(300)}$ 0.00037562107425215551890931  & 1 & 1   \\ \hline
\hline
\end{tabular}}
\end{center}
\caption{The approximate values of $\theta_m$ and $a^{TnC}_m$ for $1\leq m\leq 16$,
  where $\approx_{(n)}$ means the estimate based on the first n
  entries of the sequence.}
\label{tab5}
\end{table}

Based on the  obtained numerical  results we computed the  generating functions

 \ $\ds   {\cal F}^{TnC}_{m}(x) \stackrel{\rm def}{=} \sum_{n\geq 1}^{\infty} f^{TnC}_{m}(n)x^{n} $ ($2 \leq m \leq 18$), \
  $\ds   {\cal F}^{TG}_{m,p}(x) \stackrel{\rm def}{=} \sum_{n\geq 1}^{\infty} f^{TG}_{m,p}(n)x^{n} $ ($2 \leq m \leq 10$) and

    $\ds   {\cal F}^{KB}_{m,p}(x) \stackrel{\rm def}{=} \sum_{n\geq 1}^{\infty} f^{KB}_{m,p}(n)x^{n} $  ($2 \leq m \leq 10$).
    They are listed in Appendix of this paper.
 The order  of the  recurrence relations   for  $TnC_{m}(n)$, \ $TG_{m}^{(p)}(n)$ and \ $KB_{m}^{(p)}(n)$ is obtained from the  polynomial  in denominator of these rational functions (see Table~\ref{tab1}, Table~\ref{tab3} and   Table~\ref{tab4}).

\begin{table}[htb]
\begin{center}
%\small{}
\begin{tabular}{||c||r|r|r|r|r|r|r|r|r||}
\hline \hline  m   &  2 & 3 & 4 & 5& 6& 7 & 8 & 9 & 10    \\
\hline \hline
$ p=0 $  &  4& 3
& 8 &10 &26 &35& 105 & 132 &370
\\ \hline
 $ p=1 $  &  4& 3
& 8 &10 &25 &35& 80 & 132 &369
\\ \hline
 $ p=2  $  &  -  & 3 & 7 & 10 & 24 &35 & 61 & 132 & 370
\\ \hline
$ p=3 $  & -   &  - & 8 & 10 & 26 &35 & 80 & 132 & 369
\\ \hline
 $ p=4  $ & -  &  -  &  -  &  10 & 24
 &35 &  105  &132 & 370
 \\  \hline
 $ p=5  $ &-  & - & - & -
 & 25 & 35 & 80 & 132  & 369
 \\  \hline \hline
\end{tabular}
\caption{
The order  of the  recurrence relation   for the   torus graphs \ $TG_{m}^{(p)}(n)$.}
\label{tab3}
\end{center}
\end{table}

\begin{table}[htb]
\begin{center}
%\small{}
\begin{tabular}{||c||r|r|r|r|r|r|r|r|r||}
\hline \hline  m   &  2 & 3 & 4 & 5& 6& 7 & 8 & 9 & 10    \\
\hline \hline
$ p=0 $  &  4& 2
& 8 &4 &17 &8& 34 &16 &66
\\ \hline
$ p=1 $  &  4& 2
& 8 &4 &17 &8& 34 &16 &66
\\ \hline \hline
\end{tabular}
\caption{
The order  of the  recurrence relation   for the  Klein bottle $KB^{(p)}_m(n)$.}
\label{tab4}
\end{center}
\end{table}

By analyzing these functions in case of torus grid, we noticed matchings
 of some functions.
\begin{thm}  \label{thm:sim1}
$${\cal F}^{TG}_{m,p}(x) = {\cal F}^{TG}_{m,m-p}(x), \mbox{for all }  1 \leq p \leq m-1. $$
 \end{thm}
\begin{proof}
Since  $\ds  ({\cal R}_m^*)^m = I$ we have $\ds \left( ( {\cal R}_m^*)^p \right)^T = ( {\cal R}_m^*)^{m-p}$. Additionally, the matrix ${\cal T}_m^*$ and  all its degrees are symmetric ones.
Thus, $$
 f_{m,p}^{TG}(n)=  tr(({\cal R}_m^*)^p \cdot ({\cal T}_m^*)^n) = tr((({\cal R}_m^*)^p \cdot ({\cal T}_m^*)^n)^{T}) =
 tr( \left(({\cal T}_m^*)^n \right)^{T} \cdot   \left(({\cal R}_m^*)^p \right)^{T})  = $$
  $\ds tr( ({\cal T}_m^*)^n  \cdot ({\cal R}_m^*)^{m-p}) =  tr( ({\cal R}_m^*)^{m-p} \cdot  ({\cal T}_m^*)^n ) =  f_{m,m-p}^{TG}(n).$ \ $\Box$
\end{proof}

 The matching in case of  Klein bottle is less obvious less and it is
 expressed in the following theorem.
\begin{conj}
If $m$ is  odd, the generating function
  $\ds   {\cal F}^{KB}_{m,p}(x) \stackrel{\rm def}{=} \sum_{n\geq 1}^{\infty} f^{KB}_{m,p}(n)x^{n} $ is invariant for different values of $p$ ($0 \leq p \leq m-1$). If $m$ is even, the same is true for all values of $p$ which are of the same parity.
\end{conj}

%Za generisanje digrafa  ${\cal D}^{**TnC}_{18}$ bilo je potrebno oko 6 sati i 10 minuta i jos oko 1 sat i 13 minuta za odredjivanje  prvih 300 clanova niza.

%Dobijeni  su digrafovi ${\cal D}^{**TnC}_{m}$ za $m \leq 18$ i  racionalne generativne funkcije ${\cal F}^{TnC}_{m}(x)$ za $2 \leq m \leq 14$. Prvih $25$ vrednosti za $9 \leq m %\leq 18$ su ispisane u poslednjoj sekciji.

\vspace*{0.5cm}
%Fill author(s) affiliation(s), address(es) and emails here:

\noindent Faculty of Technical Sciences,
  University of Novi Sad,
  Novi Sad, Serbia\\
     E-mail: jelenadjokic@uns.ac.rs  \\
E-mail: ksenija@uns.ac.rs (corresponding author)

\vspace*{0.5cm}

 \noindent
  Dept.\ of Math.\ \&\ Info.,
  Faculty of Science,
  University of Novi Sad,
  Novi Sad, Serbia \\
   E-mail: olga.bodroza-pantic@dmi.uns.ac.rs

\newpage

\vspace{1.5cc}
\begin{center}
{\bf  5.  \  Appendix }
\end{center}

\vspace*{4mm}

{\bf  5.1.  \  \ Thin cylinder graph $TnC_{m}(n) \equiv C_{m} \times P_{n}$ ($ 2 \leq m \leq 18$) }

 $\ds   {\cal F}^{TnC}_{m}(x) \stackrel{\rm def}{=} \sum_{n\geq 1}^{\infty} f^{TnC}_{m}(n)x^{n} $

 \small

\noindent
$ \ds  {\cal F}^{TnC}_{2}(x) =  \frac{x(1 + 3 x)}{1 - 2 x - 3 x^2} = $ \\
\parbox[t]{5.8in}{
$x+5 x^2+13 x^3+41 x^4+121 x^5+365 x^6+1093 x^7+3281 x^8+9841 x^9+29525 x^{10}+88573 x^{11}+265721 x^{12}+797161 x^{13}+2391485 x^{14}+7174453
x^{15}+21523361 x^{16}+64570081 x^{17}+193710245 x^{18}+581130733 x^{19}+1743392201 x^{20}+5230176601 x^{21}+15690529805 x^{22}+47071589413 x^{23}+141214768241
x^{24}+423644304721 x^{25}+ \ldots $}

 \bc $ \mbox{--------------------------------------    } $ \ec
$ \ds  {\cal F}^{TnC}_{3}(x) =  \frac{x(1+x)}{1-3 x-x^2} = $ \\
  \parbox[t]{5.8in}{
$x+4 x^2+13 x^3+43 x^4+142 x^5+469 x^6+1549 x^7+5116 x^8+16897 x^9+55807 x^{10}+184318 x^{11}+608761 x^{12}+2010601 x^{13}+6640564 x^{14}+21932293
x^{15}+72437443 x^{16}+239244622 x^{17}+790171309 x^{18}+2609758549 x^{19}+8619446956 x^{20}+28468099417 x^{21}+94023745207 x^{22}+310539335038 x^{23}+1025641750321
x^{24}+3387464586001 x^{25}+ \ldots $}

 \bc $ \mbox{--------------------------------------    } $ \ec
 $\ds {\cal F}^{TnC}_{4}(x) = \frac{x\left(1+3 x-4 x^2\right)}{1-6 x-3 x^2+4 x^3} = $ \\
  \parbox[t]{5.8in}{
$\ds x+9 x^2+53 x^3+341 x^4+2169 x^5+13825 x^6+88093 x^7+561357 x^8+3577121 x^9+22794425 x^{10}+145252485 x^{11}+925589701 x^{12}+5898117961
x^{13}+37584466929 x^{14}+239498796653 x^{15}+1526153708861 x^{16}+9725080775409 x^{17}+61970950592425 x^{18}+394896331045333 x^{19}+2516390514947637
x^{20}+16035148280452121 x^{21}+102180475903374305 x^{22}+651122738201811645 x^{23}+4149137263799184301 x^{24}+26439469893787043521 x^{25}+
  \ldots $}

      \bc $ \mbox{--------------------------------------    } $ \ec
    $\ds {\cal F}^{TnC}_{5}(x) =  \ds \frac{x\left(1+2 x-14 x^2+3 x^3\right)}{1-9 x+4 x^2+22 x^3-3 x^4}=$ \\
  \parbox[t]{5.8in}{$
x+11 x^2+81 x^3+666 x^4+5431 x^5+44466 x^6+364061 x^7+2981201 x^8+24412606 x^9+199912706 x^{10}+1637069691 x^{11}+13405842666 x^{12}+109779463516
x^{13}+898976005896 x^{14}+7361648869421 x^{15}+60284005131851 x^{16}+493661316969811 x^{17}+4042556485091321 x^{18}+33104199931650186 x^{19}+271087876486546101
x^{20}+2219918829931214536 x^{21}+18178753234393716291 x^{22}+148864483106909524811 x^{23}+1219040384395583776646 x^{24}+9982632712465747775776 x^{25}
  + \ldots $}

 \bc $ \mbox{--------------------------------------    } $ \ec
 $\ds {\cal F}^{TnC}_{6}(x) = \frac{x\left(1+4 x-52 x^2+85 x^3+4 x^4-24 x^5\right)}{1-16 x+15 x^2+108 x^3-163 x^4-14 x^5+24 x^6} = $ \\
 \parbox[t]{5.8in}{
$  x+20 x^2+253 x^3+3725 x^4+53812 x^5+781043 x^6+11328703 x^7+164342144 x^8+2384008549 x^9+34583478677 x^{10}+501682800748 x^{11}+7277627334803
x^{12}+105572401943143 x^{13}+1531478817520040 x^{14}+22216292548032997 x^{15}+322279125907163021 x^{16}+4675120061914150660 x^{17}+67819308904658336819
x^{18}+983816158598975546575 x^{19}+14271661707453924975056 x^{20}+207030882865408620073765 x^{21}+3003279319439344022622533 x^{22}+43566865704938163454739356
x^{23}+631999752759077987597027603 x^{24}+9168061117654885779896674423 x^{25}+ \ldots $}

   \bc $ \mbox{--------------------------------------    } $ \ec
 $ {\cal F}^{TnC}_{7}(x) \ds  =
   \frac{x(1+2 x-175 x^2+557 x^3-128 x^4-909 x^5+564 x^6+55 x^7)}{1-27 x+131 x^2+319 x^3-1511 x^4+598x^5+1473 x^6-740 x^7-55 x^8} = $ \\
 \parbox[t]{5.8in}{$x+29 x^2+477 x^3+9318 x^4+181231 x^5+3562728 x^6+70182449 x^7+1384148396 x^8+27309182412 x^9+538897819048 x^{10}+10634850017387 x^{11}+209878072831673
x^{12}+4141969931423934 x^{13}+81742600445824746 x^{14}+1613208844972065013 x^{15}+31837062363892428112 x^{16}+628312180222680689296 x^{17}+12399894995078143327538
x^{18}+244714977626524860954080 x^{19}+4829510338163034053569033 x^{20}+95311576159624130076274330 x^{21}+1880997437078948822215084651 x^{22}+37121947864644639377123619563
x^{23}+732610787514210597258594558399 x^{24}+14458254399899446658426778070807 x^{25} + \ldots $}
  \bc $ \mbox{--------------------------------------    } $ \\  $ \mbox{--------------------------------------    } $ \ec

  \noindent
  $ \ds
{\cal F}^{TnC}_{8}(x) \ds  = $ \\  \parbox[t]{5.8in}{
$ \ds \frac{x\left(-1-2 x+537 x^2-4828 x^3+8833 x^4+9182 x^5-27961 x^6+10592 x^7+1152 x^8-576 x^9\right)}{(1+x) \left(-1+48
x-497 x^2+40 x^3+16813 x^4-54024 x^5+54853 x^6-15568 x^7-1216 x^8+576 x^9\right)} $} \\

 \bc $ \mbox{--------------------------------------    } $ \ec

\noindent
$ \ds {\cal F}^{TnC}_{8}(x) = $ \\
 \parbox[t]{5.8in}{$\ds  x+49 x^2+1317 x^3+44269 x^4+1474937 x^5+49622793 x^6+1670194477 x^7+56241588037 x^8+1893972519489 x^9+63782453175969 x^{10}+2147983445752757
x^{11}+ 72337143245836829 x^{12}+2436082206688156809 x^{13}+82039418438284617401 x^{14}+2762823988969738975165 x^{15}+93043034632818588850549 x^{16}+
 3133390453009241673982033 x^{17}+105522522721603955058210641 x^{18}+3553659516758976044869767493 x^{19}+119675834464582605695436465485 x^{20}+
4030297581237050067157486579097 x^{21}+135727473019414267506176850632809 x^{22}+4570865193183627479968367145336909 x^{23}+153932053323301682716318949733383589
x^{24}+ 5183936965733447442900189278987107425 x^{25}+ \ldots $}

  \bc $ \mbox{--------------------------------------    } $ \\ $ \mbox{--------------------------------------    } $ \ec

\noindent
$ \ds {\cal F}^{TnC}_{9}(x) = $ \\
 \parbox[t]{5.8in}{$\ds
 x(-1+5 x+1579 x^2-31215 x^3+171589 x^4 +122800 x^5-3535489 x^6+9094371 x^7-2433845 x^8-16805943 x^9+14544216 x^{10}+7056620 x^{11}-9282885
x^{12}+150158 x^{13}+1105012 x^{14}-43295 x^{15})/ $}\\
\parbox[t]{5.8in}{$\ds (-1+81 x-1792 x^2+7289 x^3+113338 x^4-948939 x^5+891997 x^6+9118681 x^7-25652726 x^8+9992771 x^9+33620979 x^{10}-29903008 x^{11}-9941993
x^{12}+14464685 x^{13}- 684910 x^{14}-1263348 x^{15}+43295 x^{16}) $} \\

   \bc $ \mbox{--------------------------------------    } $ \ec

\noindent
$ \ds {\cal F}^{TnC}_{9}(x) = $ \\
 \parbox[t]{5.8in}{$\ds x+76 x^2+2785 x^3+127897 x^4+5864650 x^5+273687040 x^6+12839393125 x^7+604211712448 x^8+28474336325785 x^9+1342851693261496 x^{10}+63350881300193974
x^{11}+2989171289995095295 x^{12}+141053804754239239840 x^{13}+6656352599557966594252 x^{14}+314120572018446461135485 x^{15}+14823836337375386181563560
x^{16}+699563077126129766004289882 x^{17}+33013695479540693113501652143 x^{18}+1557980033613135956810106192973 x^{19}+73524126061962906254766642604468
x^{20}+3469748286048642692532972906829690 x^{21}+163744275200677590087354798821692690 x^{22}+7727416270516028098255799127367588354 x^{23}+364672071690222348813333957553217496592
x^{24}+ $ \\ $ 17209597260329748447239428698333452150035 x^{25}+ \ldots$}

  \bc $ \mbox{--------------------------------------    } $ \\ $ \mbox{--------------------------------------    } $ \ec
$ \ds
{\cal F}^{TnC}_{10}(x) = $ \\
   \noindent \parbox[t]{5.8in}{$ \ds
   -x(-1+7 x+4520 x^2-158777 x^3+1521893 x^4+382971 x^5-74719698 x^6+329729442 x^7+106970524 x^8-3804028638 x^9+8216839522 x^{10}-1691243209
x^{11}-12244389652 x^{12}+ 11435458052 x^{13}+1418185236 x^{14}-5391746720 x^{15}+1399734576 x^{16}+456158912 x^{17}-153955584 x^{18}-12127232 x^{19}+1990656
x^{20})/ $} \\    \noindent \parbox[t]{5.8in}{$ \ds
(1-132 x+4767 x^2-37118 x^3-798843 x^4+12319598 x^5-34775915 x^6-217196970 x^7+1331532391 x^8-746302336 x^9-9092536778 x^{10}+21085576048
x^{11}-6963045159 x^{12} - 23008679746 x^{13}+ 22066750640 x^{14}+1148788328 x^{15}-7918550312 x^{16}+1899531616 x^{17}+618171648 x^{18}-184157184 x^{19}-14919680
x^{20}+1990656 x^{21})$}

   \bc $ \mbox{--------------------------------------    } $ \ec

\noindent $ \ds
{\cal F}^{TnC}_{10}(x) = $ \\
   \noindent \parbox[t]{5.8in}{$ \ds
x+125 x^2+7213 x^3+552136 x^4+42414281 x^5+3321537720 x^6+261079885983 x^7+20559551095851 x^8+1619938572971116 x^9+127666740816792660
x^{10}+10062119265462622683 x^{11}+793072716833845192356 x^{12}+62508746073022625976096 x^{13}+4926858570367533896154450 x^{14}+388329118012333522559446373
x^{15}+30607652155303378920835811641 x^{16}+2412460161413885294860293439151 x^{17}+190147364231734348390327535617975 x^{18}+14987199165829312727697716746010488
x^{19}+1181274016272528946898501556179154091 x^{20}+93106676579881222044891366351751374216 x^{21}+7338562529511716808716729856342178425105 x^{22}+ $ \\ $ 578417166185052464306299085928797905085693
x^{23}+ $ \\ $ 45590184295478178866479815354627214195160656 x^{24}+ $ \\ $ 3593366562596992824027936085406987745740233896 x^{25}+    \ldots  $}
  \bc $ \mbox{--------------------------------------    } $  \\ $ \mbox{--------------------------------------    } $ \ec

\noindent
$ \ds
{\cal F}^{TnC}_{11}(x) = $ \\
   \noindent \parbox[t]{5.8in}{$ \ds x(1 - 44 x - 12081 x^2 + 1133545 x^3 - 37416712 x^4 + 493814171 x^5 +
   373805870 x^6 - 85025176253 x^7 + 932415097823 x^8 -
   3024413960156 x^9 - 15533189594719 x^{10} + 157561589289988 x^{11} -
   369165572872227 x^{12} - 756465027265708 x^{13 }+
   5361776208341062 x^{14} - 7328194249688806 x^{15} -
   9594905882891048 x^{16} + 34024877697560525 x^{17 } -
   13685759379919151 x^{18} - 44098741194270242 x^{19} +
   45724243017873625 x^{20} + 14863779838287700 x^{21} -
   37231721334564025 x^{22} + 6572628113135888 x^{23} +
   11491371455659095 x^{24} - 4771774256671575 x^{25} -
   1096225371607414 x^{26} + 810885158433255 x^{27} -
   26200483957866 x^{28} - 37952723101002 x^{29} + 3921432582660 x^{30} +
   91766783871 x^{31}) / $} \\
   \parbox[t]{5.8in}{  $\ds (1 - 243 x + 20039 x^2 - 656471 x^3 +
   4431429 x^4 + 227318362 x^5 - 5368460917 x^6 + 31298257778 x^7 +
   293649763644 x^8 - 5015233312469 x^9 + 21937438967225 x^{10} +
   31616042052087 x^{11} - 598765550430721 x^{12} +
   1679928753366531 x^{13}+ 1326205582238730 x^{14} -
   15835876424317510 x^{15} + 25160905266696190 x^{16} +
   16435286891736108 x^{17} - 83392084716072564 x^{18} +
   47653451679066927 x^{19} + 81495737922402983 x^{20} -
   100306224843907275 x^{21} - 13857775303557521 x^{22} +
   64525643671324017 x^{23} - 15743156527963701 x^{24} -
   16461505606017185 x^{25}+ 7609935220125145 x^{26} +
   1263109560506538 x^{27} - 1083745124145697 x^{28} +
   47282053928002 x^{29} + 43514398016402 x^{30} - 4294326180612 x^{31} -
   91766783871 x^{32}) $ }

   \bc $ \mbox{--------------------------------------    } $ \ec

     \noindent
  \parbox[t]{5.8in}{   ${\cal F}^{TnC}_{11}(x) =$ \\ $ \ds
   x+199 x^2+16237 x^3+1747846 x^4+188142923 x^5+20737405360 x^6+2304956480025 x^7+257573508034492 x^8+28865363400315608 x^9+3240058791241468318 x^{10}+364014089441637130211 x^{11}+40916785291407964066602 x^{12}+4600517305417776076239962 x^{13}+517344667951564049852195183 x^{14}+58182367489407130596945718621 x^{15}+6543710936410429529889754407939 x^{16}+735984616119817916488132820828362 x^{17}+82778971776385485828799271961192790 x^{18}+9310543698511076212843258446713194344 x^{19}+1047206044378161109688710809509094458969 x^{20}+ $ \\ $ 117785112741978205895896968461190504584242 x^{21}+ $ \\ $ 13247969348778085027453077072904755913393268 x^{22}+ $ \\ $ 1490076588308426360324950466197023081766178323 x^{23}+ $ \\ $ 167597706555404885963702412107237684319747327535 x^{24}+ $ \\ $ 18850707958004578841577999432299025065849894780530 x^{25}+ \ldots $
  }

  \bc $ \mbox{--------------------------------------    } $  \\ $ \mbox{--------------------------------------    } $ \ec

   \noindent \parbox[t]{5.8in}{${\cal F}^{TnC}_{12}(x) = $ \\ $ \ds
\ds x(-1 + 56 x + 34371 x^2 - 4864597 x^3 + 240897321 x^4 -
   4539302037 x^5 - 21305999292 x^6 + 2250499820605 x^7 -
   32835395212689 x^8 + 67641863177208 x^9 + 2973508964109849 x^{10} -
   32672625360812141 x^{11} + 101593694575101540 x^{12} +
   390701274674888073 x^{13} - 3534088510502050536 x^{14} +
   5294394569279677028 x^{15}+ 23673992151915915908 x^{16} -
   83713587334034214864 x^{17} - 16860728113863417992 x^{18} +
   389883495876254911488 x^{19} - 298642788278585906400 x^{20} -
   782820582097976003264 x^{21} + 1028399159901266187040 x^{22} +
   651762671657584482752 x^{23} - 1300946698118184948672 x^{24} -
   124207182380513596608 x^{25} + 687514492586295710080 x^{26} -
   61062685619092030208 x^{27} - 169059279633977012736 x^{28} +
   29452032911443639808 x^{29} + 19587344454030503936 x^{30} -
   4695885856421656576 x^{31} - 929084384166322176 x^{32} +
   320241016195547136 x^{33} + 3147630768685056 x^{34} -
   7866096317628416 x^{35} + 656099510321152 x^{36} +
   5025355005952 x^{37} -
   1480421539840 x^{38} ) / $  }  \\
   \parbox[t]{5.8in}{  $\ds ((-1 + 2 x) (1 - 378 x + 47332 x^2 -
     2280904 x^3 + 12808719 x^4 + 2389875792 x^5 - 75238234902 x^6 +
     543593806898 x^7 + 11175795915724 x^8 - 224307765986366 x^9 +
     783156388517842 x^{10} + 12757301632650108 x^{11 } -
     135144380202845134 x^{12} + 308192305830674750 x^{13} +
     1763902421105202275 x^{14} - 10128104308663148076 x^{15} +
     3698870897167759532 x^{16} + 75112077594370269840 x^{17} -
     119470305233424050008 x^{18} - 219141200424866705360 x^{19} +
     567082778690898661792 x^{20} + 229106439400398106656 x^{21} -
     1160421178627356937040 x^{22} + 59238560866623765184 x^{23} +
     1128088144206736669952 x^{24 } - 236940152142480533248 x^{25} -
     505989846956029581824 x^{26} + 133058523045466635008 x^{27} +
     107835619636777944064 x^{28} - 31591280389621688320 x^{29} -
     10611922573161156608 x^{30} + 3655312114963147776 x^{31} +
     384446534883069952 x^{32} - 199316640957022208 x^{33} +
     3772718451195904 x^{34}+ 4066367130304512 x^{35} -
     377529602932736 x^{36} - 699161116672 x^{37} + 740210769920 x^{38})) $ }  \bc $ \mbox{--------------------------------------    } $ \ec

        \noindent \parbox[t]{5.8in}{$ {\cal F}^{Tnc}_{12}(x) = $ \\ $ \ds
   x+324 x^2+40661 x^3+7110833 x^4+1258226556 x^5+229320317359 x^6+42157390580371 x^7+7784584974538368 x^8+1440027274442086769 x^9+266592485903824019297
x^{10}+49370866298667719771964 x^{11}+9144433503092353217515639 x^{12}+1693831480088780441382206551 x^{13}+313758480051614385713328975696 x^{14}+58120041000674374495827699035057
x^{15}+10766104822128945772009775893697705 x^{16}+1994308218312032136468520737923798952 x^{17}+369425086385808570102911307947404967215 x^{18}+68432227400397660077853587903266565652955
x^{19}+ $ \\ $ 12676374200261687110601982960285590881424640 x^{20}+ $ \\ $ 2348169602273990391519270417144794820324148805 x^{21}+ $ \\ $ 434974590875730148654647846938627993696817466361
x^{22}+ $ \\ $ 80574630359272258570224416141640937089566975901396 x^{23}+ $ \\ $ 14925632988960550558242562591180701510404613852892927 x^{24}+ $ \\ $ 2764822125913864323327329261558606236715960966384783383
x^{25}+ \ldots $}

  \bc $ \mbox{--------------------------------------    } $  \\ $ \mbox{--------------------------------------    } $ \ec

   \noindent \parbox[t]{5.8in}{ $ {\cal F}^{TnC}_{13}(x) =$ \\ $ \ds
 -x(1 - 208 x - 79892 x^2 + 33042137 x^3 - 4869390172 x^4 +
    373483066597 x^5 - 15135800458651 x^6 + 191818992781737 x^7 +
    11079457007481830 x^8 - 632841913799364517 x^9 +
    13918241384529471242 x^{10} - 87693069841230523094 x^{11} -
    2788059553915657926426 x^{12} + 77707843921367359992989 x^{13} -
    755415318769706674613309 x^{14} - 925018680108455680670626 x^{15} +
    101487342677058526530184995 x^{16} -
    1023078316773348885386226273 x^{17 }+
    2303832047406447741814053167 x^{18} +
    40152955247861642091072539182 x^{19} -
    395878025158363885643482233411 x^{20} +
    1001811435396844163799570124076 x^{21} +
    6394921217172604112892229840005 x^{22} -
    55668278101946512484697082135321 x^{23 }+
    109212038343275044549244922004133 x^{24 }+
    513572332823648453020271396810846 x^{25} -
    3237182867334606281012564610743853 x^{26} +
    3841140916525575444614058990336883 x^{27} +
    20419196651899602668992374438670226 x^{28} -
    78567646837662277795702316507343481 x^{29} +
    30690951610382047776198221838765067 x^{30} +
    378094412006368220082564812727239755 x^{31}-
    783528468398079487572618201668412839 x^{32 }-
    293990881816337089761981120737863631 x^{33} +
    3008104705893398867249496210169025307 x^{34} -
    2929015291786575133616708599085469355 x^{35} -
    3796739592887618722452573683809355005 x^{36 }+
    9746084284316976443626738883590246361 x^{37} -
    2754836522155994780666309278884205938 x^{38} -
    11402646654459649578350510433014541219 x^{39 }+
    12089306032836145766107047838981155562 x^{40} +
    2844360951099792863631702261076008755 x^{41} -
    11913725326192563296147055326251452013 x^{42} +
    4933800759563095968000854508039187618 x^{43} +
    4360344184704242580238180787164524291 x^{44} -
    4475179226084075430303128388600609667 x^{45} +
    146192922260515937248678600430383775 x^{46} +
    1450970582806272131232442714816566530 x^{47} -
    534809702639522619653312060561575608 x^{48} -
    175840325453497586039101584055405323 x^{49} +
    150523237812662580783224696921029933 x^{50} -
    7668370815393412314175579358355667 x^{51} -
    18233111926513813771882264772114438 x^{52} +
    4196696982178710871202017611050486 x^{53} +
    925936608621152268812931398292294 x^{54 }-
    429825821051533195593212358313691 x^{55} +
    897165031534360276341370410761 x^{56} +
    18895674660195435967037426744133 x^{57} -
    1784982930656432777798323730463 x^{58} -
    323509207084312021328443656409 x^{59} +
    51766170562066232232343663129 x^{60} +
    635531816742339296550606824 x^{61} -
    292979227870435856545576584 x^{62} +
    2903180987112106367511999 x^{63})
/ $} \\
   \parbox[t]{5.8in}{  $\ds
(-1 + 729 x - 205276 x^2 +
   28774364 x^3 - 2069242947 x^4 + 54521238784 x^5 +
   2540771984406 x^6 - 258270158875831 x^7 + 8254277869695393 x^8 -
   50438475178054008 x^9 - 4577922591875751515 x^{10} +
   156920730246337867838 x^{11} - 1947518869707068863368 x^{12} -
   6258787718073169065654 x^{13} + 537497696236859675204944 x^{14} -
   6981499288328474688957601 x^{15} + 20034421337134822343796215 x^{16} +
   493095821378669850823979955 x^{17} -
   6544872938919211955859448126 x^{18} +
   25406303573141864815862984199 x^{19} +
   141499498127500753890161896491 x^{20} -
   2000864943348121099938431452301 x^{21} +
   7154865464783049876481048644373 x^{22} +
   16629955824567304984660205720833 x^{23} -
   233729092639295072145786839574244 x^{24} +
   646116458926462503983936842610559 x^{25} +
   1177121779980131876904375050463561 x^{26 }-
   11791744839674393668221649162114373 x^{27} +
   21504006059868060173719331448152805 x^{28} +
   46190334441390263074747993336223894 x^{29} -
   256360433605959364284217420577841198 x^{30} +
   235009735911625906812861513875038323 x^{31} +
   870707795924250278511852812478304166 x^{32} -
   2387383239161720320955338589768157719 x^{33} +
   302607738682234190013151193583053376 x^{34} +
   6860585933858413756505798261386039563 x^{35} -
   8822306772787739848623153058208091328 x^{36} -
   5552856626388077752829058463035375175 x^{37}+
   21613448565579102233889724765461627035 x^{38} -
   10368056116790690392474636985351863102 x^{39} -
   20080762265804261087167442626526631317 x^{40} +
   26113021283068413700974698372414984970 x^{41} +
   1390004873389173619205792799047533594 x^{42} -
   21304487637811693028239371056298770161 x^{43} +
   10781189932775001881834059026377855347 x^{44} +
   6363412216101912926979518415704916003 x^{45} -
   7821204017924481120760773240982875216 x^{46} +  % $} \\   \noindent \parbox[t]{5.8in}{$ \ds
   722046678331614498126502980945559883 x^{47} +
   2240930973472927038251157158623351578 x^{48} -
   919048221800555052277331785586930648 x^{49} -
   231337977658386086207508579809234906 x^{50} +
   228708614752344081036609774647295785 x^{51} -
   16913766876412520957797308390311353 x^{52} -
   25304540394430534938380404952836106 x^{53} +
   6236853171309462366457570112380092 x^{54} +
   1144877899955153902428948327877978 x^{55} -
   577206037028759042190137801400852 x^{56} +
   7529973985103843173848511760413 x^{57} +
   23326357953957345784305823318410 x^{58} -
   2318542877384283583434937750011 x^{59} -
   363854240174009045344913754630 x^{60 }+
   59964026528900193652041528729 x^{61} +
   515500743628959729859994424 x^{62} -
   305848740124993418971257480 x^{63} + 2903180987112106367511999 x^{64})
 $ }  \bc $ \mbox{--------------------------------------    } $ \ec

   \noindent \parbox[t]{5.8in}{ $ {\cal F}^{TnC}_{13}(x) = $ \\ $ \ds x+521 x^2+94641 x^3+23860994 x^4+6019949235 x^5+1563616901289 x^6+410685863818829 x^7+108680224298965775 x^8+28885443646068696636 x^9+7697782672223977809178
x^{10}+2054738065680739228638707 x^{11}+549007484897518979688253782 x^{12}+146778979981795502904523460553 x^{13}+39256455963159335499059410183014
x^{14}+10501645621993870751107771050851885 x^{15}+2809728365013355124287589270028794708 x^{16}+751810737333491167502810631395327329121 x^{17}+ $ \\ $ 201175700430372387590350775917384172630403
x^{18}+ $ \\ $ 53833981683801464428474326205942754537286776 x^{19}+ $ \\ $ 14406087682340513042806512640834731766157769808 x^{20}+ $ \\ $ 3855146257880058549107041721153429364184928394618
x^{21}+ $ \\ $ 1031665446779243847361514470733917000217365153177412 x^{22}+ $ \\ $ 276082504055472631358905615180969234514638921778807699 x^{23}+ $ \\ $ 73882247078607453109055517073219210939592385510318453725
x^{24}+ $ \\ $ 19771610577230034420482354351709347816048938574233828605390 x^{25}+
  \ldots  $}

  \bc $ \mbox{--------------------------------------    } $  \\ $ \mbox{--------------------------------------    } $ \ec

   \noindent \parbox[t]{5.8in}{${\cal F}^{TnC}_{14}(x) = $ \\ $ \ds
\ds -x(1 - 250 x - 241465 x^2 + 130181673 x^3 - 26427230137 x^4 +
   2708647608986 x^5 - 130217217118640 x^6 - 470836804131480 x^7 +
   419309310152602865 x^8 - 22974790326493373041 x^9 +
   474262067769429561796 x^{10} + 5598507190928637868102 x^{11} -
   571508536513909558942594 x^{12} + 13917902583796228345661803 x^{13} -
   117113367375450720200283133 x^{14} -
   1656259014654120521313749934 x^{15} +
   54001661463527122847686547042 x^{16} -
   486325436495354321002640736880 x^{17}-
   1605152466339233564549002956640 x^{18} +
   74099155369156694584230365258071 x^{19} -
   570102846866816503384840626169219 x^{20} -
   822300642088766387503327511823792 x^{21} +
   42981048873345720546910836849946004 x^{22} -
   257161489141860033549386184964574010 x^{23} -
   253510390114999310217807632433799594 x^{24} +
   10673567481699921074736526517028335402 x^{25} -
   46946430826474049597337696289782958858 x^{26} -
   46277637595578587497426960505134905378 x^{27} +
   1149337199348125334199654551037368029623 x^{28} -
   3522175222040222760434269297713180738268 x^{29} -
   4369856596077501178153759366651850525419 x^{30} +
   55566876535896079206032017528855663008535 x^{31} -
   113459184423791680593513823627816708409313 x^{32} -   $}
\\   \noindent \parbox[t]{5.8in}{$ \ds
   173943889851419628194630861790033520405456 x^{33} +    $  \\  $
   1215382299008078243475661349251217642694642 x^{34} -    $  \\  $
   1608119302284289389476015331241378182938774 x^{35} -    $  \\  $
   2879664522920911409039572843365761888793143 x^{36} +    $  \\  $
   12055789727468931549127069341186159738828857 x^{37} -    $  \\  $
   9985929328812497581879337689906031165087764 x^{38} -    $  \\  $
   20942258352471870193928568975579222497456902 x^{39} +    $  \\  $
   55392520723838044259412878168157515125889636 x^{40} -    $  \\  $
   27609686230977138495026565377064966130425637 x^{41} -    $  \\  $
   67850586447897636496464132692072872436758585 x^{42} +    $  \\  $
   118802785837761464661211758185967151322113034 x^{43} -    $  \\  $
   33977631755868264948399893457763079268493366 x^{44}-    $  \\  $
   99421473490735074981813309603831089689951584 x^{45} +    $  \\  $
   120256511659972649735756899060884018263658966 x^{46} -    $  \\  $
   20489478639226596484134815318696239165821201 x^{47} -    $  \\  $
   64166501630310175972836888096097593046941905 x^{48} +    $  \\  $
   56899019707905833112640870171779685334896140 x^{49} -    $  \\  $
   7988423180891978814384159379402438588047584 x^{50} -    $  \\  $
   16498043402068033174159692237075411690150480 x^{51} +    $  \\  $
   11566441375465033208057690381308526690265440 x^{52} -    $  \\  $
   1858480556728317777694402485432190396917760 x^{53 }-    $  \\  $
   1293380582943050706948480926685812206707200 x^{54} +    $  \\  $
   727602597725014857232079357140843386469888 x^{55} -    $  \\  $
   76710854904929243688280535862161670205440 x^{56} -    $  \\  $
   45971146675336357055513757997856051377152 x^{57} +
   15827629533235547990659315987758396997632 x^{58} -
   296353970289794533339234242414949163008 x^{59} -
   690854642288128089124070699167113347072 x^{60} +
   99831018249592940108202590217923198976 x^{61} +
   10408795298536646682314789739889688576 x^{62} -
   3325028631626157528481847618427682816 x^{63} +
   25338546848544478360952614168821760 x^{64 }+
   50169288331599343596914886980927488 x^{65} -
   2402700679322372631085805681508352 x^{66 }-
   395691346577380311483617642545152 x^{67} +
   25633128015062785391957603516416 x^{68} +
   1640091750757852426975263588352 x^{69} -
   100486229785123775534189248512 x^{70} -
   3089733801952044024971919360 x^{71} +
   87018482365782313402368000 x^{72}) / $} \\
   \parbox[t]{5.8in}{  $\ds ((1 + x) (-1 + 15 x) (1 + 11 x + 24 x^2 + 13 x^3) (1 + 39 x +
    164 x^2 + 139 x^3) (1 + 15 x + 43 x^2 + 14 x^3 + x^4) (1 +
    123 x + 4287 x^2 + 56388 x^3 + 340889 x^4 + 1023737 x^5 +
    1586768 x^6 + 1283699 x^7 + 557419 x^8 + 135360 x^9 +
    18857 x^{10 }+ 1486 x^{11} + 61 x^{12} + x^{13}) (-1 + 372 x -
    42413 x^2 + 2173210 x^3 - 60032212 x^4 + 993128332 x^5 -
    10485047801 x^6 + 73443641910 x^7 - 349700572708 x^8 +
    1150917656464 x^9 - 2655012399456 x^{10} + 4349572253312 x^{11} -
    5116898626432 x^{12 }+ 4355854109952 x^{13} - 2691841916928 x^{14} +
    1205257859072 x^{15} - 387743512576 x^{16} + 88113405952 x^{17} -
    13717635072 x^{18} + 1384644608 x^{19} - 81264640 x^{20} +
    2097152 x^{21}) (-1 + 897 x - 284613 x^2 + 44616593 x^3 -
    4029293589 x^4 + 228607904694 x^5 - 8623998744010 x^6 +
    225243247921559 x^7 - 4196824245744747 x^8 +
    57062519353551486 x^9 - 576185381851639467 x^{10} +
    4381096629347262750 x^{11} - 25362753972981226023 x^{12 }+
    112745987919002029131 x^{13} - 387224132372088067387 x^{14} +
    1031345440365206915058 x^{15} - 2132798841380783388554 x^{16} +
    3419096019693872449756 x^{17} - 4228713744245639777841 x^{18} +
    4000304340185270192004 x^{19} - 2855800746727609888656 x^{20 }+
    1508442695853085266432 x^{21} - 572933701954998116352 x^{22} +
    150113917428371795968 x^{23} - 25482017209198657536 x^{24} +
    2527785821948608512 x^{25} - 119775081782050816 x^{26} +
    1530848398540800 x^{27}) $ }
   \bc $ \mbox{--------------------------------------    } $ \ec

   \noindent \parbox[t]{5.8in}{${\cal F}^{TnC}_{14}(x) = $ \\ $ \ds  x+845 x^2+232861 x^3+93547152 x^4+38161604217 x^5+16179908259446 x^6+6955967193562675 x^7+3013814772945044388 x^8+1310571139307486059744
x^9+570941650352643832290496 x^{10}+248946408250205225732190657 x^{11}+108594562872624965291980865517 x^{12}+47380824041848615032082234543538 x^{13}+20674860253582882636086734473452690
x^{14}+9022044500672394796169951103350465307 x^{15}+3937117874038266866017108373849625832406 x^{16}+ $ \\ $ 1718135244552626203898654459395975826590780 x^{17}+ $ \\ $ 749788869420890785643183575437392430014238172
x^{18}+ $ \\ $ 327206540946232850646025525636752798786883194114 x^{19}+ $ \\ $ 142792577247139636166593093344516530392048619109541 x^{20}+ $ \\ $ 62314572410495518341343134325594627886014538946907912
x^{21}+ $ \\ $ 27194042460317564525972154320581961864323701254872880977 x^{22}+ $ \\ $ 11867466249360362556968273924667166748336830671415324128093 x^{23}+ $ \\ $ 5178956700196225617404438425797251018982847819819861731612629
x^{24}+ $ \\ $ 2260094378398440680979300686322330296192434522169114137118638649 x^{25}+
\ldots $  }
  \bc $ \mbox{--------------------------------------    } $  \\ $ \mbox{--------------------------------------    } $ \ec

    \noindent
\parbox[t]{5.8in}{ ${\cal F}^{TnC}_{15}(x) = $ \\ $ \ds x+1364 x^2+551613 x^3+325657968 x^4+192461352562 x^5+117686516544594 x^6+72945933567920729 x^7+45643040884844001536 x^8+28724255612327969672932
x^9+18144383836051636673861867 x^{10}+11488975420539457121794585998 x^{11}+7286202085630847737325323395906 x^{12}+4625571064261896320140372765529296
x^{13}+2938460832220940936831456750219444994 x^{14}+  $  \\  $1867516598405888914804775809449705242513 x^{15}+   $  \\  $1187226104290641513303717685425133226997878 x^{16}+  $  \\  $754890354123995415719107703048801759286901762
x^{17}+  $  \\  $480051335501492838673135121095806203328961433789 x^{18}+  $  \\  $305299750808514058182597036323917736265092146019244 x^{19}+  $  \\  $194172678923920991475888643778258816520058079987883396
x^{20}+  $  \\  $123499386994065384580259878272101608357026769697620157317 x^{21}+  $  \\  $78550926840805376586585823096710973035451347336859903278567 x^{22}+  $  \\  $49962512246307590120896780660550059579267388047814570691961898
x^{23}+  $  \\  $31779088949936319299332029947934364090037973973942235275728958766 x^{24}+  $  \\  $20213493680676849470640814622587215937195100630627182727303982687346
x^{25}+    \ldots $ }

    \bc $ \mbox{--------------------------------------    } $  \\ $ \mbox{--------------------------------------    } $ \ec

   \noindent
\parbox[t]{5.8in}{${\cal F}^{TnC}_{16}(x) =$ \\ $ \ds
x+2209 x^2+1344837 x^3+1247739997 x^4+1176032485833 x^5+1160288515380121 x^6+1166255193469341005 x^7+1185537358216390245429 x^8+1212265596856254020971761
x^9+1243579048192809989812145169 x^{10}+1277895262634953569701628954885 x^{11}+1314373438943841131293557932317677 x^{12}+1352565836514823616126618920708268489
x^{13}+1392241531890106476062759414188473371689 x^{14}+ $ \\ $ 1433288723930969312683797470832044615749917 x^{15}+ $  \\  $1475661751488339174448720884931066509880982565
x^{16}+ $  \\  $1519351974798138063352510239397485429264719411585 x^{17}+ $  \\  $1564371765333443333756180310291660713930922038383681 x^{18}+ $  \\  $1610745673603442972346912013817118708336327750901642421
x^{19}+ $  \\  $1658505551969350049980484820890118350456507321682207716957 x^{20}+ $  \\  $1707687862064664321324867899251182896438911982818451145146329 x^{21}+ $  \\  $1758332192736517484633412955904639221829026592583003083532989913
x^{22}+ $  \\  $1810480450502483013757516923306739236326466951416769104534271300797 x^{23}+ $  \\  $1864176424622422017964522716857650762175540568219077044979790831688277
x^{24}+ $  \\  $1919465561436274265978290290338211439781915616861551057610847225732594161 x^{25}+  \ldots
 $}

     \bc $ \mbox{--------------------------------------    } $  \\ $ \mbox{--------------------------------------    } $ \ec

   \noindent
 \parbox[t]{5.8in}{${\cal F}^{TnC}_{17}(x) =$ \\ $ \ds
x+3571 x^2+3215041 x^3+4444339751 x^4+6151532634859 x^5+8851836530897104 x^6+12939617301085024529 x^7+19127617320816121366452 x^8+28474590328963725085347639
x^9+42589885261525799514745869499 x^{10}+63905315558068166767129771712679 x^{11}+96096516725583526494032612816236112 x^{12}+144717421460086179237227147299450659836
x^{13}+218159686854741111724361184739620835897302 x^{14}+ $  \\  $329102191183638563664798864443258409828820241 x^{15}+ $  \\  $496701121655668982938232734328368052150528489952
x^{16}+ $  \\  $749898916351637898118940193686340894727637366003716 x^{17}+ $  \\  $1132423640596136202261956753918967417943579199078073338 x^{18}+ $  \\  $1710342364351588023476012453033776190031904865314169651567
x^{19}+ $  \\  $2583473277479531260220823597300820365910480560545500758760898 x^{20}+ $  \\  $3902627819968835955248721578281365478246482855981064422162312401 x^{21}+ $  \\  $5895661154947594901443977267554991231621289042216368497546032187540
x^{22}+ $  \\  $8906830825593693661442159214932408949457259307858939336106958487041559 x^{23}+ $  \\  $13456262232647550295196744180048374236536107862397316416427857701132566515
x^{24}+ $  \\  $20329792730599700675924568044783210931387539949914868325206259268106800414972 x^{25}+  \ldots
 $ }

     \bc $ \mbox{--------------------------------------    } $  \\ $ \mbox{--------------------------------------    } $ \ec

   \noindent \parbox[t]{5.8in}{${\cal F}^{TnC}_{18}(x) =$ \\ $ \ds x+5780 x^2+7801165 x^3+16788557123 x^4+36655273043812 x^5+84238150673398010 x^6+197983223809129991221 x^7+472116548787108690744794
x^8+1135126364276555282932647235 x^9+2742392002573492985349801499994 x^{10}+6643863136834483490999934197961784 x^{11}+16121626293866934494559934285910094227
x^{12}+39156115632530892545028999376876628331958 x^{13}+ $ \\ $ 95153141325329485855450526482781928401051570 x^{14}+ $  \\  $231303071218305347597566406802770824393462101445
x^{15}+ $  \\  $562364331191836924937840680006541962007395307421326 x^{16}+ $  \\  $1367412236724940008404355387662509090232141681352022714 x^{17}+ $  \\  $3325121329527473398490899002791157147316127401748635779509
x^{18}+ $  \\  $8085946346778760589100083264115676676177788707823547122857793 x^{19}+ $  \\  $19663604828744357778377718293621252172927077215019804895654629084 x^{20}+ $  \\  $47819014551205735095235300385428033892361857458316584177576367043214
x^{21}+ $  \\  $116289670661881088029330391385128725361964765661104284030666283103331222 x^{22}+282802621230386706038119542241848749659558713589758116655044031017408316646
x^{23}+ $  \\  $687743912708372591876329863661928036510919481432015721376343008774762855424214 x^{24}+ $  \\  $1672517552800166546609792709261336772742668582850638721533836534157959086106144871
x^{25}+ \ldots
$ }

\newpage

{\bf  5.1.   \ \ Torus Grid  $TG^{(p)}_{m}(n)$ and Klein bottle  $KB^{(p)}_{m}(n)$ ($2 \leq m \leq 10$) }

\normalsize

  $$\ds    {\cal F}^{TG}_{m,p}(x) \stackrel{\rm def}{=} \sum_{n\geq 1}^{\infty} f^{TG}_{m,p}(n)x^{n}
\mbox{   and  } \ds   {\cal F}^{KB}_{m,p}(x) \stackrel{\rm def}{=} \sum_{n\geq 1}^{\infty} f^{KB}_{m,p}(n)x^{n}  $$

\small

\bc $ 0 \leq p \leq m-1$,
$TG^{(0)}_{m}(n) \equiv C_{m} \times C_{n}$)\ec

   \small
 \noindent \parbox[t]{5.8in}{
$ {\cal F}^{TG}_{2,0}(x)  =   {\cal F}^{KB}_{2,1}(x) = $ \\ $  \ds  -\frac{2 x (1+3 x)}{(1+x) (-1+3 x)} -\frac{8 x^2}{(-1+2 x) (1+2 x)} = $ \\ $  2 x+18 x^2+26 x^3+114 x^4+242 x^5+858 x^6 + 2186 x^7+7074 x^8+19682 x^9+61098 x^{10}+177146 x^{11}+539634 x^{12}+1594322 x^{13}+4815738
x^{14}+14348906 x^{15}+43177794 x^{16}+129140162 x^{17}+387944778 x^{18}+1162261466 x^{19}+3488881554 x^{20}+10460353202 x^{21}+31389448218 x^{22}+94143178826
x^{23}+282463090914 x^{24}+847288609442 x^{25}+ \ldots $}
\bc $ \mbox{--------------------------------------    } $ \ec
\noindent
\parbox[t]{5.8in}{
$ \ds  {\cal F}^{TG}_{2,1}(x) =  {\cal F}^{KB}_{2,0}(x) = $ \\ $
\ds
-\frac{2 x (1+3 x)}{(1+x) (-1+3 x)} + \frac{4 x}{1-4 x^2} = $ \\ $ 6 x+10 x^2+42 x^3+82 x^4+306 x^5+730 x^6+2442 x^7+6562 x^8+20706 x^9+59050 x^{10}+181242 x^{11}+531442 x^{12}+1610706 x^{13}+4782970 x^{14}+14414442
x^{15}+43046722 x^{16}+129402306 x^{17}+387420490 x^{18}+1163310042 x^{19}+3486784402 x^{20}+10464547506 x^{21}+31381059610 x^{22}+94159956042 x^{23}+282429536482
x^{24}+847355718306 x^{25}
+  \dots  $}
   \bc $ \mbox{--------------------------------------    } $ \\ $ \mbox{--------------------------------------    } $ \ec
 \noindent
\parbox[t]{5.8in}{
$ \ds  {\cal F}^{TG}_{3,0}(x) = $ \\ $ \ds  -\frac{2x(1+ 11x + 4 x^2)}{(1+x)(-1 + 3 x + x^2)}$ = \\ $ 2 x+26 x^2+68 x^3+242 x^4+782 x^5+2600 x^6+8570 x^7+28322 x^8+93524 x^9+308906 x^{10}+1020230 x^{11}+3369608 x^{12}+11129042 x^{13}+36756746
x^{14}+121399268 x^{15}+400954562 x^{16}+1324262942 x^{17}+4373743400 x^{18}+14445493130 x^{19}+47710222802 x^{20}+157576161524 x^{21}+520438707386
x^{22}+1718892283670 x^{23}+5677115558408 x^{24}+18750238958882 x^{25} +  \dots  $}
\bc $ \mbox{--------------------------------------    } $ \ec
\noindent
\parbox[t]{5.8in}{
$ \ds  {\cal F}^{TG}_{3,1}(x) = {\cal F}^{TG}_{3,2}(x) = $ \\ $ \ds
-\frac{2x \left(4+2 x+x^2\right)}{(1+x) \left(-1+3 x+x^2\right)} = $ \\ $
8 x+20 x^2+74 x^3+236 x^4+788 x^5+2594 x^6+8576 x^7+28316 x^8+93530 x^9+308900 x^{10}+1020236 x^{11}+3369602 x^{12}+11129048 x^{13}+36756740
x^{14}+121399274 x^{15}+400954556 x^{16}+1324262948 x^{17}+4373743394 x^{18}+14445493136 x^{19}+47710222796 x^{20}+157576161530 x^{21}+520438707380
x^{22}+1718892283676 x^{23}+5677115558402 x^{24}+18750238958888 x^{25} +  \dots  $}
\bc $ \mbox{--------------------------------------    } $ \ec
\noindent
\parbox[t]{5.8in}{
$ \ds {\cal F}^{KB}_{3,0}(x) = {\cal F}^{KB}_{3,1}(x) = {\cal F}^{KB}_{3,2}(x) = $ \\ $ \ds
-\frac{(2x (3 + 2 x)}{-1 + 3 x + x^2}= $ \\ $6 x+22 x^2+72 x^3+238 x^4+786 x^5+2596 x^6+8574 x^7+28318 x^8+93528 x^9+308902 x^{10}+1020234 x^{11}+3369604 x^{12}+11129046 x^{13}+36756742
x^{14}+121399272 x^{15}+400954558 x^{16}+1324262946 x^{17}+4373743396 x^{18}+14445493134 x^{19}+47710222798 x^{20}+157576161528 x^{21}+520438707382
x^{22}+1718892283674 x^{23}+5677115558404 x^{24}+18750238958886 x^{25}
+  \dots  $}
   \bc $ \mbox{--------------------------------------    } $ \\ $ \mbox{--------------------------------------    } $ \ec
\noindent
\parbox[t]{5.8in}{${\cal F}^{TG}_{4,0}(x) = \ds
\frac{2 x (1+21 x+6 x^2-20 x^3)}{(1+x) (1+2 x) (1-7 x+4 x^2)} -\frac{8
x^2 (-7+25 x^2)}{(-1+x) (1+x) (-1+5 x) (1+5 x)} + \frac{8 x^2}{1-4 x^2} = $ \\
$ \ds 2 x+114 x^2+242 x^3+2970 x^4+10442 x^5+98466 x^6+426386 x^7+3500970 x^8+17323226 x^9+129930354 x^{10}+703463906 x^{11}+4970993658 x^{12}+28564983722
x^{13}+194231313474 x^{14}+1159909450802 x^{15}+7696445791050 x^{16}+47099249906042 x^{17}+307759067766546 x^{18}+1912510703585666 x^{19}+12377791111168410
x^{20}+77659350883118666 x^{21}+499635602835227874 x^{22}+3153433215088906706 x^{23}+20213772870411999978 x^{24}+128048212205068924442 x^{25}+
  \ldots $}
 \bc $ \mbox{--------------------------------------    } $ \ec
\parbox[t]{5.8in}{${\cal F}^{TG}_{4,1}(x) = {\cal F}^{TG}_{4,3}(x) = {\cal F}^{KB}_{4,0}(x)= {\cal F}^{KB}_{4,2}(x) = $ \\ $\ds
-\frac{6 (-1+x) x (1+2 x)}{(1+x) \left(1-7 x+4 x^2\right)} + \frac{8 \left(x+5 x^3\right)}{(-1+x)
(1+x) (-1+5 x) (1+5 x)} + \frac{4 x}{1-4 x^2} = $ \\$
\noindent
18 x+42 x^2+522 x^3+1650 x^4+16818 x^5+66954 x^6+583146 x^7+2718690 x^8+21231522 x^9+110395002 x^{10}+801128346 x^{11}+4482696018 x^{12}+31006422738
x^{13}+182024216682 x^{14}+1220944738122 x^{15}+7391269747650 x^{16}+48625129336578 x^{17}+300129672186714 x^{18}+1950657678339066 x^{19}+12187056243692850
x^{20}+78613025207913522 x^{21}+494867231236419402 x^{22}+3177275073032617386 x^{23}+20094563580794109858 x^{24}+128644258652957048418 x^{25}+
 \ldots $}
 \bc $ \mbox{--------------------------------------    } $ \ec
\parbox[t]{5.8in}{${\cal F}^{TG}_{4,2}(x) = \ds
-\frac{2 x \left(-5+3 x+6 x^2+4 x^3\right)}{(1+x) (1+2 x) \left(1-7 x+4 x^2\right)} +\frac{48 x^2}{(-1+x)
(1+x) (-1+5 x) (1+5 x)} + \frac{8 x^2}{1-4 x^2} = $\\$
\noindent
10 x+90 x^2+274 x^3+2898 x^4+10570 x^5+98202 x^6+426898 x^7+3499938 x^8+17325274 x^9+129926250 x^{10}+703472098 x^{11}+4970977266 x^{12}+28565016490
x^{13}+194231247930 x^{14}+1159909581874 x^{15}+7696445528898 x^{16}+47099250430330 x^{17}+307759066717962 x^{18}+1912510705682818 x^{19}+12377791106974098
x^{20}+77659350891507274 x^{21}+499635602818450650 x^{22}+3153433215122461138 x^{23}+20213772870344891106 x^{24}+128048212205203142170 x^{25}+  \ldots $
}
 \bc $ \mbox{--------------------------------------    } $ \ec
\parbox[t]{5.8in}{${\cal F}^{KB}_{4,1}(x) = {\cal F}^{KB}_{4,3}(x)  = $ \\
 $\ds  -\frac{6 (-1+x) x (1+2 x)}{(1+x) \left(1-7 x+4 x^2\right)} + \frac{52 x^2-100 x^4}{(-1+x) (1+x)
(-1+5 x) (1+5 x)} +\frac{8 x^2}{1-4 x^2}
  = $ \\$
\noindent 6 x+102 x^2+258 x^3+2934 x^4+10506 x^5+98334 x^6+426642 x^7+3500454 x^8+17324250 x^9+129928302 x^{10}+703468002 x^{11}+4970985462 x^{12}+28565000106
x^{13}+194231280702 x^{14}+1159909516338 x^{15}+7696445659974 x^{16}+47099250168186 x^{17}+307759067242254 x^{18}+1912510704634242 x^{19}+12377791109071254
x^{20}+77659350887312970 x^{21}+499635602826839262 x^{22}+3153433215105683922 x^{23}+20213772870378445542 x^{24}+128048212205136033306 x^{25}
+ \ldots $}
   \bc $ \mbox{--------------------------------------    } $ \\ $ \mbox{--------------------------------------    } $ \ec
\noindent
\parbox[t]{5.8in}{$ \ds {\cal F}^{TG}_{5,0}(x) = \ds
\frac{2x \left(1+116 x-250 x^2-1026 x^3+1646 x^4+1675 x^5-2401 x^6+58 x^7+372 x^8-48 x^9\right)}{\left(-1+9
x-4 x^2-22 x^3+3 x^4\right) \left(-1-4 x+4 x^2+10 x^3-8 x^4-x^5+x^6\right)} =$ \\
$ \ds 2 x+242 x^2+782 x^3+10442 x^4+67832 x^5+628382 x^6+4831612 x^7+40904442 x^8+329212322 x^9+2720543472 x^{10}+22172526752 x^{11}+182023143782
x^{12}+1488626009132 x^{13}+12198610087752 x^{14}+99857648185292 x^{15}+817880965553242 x^{16}+6696906730979202 x^{17}+54843316086466622 x^{18}+449095741864069442
x^{19}+3677664470840904912 x^{20}+30115901734205683702 x^{21}+246617882571678884332 x^{22}+2019531708324695230722 x^{23}+16537815439466805799542
x^{24}+135426887276826051341032 x^{25}+ \ldots $}
   \bc $ \mbox{--------------------------------------    } $  \ec

\noindent
\parbox[t]{5.8in}{$ \ds {\cal F}^{TG}_{5,1}(x) = {\cal F}^{TG}_{5,4}(x) = $ \\ $ \ds
\frac{-2 x (-16 + 34 x + 125 x^2 + 61 x^3 - 421 x^4 - 290 x^5 + 566 x^6 -
    133 x^7 + 3 x^8 + 3 x^9)}{(-1 + 9 x - 4 x^2 - 22 x^3 +
    3 x^4) (-1 - 4 x + 4 x^2 + 10 x^3 - 8 x^4 - x^5 + x^6)}=$ \\
$ \ds
32 x+92 x^2+1362 x^3+7952 x^4+78402 x^5+583082 x^6+5025732 x^7+40071652 x^8+332785472 x^9+2705210252 x^{10}+22238327242 x^{11}+181740763342
x^{12}+1489837842362 x^{13}+12193409497792 x^{14}+99879966576102 x^{15}+817785185872272 x^{16}+6697317770881812 x^{17}+54841552102546412 x^{18}+449103312027847732
x^{19}+3677631983367604002 x^{20}+30116041154693649302 x^{21}+246617284246523379292 x^{22}+2019534276046249700682 x^{23}+16537804420050484264322
x^{24}+135426934566819733665752 x^{25}
+ \ldots $}   \bc $ \mbox{--------------------------------------    } $  \ec

\noindent
\parbox[t]{5.8in}{$ \ds {\cal F}^{TG}_{5,2}(x) = {\cal F}^{TG}_{5,3}(x) = $ \\ $ \ds
\frac{2x \left(6+46 x-85 x^2-201 x^3-84 x^4+265 x^5+104 x^6-67 x^7+12 x^8-3 x^9\right)}{\left(-1+9 x-4 x^2-22 x^3+3 x^4\right)
\left(-1-4 x+4 x^2+10 x^3-8 x^4-x^5+x^6\right)}=$ \\
$ \ds
12 x+152 x^2+1022 x^3+9412 x^4+71952 x^5+610862 x^6+4906052 x^7+40585712 x^8+330578112 x^9+2714684802 x^{10}+22197663412 x^{11}+181915278242
x^{12}+1489088898082 x^{13}+12196623621232 x^{14}+99866173082952 x^{15}+817844380915012 x^{16}+6697063734348162 x^{17}+54842642304392132 x^{18}+449098633409635132
x^{19}+3677652061729775452 x^{20}+30115954988094303612 x^{21}+246617654031804132482 x^{22}+2019532689107058341092 x^{23}+16537811230424301945262
x^{24}+135426905339996319417002 x^{25} + \ldots $}

   \bc $ \mbox{--------------------------------------    } $  \ec

\noindent
\parbox[t]{5.8in}{$ \ds {\cal F}^{KB}_{5,0}(x) = {\cal F}^{KB}_{5,1}(x) ={\cal F}^{KB}_{5,2}(x) ={\cal F}^{KB}_{5,3}(x) ={\cal F}^{KB}_{5,4}(x) = $ \\ $ \ds
2\frac{x \left(-9+8 x+66 x^2-12 x^3\right)}{-1+9 x-4 x^2-22 x^3+3 x^4}=$ \\
$ \ds
18 x+146 x^2+1110 x^3+9034 x^4+73708 x^5+603254 x^6+4939036 x^7+40443834 x^8+331187898 x^9+2712066716 x^{10}+22208901612 x^{11}+181867045390
x^{12}+1489295898004 x^{13}+12195735265160 x^{14}+99869985500680 x^{15}+817828019825562 x^{16}+6697133948287830 x^{17}+54842340980068742 x^{18}+449099926547807034
x^{19}+3677646512207132764 x^{20}+30115978803956317906 x^{21}+246617551825666781576 x^{22}+2019533127726262262854 x^{23}+16537809348083275643742
x^{24}+135426913418091631501308 x^{25} + \ldots $}

\bc $ \mbox{--------------------------------------    } $ \\ $ \mbox{--------------------------------------    } $ \ec
\parbox[t]{5.8in}{
$ {\cal F}^{TG}_{6,0}(x) = $ \\ $ \ds
-(2 x (1+155 x+215 x^2-3684 x^3-4669 x^4+22509 x^5+18491 x^6-43072 x^7-10982 x^8+6708 x^9+1368
x^{10})) / $ \\ $((-1+x) (-1+2 x) (1+2 x) (1+3 x) (1+3 x+x^2) (1+6 x+x^2) (-1+17 x-37 x^2+12 x^3)) - $ \\ $(12 x^2 (35-2508 x^2+38371 x^4-212774 x^6+506293 x^8-564506 x^{10}+295808 x^{12}-68120 x^{14}+5376 x^{16})) /$ \\ $((-1+x) (1+x)
(-1+2 x) (1+2 x) (1-3 x+x^2) (1+3 x+x^2)(1-5 x+3 x^2) (1+5 x+3 x^2)(-1+14 x-22 x^2+8 x^3)
(1+14 x+22 x^2+8 x^3))  $ \\ $ \ds  - \frac{12 x^2 (-3+7 x) (3+7 x)}{(-1+x) (1+x) (-1+7 x) (1+7 x)} + \frac{8 x^2}{1-4 x^2}$
}
\bc $ \mbox{--------------------------------------    } $ \ec
\parbox[t]{5.8in}{
$ {\cal F}^{TG}_{6,0}(x) =  $ \\ $\ds 2 x+858 x^2+2600 x^3+98466 x^4+628382 x^5+16448400 x^6+134721638 x^7+2995602834 x^8+28432011752 x^9+566597492178 x^{10}+5985882924254
x^{11}+109903205061360 x^{12}+1259741716585718 x^{13}+
 (21736984452051810 x^{14}+265098868583817320 x^{15}+4365796637993895186 x^{16}+55786599952981377950 x^{17}+887421840845709378960 x^{18}+11739543540193217824262
x^{19}+182043417096228583263666 x^{20})+ 2470429220177037909717224 x^{21}+37601976458209902773864082 x^{22}+519868617913174476410785502 x^{23}+7806648458062158266900754864 x^{24}+109399360988897448327149426582
x^{25}+ \ldots $ }
\bc $ \mbox{--------------------------------------    } $ \ec
\parbox[t]{5.8in}{
$ {\cal F}^{TG}_{6,1}(x) =  {\cal F}^{TG}_{6,5}(x) =$ \\$ \ds
-(2 x (10+38 x-163 x^2-1254 x^3-595 x^4+6357 x^5+4568 x^6-6865 x^7-1700 x^8+1548 x^9+216 x^{10}))/$ \\ $((-1+x) (-1+2 x) (1+2 x) (1+3
x)(1+3 x+x^2) (1+6 x+x^2) (-1+17 x-37 x^2+12 x^3))-$ \\ $(6 x (5-282 x^2+3121 x^4-24860 x^6+110739 x^8-180644
x^{10}+115360 x^{12}-26936 x^{14}+1472 x^{16}))/$ \\ $((-1+x) (1+x) (-1+2 x) (1+2 x) (1-3 x+x^2) (1+3 x+x^2) (1-5 x+3 x^2)
(1+5 x+3 x^2) (-1+14 x-22 x^2+8 x^3) (1+14 x+22 x^2+8 x^3))+$ \\ $ \ds
\frac{12 x (1+7 x^2)}{(-1+x) (1+x) (-1+7 x) (1+7 x)}+\frac{4 x}{1-4 x^2}
$}
\bc $ \mbox{--------------------------------------    } $ \ec
\parbox[t]{5.8in}{
$ {\cal F}^{TG}_{6,1}(x) =  {\cal F}^{TG}_{6,5}(x) =$ \\$ \ds
66 x+196 x^2+7728 x^3+43276 x^4+1238826 x^5+9280810 x^6+220306950 x^7+1959709660 x^8+41039318112 x^9+412627702756 x^{10}+7869400145706
x^{11}+86839919979346 x^{12}+1542287856771702 x^{13}+18274564632840964 x^{14}+307535029174761288 x^{15}+3845645631265057852 x^{16}+62162525468824018554
x^{17}+809264730153945546874 x^{18}+12697619995607789219814 x^{19}+170298892278149136083116 x^{20}+2614399472275480088901744 x^{21}+35837112519944207795899156
x^{22}+541503291293519991461336874 x^{23}+7541438501057517562114538146 x^{24}+112650454910746254596670459126 x^{25} + \ldots
$}

\bc $ \mbox{--------------------------------------    } $ \ec
\parbox[t]{5.8in}{
$ {\cal F}^{TG}_{6,2}(x) =  {\cal F}^{TG}_{6,4}(x) =$ \\$ \ds
-(2 x (10+32 x-109 x^2-714 x^3-1207 x^4+2265 x^5+6872 x^6-223 x^7-1142 x^8+624 x^9+72 x^{10}))/$ \\ $((-1+x) (-1+2 x) (1+2 x) (1+3
x) (1+3 x+x^2) (1+6 x+x^2) (-1+17 x-37 x^2+12 x^3))+$ \\ $ (6 x^2 (-43+466 x^2+97 x^4-3308 x^6+10207
x^8-25688 x^{10}+22100 x^{12}-6624 x^{14}+768 x^{16}))/$ \\ $((-1+x) (1+x) (-1+2 x) (1+2 x) (1-3 x+x^2) (1+3 x+x^2) (1-5
x+3 x^2) (1+5 x+3 x^2) (-1+14 x-22 x^2+8 x^3) (1+14 x+22 x^2+8 x^3)) + $ \\ $ \ds
\frac{96 x^2}{(-1+x) (1+x) (-1+7 x) (1+7 x)} +\frac{8 x^2}{1-4 x^2}
$}
\bc $ \mbox{--------------------------------------    } $ \ec
\parbox[t]{5.8in}{
$ {\cal F}^{TG}_{6,2}(x) =  {\cal F}^{TG}_{6,4}(x) =$ \\$ \ds
20 x+546 x^2+3266 x^3+92394 x^4+649280 x^5+16288572 x^6+135413648 x^7+2990869818 x^8+28455356186 x^9+566448456186 x^{10}+5986674473624
x^{11}+109898350832964 x^{12}+1259768592693176 x^{13}+21736823342530506 x^{14}+265099781460879026 x^{15}+4365791234629904250 x^{16}+55786630962853983728
x^{17}+887421658574519641164 x^{18}+11739544593606366893792 x^{19}+182043410928107923513914 x^{20}+2470429255961988323286218 x^{21}+37601976249114446028812058
x^{22}+519868619128808596283356232 x^{23}+7806648450967193961130178964 x^{24}+109399361030193216423045910280 x^{25} + \ldots
$}
\bc $ \mbox{--------------------------------------    } $ \ec
\parbox[t]{5.8in}{
$ {\cal F}^{TG}_{6,3}(x) = $ \\$ \ds
-(2 x (7+95 x-295 x^2-2292 x^3+731 x^4+11745 x^5+2153 x^6-15772 x^7-4148 x^8+2808 x^9+648 x^{10}))/ $ \\ $((-1+x) (-1+2 x) (1+2 x) (1+3
x) (1+3 x+x^2) (1+6 x+x^2) (-1+17 x-37 x^2+12 x^3)) -$ \\ $
(12 x (1+102 x^2-1225 x^4+9488 x^6-31443 x^8+41180 x^{10}-14362 x^{12}-2836 x^{14}+1120 x^{16}))/ $ \\ $((-1+x) (1+x) (-1+2 x) (1+2
x) (1-3 x+x^2) (1+3 x+x^2) (1-5 x+3 x^2) (1+5 x+3 x^2) (-1+14 x-22 x^2+8 x^3) (1+14 x+22
x^2+8 x^3)) + $ \\ $ \ds \frac{12 x (1+7 x^2)}{(-1+x) (1+x) (-1+7 x) (1+7 x)} + \frac{4 x}{1-4 x^2}
$}
\bc $ \mbox{--------------------------------------    } $ \ec
\parbox[t]{5.8in}{
$ {\cal F}^{TG}_{6,3}(x) = $ \\$ \ds
42 x+274 x^2+6840 x^3+46498 x^4+1211262 x^5+9396232 x^6+219469278 x^7+1963685122 x^8+41013093312 x^9+412763242834 x^{10}+7868553742014
x^{11}+86844528775432 x^{12}+1542259952215518 x^{13}+18274721236048642 x^{14}+307534097147126400 x^{15}+3845650951522451074 x^{16}+62162494103450793390
x^{17}+809264910889307863816 x^{18}+12697618935604681128702 x^{19}+170298898417860124862338 x^{20}+2614399436368451516402304 x^{21}+35837112728513906398644274
x^{22}+541503290075625092403221982 x^{23}+7541438508142749946738210312 x^{24}+112650454869408625049062420062 x^{25} + \ldots
$}
\bc $ \mbox{--------------------------------------    } $ \ec
\parbox[t]{5.8in}{
$ {\cal F}^{KB}_{6,0}(x) =  {\cal F}^{KB}_{6,2}(x) =  {\cal F}^{KB}_{6,4}(x) = $ \\ $ \ds
\frac{2 x (7+17 x-207 x^2-106 x^3+850 x^4+12 x^5-168 x^6)}{(-1+2 x) (1+2 x)(1+3 x+x^2)
(-1+17 x-37 x^2+12 x^3)}-  $ \\ $ \ds \frac{24 x (-1+3 x^2-66 x^4+48 x^6)}{(-1+2 x) (1+2 x)(-1+14 x-22 x^2+8 x^3) \left(1+14
x+22 x^2+8 x^3\right)}+  $ \\ $ \ds \frac{12 x \left(1+7 x^2\right)}{(-1+x) (1+x) (-1+7 x) (1+7 x)} -\frac{4 x}{(-1+2 x) (1+2 x)}
$}
\bc $ \mbox{--------------------------------------    } $ \ec
\parbox[t]{5.8in}{
$ {\cal F}^{KB}_{6,0}(x) =  {\cal F}^{KB}_{6,2}(x) =  {\cal F}^{KB}_{6,4}(x) = $ \\ $ \ds
54 x+230 x^2+7416 x^3+44382 x^4+1229574 x^5+9319412 x^6+220027470 x^7+1961035326 x^8+41030575488 x^9+412672884830 x^{10}+7869118007046
x^{11}+86841456252900 x^{12}+1542278555236590 x^{13}+18274616833942958 x^{14}+307534718498817456 x^{15}+3845647404684319998 x^{16}+62162515013699348022
x^{17}+809264790399066843476 x^{18}+12697619642273418807534 x^{19}+170298894324719467773342 x^{20}+2614399460306470560540960 x^{21}+35837112589467440671869470
x^{22}+541503290887555025091854694 x^{23}+7541438503419261690355983300 x^{24}+112650454896967044747400670574 x^{25}+  \ldots
$}
\bc $ \mbox{--------------------------------------    } $ \ec
\parbox[t]{5.8in}{
$ {\cal F}^{KB}_{6,1}(x) =  {\cal F}^{KB}_{6,3}(x) =  {\cal F}^{KB}_{6,5}(x) = $ \\ $ \ds
-\frac{2 x \left(-9+15 x+177 x^2-110 x^3-524 x^4+16 x^5+120 x^6\right)}{(-1+2 x) (1+2 x) \left(1+3 x+x^2\right)
\left(-1+17 x-37 x^2+12 x^3\right)}-  $ \\ $ \ds \frac{8 x^2 \left(-39+434 x^2-828 x^4+256 x^6\right)}{(-1+2 x) (1+2 x) \left(-1+14 x-22 x^2+8 x^3\right) \left(1+14
x+22 x^2+8 x^3\right)}  $ \\ $ \ds -\frac{4 x^2 (-5+7 x) (5+7 x)}{(-1+x) (1+x) (-1+7 x) (1+7 x)}-\frac{8 x^2}{(-1+2 x) (1+2 x)}
$}
\bc $ \mbox{--------------------------------------    } $ \ec
\parbox[t]{5.8in}{
$ {\cal F}^{KB}_{6,1}(x) =  {\cal F}^{KB}_{6,3}(x) =  {\cal F}^{KB}_{6,5}(x) = $ \\ $ \ds
18 x+642 x^2+3060 x^3+94386 x^4+642378 x^5+16341720 x^6+135183234 x^7+2992446978 x^8+28447575732 x^9+566498132802 x^{10}+5986410627930
x^{11}+109899968900904 x^{12}+1259759634007074 x^{13}+21736877045671506 x^{14}+265099477168590660 x^{15}+4365793035751103490 x^{16}+55786620626230043946
x^{17}+887421719331582362808 x^{18}+11739544242468651585858 x^{19}+182043412984148141333346 x^{20}+2470429244033671522957524 x^{21}+37601976318812931602107458
x^{22}+519868618723597223009276538 x^{23}+7806648453332182063020149832 x^{24}+109399361016427960391147524578 x^{25}+ \ldots
$}
   \bc $ \mbox{--------------------------------------    } $  \\ $ \mbox{--------------------------------------    } $ \ec  \noindent
    \small
    \parbox[t]{5.8in}{ ${\cal F}^{TG}_{7,0}(x) = $ \\ $ \ds
  {-2x }
(1+1079 x-11309 x^2-163017 x^3 + 1411227 x^4 + 8050358 x^5-59600093 x^6-173474945 x^7 +  1206163636 x^8 + 1756819829 x^9-13269562704 x^{10}-7839247762
x^{11} + 83351457555 x^{12} + 7438024735 x^{13}-
308734737505 x^{14} + 53387827262 x^{15} + 698371664770 x^{16}-196250242748 x^{17}-995090049606 x^{18} + 287671218042 x^{19} + 906343826977 x^{20}-210199757682
x^{21}-523958371651 x^{22} + 71335687392 x^{23} +
185805053315 x^{24}-4904925295 x^{25}-37609720308 x^{26}-2746138593 x^{27} + 3804635885 x^{28} + 559321665 x^{29}-156088534 x^{30}-33319916
x^{31} + 1060080 x^{32} + 523880 x^{33} \left.  + 23870 x^{34})\right/ $} \\
 \noindent      \parbox[t]{5.8in}{$ \ds
((-1 + 27 x-131 x^2-319 x^3 + 1511 x^4-598 x^5-1473 x^6 + 740 x^7 + 55 x^8)
(1 + 13 x-72 x^2-1030 x^3 + 1281 x^4 + 25514 x^5-10606 x^6-289247 x^7 + 70298 x^8 + 1694810 x^9-310132 x^{10}-5443990 x^{11} + 538711 x^{12} + 9882395
x^{13} + 197366 x^{14}-10127513 x^{15}-1486842 x^{16} + 5648377 x^{17} + 1517144 x^{18}-1533577 x^{19}-597207 x^{20} + 147951 x^{21} + 84616 x^{22}-1225 x^{23}-4389 x^{24}-336 x^{25} + 63 x^{26} + 7 x^{27}))$}
   \bc $ \mbox{--------------------------------------    } $  \ec
     \noindent
   \parbox[t]{5.8in}{${\cal F}^{TG}_{7,0}(x) = $  \\   \small $ \ds
       2 x + 2186 x^2 + 8570 x^3 + 426386 x^4 + 4831612 x^5 + 134721638 x^6 + 2142911388 x^7 + 48506694658 x^8 + 878082461552 x^9 + 18309402983496 x^{10} + 348982663373192
x^{11} + 7041475309154414 x^{12} + 137028579300015678 x^{13} +
2728527055102059808 x^{14} + 53544239536458527320 x^{15} + 1060515253541792576418 x^{16} + 20881817475671338188760 x^{17} + 412705594209438378631232
x^{18} + 8137355045911556964755608 x^{19} +
160686577635260457567848056 x^{20} + 3170012863499006576524131402 x^{21} + 62575725299441382116431374292 x^{22} + 1234762642850384408406181694538
x^{23} + 24370659626838874175976096699838 x^{24} +
480931967635756920300968526068362 x^{25}+  \ldots $}
   \bc $ \mbox{--------------------------------------    } $  \ec
     \noindent
  \parbox[t]{5.8in}{${\cal F}^{TG}_{7,1}(x) = {\cal F}^{TG}_{7,6}(x) = $ \\$ \ds
2x
(-64+685 x+9405 x^2-68067 x^3-491175 x^4+2286528 x^5+11364290 x^6-39380117 x^7-133150016 x^8+419766765 x^9+843825807 x^{10}-2859633578
x^{11}-2793548158 x^{12}+11845391151 x^{13}+4887555986 x^{14}-
30035595233 x^{15}-6484474349 x^{16}+50321690006 x^{17}+10166277620 x^{18}-56681098507 x^{19}-14630592369 x^{20}+41052266316 x^{21}+13322442008
x^{22}-17711965765 x^{23}-6758877104 x^{24}+4060686354 x^{25}+
1753182339 x^{26}-397957684 x^{27}-199772595 x^{28}+10591357 x^{29}+8427503 x^{30}+300951 x^{31}-24213 x^{32}+7280 x^{33} +385 x^{34} )\left. \right/$}
\\ \parbox[t]{5.8in}{$ \ds((-1+27 x-131 x^2-319 x^3+1511 x^4-598 x^5-1473 x^6+740 x^7+55 x^8)
(1+13 x-72 x^2-1030 x^3+1281 x^4+25514 x^5-10606 x^6-289247 x^7+70298 x^8+1694810 x^9-310132 x^{10}-5443990 x^{11}+538711 x^{12}+9882395
x^{13}+197366 x^{14}-10127513 x^{15}-1486842 x^{16}+5648377 x^{17}+
1517144 x^{18}-1533577 x^{19}-597207 x^{20}+147951 x^{21}+84616 x^{22}-1225 x^{23}-4389 x^{24}-336 x^{25}+63 x^{26}+7 x^{27}))
 $}
   \bc $ \mbox{--------------------------------------    } $  \ec
     \noindent
  \parbox[t]{5.8in}{ ${\cal F}^{TG}_{7,1}(x) = {\cal F}^{TG}_{7,6}(x) =  $  \\   \small $ \ds
 128 x+422 x^2+24474 x^3+226186 x^4+7202610 x^5+104855690 x^6+2512184138 x^7+43867474800 x^8+936009899568 x^9+17582816323228 x^{10}+358078499108052
x^{11}+6927458003305128 x^{12}+138456926552174358 x^{13}+2710626515293975644 x^{14}+53768532907100055242 x^{15}+1057704536580601904710 x^{16}+20917037738011516021104
x^{17}+412264244230916452785342 x^{18}+8142885565313287985192096 x^{19}+160617274442569182417892638 x^{20}+3170881300284945353876423922 x^{21}+62564842903179537055071304118
x^{22}+1234899010119684176870242642914 x^{23}+24368950807461578197509718153998 x^{24}+480953380854744384385606313513380 x^{25}
 +  \ldots $}
    \bc $ \mbox{--------------------------------------    } $  \ec
     \noindent
 ${\cal F}^{TG}_{7,2}(x) = {\cal F}^{TG}_{7,5}(x) = $ \\   \small \parbox[t]{5.8in}{$ \ds
2x (-15-358 x+6500 x^2+20581 x^3-379686 x^4-614349 x^5+9623516 x^6+6871193 x^7-133665881 x^8+16050270 x^9+1034270306 x^{10}-832552543
x^{11}-4149076038 x^{12}+5536086744 x^{13}+8088738201 x^{14}-
16598435294 x^{15}-3468558824 x^{16}+24425985465 x^{17}-13759119120 x^{18}-16229019938 x^{19}+26995435172 x^{20}+1959375906 x^{21}-22448277564
x^{22}+2696371002 x^{23}+9963775850 x^{24}-1062648142 x^{25}-
2407422144 x^{26}+69471796 x^{27}+295592527 x^{28}+18818870 x^{29}-13539904 x^{30}-1719872 x^{31}+95417 x^{32}+18060 x^{33}+385
x^{34})\left. \right/$ } \\ \parbox[t]{5.8in}{$ \ds
((-1+27 x-131 x^2-319 x^3+1511 x^4-598 x^5-1473 x^6+740 x^7+55 x^8)
(1+13 x-72 x^2-1030 x^3+1281 x^4+25514 x^5-10606 x^6-289247 x^7+70298 x^8+1694810 x^9-310132 x^{10}-5443990 x^{11}+538711 x^{12}+9882395
x^{13}+197366 x^{14}-10127513 x^{15}-1486842 x^{16}+5648377 x^{17}+
1517144 x^{18}-1533577 x^{19}-597207 x^{20}+147951 x^{21}+84616 x^{22}-1225 x^{23}-4389 x^{24}-336 x^{25}+63 x^{26}+7 x^{27}))
 $}
   \bc $ \mbox{--------------------------------------    } $  \ec
     \noindent
   ${\cal F}^{TG}_{7,2}(x) ={\cal F}^{TG}_{7,5}(x) =  $  \\   \small \parbox[t]{5.8in}{$ \ds
30 x+1136 x^2+11664 x^3+365766 x^4+5281838 x^5+127985958 x^6+2214495516 x^7+47553306474 x^8+889459669478 x^9+18163991799382 x^{10}+350778866259390
x^{11}+7018826147346050 x^{12}+137311199904697394 x^{13}+2724978632865004126 x^{14}+53588649260529868984 x^{15}+1059958419863026829840 x^{16}+20888792565117299348136
x^{17}+412618173146784864901112 x^{18}+8138450397877849657730138 x^{19}+160672850996776147614394756 x^{20}+3170184866337306023656675304 x^{21}+62573569893828338148584827972
x^{22}+1234789651977316214729121486960 x^{23}+24370321173568753474896143092952 x^{24}+480936208782384939197706128173860 x^{25}
 +  \ldots $}
    \bc $ \mbox{--------------------------------------    } $  \ec
     \noindent
\parbox[t]{5.8in}{ ${\cal F}^{TG}_{7,3}(x) = {\cal F}^{TG}_{7,4}(x) = $ \\  \small $ \ds
2x (-15-99 x+515 x^2+19678 x^3+5209 x^4-643315 x^5-679399 x^6+8497104 x^7+5597117 x^8-72162939 x^9+81359410 x^{10}+372723641 x^{11}-1374105581
x^{12}-605370011 x^{13}+7017383133 x^{14}-1228664409 x^{15}-
18147310002 x^{16}+6393165275 x^{17}+26715714916 x^{18}-11222971387 x^{19}-22578171631 x^{20}+10522127180 x^{21}+10866240778 x^{22}-5344380437
x^{23}-3041644711 x^{24}+1293688512 x^{25}+511424493 x^{26}-110852572 x^{27}-31933238 x^{28}+1560783 x^{29}-1355067 x^{30}-368263 x^{31}+50617 x^{32}+12670 x^{33}+385 x^{34}) \left. \right/$ }\\
\parbox[t]{5.8in}{$ \ds
((-1+27 x-131 x^2-319 x^3+1511 x^4-598 x^5-1473 x^6+740 x^7+55 x^8)
(1+13 x-72 x^2-1030 x^3+1281 x^4+25514 x^5-10606 x^6-289247 x^7+70298 x^8+1694810 x^9-310132 x^{10}-5443990 x^{11}+538711 x^{12}+9882395
x^{13}+197366 x^{14}-10127513 x^{15}-1486842 x^{16}+5648377 x^{17}+
1517144 x^{18}-1533577 x^{19}-597207 x^{20}+147951 x^{21}+84616 x^{22}-1225 x^{23}-4389 x^{24}-336 x^{25}+63 x^{26}+7 x^{27}))
 $}
   \bc $ \mbox{--------------------------------------    } $  \ec
     \noindent
   \parbox[t]{5.8in}{ ${\cal F}^{TG}_{7,3}(x) ={\cal F}^{TG}_{7,4}(x) =  $  \\   \small $ \ds
30 x+618 x^2+16382 x^3+282368 x^4+6242980 x^5+114844424 x^6+2375085764 x^7+45491597510 x^8+915087730818 x^9+17840615564606 x^{10}+354820611593730
x^{11}+6968077366210606 x^{12}+137946609575486742 x^{13}+2717011641917464824 x^{14}+53688457838042204844 x^{15}+1058707502097050507466 x^{16}+20904466559904256508804
x^{17}+412421752194947373969974 x^{18}+8140911688520111488953128 x^{19}+160642008082884457694695846 x^{20}+3170571356074785654054128186 x^{21}+62568726767188747969016427850
x^{22}+1234850341076421412471622715496 x^{23}+24369560677167046111278862657074 x^{24}+480945738559287010379566254808738 x^{25}
 +  \ldots $}

   \bc $ \mbox{--------------------------------------    } $  \ec

\noindent
\parbox[t]{5.8in}{$ \ds {\cal F}^{KB}_{7,0}(x) = {\cal F}^{KB}_{7,1}(x) ={\cal F}^{KB}_{7,2}(x) ={\cal F}^{KB}_{7,3}(x) ={\cal F}^{KB}_{7,4}(x) = {\cal F}^{KB}_{7,5}(x) ={\cal F}^{KB}_{7,6}(x) =$ \\
$ \ds
\frac{2 x (-27+262 x+957 x^2-6044 x^3+2990 x^4+8838 x^5-5180 x^6-440 x^7)}{-1+27 x-131 x^2-319 x^3+1511 x^4-598 x^5-1473 x^6+740
x^7+55 x^8}$}

  \bc $ \mbox{--------------------------------------    } $  \ec
\noindent  \parbox[t]{5.8in}{$ \ds {\cal F}^{KB}_{7,0}(x) = {\cal F}^{KB}_{7,1}(x) ={\cal F}^{KB}_{7,2}(x) ={\cal F}^{KB}_{7,3}(x) ={\cal F}^{KB}_{7,4}(x) = {\cal F}^{KB}_{7,5}(x) ={\cal F}^{KB}_{7,6}(x) =$ \\
$ \ds  54 x+934 x^2+16230 x^3+310718 x^4+6040924 x^5+118584826 x^6+2335206032 x^7+46047350318 x^8+908456723040 x^9+17926321479704 x^{10}+353762659613648
x^{11}+6981456906125426 x^{12}+137779721623533238 x^{13}+2719108662179278428 x^{14}+53662217078257540780 x^{15}+1059036595803307294350 x^{16}+20900344457391068849264
x^{17}+412473419050676537420584 x^{18}+8140264335619150746929476 x^{19}+160650120668531433288830648 x^{20}+3170469701270440091385512318 x^{21}+62570000632547804065968070596
x^{22}+1234834378456746859506879340754 x^{23}+24369760706176232820477934929698 x^{24}+480943232004084226889532274151474 x^{25} +  \ldots $}

  \bc $ \mbox{--------------------------------------    } $  \\ $ \mbox{--------------------------------------    } $ \ec

   \noindent  \parbox[t]{5.8in}{  ${\cal F}^{TG}_{8,0}(x) = $  \\ $
2 x (1+ 1155 x- 12424 x^2- 384908 x^3+ 3652149 x^4+ 39134559 x^5- 397978970 x^6- 1325446682 x^7+ 19166357954 x^8- 4538659130 x^9- 385677721008
x^{10}+ 853933290216 x^{11}+ 2271727081046 x^{12}- 8951163335278 x^{13}- 1303075634316 x^{14}+ 32748415590308 x^{15}- 16992774066247 x^{16}- 56226830515941 x^{17}+ 42301303866136 x^{18}+ 53709870617604 x^{19}- 40205849888523
x^{20}- 32521295586161 x^{21}+ 17363173153318 x^{22}+
12143542140342 x^{23}- 2895586231932 x^{24}- 2156940376804 x^{25}+ 138822154976 x^{26}+ 146741481824 x^{27}- 1366377600 x^{28}- 3443059968
x^{29}-38983680 x^{30}+16671744 x^{31})   \left. \right/  $ \\ $
((-1+x) (1+x) (-1+3 x) (-1+9 x) (1+9 x) (1+6 x+x^2) (1+10 x+9 x^2+2 x^3) (1-8 x+20 x^2-16 x^3+x^4)(1-10
x+29 x^2-26 x^3+4 x^4) (1-45 x+415 x^2-1171 x^3+1264 x^4-512 x^5+64 x^6)(1+28 x+234 x^2+776 x^3+1213 x^4+964 x^5+390 x^6+72 x^7+4 x^8)) $ \\
 $-16 x^2 (210-167438 x^2+40334531 x^4-4621804256 x^6+300002141022 x^8-12034896543418 x^{10}+313705813053022 x^{12}-5495025669054056
x^{14}+66442956971597534 x^{16}-567362486710974600 x^{18}+3486823149562258817 x^{20}-15656403239930721918 x^{22}+51952386721895018640 x^{24}-128453782496846285630 x^{26}+237955842490205821743 x^{28}-331260008529177488392
x^{30}+346819322354225766064 x^{32}-272717367155082963572 x^{34}+160497980150868303944 x^{36}-70270735446821798072 x^{38}+22683671153127961776 x^{40}-5329748411900492656 x^{42}+895330914574433520
x^{44}-104854721754153696 x^{46}+8245570080796608 x^{48}-409892218061056 x^{50}+11569362093056 x^{52}-148430451200 x^{54}+472817664 x^{56}) \left. \right/ $ \\ $((-1+2 x) (1+2 x) (-1-2 x+2 x^2)(-1+2 x+2 x^2) (-1+6 x-5 x^2+x^3) (1+6 x+5 x^2+x^3) (-1+8
x-6 x^2-4 x^3+2 x^4) (-1-8 x-6 x^2+4 x^3+2 x^4) (1-38 x+265 x^2-509 x^3+264 x^4)(1+38 x+265 x^2+509 x^3+264 x^4)
(-1+12 x-37 x^2+41 x^3-16 x^4+2 x^5) (1+12 x+37 x^2+41 x^3+16 x^4+2 x^5) (1-28 x+276 x^2-1310 x^3+3381 x^4-4998 x^5+4269
x^6-2044 x^7+508 x^8-56 x^9+2 x^{10})
(1+28 x+276 x^2+1310 x^3+3381 x^4+4998 x^5+4269 x^6+2044 x^7+508 x^8+56 x^9+2 x^{10})) $ \\ $ -8 x^2 (147-33512 x^2+1954689 x^4-48736404 x^6+618388328 x^8-4267898812 x^{10}+16333809700 x^{12}-35136967144 x^{14}+44267544512
x^{16}-34076687840 x^{18}+16407177664 x^{20}-4932389632 x^{22}+893827072 x^{24}-
88756224 x^{26}+3686400 x^{28})  \left. \right/$ \\ $ ((-1+2 x) (1+2 x) (-1+3 x) (1+3 x) (1-5 x+2 x^2)(1-4 x+2 x^2)
(-1-2 x+2 x^2)(-1+2 x+2 x^2) (1+4 x+2 x^2) (1+5 x+2 x^2) (-1+21 x-26 x^2+8 x^3)(1+21 x+26 x^2+8 x^3)(-1+8 x-6 x^2-4 x^3+2 x^4) (-1-8 x-6 x^2+4 x^3+2 x^4))  $ \\$
 \ds + \frac{16 x^2 (11-81 x^2)}{1-82 x^2+81 x^4} +
  \frac{8 x^2}{1-4 x^2}$ }

    \bc $ \mbox{--------------------------------------    } $  \ec  \noindent
   \parbox[t]{5.8in}{ ${\cal F}^{TG}_{8,0}(x) = $ \\ $ \ds    2 x+7074 x^2+28322 x^3+3500970 x^4+40904442 x^5+2995602834 x^6+48506694658 x^7+2901094068042 x^8+55552654729898 x^9+2938658810744994 x^{10}+63145273515340370
x^{11}+3049469798814523722 x^{12}+71652138555461676474 x^{13}+3218226920211928101138 x^{14}+81272482607385417610722 x^{15}+3441994018739506356665514 x^{16}+92175870041745601933702730
x^{17}+3722302328182703609943479250 x^{18}+104539748840959783379179264498 x^{19}+4062808105933711353421113028170 x^{20}+118561430934917292221051125848154 x^{21}+4468567183903910361175202236507650
x^{22}+134463648376982602891252560891661954 x^{23} + 4945840131592037075156768922329963514 x^{24}+152498730067609316115438024585263736042 x^{25}+ \ldots $}

        \bc $ \mbox{--------------------------------------    } $  \ec  \noindent
   \parbox[t]{5.8in}{ ${\cal F}^{TG}_{8,1}(x) ={\cal F}^{TG}_{8,7}(x) = $ \\ $ \ds
   -2 x (35-352 x-8820 x^2+31894 x^3+678486 x^4+16722 x^5-28665874 x^6+11376898 x^7+580383868 x^8-1422165502 x^9-1155510466 x^{10}+6534180246
x^{11}-970156690 x^{12}-13680938914 x^{13}+8245651762 x^{14}+
8896179374 x^{15}-5387736095 x^{16}-3166721746 x^{17}+1081209174 x^{18}+605907460 x^{19}-34699780 x^{20}-32512224 x^{21}-1247616
x^{22}+155520 x^{23}) \left. \right/ $\\ $
((-1+x) (-1+3 x) (-1+9 x) (1+9 x) (1+6 x+x^2) (1-8 x+20 x^2-16 x^3+x^4)(1-45 x+415 x^2-1171 x^3+1264
x^4-512 x^5+64 x^6)(1+28 x+234 x^2+776 x^3+1213 x^4+964 x^5+390 x^6+72 x^7+4 x^8))
- $ \\
$16 x (7-4835 x^2+777672 x^4-53700277 x^6+1973210074 x^8-42111638648 x^{10}+530559598533 x^{12}-3824608351502 x^{14}+16263660067215
x^{16}-41026773171432 x^{18}+58406879392463 x^{20}-36858830222949 x^{22}-11435596336624 x^{24}+36679846968046 x^{26}-26764580426334 x^{28}+9779554720360 x^{30}-1896056631852 x^{32}+175950474200 x^{34}-4784086536
x^{36}-166360080 x^{38}+2917200 x^{40}) \left. \right/ $\\
$((-1+6 x-5 x^2+x^3)(1+6 x+5 x^2+x^3) (-1+8 x-6 x^2-4 x^3+2 x^4) (-1-8 x-6 x^2+4 x^3+2 x^4)
(1-38 x+265 x^2-509 x^3+264 x^4) (1+38 x+265 x^2+509 x^3+264 x^4) (1-28 x+276 x^2-1310 x^3+3381 x^4-4998 x^5+4269 x^6-2044 x^7+508 x^8-56 x^9+2 x^{10}) (1+28 x+276 x^2+1310 x^3+3381
x^4+4998 x^5+4269 x^6+2044 x^7+508 x^8+56 x^9+2 x^{10})) - $ \\ $ \ds
8 x (7-1243 x^2+30344 x^4-527732 x^6+4970334 x^8-18158020 x^{10}+26844948 x^{12}-18432776 x^{14}+6174656 x^{16}-948800 x^{18}+49920
x^{20})    \left. \right/ ((-1+2 x) (1+2 x) (-1+3 x) (1+3 x) (1-4 x+2 x^2) (1+4 x+2 x^2) (-1+21 x-26 x^2+8 x^3) (1+21 x+26
x^2+8 x^3) (-1+8 x-6 x^2-4 x^3+2 x^4) (-1-8 x-6 x^2+4 x^3+2 x^4))
+ $ \\ $ \ds
\frac{16 x \left(1+9 x^2\right)}{(-1+x) (1+x) (-1+9 x) (1+9 x)}
+\frac{4 x}{1-4 x^2}$}
        \bc $ \mbox{--------------------------------------    } $  \ec  \noindent
   \parbox[t]{5.8in}{ ${\cal F}^{TG}_{8,1}(x) ={\cal F}^{TG}_{8,7}(x) = $ \\ $ \ds
 258 x+906 x^2+122970 x^3+1193458 x^4+98122050 x^5+1432602986 x^6+92170349130 x^7+1647401456738 x^8+91916712721410 x^9+1874451697845274
x^{10}+94419822463975098 x^{11}+2127485426337270546 x^{12}+98878616824501947666 x^{13}+2413267137345084941770 x^{14}+105087917204613414240762 x^{15}+2737063321207611201239234
x^{16}+113047671005824033157621058 x^{17}+3104204990401867912554208122 x^{18}+122846314863361522066752950730 x^{19}+3520567590733086151373754801522
x^{20}+134623435267957989067437505160946 x^{21}+3992769309773621586556190357722666 x^{22}+148558323044818455419976063830003610 x^{23}+4528303974210709640776752177341476642
x^{24}+164867809702181958255884858725382062098 x^{25} + \ldots  $}

 \bc $ \mbox{--------------------------------------    } $  \ec  \noindent
   \parbox[t]{5.8in}{ ${\cal F}^{TG}_{8,2}(x) = {\cal F}^{TG}_{8,6}(x) =$ \\ $ \ds
2 x (-21+301 x+2142 x^2-31558 x^3-85695 x^4+1226087 x^5+2599204 x^6-46907676 x^7+143510797 x^8-97239221 x^9-304897698 x^{10}+687865258
x^{11}-544697417 x^{12}+142538593 x^{13}+62742304 x^{14}-63141192 x^{15}+
18193424 x^{16}-2640960 x^{17}+62208 x^{19})\left. \right/ $  \\ $((-1+x) (1+x) (-1+3 x) (-1+9 x) (1+9 x)(1+6 x+x^2) (1+10
x+9 x^2+2 x^3)(1-10 x+29 x^2-26 x^3+4 x^4)(1-45 x+415 x^2-1171 x^3+1264 x^4-512 x^5+64 x^6)) + $
 \\ $
16 x^2 \left(-100+6734 x^2 \right. -290409 x^4+7598896 x^6-87806785 x^8+600533949 x^{10}-2497352040 x^{12}+5307004114 x^{14}-3877527102 x^{16}-1464694334
x^{18}+3686927028 x^{20}-2013217248 x^{22}+477885104 x^{24}-47721632 x^{26}  +1672704 x^{28}) \left. \right/ $ \\ $((-1+2 x) (1+2 x) (-1-2 x+2 x^2) (-1+2
x+2 x^2)(-1+6 x-5 x^2+x^3) (1+6 x+5 x^2+x^3)(1-38 x+265 x^2-509 x^3+264 x^4)(1+38 x+265 x^2+509 x^3+264
x^4) (-1+12 x-37 x^2+41 x^3-16 x^4+2 x^5) (1+12 x+37 x^2+41 x^3+16 x^4+2 x^5))
+ $ \\
$16 x^2 (-43-187 x^2+19686 x^4-194658 x^6+631266 x^8-844368 x^{10}+513280 x^{12}-137984 x^{14}+13824 x^{16}) \left. \right/ $ \\ $((-1+2 x) (1+2
x) (-1+3 x) (1+3 x) (1-5 x+2 x^2)(-1-2 x+2 x^2)(-1+2 x+2 x^2) (1+5 x+2 x^2)(-1+21 x-26 x^2+8
x^3)(1+21 x+26 x^2+8 x^3)) +$ \\ $ \ds
\frac{160 x^2}{(-1+x) (1+x) (-1+9 x) (1+9 x)} \
-\frac{8 x^2}{(-1+2 x) (1+2 x)}$}
 \bc $ \mbox{--------------------------------------    } $  \ec  \noindent
   \parbox[t]{5.8in}{ ${\cal F}^{TG}_{8,2}(x) = {\cal F}^{TG}_{8,6}(x) =$ \\ $ \ds
   42 x+3618 x^2+39522 x^3+3128370 x^4+43440602 x^5+2927798394 x^6+49136765762 x^7+2886264949986 x^8+55716321149994 x^9+2935165787515578
x^{10}+63188436935121298 x^{11}+3048615263959972914 x^{12}+71663576269352269946 x^{13}+3218012820930136965210 x^{14}+81275517882661740507362 x^{15}+3441939475524946178698050
x^{16}+92176675886656076462974410 x^{17}+3722288264942729324525277738 x^{18}+104539962816630877105039922482 x^{19}+4062804448343867896109857271730
x^{20}+118561487754176728036632187133978 x^{21}+4468566226712191347899136905075514 x^{22}+134463663465005452261663283406005890 x^{23}+4945839879990147997782340020550314978
x^{24}+152498734074161580730944989626029361002 x^{25}+ \ldots  $}
\bc $ \mbox{--------------------------------------    } $  \ec  \noindent
   \parbox[t]{5.8in}{ ${\cal F}^{TG}_{8,3}(x) = {\cal F}^{TG}_{8,5}(x) = $ \\ $ \ds
-2 x (15+412 x-9492 x^2-55978 x^3+870046 x^4+2529242 x^5-27049786 x^6-38174246 x^7+436907996 x^8-66374206 x^9-3321440314 x^{10}+5209671710
x^{11}+2740680742 x^{12}-9419375018 x^{13}+41160282 x^{14}+ 8248514518 x^{15}-2000260459 x^{16}-3066704782 x^{17}+914663662 x^{18}+528379964 x^{19}-82926724 x^{20}-26623200 x^{21}+867456
x^{22}+155520 x^{23})  \left. \right/ $ \\ $
((-1+x) (-1+3 x) (-1+9 x) (1+9 x) (1+6 x+x^2) (1-8 x+20 x^2-16 x^3+x^4) (1-45 x+415 x^2-1171 x^3+1264
x^4-512 x^5+64 x^6) (1+28 x+234 x^2+776 x^3+1213 x^4+964 x^5+390 x^6+72 x^7+4 x^8)) + $ \\
$
16 x (-1-1558 x^2+290772 x^4-18543808 x^6+568017419 x^8-8628095973 x^{10}+42904133425 x^{12}+366587839840 x^{14}-4324466854832 x^{16}+17596586188137
x^{18}-39661828333732 x^{20}+56167093516791 x^{22}-51844948050210 x^{24}+30943111617414 x^{26}-11455460450038 x^{28}+2304439605736 x^{30}-110372301292 x^{32}-46175831064 x^{34}+8206127304
x^{36}-413798736 x^{38}+3773616 x^{40}) \left. \right/ $ \\ $
((-1+6 x-5 x^2+x^3) (1+6 x+5 x^2+x^3) (-1+8 x-6 x^2-4 x^3+2 x^4) (-1-8 x-6 x^2+4 x^3+2 x^4)
(1-38 x+265 x^2-509 x^3+264 x^4) (1+38 x+265 x^2+509 x^3+264 x^4) (1-28 x+276 x^2-1310 x^3+3381 x^4-4998 x^5+4269 x^6-2044 x^7+508 x^8-56 x^9+2 x^{10}) (1+28 x+276 x^2+1310 x^3+3381 x^4+4998
x^5+4269 x^6+2044 x^7+508 x^8+56 x^9+2 x^{10}))
- $ \\
$ 8 x (3+397 x^2-5048 x^4+41044 x^6-390886 x^8+2407564 x^{10}-3292980 x^{12}+1718856 x^{14}-239168 x^{16}-57664 x^{18}+13056 x^{20}) \left. \right/ $ \\ $ ((-1+2
x) (1+2 x) (-1+3 x) (1+3 x) (1-4 x+2 x^2) (1+4 x+2 x^2) (-1+21 x-26 x^2+8 x^3) (1+21 x+26 x^2+8 x^3)
(-1+8 x-6 x^2-4 x^3+2 x^4) (-1-8 x-6 x^2+4 x^3+2 x^4))
+ $ \\ $ \ds
\frac{16 x (1+9 x^2)}{(-1+x) (1+x) (-1+9 x) (1+9 x)} + \frac{4 x}{1-4 x^2} $}
 \bc $ \mbox{--------------------------------------    } $  \ec  \noindent
   \parbox[t]{5.8in}{ ${\cal F}^{TG}_{8,3}(x) = {\cal F}^{TG}_{8,5}(x) =$ \\ $ \ds
   90 x+1514 x^2+94362 x^3+1387122 x^4+92162898 x^5+1485327978 x^6+90828455850 x^7+1661455264354 x^8+91600780892322 x^9+1878185846862106
x^{10}+94343121781431066 x^{11}+2128477108273623058 x^{12}+98859556057492818690 x^{13}+2413530478272486804170 x^{14}+105083094553825097045610 x^{15}+2737133250428499047824578
x^{16}+113046434109082928272163778 x^{17}+3104223559811864961759678074 x^{18}+122845994423835057869228842458 x^{19}+3520572521759994096242175873138
x^{20}+134623351644902885173940369410098 x^{21}+3992770619186522756263173185080106 x^{22}+148558301108128862764536464874694458 x^{23}+4528304321919647817704326349628896546
x^{24}+164867803926291157295966137332587684850 x^{25}+ \ldots  $}
\bc $ \mbox{--------------------------------------    } $  \ec  \noindent
   \parbox[t]{5.8in}{ ${\cal F}^{TG}_{8,4}(x) = $ \\ $ \ds
2 x (41-413 x-15840 x^2+50300 x^3+2359725 x^4-4375985 x^5-165478450 x^6+443450094 x^7+5192096930 x^8-21582548570 x^9-53441593856
x^{10}+399917164216 x^{11}-319969463706 x^{12}-1690918438062 x^{13}+ 2931087071332 x^{14}+2518195738036 x^{15}-7381107127935 x^{16}-1396275636325 x^{17}+9358426269312 x^{18}-5633688660 x^{19}-6653330191507 x^{20}-1343499553
x^{21}+2667240608782 x^{22}+222728854430 x^{23}-542256944460 x^{24}-96086497060 x^{25}+43324394784 x^{26}+10305897696 x^{27}-1079294080 x^{28}-305921280 x^{29}+165888 x^{30}+746496
x^{31})) \left. \right/ $ \\ $
((-1+x) (1+x) (-1+3 x) (-1+9 x) (1+9 x) (1+6 x+x^2) (1+10 x+9 x^2+2 x^3) (1-8 x+20 x^2-16 x^3+x^4) (1-10
x+29 x^2-26 x^3+4 x^4) (1-45 x+415 x^2-1171 x^3+1264 x^4-512 x^5+64 x^6) (1+28 x+234 x^2+776 x^3+1213 x^4+964 x^5+390 x^6+72 x^7+4 x^8))
+ $\\ $
32 x^2 (-34+288 x^2+951531 x^4-127953686 x^6+7707175326 x^8-251997212110 x^{10}+4572837437749 x^{12}-40496311634298 x^{14}+8552238075793
x^{16}+3581342748552186 x^{18}-39624588209529114 x^{20}+231144282506449358 x^{22}-858372624857577962 x^{24}+2163429734007596828 x^{26}-3814083013353228292 x^{28}+4773225202601871392 x^{30}-4256752027902923568
x^{32}+2684197421453858678 x^{34}-1166840204758900552 x^{36}+328560423973435556 x^{38}-49409819933986792 x^{40}-437884684114088 x^{42}+1751272126622784 x^{44}-350647188754800
x^{46}+35628237884640 x^{48}-2185939682944 x^{50}+87861671296 x^{52}-2103543040 x^{54}+13381632 x^{56}) \left. \right/ $ \\ $
((-1+2 x) (1+2 x) (-1-2 x+2 x^2) (-1+2 x+2 x^2) (-1+6 x-5 x^2+x^3) (1+6 x+5 x^2+x^3) (-1+8
x-6 x^2-4 x^3+2 x^4) (-1-8 x-6 x^2+4 x^3+2 x^4) (1-38 x+265 x^2-509 x^3+264 x^4) (1+38 x+265 x^2+509 x^3+264 x^4)(-1+12 x-37 x^2+41 x^3-16 x^4+2 x^5) (1+12 x+37 x^2+41 x^3+16 x^4+2 x^5) (1-28 x+276 x^2-1310 x^3+3381 x^4-4998 x^5+4269
x^6-2044 x^7+508 x^8-56 x^9+2 x^{10})(1+28 x+276 x^2+1310 x^3+3381 x^4+4998 x^5+4269 x^6+2044 x^7+508 x^8+56 x^9+2 x^{10}))- $ \\ $8 x^2 (83-4480 x^2+221313 x^4-6117684 x^6+87506312 x^8-649648852 x^{10}+2459686052 x^{12}-4746232984 x^{14}+5239933088 x^{16}-3510660768
x^{18}+1460030784 x^{20}-375567616 x^{22}+57466880 x^{24}-4681728 x^{26}+147456 x^{28}) \left. \right/ $ \\ $ ((-1+2 x) (1+2 x) (-1+3 x) (1+3 x) (1-5 x+2 x^2)
(1-4 x+2 x^2) (-1-2 x+2 x^2) (-1+2 x+2 x^2) (1+4 x+2 x^2) (1+5 x+2 x^2) (-1+21 x-26 x^2+8
x^3) (1+21 x+26 x^2+8 x^3) (-1+8 x-6 x^2-4 x^3+2 x^4) (-1-8 x-6 x^2+4 x^3+2 x^4)) + $ \\
$ \ds \frac{160 x^2}{(-1+x) (1+x) (-1+9 x) (1+9 x)}
-\frac{8 x^2}{(-1+2 x) (1+2 x)} $}
 \bc $ \mbox{--------------------------------------    } $  \ec  \noindent
   \parbox[t]{5.8in}{ ${\cal F}^{TG}_{8,4}(x) = $ \\ $ \ds
82 x+2898 x^2+46498 x^3+2891562 x^4+45545722 x^5+2868866562 x^6+49729274306 x^7+2872061882730 x^8+55876850613994 x^9+2931718484241138
x^{10}+63231342283558674 x^{11}+3047764139803694010 x^{12}+71674992873372235130 x^{13}+3217798980157715938530 x^{14}+81278551435034414932258 x^{15}+3441884952158681678200266
x^{16}+92177481591073602079199370 x^{17}+3722274203243322471377657394 x^{18}+104540176780849548385092275378 x^{19}+4062800790874696129687192838202
x^{20}+118561544572502736215457104022554 x^{21}+4468565269529996976224046889468770 x^{22}+134463678552952226901453862332146114 x^{23}+4945839628389015474829403515287198570
x^{24}+152498738080707645357839362774243546922 x^{25}
+ \ldots $}

  \bc $ \mbox{--------------------------------------    } $  \ec
   \noindent  \parbox[t]{5.8in}{ ${\cal F}^{KB}_{8,0}(x) =  {\cal F}^{KB}_{8,2}(x) = {\cal F}^{KB}_{8,4}(x) =  {\cal F}^{KB}_{8,6}(x) =  $  \\ $
 2 x(-19+44 x+6760 x^2-32292 x^3-265118 x^4+1305548 x^5-1068528 x^6-1614876 x^7+2507865
x^8-798040 x^9-64800 x^{10}+33696 x^{11}) \left. \right/ $ \\ $(
(-1+x) (1+x) (-1+9 x) (1+9 x)(1+6 x+x^2) (1-45 x+415 x^2-1171 x^3+1264 x^4-512
x^5+64 x^6))+$ \\ $ \ds
8 x (-8+1005 x^2-27576 x^4+346147 x^6-267177 x^8-156480 x^{10}+53526 x^{12}) \left. \right/ $ \\ $((-1+6 x-5 x^2+x^3) (1+6 x+5 x^2+x^3)
(1-38 x+265 x^2-509 x^3+264 x^4) (1+38 x+265 x^2+509 x^3+264 x^4))
-   $ \\ $\ds  \frac{8 x (5-103 x^2+2436 x^4-5868 x^6+1968 x^8)}{(-1+2 x) (1+2 x) (-1+3 x) (1+3 x) (-1+21 x-26 x^2+8 x^3) (1+21
x+26 x^2+8 x^3)}+ $ \\ $ \ds \frac{16 x (1+9 x^2)}{(-1+x) (1+x) (-1+9 x) (1+9 x)}+\frac{4 x}{1-4 x^2} $}
 \bc $ \mbox{--------------------------------------    } $  \ec
   \noindent  \parbox[t]{5.8in}{  ${\cal F}^{KB}_{8,0}(x) =  {\cal F}^{KB}_{8,2}(x) = {\cal F}^{KB}_{8,4}(x) =  {\cal F}^{KB}_{8,6}(x) =  $  \\ $
162 x+1394 x^2+107262 x^3+1303578 x^4+95024862 x^5+1460029826 x^6+91489840926 x^7+1654514467114 x^8+91757972005254 x^9+1876325746040594
x^{10}+94381409360938158 x^{11}+2127981832165582650 x^{12}+98869081357268915166 x^{13}+2413398853562380348834 x^{14}+105085505467437020151702 x^{15}+2737098289524095588329034
x^{16}+113047052524099117573839606 x^{17}+3104214275407055708525782322 x^{18}+122846154640896586534969435350 x^{19}+3520570056270855454733052763738
x^{20}+134623393456211599142446833273510 x^{21}+3992769964482041713214111755986370 x^{22}+148558312076455933216017147636514830 x^{23}+4528304148065338262126693573363986282
x^{24}+164867806814235121979950116018384551862 x^{25}  + \ldots
   $} \bc $ \mbox{--------------------------------------    } $  \ec
   \noindent  \parbox[t]{5.8in}{  ${\cal F}^{KB}_{8,1}(x) =  {\cal F}^{KB}_{8,3}(x) = {\cal F}^{KB}_{8,5}(x) =  {\cal F}^{KB}_{8,7}(x) =  $  \\ $
- 2 x (-27+438 x+2322 x^2-27866 x^3-122068 x^4+704506 x^5-392098 x^6-1315518 x^7+1666463 x^8-501048 x^9-48672
x^{10}+23328 x^{11}) \left. \right/ $ \\ $ ((-1+x) (1+x) (-1+9 x) (1+9 x) (1+6 x+x^2) (1-45 x+415 x^2-1171 x^3+1264 x^4-512 x^5+64 x^6))
- $ \\$  \ds  4 x^2 (470-55846 x^2+1447257 x^4-7171386 x^6+8483145 x^8-3075627 x^{10}+243936 x^{12}) \left. \right/ $ \\ $((-1+6 x-5 x^2+x^3) (1+6
x+5 x^2+x^3) (1-38 x+265 x^2-509 x^3+264 x^4) (1+38 x+265 x^2+509 x^3+264 x^4))
- $ \\ $ \ds  \frac{4 x^2 (201-5433 x^2+27732 x^4-26144 x^6+5760 x^8)}{(-1+2 x) (1+2 x) (-1+3 x) (1+3 x) (-1+21 x-26 x^2+8 x^3) (1+21
x+26 x^2+8 x^3)}
- $ \\ $ \ds \frac{4 x^2 (-41+81 x^2)}{(-1+x) (1+x) (-1+9 x) (1+9 x)}+ \frac{8 x^2}{1-4 x^2}   $}
 \bc $ \mbox{--------------------------------------    } $  \ec
   \noindent  \parbox[t]{5.8in}{  ${\cal F}^{KB}_{8,1}(x) =  {\cal F}^{KB}_{8,3}(x) = {\cal F}^{KB}_{8,5}(x) =  {\cal F}^{KB}_{8,7}(x) =  $  \\ $
54 x+4086 x^2+39870 x^3+3148902 x^4+43450454 x^5+2928951690 x^6+49136936686 x^7+2886335354070 x^8+55716311712582 x^9+2935170243809226
x^{10}+63188435179050334 x^{11}+3048615551774372274 x^{12}+71663576075613080886 x^{13}+3218012839803884885586 x^{14}+81275517863718063880910 x^{15}+3441939476780979633744054
x^{16}+92176675884887202375765542 x^{17}+3722288265027681911221986654 x^{18}+104539962816469474926609307454 x^{19}+4062804448349720487906909287442
x^{20}+118561487754162209105685255046678 x^{21}+4468566226712602966494950715892446 x^{22}+134463663465004159455247364224789166 x^{23}+4945839879990177603501558809666430594
x^{24}+152498734074161466529767223663491822854 x^{25} + \ldots
   $}

 \bc $ \mbox{--------------------------------------    } $  \\ $ \mbox{--------------------------------------    } $  \ec  \noindent
\small
   \noindent  \parbox[t]{5.8in}{  ${\cal F}^{TG}_{9,0}(x) = $  \\ $
   -(2x (-1-9800 x+359083 x^2+16427291 x^3-610797257 x^4-10367246765 x^5+419467480970 x^6+3271713224707 x^7-159387915091771 x^8-532793040194260
x^9+38298388185121443 x^{10}+27528294151330576 x^{11}-6270569447266278576 x^{12}+6330866652133000181 x^{13}+734139949111726104650 x^{14}-1646671463137717739484 x^{15}-63589707904290138859679 x^{16}+202605153305015367613358
x^{17}+4180235346070316976909123 x^{18}-16549316774066151856124921 x^{19}-212717703591028771713946565 x^{20}+987196592557237768194072372 x^{21}+8510343085758647848438794430
x^{22}-45039456040594841954111521297 x^{23}-270951310971092477590360357988 x^{24}+1616949814715881221037930917164 x^{25}+6926989952773808495556682950829
x^{26}-46595391796390347805018716401847 x^{27}-143029807464313492758527896722438 x^{28}+1094016814924062628435464659431235 x^{29}+2389956227807518140976599231134577 x^{30}-21176292404353007582432350317178603
x^{31}-32200987369979466783264528354915163 x^{32}+341166273603656291654746808444027176 x^{33}+345092568242648114282783041476327934 x^{34}-4611138156199533636619770727592623594
x^{35}-2834587753227632927274941414952428056 x^{36}+52634633889435772825983348626968761355 x^{37}+15890857035822729086776887751918914308 x^{38}-510306487717336142003564865151071822647
x^{39}-27324678365155749893067330382835184065 x^{40}+4223052141706574917364770209959800039278 x^{41}-608712206598178809588916714383700113406 x^{42}-29958887771534993168817018210559085822548
x^{43}+9020739953725239922620163083622478673412 x^{44}+182881624736372061109142884993629891705916 x^{45} $ \\ $  -77669314780852491882422759538440392914863
x^{46} $ \\ $  - 963840298509073190956930797415637277181057 x^{47} $ \\ $  + 501657036824161407289050322441222062233670 x^{48} $ \\ $  + 4398455890744031925855351032495505883371193 x^{49} $ \\ $  - 2607235304757243219612258403963369571058152
x^{50} $ \\ $  - 17424503752343213730439006353333425448579947 x^{51} $ \\ $  + 11239123592612211062639562778159999105104346 x^{52} $ \\ $  + 60052504197509240391983675422140158245281774
x^{53} $ \\ $  - 40844642854390528671749230032414150257885968 x^{54} $ \\ $  - 180384953234555852463785364855361724601213457 x^{55} $ \\ $  + 126393394998188618722631962173561519719053026
x^{56} $ \\ $  + 472932268656311302299305113288163321238735414 x^{57} $ \\ $  - 335278270632828284263245112930823019579179230 x^{58} $ \\ $  - 1083433544313444881365952188508642523005068109
x^{59} $ \\ $  + 766037676985480927415111468593924900501823951 x^{60} $ \\ $  + 2170366382682205296407762434906709643885957386 x^{61} $ \\ $  - 1512881154229893990745371071933083045523711759
x^{62} $ \\ $  - 3803359331926731601173575617393141913236909440 x^{63} $ \\ $  + 2589771602796994700487014678672589603582230100 x^{64} $ \\ $  + 5830995414946424862562121177906382036656511359
x^{65} $ \\ $  - 3850861209510910475036354303453450580817325600 x^{66} $ \\ $  - 7819291373856199242106778977212668791015616982 x^{67} $ \\ $  + 4982368410441227670377883403072114524318945194
x^{68} $ \\ $  + 9167053748855525034544934831120538350111880652 x^{69} $ \\ $  - 5616585607685453687222510622094190031802133025 x^{70} $ \\ $  - 9388704056669986782487498406887434105465753160
x^{71} $ \\ $  + 5521940417786212009445539719556946163576682827 x^{72} $ \\ $  + 8391947729795989236879659154142149485746478836 x^{73} $ \\ $  - 4737653624282941370684007595500260842382714340
x^{74} $ \\ $  - 6538251345452713068078942885960343048732835523 x^{75} $ \\ $  + 3548052958227782985668965268515165868501505596 x^{76} $ \\ $  + 4433557480976635259284838732269310542392490573
x^{77} $ \\ $  - 2318947371053232849711104818657538657334219389 x^{78} $ \\ $  - 2611971844711786550309516743150809566433200994 x^{79} $ \\ $  + 1321846614354603660037651659426951751024353803
x^{80} $ \\ $  + 1334191278993962390114122100232110558267411345 x^{81} $ \\ $  - 656380259392679891778248299080766897969529065 x^{82} $ \\ $  - 589485711127887576027631906259151530931057378
x^{83} $ } \\
  \noindent  \parbox[t]{5.8in}{$  + 283450691350787209149331063609782519240650072 x^{84} $ \\ $  + 224677172180454596376482191289111161857532481 x^{85} $ \\ $  - 106212550099712195733516319877763466435808558
x^{86} $ \\ $  - 73644608565084126739027705044830155465216009 x^{87} $ \\ $  + 34439023112344323340694497351998623823291655 x^{88} $ \\ $  + 20687758554886147710876930337979324355430384
x^{89} $ \\ $  - 9631385000502561585439191271206784624962829 x^{90} $ \\ $  - 4961168852829536804935377327620615500073172 x^{91} $ \\ $  + 2314629901504294868886192907618384653379998
x^{92} $ \\ $  + 1011240586274196238351415160659890347313734 x^{93} $ \\ $  - 476045699325188881713822005593441880566914 x^{94} $ \\ $  - 174338436438015256755848229178002905245395
x^{95} $ \\ $  + 83417351982555357787029253609794579110355 x^{96} + $ \\ $  25280865375617832798040497534864896365038 x^{97} - 12394586777360308482213056922298187716656 x^{98}   - 3064064857180539080760448553722127024241 x^{99}+1553704891388032396756223000182870623210
x^{100}+308105318378920920842921174948070662487 x^{101}-163424572550351656285793539371759738312 x^{102}-25475679190337193798690592444222485630 x^{103}+14340549440696628740066875031882933475
x^{104}+1712710735758596942785458635211725859 x^{105}-1043260137093016999512250554401757765 x^{106}-92199808860212864167223745220780428 x^{107}+62487654943816613503297369926930645
x^{108}+3883823249084069631509213997618429 x^{109}-3057526712398502455144785071787771 x^{110}-122933510722791389849648662115766 x^{111}+121103155670861695041765346646592
x^{112}+2665737825027239527869597836832 x^{113}-3840436635553736102851239030138 x^{114}-27214348799419765458187216077 x^{115}+96191119182392354333208577518
x^{116}-481454351622633105341159175 x^{117}-1870161222854269116288436671 x^{118}+27388193160562613049845106 x^{119}+27586882126692398139141438 x^{120}-611375980770868106877705 x^{121}-299295232121840994450954
x^{122}+ 8329083449137020167535 x^{123}+2284087833532022975181 x^{124}-72759349025611434882 x^{125}-11442895851123712422 x^{126}+393304622439348567 x^{127}+33291200297708901
x^{128}-1179649024779960 x^{129}-  42085722577626 x^{130}+1479376641960 x^{131})) \left. \right/ $ \\ $((1+3 x) (-1+2 x+7 x^2) (-1+81 x-1792 x^2+7289 x^3+113338 x^4-948939 x^5+891997 x^6+9118681 x^7-25652726 x^8+9992771 x^9+33620979
x^{10}-29903008 x^{11}-9941993 x^{12}+14464685 x^{13}-684910 x^{14}-1263348 x^{15}+43295 x^{16})(-1-27 x-30 x^2+3263 x^3+10255 x^4-128271 x^5-422489 x^6+2034868 x^7+6331921 x^8-14787664 x^9-46140893 x^{10}+49195065 x^{11}+177327440
x^{12}-56000703 x^{13}-357184328 x^{14}-43503000 x^{15}+ 357901011 x^{16}+120196901 x^{17}-183392047 x^{18}-80222831 x^{19}+48816126 x^{20}+23044935 x^{21}-6865887 x^{22}-3172062 x^{23}+503757
x^{24}+213075 x^{25}-17955 x^{26}-6750 x^{27}+243 x^{28}+81 x^{29})(-1-12 x+1296 x^2+6760 x^3-558339 x^4-1161243 x^5+118840379 x^6+34612254 x^7-14636101794 x^8+11650032630 x^9+1139067517383 x^{10}-1675685198967
x^{11}-59425014210170 x^{12}+112413543683178 x^{13}+2171147767245072 x^{14}-4673729741871850 x^{15}-57423564754079904 x^{16}+132752592370932042 x^{17}+1127413588825996376 x^{18}-2710891411914111867
x^{19}-16748804440629896181 x^{20}+41076234964519891361 x^{21}+191073085446307327887 x^{22}-471751936975189881282 x^{23}-1693441623773186077195 x^{24}+4168981592750804837106 x^{25}+11769841572576724100388
x^{26}-28667256121124059466103 x^{27}-64650834761329667094285 x^{28}+154709707930302466051638 x^{29}+282472949704888959670555 x^{30}-659784412595668698793188
x^{31}-986747002345947660623415 x^{32}+2235928540752665118157037 x^{33}+2766102104157402845641392 x^{34}-6048596568103168034002953 x^{35}-6235862859792016465952094 x^{36}+13109354010298653133950159
x^{37}+11311529009445116059397322 x^{38}-22829191668429273215380050 x^{39}-16491985570921054148696568 x^{40}+32012050759806836634743418 x^{41}+19274557174365641463936649
x^{42}-36194090036067799457234367 x^{43}-17977866789394483635892473 x^{44}+33010263268427485989896780 x^{45}+13294805056149648548039391 x^{46}-24269111871425405879756730
x^{47}-7719715827397162309377098 x^{48}+14354944271362017116072472 x^{49}+3467347525245472903206939 x^{50}-6808062303404022976704068 x^{51}-1174671001995294266317512
x^{52}+2576330363368357032796665 x^{53}+285539149598812572013492 x^{54}-772973767306693551468021 x^{55}-43493796133268719517643 x^{56}+182435806074644884579263 x^{57}+1578291814221380714235
x^{58}-33560515290021718006404 x^{59}+1085611778689724348088 x^{60}+4761258378023303618643 x^{61}-319945226855087846232 x^{62}-514768561243945226721 x^{63}+49664633878454946921 x^{64}+41857441719264765459
x^{65}- 5059221729603564351 x^{66}-2523288763225221867 x^{67}+356134240037245131 x^{68}+111006922260938166 x^{69}-17560157208954669 x^{70}-3498962768541387
x^{71}+605417634344259 x^{72}+77202421595082 x^{73}-  14421262679676 x^{74}-1155401508177 x^{75}+232152034824 x^{76}+11204841699 x^{77}-2436730722 x^{78}-65443302 x^{79}+15699744
x^{80}+201609 x^{81}-55404 x^{82}-243 x^{83}+81 x^{84}))  $
    \bc $ \mbox{--------------------------------------    } $  \ec  \noindent }
 \\
  \noindent  \parbox[t]{5.8in}{
${\cal F}^{TG}_{9,1}(x) ={\cal F}^{TG}_{9,8}(x)  =  $ \\
$ \ds 2x(256-9523 x-425476 x^2+15516100 x^3+267654407 x^4-10137498520 x^5-86210435846 x^6+3657020332670 x^7+15599413664458 x^8-836292224795891
x^9-1490639437226235 x^{10}+131094693618246140 x^{11}+
20976358801026762 x^{12}-14815609503165511484 x^{13}+14492478755925630436 x^{14}+1250575592131503419406 x^{15}-2271988501858134682240 x^{16}-80970413199362959442624
x^{17}+
198898988548906024176885 x^{18}+4108844693427633645615839 x^{19}-12150496891363701675253969 x^{20}-166437784650071719589558637 x^{21}+558739345231834539436075547
x^{22}+
5468369973384975789385493599 x^{23}- $ \\ $ 20134072158417162774086404948 x^{24}-147750598875752939567215299638 x^{25}+583709057127073169103879791510
x^{26}+3320946768587065127907745387944 x^{27}-
13877311455886941191075848438065 x^{28}-62658295593998069225665025082929 x^{29}+274487699834608070500714875287505 x^{30}+998867403404942479967869877460361
x^{31}-
4566428265605950238348696101620719 x^{32}-13510132251931951616965396477885966 x^{33}+64410621533342852521655954176030862 x^{34}+155376601145314837770328687399012315
x^{35}-
774762131817861834853075851810609218 x^{36}-1520461003637141051553960597728704459 x^{37}+7979306252735603247315458712928648651 x^{38}+12653914204973825498400623595573464153
x^{39}-
70564270344673990326372139554453504754 x^{40}-89439474182377408689229748336634539082 x^{41}+536960482652747303931383391550264875643 x^{42}+535621450181306275017553446602923056181
x^{43}-
3521915369904598389768992931495198981586 x^{44}-2707520536475096285516014157639669063662 x^{45}+19941622219504651944162038449558536789212 x^{46}+11480984932881191144649176349360767005150
x^{47}-
97618148897501308156163488447298504013429 x^{48}-40397739782677654984144072525505515248274 x^{49}+ $ \\
$ 413742533309376061550231726327049486928382
x^{50}+$ \\ $115496607610595444066426688329271091874423 x^{51}-$ \\ $
1520533097062640973758522438455020518390002 x^{52}-$ \\ $255809501294899644856083752911674901285635 x^{53}+$ \\ $4852281970792844264066622510725183392942243
x^{54}+$ \\ $
378681801357985046954316453403664244570301 x^{55}-$ \\ $13463621437560335371111673718445551950074503 x^{56}-$ \\ $78718146032641225653212252196724609685067
x^{57}+$ \\ $
32521337726560030688531324592108575265981475 x^{58}-$ \\ $1651817900389582126258771587697633383429601 x^{59}-$ \\ $68456895074337938780038057564187225262617447
x^{60}+$ \\ $
6440670982751893282261277102199892124029666 x^{61}+$ \\ $125682081401107225593375491985839763200040569 x^{62}-$ \\ $15841100492120447680067399537769336637444593
x^{63}-$ \\ $
201375398182943627967965590076896887434807580 x^{64}+$ \\ $29904901914589641363919869944683676054976748 x^{65}+$ \\ $281698984295293927652899230541442671627458542
x^{66}-$ \\ $
45958967954255302257914413194599996067072438 x^{67}-$ \\ $344093609485730064450467471751146711287987586 x^{68}+$ \\ $59012854435261524887207805037030841974746836
x^{69}+$ \\ $
366982297737121921439062026766298726578617288 x^{70}-$ \\ $64196281289772335491808082360583094148327719 x^{71}-$ \\ $341629127493946825515270959139230647709145696
x^{72}+$ \\ $
59663405859467602954138972684198769698030151 x^{73}+$ \\ $277440828487324889553426577489262846382276270 x^{74}-$ \\ $47636462010951632528927910464520074934927354
x^{75}-$ \\ $
196406629311498817796232165055300113951824228 x^{76}+$ \\ $32801026862515441374834240263055383321908676 x^{77}+$ \\ $121080677017851731403575544976125066953471108
x^{78}-$ \\ $
19534356322584265992196134386337807177924585 x^{79}-$ \\ $64920613619166573214746021493294384712149760 x^{80}+$ \\ $10084157728533266563086982972641680100823915
x^{81}+$ \\ $
30228844205052614736407790438450170638135615 x^{82}-$ \\  $ 4520335900319551526789576365649057982052473 x^{83}-$ \\ $12201361576343292063684127121733914101032788
x^{84}+$ \\$
1761858337107910678532883426113231012214205 x^{85}+$ } \\
\parbox[t]{5.8in}{  $4260085359740534051688678789905890550774063 x^{86}-$ \\ $597587782849770481763897967486230851111151
x^{87}-$ \\ $
1283407364111767833124140228695015719906450 x^{88}+$ \\ $176408971409475384767276189214102963646835 x^{89}+$ \\ $332636884549440503765301009973289437578163
x^{90}-$ \\ $45285347823218354185604151043163568598857 x^{91}-$ \\ $ \ds
73915682432345613699249975400553006996844 x^{92}+10087448764957430552694919252562460165399 x^{93}+14024959035464121143031586576454174029923
x^{94}-1942663998428242508893071990051902697090 x^{95}-
2261496459716668799978123661194784794757 x^{96}+321749955767205283765511886870682360263 x^{97}+308167910462815456767564735079985130867 x^{98}-45520900151234895464942959860657177702
x^{99}-
35254283577574707707768022441668089992 x^{100}+5457692019799603423053969955717477800 x^{101}+3359491704405954218279641111223658603 x^{102}-549599214119596216665325469504489403
x^{103}-
264169841958960468901923458682046944 x^{104}+46042435024809013901270009519688051 x^{105}+16941988825835758858950012199612656 x^{106}-3176315605841964497522380112175234
x^{107}-
872711463468836289901577299082829 x^{108}+178486559593162069024680082022466 x^{109}+35327941631953592253811555024494 x^{110}-8071817641852705088270697144027
x^{111}-
1084481308801276995119276027790 x^{112}+289698452417658032438373993051 x^{113}+23478102343958042392800657822 x^{114}-8110050663068985394153808811
x^{115}-285270912875643824296125819 x^{116}+
173081504892834190211730111 x^{117}-1058754832038077762527839 x^{118}-2724484919094654297427530 x^{119}+126321035614224190035696 x^{120}+29983974235231632304338
x^{121}-2751491806187380745094 x^{122}-
207397402857267286743 x^{123}+33241924908973835892 x^{124}+637564544815655484 x^{125}-238291722911777220 x^{126}+1701602497586205
x^{127}+937355756130927 x^{128}-19311690472965 x^{129}-1538244814239 x^{130}+41756598765 x^{131})     \left.    \right/ $ \\
$ \ds ((1+3 x) (-1+2 x+7 x^2) (-1+81 x-1792 x^2+7289 x^3+113338 x^4-948939 x^5+891997 x^6+9118681 x^7-25652726 x^8+9992771 x^9+33620979
x^{10}-29903008 x^{11}-9941993 x^{12}+14464685 x^{13}-684910 x^{14}-1263348 x^{15}+43295 x^{16})
(-1-27 x-30 x^2+3263 x^3+10255 x^4-128271 x^5-422489 x^6+2034868 x^7+6331921 x^8-14787664 x^9-46140893 x^{10}+49195065 x^{11}+177327440
x^{12}-56000703 x^{13}-357184328 x^{14}-43503000 x^{15}+357901011 x^{16}+120196901 x^{17}-183392047 x^{18}-80222831 x^{19}+48816126 x^{20}+23044935 x^{21}-6865887 x^{22}-3172062 x^{23}+503757
x^{24}+213075 x^{25}-17955 x^{26}-6750 x^{27}+243 x^{28}+81 x^{29})
(-1-12 x+1296 x^2+6760 x^3-558339 x^4-1161243 x^5+118840379 x^6+34612254 x^7-14636101794 x^8+11650032630 x^9+1139067517383 x^{10}-1675685198967
x^{11}-59425014210170 x^{12}+112413543683178 x^{13}+
2171147767245072 x^{14}-4673729741871850 x^{15}-57423564754079904 x^{16}+132752592370932042 x^{17}+
$ \\ $
1127413588825996376 x^{18}-2710891411914111867
x^{19}-16748804440629896181 x^{20}+
41076234964519891361 x^{21}+191073085446307327887 x^{22}-471751936975189881282 x^{23}-1693441623773186077195 x^{24}+4168981592750804837106 x^{25}+11769841572576724100388
x^{26}-
28667256121124059466103 x^{27}-64650834761329667094285 x^{28}+154709707930302466051638 x^{29}+282472949704888959670555 x^{30}-659784412595668698793188
x^{31}-986747002345947660623415 x^{32}+
2235928540752665118157037 x^{33}+2766102104157402845641392 x^{34}-6048596568103168034002953 x^{35}-6235862859792016465952094 x^{36}+13109354010298653133950159
x^{37}+11311529009445116059397322 x^{38}-
22829191668429273215380050 x^{39}-16491985570921054148696568 x^{40}+32012050759806836634743418 x^{41}+19274557174365641463936649 x^{42}-36194090036067799457234367
x^{43}-
17977866789394483635892473 x^{44}+33010263268427485989896780 x^{45}+13294805056149648548039391 x^{46}-24269111871425405879756730 x^{47}-7719715827397162309377098
x^{48}+
14354944271362017116072472 x^{49}+3467347525245472903206939 x^{50}-6808062303404022976704068 x^{51}-1174671001995294266317512 x^{52}+2576330363368357032796665
x^{53}+285539149598812572013492 x^{54}-
772973767306693551468021 x^{55}-43493796133268719517643 x^{56}+182435806074644884579263 x^{57}+1578291814221380714235 x^{58}-33560515290021718006404
x^{59}+1085611778689724348088 x^{60}+
4761258378023303618643 x^{61}-319945226855087846232 x^{62}-514768561243945226721 x^{63}+49664633878454946921 x^{64}+41857441719264765459 x^{65}-5059221729603564351
x^{66}-2523288763225221867 x^{67}+
356134240037245131 x^{68}+111006922260938166 x^{69}-17560157208954669 x^{70}-3498962768541387 x^{71}+605417634344259 x^{72}+77202421595082 x^{73}-14421262679676
x^{74}-1155401508177 x^{75}+
232152034824 x^{76}+11204841699 x^{77}-2436730722 x^{78}-65443302 x^{79}+15699744 x^{80}+201609 x^{81}-55404 x^{82}-243 x^{83}+81
x^{84}))
$
    \bc $ \mbox{--------------------------------------    } $  \ec } \noindent \\
\parbox[t]{5.8in}{ ${\cal F}^{TG}_{9,2}(x) = {\cal F}^{TG}_{9,7}(x)=$ \\
$ \ds
2 x (31+2969 x-178588 x^2-2861945 x^3+180708818 x^4+1149202778 x^5-85242050741 x^6-226448839081 x^7+23776862205904 x^8+16072225985047
x^9-4394001689464458 x^{10}+2568720303994433 x^{11}+575000137331856909 x^{12}-853237425710411600 x^{13}-55549691402673241247 x^{14}+120611596025123106081
x^{15}+4074100239821587910807 x^{16}-11049124193058988617344 x^{17}-231651086214573892657119 x^{18}+726598274767858525180241 x^{19}+10395903439221716812362362
x^{20}-35967734629047318471860550 x^{21}-374350798195441830980115145 x^{22}+1380643975349217128077557322 x^{23}+10985931159010305843941081111 x^{24}-41994067450210537190323155770
x^{25}-266580571155366898580252852317 x^{26}+1029293042565437005412244801387 x^{27}+5419220993762883134681465449167 x^{28}-20602725029393270159979865339085
x^{29}-93342310884129617069853456607623 x^{30}+340293085937554201075292735099812 x^{31}+1375001344843856823537747969120517 x^{32}-4673427761373051677006233539478180
x^{33}-17450184944820208796486251746885217 x^{34}+53625526187963433125709389656052410 x^{35}+191876012293794229781371262223152326 x^{36}-515020882898039672642880265310157751
x^{37}-1835997549968244259104326149230918764 x^{38}+4131381529856504490320694517824108233 x^{39}+15342951757504366802207846643765090578 x^{40}-27467412642989500928377788927616463520
x^{41}-112320736311313274776222413505735921514 x^{42}+148618757470175356513982212112065730836 x^{43}+722210359114109441086262850089014049051 x^{44}-626589469496183000386335895341108188572
x^{45}-4087246640936655852811390409069436852703 x^{46}+1802437137004171445381319759301492053166 x^{47}+20388166887105599108452731988442338304844
x^{48}-1188949068275551169681319090444271394842 x^{49}-89698130295496733478577574805216574678801 x^{50}-24174504399367301570894868208072473932995
x^{51}+348009073220089551779303460480363214185878 x^{52}+ $ \\ $
186800970253536885265831483040823831777795 x^{53}-$ \\ $1189788117629036097745938209961142642831691
x^{54}-$ \\ $884127707018248102166280265658293890948689 x^{55}+$ \\ $3580219899751001334341190713219406938898021 x^{56}+$ \\ $3195595687319789114805357206158406236895958
x^{57}-$ \\ $9469305671408718050651968929713277672867389 x^{58}-$ \\ $9395397714499887620681608775578311379188775 x^{59}+$ \\ $21983759090035783503570240321665045767556290
x^{60}+$ \\ $23089315278846659080279604287497559865755600 x^{61}-$ \\ $44742511643770297574644269756120419423302276 x^{62}-$ \\ $48106348852914907447023580006937222975631102
x^{63}+$ \\ $79745763317487228620111402954831443544029593 x^{64}+$ \\ $85677698388628384560906639864787981703271280 x^{65}-$ \\ $124357567191443292400320089382851877312229309
x^{66}-$ \\ $131099193749904757818250922133349834981761585 x^{67}+$ \\ $169543046422270372358883325185383808128536414 x^{68}+$ \\ $172883413758867790300329745431481180456699697
x^{69}-$ \\ $201942715025410959758593606214741481886910238 x^{70}-$ \\ $196843406824013496415180594257640298932846526 x^{71}+$ \\ $209998245511058586707360052055251673236621132
x^{72}+
 $ \\ $ \ds
193680684714428101983041090715963104679724793 x^{73}-$ \\ $190507458731589925952197130209633016629575864 x^{74}-$ \\ $164706716016946172149532697814714320132114042
x^{75}+$ \\ $150635545839667647748787809408843420378723324 x^{76}+$ \\ $121000747887454109608038914425883779639961156 x^{77}-$ \\ $103701471821795876139514970310930247576388473
x^{78}-$ \\ $76713468855734350770152122311586591686946320 x^{79}+$ \\ $62071739533568112371533902846993680373626263 x^{80}+$ \\ $41908302241796690551010168537091218135802169
x^{81}-$ \\ $32249914173187990695631054624846877179733594 x^{82}-$ \\ $19688235006606097730929640521620603050572728 x^{83}+$ \\ $14514765545875912407282353930805101026116733
x^{84}+$ \\ $7934646525596436849793390719925058738837479 x^{85}-$ } \\
\parbox[t]{5.8in}{
$5645404037588077033182365290690555678299265 x^{86}-$ \\ $2735292101266018615595276168734759285846976
x^{87}+$ \\ $1892218375039972077225897076936898869102520 x^{88}+$ \\ $803860600136126273282934927210016301528801 x^{89}-$ \\ $544828281962994437234978795022461704880888
x^{90}-$ \\ $200632792292911936949649311442422338786875 x^{91}+$ \\ $134283422448481013810942977728286679847522 x^{92}+$ \\ $42344069526588886372533397904105415336039
x^{93}-28221107701768028425342801953152242413573 x^{94}-7520211315740421283592871876736129118640 x^{95}+5036068825144956148718836661767179757239
x^{96}+1117606598148513717714224303230901592090 x^{97}-759660364399428444911663555421550858341 x^{98}-138084924367825577886967157484745050195 x^{99}+96396944985542559606707193788825753772
x^{100}+14073920724117765760938596023458397860 x^{101}-10236922338603655147424133861099050511 x^{102}-1171781592505633915824572318862001593 x^{103}+904655996709262659832942892054536686
x^{104}+78661555161397042618740910274472546 x^{105}-66112359861020188909244086108822239 x^{106}-4177003654076905873494470430919143 x^{107}+3967113662233114619275635847395231
x^{108}+169958930157680923839847952357424 x^{109}-193844607140658668314657139048268 x^{110}-4965148689969141146255082849882 x^{111}+7636646334105522578743681101642
x^{112}+85361049460727497994792147991 x^{113}-239619583914771636922544009730 x^{114}+184378699519960949414845452 x^{115}+5897359565401568126967397629
x^{116}-  $ \\ $  61956659223934412321402379 x^{117}-111625272308018377588349754 x^{118}+2040430202641315227302757 x^{119}+  $ \\ $ 1583391582962589495924198 x^{120}-38481714006779143468044
x^{121}-16249756963364137225638 x^{122}+  $ \\ $ 467176331902195299924 x^{123}+114749582241566383659 x^{124}-3641175954753495570 x^{125}-516181843115882946
x^{126}+17172088744681188 x^{127}+1292513482761624 x^{128}-42964979521788 x^{129}-1323496592019 x^{130}+41756598765 x^{131}) \left.   \right/    $ \\ $ \ds
((1+3
x)(-1+2 x+7 x^2)(-1+81 x-1792 x^2+7289 x^3+113338 x^4-948939 x^5+891997 x^6+9118681 x^7-25652726 x^8+9992771 x^9+33620979 x^{10}-29903008
x^{11}-9941993 x^{12}+14464685 x^{13}-684910 x^{14}-1263348 x^{15}+43295 x^{16}) (-1-27 x-30 x^2+3263 x^3+10255 x^4-128271 x^5-422489
x^6+2034868 x^7+6331921 x^8-14787664 x^9-46140893 x^{10}+49195065 x^{11}+177327440 x^{12}-56000703 x^{13}-357184328 x^{14}-43503000 x^{15}+357901011
x^{16}+120196901 x^{17}-183392047 x^{18}-80222831 x^{19}+48816126 x^{20}+23044935 x^{21}-6865887 x^{22}-3172062 x^{23}+503757 x^{24}+213075 x^{25}-17955
x^{26}-6750 x^{27}+243 x^{28}+81 x^{29}) (-1-12 x+1296 x^2+6760 x^3-558339 x^4-1161243 x^5+118840379 x^6+34612254 x^7-14636101794 x^8+11650032630
x^9+1139067517383 x^{10}-1675685198967 x^{11}-59425014210170 x^{12}+112413543683178 x^{13}+2171147767245072 x^{14}-4673729741871850 x^{15}-57423564754079904
x^{16}+132752592370932042 x^{17}+1127413588825996376 x^{18}-2710891411914111867 x^{19}-16748804440629896181 x^{20}+41076234964519891361 x^{21}+191073085446307327887
x^{22}-471751936975189881282 x^{23}-1693441623773186077195 x^{24}+4168981592750804837106 x^{25}+11769841572576724100388 x^{26}-28667256121124059466103
x^{27}-64650834761329667094285 x^{28}+154709707930302466051638 x^{29}+282472949704888959670555 x^{30}-659784412595668698793188 x^{31}-986747002345947660623415
x^{32}+2235928540752665118157037 x^{33}+2766102104157402845641392 x^{34}-6048596568103168034002953 x^{35}-6235862859792016465952094 x^{36}+13109354010298653133950159
x^{37}+11311529009445116059397322 x^{38}-22829191668429273215380050 x^{39}-16491985570921054148696568 x^{40}+32012050759806836634743418 x^{41}+19274557174365641463936649
x^{42}-36194090036067799457234367 x^{43}-17977866789394483635892473 x^{44}+33010263268427485989896780 x^{45}+13294805056149648548039391 x^{46}-24269111871425405879756730
x^{47}-7719715827397162309377098 x^{48}+14354944271362017116072472 x^{49}+3467347525245472903206939 x^{50}-6808062303404022976704068 x^{51}-1174671001995294266317512
x^{52}+2576330363368357032796665 x^{53}+285539149598812572013492 x^{54}-772973767306693551468021 x^{55}-43493796133268719517643 x^{56}+182435806074644884579263
x^{57}+1578291814221380714235 x^{58}-33560515290021718006404 x^{59}+1085611778689724348088 x^{60}+4761258378023303618643 x^{61}-319945226855087846232
x^{62}-514768561243945226721 x^{63}+49664633878454946921 x^{64}+41857441719264765459 x^{65}-5059221729603564351 x^{66}-2523288763225221867 x^{67}+356134240037245131
x^{68}+111006922260938166 x^{69}-17560157208954669 x^{70}-3498962768541387 x^{71}+605417634344259 x^{72}+77202421595082 x^{73}-14421262679676 x^{74}-1155401508177
x^{75}+232152034824 x^{76}+11204841699 x^{77}-2436730722 x^{78}-65443302 x^{79}+15699744 x^{80}+201609 x^{81}-55404 x^{82}-243 x^{83}+81 x^{84}))
$
    \bc $ \mbox{--------------------------------------    } $  \ec } \noindent \\
\parbox[t]{5.8in}{ ${\cal F}^{TG}_{9,3}(x) = {\cal F}^{TG}_{9,6}(x)=$ \\
$ \ds
2 x (37+116 x-23347 x^2-1857623 x^3+23054138 x^4+1658327258 x^5-17417701382 x^6-649807102243 x^7+7250324620789 x^8+145892379295639
x^9-1826533024906947 x^{10}-20882240179401760 x^{11}+304971711617604024 x^{12}+1994020382395576351 x^{13}-35899777358788691585 x^{14}-126482024084706850179
x^{15}+3107954076294109597487 x^{16}+4769063027348182649695 x^{17}-203997283075072018280322 x^{18}-30796727137599657444262 x^{19}+10391684355656479304128160
x^{20}-8889790953446751473431800 x^{21}-418692808075182722899176511 x^{22}+677534735097561716635658086 x^{23}+13557081006078163679004027536 x^{24}-30096303941410726327492346153
x^{25}-357572117456073321570928032796 x^{26}+970629421114119355188072797346 x^{27}+7770334598705133902395140968889 x^{28}-24385655353329307903247323759838
x^{29}-140447020509725992912002934371873 x^{30}+494002959452453981192005910339782 x^{31}+2127826709463796246366810653957409 x^{32}-8237450711645200462095793863515701
x^{33}-27188253902642500039450652397258859 x^{34}+114626058616922437235619476025685813 x^{35}+294396135512961980192556588404013907 x^{36}-1344050071175848514152407548203052272
x^{37}-2711440926459327393705605565546421088 x^{38}+13374818665468335020701471366820039873 x^{39}+21303190827817726536861084202487105147 x^{40}-113566908548934930480545351369497111521
x^{41}-143116707967186098828707117011637301005 x^{42}+826317235300924604308880090011290190615 x^{43}+823822805373559725406060735609933523144 x^{44}-5169618364457182037830149789348587769832
x^{45}-4071307680553448788985143986422752984432 x^{46}+27888246117260876974428061303414407314818 x^{47}+$ \\ $17309536584610691517766617348306798760623
x^{48}-$ \\ $130041370174145404585997620607523828835036 x^{49}-$ \\ $63456803833519549992042725911808727002578 x^{50}+$ \\ $525210717726756596993094082616487926160344
x^{51}+$ \\ $201114272119935019403514705790237651321418 x^{52}-$ \\ $1840501651309957697447669266685914146009438 x^{53}-$ \\ $552686500538060751093637947981163333080029
x^{54}+$ \\ $5604289116866772449217279054349337415691054 x^{55}+$ \\ $1321442212373257033736668530226649291870862 x^{56}-$ \\ $14845457980693056804830460113428745911144308
x^{57}-$ \\ $2758718910988368220120863673268295731769638 x^{58}+$ \\ $34240673073616315886807632692993077149353977 x^{59}+$ \\ $5045918324752776979782843081498003092194639
x^{60}-$ \\ $68808008924629424902653118569849935809327403 x^{61}-$ \\ $8107039923472720701621914994517330147693006 x^{62}+$ \\ $120517229401056560197301293293022258976330916
x^{63}+$ \\ $11450048653857330340685697833843219365607214 x^{64}-$ \\ $184009253208087679596679163597644603293154607 x^{65}-$ \\ $-14185430326650963152006871743476054189110730
x^{66}+$ \\ $244898842258228863416490985378354168016898939 x^{67}+$ \\ $15320039924369011905195347260630645952220268 x^{68}-$ \\ $284041241478779349093029310904430679337930648
x^{69}-$ \\ $14254331689316738359831240299313941860932683 x^{70}+$ \\ $286967088390849635071109608801431443079995101 x^{71}+$ \\ $11197732400520137313058811735950130785356399
x^{72}-$ \\ $252384982272215060945883244229103969149946955 x^{73}-$ \\ $7157254869792406692086349601183785585603060 x^{74}+$ \\ $193068914026006736251245981970019452947635808
x^{75}+$ \\ $3418104114406704697507050232820883115647634 x^{76}-$ \\ $128327770987969825505499052077773487791248021 x^{77}-$ \\ $866676417948185498200097129379562121875510
x^{78}+$ \\ $74016396208936238604803388083982718570997751 x^{79}-$ \\ $343244762360597330558729349575909912194103 x^{80}-$ \\ $36988224395917779836206444587956521556424311
x^{81}+$ \\ $609688590350078061889758313752736630349923 x^{82}+$ \\ $15985989888689235493205146921647609138849984 x^{83}-$ \\ $458754682980422641994255289295454118916709
x^{84}-$ \\ $5962740932128082535154159620694597548129155 x^{85}+$} \\
\parbox[t]{5.8in}{ $ \ds
245188667698978289907968803368854501122616 x^{86}+$ \\ $1914899291477973168196809825318011453169712
x^{87}-$ \\ $102538207963238991005653039036052574637783 x^{88}-$ \\ $528046087922705788550263368051218290561957 x^{89}+$ \\ $34721966677282892897306077574246805796453
x^{90}+$ \\ $124657517716325433033826461848132979155649 x^{91}-$ \\ $9663653748921129291559963941103333014753 x^{92}-25109136235865310705212439539798645313822
x^{93}+2225690643221080168123371969412652724585 x^{94}+4299215098313300650501034977190737398552 x^{95}-425333796028671614696987485463114217945 x^{96}-623117616085336186919220078676483385817
x^{97}+67447521779360066018021309585618803671 x^{98}+76087851926994247260583756758172853070 x^{99}-8858331910663314132940890286526034342 x^{100}-7785358123155709143171306133933539774
x^{101}+960043961885231999108424386426197110 x^{102}+663394063022520816346501459443542127 x^{103}-85381995220129919784268078439018598 x^{104}-46742094844290537596318253899040822
x^{105}+6182891448591506876905377453622119 x^{106}+2701150294363639182935799864099408 x^{107}-360678361583665996972799854563252 x^{108}-126837646150147970033595650674437
x^{109}+16698093437864022496266948686496 x^{110}+4788801023595342364085346938979 x^{111}-600174248828417135482901674131 x^{112}-143694131752400660841117412170
x^{113}+16155848412954608122413857637 x^{114}+3386503800730502899294338918 x^{115}-303369068308412780079057189 x^{116}-62104863688059170776717638
x^{117}+3227565770769017117973798 x^{118}+886018513223091179911365 x^{119}+4339109835241241753637 x^{120}-10021783495826405771145 x^{121}-771872417151952497255
x^{122}+93278865090048547914 x^{123}+12830634432227320731 x^{124}-720260973021852279 x^{125}-106840747154565450 x^{126}+4211112043399635 x^{127}+458551859100894
x^{128}-15261801128280 x^{129}-794483135775 x^{130}+23860913580 x^{131}) \left. \right/ $ \\ $
((1+3 x) (-1+2 x+7 x^2)(-1+81 x-1792
x^2+7289 x^3+113338 x^4-948939 x^5+891997 x^6+9118681 x^7-25652726 x^8+9992771 x^9+33620979 x^{10}-29903008 x^{11}-9941993 x^{12}+14464685 x^{13}-684910
x^{14}-1263348 x^{15}+43295 x^{16})(-1-27 x-30 x^2+3263 x^3+10255 x^4-128271 x^5-422489 x^6+2034868 x^7+6331921 x^8-14787664 x^9-46140893
x^{10}+49195065 x^{11}+177327440 x^{12}-56000703 x^{13}-357184328 x^{14}-43503000 x^{15}+357901011 x^{16}+120196901 x^{17}-183392047 x^{18}-80222831
x^{19}+48816126 x^{20}+23044935 x^{21}-6865887 x^{22}-3172062 x^{23}+503757 x^{24}+213075 x^{25}-17955 x^{26}-6750 x^{27}+243 x^{28}+81 x^{29})
(-1-12 x+1296 x^2+6760 x^3-558339 x^4-1161243 x^5+118840379 x^6+34612254 x^7-14636101794 x^8+11650032630 x^9+1139067517383 x^{10}-1675685198967
x^{11}-59425014210170 x^{12}+112413543683178 x^{13}+2171147767245072 x^{14}-4673729741871850 x^{15}-57423564754079904 x^{16}+132752592370932042 x^{17}+1127413588825996376
x^{18}-2710891411914111867 x^{19}-16748804440629896181 x^{20}+41076234964519891361 x^{21}+191073085446307327887 x^{22}-471751936975189881282 x^{23}-1693441623773186077195
x^{24}+4168981592750804837106 x^{25}+11769841572576724100388 x^{26}-28667256121124059466103 x^{27}-64650834761329667094285 x^{28}+154709707930302466051638
x^{29}+282472949704888959670555 x^{30}-659784412595668698793188 x^{31}-986747002345947660623415 x^{32}+2235928540752665118157037 x^{33}+2766102104157402845641392
x^{34}-6048596568103168034002953 x^{35}-6235862859792016465952094 x^{36}+13109354010298653133950159 x^{37}+11311529009445116059397322 x^{38}-22829191668429273215380050
x^{39}-16491985570921054148696568 x^{40}+32012050759806836634743418 x^{41}+19274557174365641463936649 x^{42}-36194090036067799457234367 x^{43}-17977866789394483635892473
x^{44}+33010263268427485989896780 x^{45}+13294805056149648548039391 x^{46}-24269111871425405879756730 x^{47}-7719715827397162309377098 x^{48}+14354944271362017116072472
x^{49}+3467347525245472903206939 x^{50}-6808062303404022976704068 x^{51}-1174671001995294266317512 x^{52}+2576330363368357032796665 x^{53}+285539149598812572013492
x^{54}-772973767306693551468021 x^{55}-43493796133268719517643 x^{56}+182435806074644884579263 x^{57}+1578291814221380714235 x^{58}-33560515290021718006404
x^{59}+1085611778689724348088 x^{60}+4761258378023303618643 x^{61}-319945226855087846232 x^{62}-514768561243945226721 x^{63}+49664633878454946921
x^{64}+41857441719264765459 x^{65}-5059221729603564351 x^{66}-2523288763225221867 x^{67}+356134240037245131 x^{68}+111006922260938166 x^{69}-17560157208954669
x^{70}-3498962768541387 x^{71}+605417634344259 x^{72}+77202421595082 x^{73}-14421262679676 x^{74}-1155401508177 x^{75}+232152034824 x^{76}+11204841699
x^{77}-2436730722 x^{78}-65443302 x^{79}+15699744 x^{80}+201609 x^{81}-55404 x^{82}-243 x^{83}+81 x^{84}))$
    \bc $ \mbox{--------------------------------------    } $  \ec  } \\
\parbox[t]{5.8in}{ ${\cal F}^{TG}_{9,4}(x) = {\cal F}^{TG}_{9,5}(x)=$ \\
$ \ds
2 x (31+602 x-62128 x^2-58391 x^3+34270988 x^4-285377827 x^5-8059003382 x^6+177790238036 x^7+618410529394 x^8-52720475798255
x^9+123512524582362 x^{10}+9704448702305672 x^{11}-41847987558005235 x^{12}-1233105448397528192 x^{13}+6098719240153252291 x^{14}+115383858654790586625
x^{15}-571022906872928214019 x^{16}-8335558852092814790267 x^{17}+38284698185589306324930 x^{18}+481728026053094138134598 x^{19}-1946183370909462881959906
x^{20}-22818007427343440525301891 x^{21}+78196816060271487050861855 x^{22}+897888019651364101278384526 x^{23}-2571362729896620700282934032 x^{24}-29513384773013776775289099986
x^{25}+71221872670999826288486788814 x^{26}+811521320687415415127242509546 x^{27}-1695750562083087963017621966625 x^{28}-18678314840799062874190631743739
x^{29}+35050990985687022136494917074428 x^{30}+360364299102434639565886383276124 x^{31}-629420079142777899235394235948710 x^{32}-5842638826243642457691370704749449
x^{33}+9776264862610987426315534812933455 x^{34}+79876007029710867764051068254909937 x^{35}-130636055605627341007741087359981362 x^{36}-924380571330583710172747815743562106
x^{37}+1495845679305990225160093639189628782 x^{38}+9092089025833971680305452772394630339 x^{39}-14648806418206992446073170945966253940 x^{40}-76307751285063665359668050105532336579
x^{41}+122662870819778111328096620086311910519 x^{42}+548504502704133783250728583747951528543 x^{43}-879059595240098569007623129581037237042 x^{44}-3388291610718301850600761624622016147643
x^{45}+5399927011716624504134817183990138849392 x^{46}+18042901001129661376917216739496915146711 x^{47}-28484216251741782505424257091794045950741
x^{48}-83050338609167435721338619853247026650376 x^{49}+129259919802387560654542295961547523557814 x^{50}+331220362657552916931042314031583957072778
x^{51}-505501919662968615237649815844526775724042 x^{52}-$ \\ $1146854709876173546732874146878392724664946 x^{53}+$ \\ $1706361505310498240935428889886622938002252
x^{54}+$ \\ $3453264181961807015744654464912115034645046 x^{55}-$ \\ $4978787099684469109990480419928978891219014 x^{56}-$ \\ $9053739159698123236542724578977065381130337
x^{57}+$ \\ $12572467013408961417486603793792708214963974 x^{58}+$ \\ $20685972664237534250707926579221660430719186 x^{59}-$ \\ $27506266088296020671511251310888187007016236
x^{60}-$ \\ $41206883029254161321568625742010295742424968 x^{61}+$ \\ $52186889070388129542837432174698053357522283 x^{62}+$ \\ $71569830596167508886445266871716499926763782
x^{63}-$ \\ $85931425830090720836068628754512823117957630 x^{64}-$ \\ $108344764397933829801236245066891148860744613 x^{65}+$ \\ $122880170754388257331558487263426496325261545
x^{66}+$ \\ $142854779958750270533725889429027284949588124 x^{67}-$ \\ $152670962260577911276211698379512123365609206 x^{68}-$ \\ $163881552357509524524997873677791178511382467
x^{69}+$ \\ $164850727062374330365085785299506868797425821 x^{70}+$ \\ $163344789816947261923391516958629047512365245 x^{71}-$ \\ $154696547131462428865119407574750433702116306
x^{72}-$ \\ $141207812031643786085233462443607902927795091 x^{73}+$ \\ $126113657648345685800007320261549058372375129 x^{74}+$ \\ $105647410281875085219152457387827318680106682
x^{75}-$ \\ $89240522194618244088005119943722145425528199 x^{76}-$ \\ $68229849082260969599361519502317406791030427 x^{77}+$ \\ $54731386456054392221266195284707825852014712
x^{78}+$ \\ $37916620206596841759322248618333094180954680 x^{79}-$ \\ $29026687823631807740938289321189220296729480 x^{80}-$ \\ $18061000773875435280552791453437231629650153
x^{81}+$ \\ $13268607664479011329783660785340186235965549 x^{82}+$ \\ $7338965547068730635108321833618810117214721 x^{83}-$ \\ $5204174614680949367097056078540062478516681
x^{84}-$ \\ $2528796174779026440622152570792252181745219 x^{85}+$} \\
\parbox[t]{5.8in}{$1740560729330149347898601197925776748230130 x^{86}+$ \\ $733279596195519438448753494194702437540504
x^{87}-$ \\ $492198695473440115367180945786577355455370 x^{88}-$ \\ $177150997364736816040887788597739907143955 x^{89}+$ \\ $116262710292005292850067102681927755061929
x^{90}+$ \\ $35163164631392939675837449983938085577959 x^{91}-$ \\ $22517426427067004704098734282185266987102 x^{92}-$ \\ $5615388955065139699307482825410977173149
x^{93}+3461626757265782976025446241428824829498 x^{94}+695779812741289700648526772084179309606 x^{95}-393401993279283138159958768904853935916 x^{96}-61822699353433287175557137924848600176
x^{97}+25812270996211253576946595513230789750 x^{98}+2985375898085984475566575780423749771 x^{99}+937369013566333099250034234345529434 x^{100}+106565329356341600420634395456692455
x^{101}-541489212594673337864196311479176765 x^{102}-35940101033480717301336403308550782 x^{103}+86328675926524247747720836894334109 x^{104}+3538747260753841644126614490303756
x^{105}-9074344187155068160071771704407443 x^{106}-189320299469529894463856888507679 x^{107}+707853639339186336678519901542447 x^{108}+3329457270154587620097507014634
x^{109}-42482614550803139963140441822269 x^{110}+344897686005821078120470950138 x^{111}+1984420174540513431179145762441 x^{112}-35905415462483512503632787180
x^{113}-72169164215957833763740271892 x^{114}+1851054193769077238320669635 x^{115}+2029255579724550560138579283 x^{116}-63240177787467630324374907
x^{117}-43525404524172780590079771 x^{118}+1522327290640351898257629 x^{119}+697515553781480029001832 x^{120}-26066459309785602251445 x^{121}-8101345709182717927089
x^{122}+312521014755158955552 x^{123}+65193537548917994607 x^{124}-2525557581716525577 x^{125}-338785060827590559 x^{126}+12837256720911612 x^{127}+1004062961221668
x^{128}-36087360798015 x^{129}-1269809536464 x^{130}+41756598765 x^{131}) \left. \right/ $ \\ $((1+3 x) (-1+2 x+7 x^2) (-1+81 x-1792
x^2+7289 x^3+113338 x^4-948939 x^5+891997 x^6+9118681 x^7-25652726 x^8+9992771 x^9+33620979 x^{10}-29903008 x^{11}-9941993 x^{12}+14464685 x^{13}-684910
x^{14}-1263348 x^{15}+43295 x^{16}) (-1-27 x-30 x^2+3263 x^3+10255 x^4-128271 x^5-422489 x^6+2034868 x^7+6331921 x^8-14787664 x^9-46140893
x^{10}+49195065 x^{11}+177327440 x^{12}-56000703 x^{13}-357184328 x^{14}-43503000 x^{15}+357901011 x^{16}+120196901 x^{17}-183392047 x^{18}-80222831
x^{19}+48816126 x^{20}+23044935 x^{21}-6865887 x^{22}-3172062 x^{23}+503757 x^{24}+213075 x^{25}-17955 x^{26}-6750 x^{27}+243 x^{28}+81 x^{29})
(-1-12 x+1296 x^2+6760 x^3-558339 x^4-1161243 x^5+118840379 x^6+34612254 x^7-14636101794 x^8+11650032630 x^9+1139067517383 x^{10}-1675685198967
x^{11}-59425014210170 x^{12}+112413543683178 x^{13}+2171147767245072 x^{14}-4673729741871850 x^{15}-57423564754079904 x^{16}+132752592370932042 x^{17}+1127413588825996376
x^{18}-2710891411914111867 x^{19}-16748804440629896181 x^{20}+41076234964519891361 x^{21}+191073085446307327887 x^{22}-471751936975189881282 x^{23}-1693441623773186077195
x^{24}+4168981592750804837106 x^{25}+11769841572576724100388 x^{26}-28667256121124059466103 x^{27}-64650834761329667094285 x^{28}+154709707930302466051638
x^{29}+282472949704888959670555 x^{30}-659784412595668698793188 x^{31}-986747002345947660623415 x^{32}+2235928540752665118157037 x^{33}+2766102104157402845641392
x^{34}-6048596568103168034002953 x^{35}-6235862859792016465952094 x^{36}+13109354010298653133950159 x^{37}+11311529009445116059397322 x^{38}-22829191668429273215380050
x^{39}-16491985570921054148696568 x^{40}+32012050759806836634743418 x^{41}+19274557174365641463936649 x^{42}-36194090036067799457234367 x^{43}-17977866789394483635892473
x^{44}+33010263268427485989896780 x^{45}+13294805056149648548039391 x^{46}-24269111871425405879756730 x^{47}-7719715827397162309377098 x^{48}+14354944271362017116072472
x^{49}+3467347525245472903206939 x^{50}-6808062303404022976704068 x^{51}-1174671001995294266317512 x^{52}+2576330363368357032796665 x^{53}+285539149598812572013492
x^{54}-772973767306693551468021 x^{55}-43493796133268719517643 x^{56}+182435806074644884579263 x^{57}+1578291814221380714235 x^{58}-33560515290021718006404
x^{59}+1085611778689724348088 x^{60}+4761258378023303618643 x^{61}-319945226855087846232 x^{62}-514768561243945226721 x^{63}+49664633878454946921
x^{64}+41857441719264765459 x^{65}-5059221729603564351 x^{66}-2523288763225221867 x^{67}+356134240037245131 x^{68}+111006922260938166 x^{69}-17560157208954669
x^{70}-3498962768541387 x^{71}+605417634344259 x^{72}+77202421595082 x^{73}-14421262679676 x^{74}-1155401508177 x^{75}+232152034824 x^{76}+11204841699
x^{77}-2436730722 x^{78}-65443302 x^{79}+15699744 x^{80}+201609 x^{81}-55404 x^{82}-243 x^{83}+81 x^{84}))
$    \bc $ \mbox{--------------------------------------    } $  \ec  } \\
\parbox[t]{5.8in}{ ${\cal F}^{TG}_{9,0}(x) = $ \\ $ \ds  2 x+19682 x^2+93524 x^3+17323226 x^4+329212322 x^5+28432011752 x^6+878082461552 x^7+55552654729898 x^8+2132941421905988 x^9+116306062547543522
x^{10}+4958358771587057828 x^{11}+251329211093805920936 x^{12}+11279496913606980300530 x^{13}+551467123545121317792872 x^{14}+25385504162517911001366644 x^{15}+1219155378352018709924802122
x^{16}+  $ \\ $  56830091153236092352304798696 x^{17}+2705278285367600567705111414600 x^{18}+  $ \\ $ 126891148055785602672485993306402 x^{19}+6014007162408033910079155162601306 x^{20}+  $ \\ $ 282956991958974571162226344697994140
x^{21}+13381716290026149071316106980725715704 x^{22}+630565873382416175092989923883359477120 x^{23}+29788991646554019144315781518130939668488 x^{24}+1404760707866092787952973880341347611130872
x^{25}+  \ldots $}
    \bc $ \mbox{--------------------------------------    } $  \ec  \noindent
\parbox[t]{5.8in}{ ${\cal F}^{TG}_{9,1}(x) = {\cal F}^{TG}_{9,8}(x) = $ \\
$ \ds
512 x+1946 x^2+439202 x^3+6299282 x^4+656565692 x^5+17594491250 x^6+1226485102814 x^7+43937205514832 x^8+2514403471437956 x^9+103599732118131110
x^{10}+5378551863832846706 x^{11}+237352093034087642312 x^{12}+11742882034081066779986 x^{13}+536066318783898296950466 x^{14}+25896591033787222616098172
x^{15}+1202176502875288256476167434 x^{16}+57393772736506764620828358704 x^{17}+2686555983308620473457947952094 x^{18}+127512813472592633115034928692166
x^{19}+5993360903047359045133640862863048 x^{20}+283642590152288706874076231594752178 x^{21}+13358947707870269363501184403403075606 x^{22}+631321970380500838514137593512361615842
x^{23}+29763882285739346946762455633018909249438 x^{24}+1405594548223937510578673458656900828693482 x^{25} + \ldots
$}    \bc $ \mbox{--------------------------------------    } $  \ec  \noindent
\parbox[t]{5.8in}{ ${\cal F}^{TG}_{9,2}(x) = {\cal F}^{TG}_{9,7}(x)=$ \\
$ \ds
62 x+8480 x^2+137072 x^3+14210138 x^4+363462332 x^5+26532003710 x^6+915555394790 x^7+53914258528706 x^8+2176161943348964 x^9+114670545449448656
x^{10}+5007516324469120688 x^{11}+249596894049237312338 x^{12}+11334597592658842230830 x^{13}+549588085827578958461546 x^{14}+25446749939033251837348298
x^{15}+1217097404125626538920277424 x^{16}+56897881395054910975811696630 x^{17}+2703015329182348287645876379988 x^{18}+126966032665145462695913910346544
x^{19}+6011514644444855103579794806875086 x^{20}+283039637604878712217084941098443106 x^{21}+13378968973784682198341681329010893754 x^{22}+630657046658161443304774645270935448424
x^{23}+29785962563889394025535498143753699848088 x^{24}+1404861270062874242661195330285432136118762 x^{25} + \ldots
$}
    \bc $ \mbox{--------------------------------------    } $  \ec
\parbox[t]{5.8in}{ ${\cal F}^{TG}_{9,3}(x) = {\cal F}^{TG}_{9,6}(x)=$ \\
$ \ds
74 x+3266 x^2+262148 x^3+7911458 x^4+544354964 x^5+19596376844 x^6+1131327794972 x^7+46330377148466 x^8+2420680879349042 x^9+106356601817379116
x^{10}+5279958847445383736 x^{11}+240458297140884450428 x^{12}+11636274357637640578604 x^{13}+539526552908365784839508 x^{14}+25779995022870311375877998
x^{15}+1206010136418493293262113074 x^{16}+  $ \\ $ 57265642694742304204932407954 x^{17}+2690792573456768162701834939058 x^{18}+  $ \\ $ 127371723915042979126239574314290
x^{19}+5998037422344075827100765499725548 x^{20}+283487096263156378225598193422631764 x^{21}+  $ \\ $ 13364107174651935060195394772723475284 x^{22}+631150537929216637202743139271246555782
x^{23}+  $ \\ $ 29769573266566030463726598764257042784972 x^{24}+1405405513079478142247631257165439831264914 x^{25} + \ldots
$}
\bc $ \mbox{--------------------------------------    } $  \ec
\parbox[t]{5.8in}{ ${\cal F}^{TG}_{9,4}(x) = {\cal F}^{TG}_{9,5}(x)=$ \\
$ \ds
62 x+3746 x^2+175898 x^3+10217936 x^4+436406162 x^5+22829200760 x^6+1010175275654 x^7+50160909680426 x^8+2287267619640590 x^9+110686584320790536
x^{10}+5133314650417419902 x^{11}+245277146687916214622 x^{12}+11474916366982044692240 x^{13}+544859576007540551875442 x^{14}+25602264782906657432963132
x^{15}+1211899751598998868201637346 x^{16}+57069758021633766370530343418 x^{17}+2697291483166647988938039090980 x^{18}+127155761131991008479105941423036
x^{19}+6005206256008264542881364899900432 x^{20}+283248960369775483820801745840547736 x^{21}+13372013892532434686796439722545979536 x^{22}+630887933367539681987436563158641597776
x^{23}+29778293309209194216031866839553472715330 x^{24}+1405115916111221404744271544679934700480692 x^{25}+ \ldots
$
\bc $ \mbox{--------------------------------------    } $  \ec}  \\
\parbox[t]{5.8in}{ ${\cal F}^{KB}_{9,0}(x) = {\cal F}^{KB}_{9,1}(x)= {\cal F}^{KB}_{9,2}(x)= \ldots = {\cal F}^{KB}_{9,8}(x)=$ \\
$ \ds
-2 x (81-3584 x+21867 x^2+453352 x^3-4744695 x^4+5351982 x^5+63830767 x^6-205221808 x^7+89934939 x^8+336209790 x^9-328933088
x^{10}-119303916 x^{11}+188040905 x^{12}-9588740 x^{13}-18950220 x^{14}+692720 x^{15}) \left. \right/ $ \\ $(-1+81 x-1792 x^2+7289 x^3+113338 x^4-948939 x^5+891997
x^6+9118681 x^7-25652726 x^8+9992771 x^9+33620979 x^{10}-29903008 x^{11}-9941993 x^{12}+14464685 x^{13}-684910 x^{14}-1263348 x^{15}+43295 x^{16})$
    \bc $ \mbox{--------------------------------------    } $  \ec
  ${\cal F}^{KB}_{9,0}(x) = {\cal F}^{KB}_{9,1}(x)= {\cal F}^{KB}_{9,2}(x)= \ldots = {\cal F}^{KB}_{9,8}(x)=$  \\
$ \ds
162 x+5954 x^2+235704 x^3+10509978 x^4+481196722 x^5+22392890660 x^6+1049463249472 x^7+49359794952522 x^8+2325552138201720 x^9+109659221103616794
x^{10}+5173004682647753432 x^{11}+244077563657520119748 x^{12}+11517426401813873331550 x^{13}+543505354511098122072640 x^{14}+25648522857745864224017944
x^{15}+1210391440932092505047153962 x^{16}+57120466761012398267514757872 x^{17}+2695621002621818932450534427948 x^{18}+127211534491703307722312706167998
x^{19}+6003360623788571438605965682391618 x^{20}+283310395637685903715253987621223424 x^{21}+13369976865300532409887251655415617332 x^{22}+630955650005917041901241431720829943204
x^{23}+29776046055262438938714290747872606532484 x^{24}+1405190578091235043157390783155692506293072 x^{25} + \ldots
$     \bc $ \mbox{--------------------------------------    } $  \\ $ \mbox{--------------------------------------    } $  \ec } \newpage
   \noindent  \parbox[t]{5.8in}{  ${\cal F}^{TG}_{10,0}(x) = $  \\ $ -((2 x (-1-9151 x+347945 x^2+32630909 x^3-1111958814 x^4-43059427711 x^5+1454258019400 x^6+27897573145887 x^7-1026234961927855
x^8-9626609938852254 x^9+439838356451493758 x^{10}+
 1628133610322615893 x^{11}-122753309907463820397 x^{12}-12165541576960418735 x^{13}+23396064432092877164330 x^{14}-55479828982332399840627 x^{15}-3156778227985365662914842
x^{16}+
 13305008163630408742288033 x^{17}+310433388026323234253833161 x^{18}-1780846057357121293505876042 x^{19}-22792857035567496860423638641 x^{20}+162922006430776639152654652722
x^{21}+
 1274313922487530381959727996331 x^{22}-10987581502323951838783094283475 x^{23}-55081249560236801856083928794799 x^{24}+569581697130976963801381811678395
x^{25}+
 1860211422063105079751121545305182 x^{26}-23332601101252367030733797918191661 x^{27}-49348269836403883592769356882049772 x^{28}+770618136484896262258099447574152201
x^{29}+
 1026992805510492088903425455653980263 x^{30}-20836036554592199246423868188724971034 x^{31}-16573974562165266647090224837533918371 x^{32}+466721674683114699071654609351639204872
x^{33}+
 200926466489825538448264946086852868981 x^{34}-8742269375230143447817788080202549189973 x^{35}-1676799734991725628732806337912008685148 x^{36}+137936049161637717961352946672178189862174
x^{37}+
 6625227492896686553628879144818552433565 x^{38}- $  \\
 $1843482710437592550239575408748402713583236 x^{39}+ 43644530383607927636339840658529182219497
x^{40}+ $ \\ $20955584957992526063541157506131710655568721 x^{41}- $ \\ $
 861790766400524637156684131612739418597862 x^{42}- $ \\ $203201390724040837312671982470373739395790156 x^{43}+ $ \\ $3831822153971494196536700675415440451263716
x^{44}+ $ \\ $
 1684091069034456953156348867312100123659566069 x^{45}+ $ \\ $52181956441857109176676523459189483560487729 x^{46}- $ \\ $11944024254666937113793585759056577061038638033
x^{47}- $ \\ $
 1173778148424135508231435116156684814132109699 x^{48}+ $ \\ $72546666541158499993317581343874374307863076526 x^{49}+ $ \\ $12675137612308776695357437268529139983044090502
x^{50}- $ \\ $
 377584025518565663291556330183036915312815237865 x^{51}- $ \\ $95808266272635295530324745024226596208209392686 x^{52}+ $ \\ $1685001088194340853348604241492944760818293716422
x^{53}+ $ \\ $
 556584614149605931664617899058800489648174541113 x^{54}- $ \\ $6452604793842113875169329824955125191445807360414 x^{55}- $ \\ $2585893805203607517622630862346213828983636147364
x^{56}+ $ \\ $
 21229132325735453853388899194622136109851492357478 x^{57}+ $ \\ $9816804787372121120299810404061487694729355927931 x^{58}- $ \\ $ \ds 60102559527735875686107702087183652815935728627741
x^{59}- $ \\ $
 30863478638081393444183045432252651407407470115347 x^{60}+  $ \\ $146726216853128534798567966621635711715961294066747 x^{61}+ $ \\ $ 81110592878143379227573850825040143276794720183670
x^{62}- $ \\ $
 309621213005285945544931156045339688800673056118659 x^{63}- $ \\ $179420700319817687802626961833497916657548401686502 x^{64}+ $ \\ $566279477652227705803407878142138840370349063877955
x^{65}+ $ \\ $
 335857008503568516947766715486327625014734116633612 x^{66}- $ \\ $900170464089421945590395323819094883834995058866319 x^{67}- $ \\ $534242939127592612991590306559086913589559561479697
x^{68}+ $ \\ $
 1247093884550591982313533609358694129347689522910186 x^{69}+ $ \\ $724442214016021932843739815397488229702226707072508 x^{70}- $ \\ $1509487343895047224650012935294153419066721459683587
x^{71}- $ \\ $
 839271168212553709169756832163656736950668232441832 x^{72}+ $ \\ $1599572707256452103095507966069918905634929051245092 x^{73}+ $ \\ $831632150351437120693073974786756141316736490976905
x^{74}- $ \\ $
 1486142516369902350036432348249790668123105372025652 x^{75}- $ \\
 $704781821972242576557127944923937138521488523545922 x^{76}+ $ \\ $1211551285099572872747379473132969089206092774952666
x^{77}+ $ \\ $ 509989157853116593058567471481648160183329689566777 x^{78}- $  \\ $866708219687141380076155070123957374746816733077413 x^{79}- $ \\ $313905899485261248470325492192336440677983634506572
x^{80}+ $ } \\   \noindent  \parbox[t]{5.8in}{$
 543696896205934524511646796995027934063312690130012 x^{81}+ $ \\ $163180063313296811897779558945233322917702114677823 x^{82}- $ \\ $298664940720626262100714965046936447441329897633830
x^{83}- $ \\ $
 70713301360488195932621639394412544181287142896122 x^{84}+ $ \\ $143366409062186817764099502330812423796748548496434 x^{85}+ $ \\ $24906704011680864154765370399726626081102437061420
x^{86}- $ \\ $
 59972219544087831608373314199303218526868045124116 x^{87}- $ \\ $6731239318115495847334328999546848337423397208807 x^{88}+ $ \\ $21787376329068211553727304382557845337126346231301
x^{89}+ $ \\ $
 1156125990564889559214368647082722873742098074257 x^{90}- $ \\ $  6845592608181034939712806942210122315516232962599 x^{91}+ $ \\ $23844497421084482974481218881731411456159403536
x^{92}+ $ \\ $
 1850998162622180993633545006202720661679479060631 x^{93}- $ \\ $106469208965877682345326448691473373424462297494 x^{94}- $ \\ $428135945761179213276661103833830436788414323811
x^{95}+ $ \\ $
 46923908040885779396362911701229134077100532333 x^{96}+ $ \\ $84090794684063816224639066567057159073265996804 x^{97}- $ \\ $13520419688157051956572237498076551647579012246
x^{98}- $ \\ $
 13896736815759293984625116732996692615222231381 x^{99}+ $ \\ $2957685154146647389968860264219751525105519082 x^{100}+ $ \\ $1909391759965272903454522450894972712174443452
x^{101}- $ \\ $
 513073792593596873714296154410994534028389867 x^{102}- $ \\ $214588305972369288186442559814146192735891502 x^{103}+ $ \\ $71626195801538526399512872232844091347811728
x^{104}+ $ \\ $
 19252571403473135237214420040871485319832852 x^{105}- $ \\ $8069188152442193498330722485469055366321591 x^{106}- $ \\ $1322708122930232301704108897975588485152163
x^{107}+ $ \\ $
 730760995108537555632664350434153924720548 x^{108}+ $ \\ $63499207149871357540336093402805870542506 x^{109}-52746321582381669438250228722772178106273
x^{110}-1496196721579050079617663944643865674440 x^{111}+
 2996281218014420061168398797091409934070 x^{112}-51980157745494094967471589784766408560 x^{113}-131704170461599489961387419506815291350 x^{114}+7109707286601427648085026491346312800
x^{115}+
 4382330595020159556031761516095691000 x^{116}-369029482469239650589666771909951600 x^{117}-107275945584063726503451845613872800 x^{118}+11816162475785331177592299057128000
x^{119}+
 1860471871970364152365855865536000 x^{120} $ \\ $ \ds -248730207344469916899991298880000 x^{121}-21734158909515171677120768880000 x^{122}+3403380725721404654558468000000
x^{123}+
 160240657641720476168152000000 x^{124}-28732432086756426026768000000 x^{125}-704779451485724979072000000 x^{126}+135499142872993306880000000
x^{127}+2036097084471582720000000 x^{128}-
298168596631449600000000 x^{129}-4771517801267200000000 x^{130}+203108376576000000000 x^{131}+3837984768000000000 x^{132})) \left. \right/ $ \\ $ \ds
 ( (-1+x) (1+x) (1+2 x) (-1+4 x) (1+4 x) ( 1-8 x+4 x^2) ( 1+5 x+5 x^2) ( 1+7 x+8 x^2) ( 1+18 x+63 x^2+50 x^3+13 x^4+x^5) ( -1+34 x-160 x^2+248 x^3-152 x^4+32 x^5) ( 1+34 x+160 x^2+248 x^3+152 x^4+32 x^5)
 ( -1+18 x-91 x^2+110 x^3+175 x^4-515 x^5+455 x^6-175 x^7+25 x^8) ( -1-18 x-91 x^2-110 x^3+175 x^4+515 x^5+455 x^6+175 x^7+25 x^8)
( -1+121 x-3901 x^2+49003 x^3-293115 x^4+896942 x^5-1381585 x^6+970100 x^7-264304 x^8+15552 x^9)
( -1+8 x+14 x^2-194 x^3+302 x^4-65 x^5-120 x^6+62 x^7+4 x^8-7 x^9+x^{10}) ( -1-8 x+14 x^2+194 x^3+302 x^4+65 x^5-120 x^6-62 x^7+4
x^8+7 x^9+x^{10})
 ( 1-36 x+504 x^2-3603 x^3+14416 x^4-33263 x^5+44489 x^6-34794 x^7+16069 x^8-4345 x^9+653 x^{10}-47 x^{11}+x^{12}) ( 1+53 x+951
x^2+7589 x^3+32774 x^4+84127 x^5+133734 x^6+133056 x^7+81936 x^8+30215 x^9+6265 x^{10}+650 x^{11}+25 x^{12})
 ( 1+92 x+3048 x^2+50453 x^3+483576 x^4+2928888 x^5+11832412 x^6+32949126 x^7+64377394 x^8+88750186 x^9+85833700 x^{10}+57206921 x^{11}+25441290
x^{12}+7182529 x^{13}+1203428 x^{14}+111956 x^{15}+5340 x^{16}+120 x^{17}+x^{18})
 ( 1-66 x+1833 x^2-28651 x^3+283174 x^4-1881547 x^5+8719340 x^6-28859906 x^7+69342834 x^8-122339404 x^9+159752386 x^{10}-155169581 x^{11}+112318552
x^{12}-60476870 x^{13}+24066800 x^{14}-6990095 x^{15}+
1450676 x^{16}-207780 x^{17}+19365 x^{18}-1050 x^{19}+25 x^{20}))) $} \\  \noindent
  \parbox[t]{5.8in}{ $ \ds
 -( ( 20 x^2 ( 1395-10498446 x^2+29534328219 x^4-44962185811984 x^6+43313020335654151 x^8-28790160274050440896 x^{10}+13952606364650668570418
x^{12}-5122525183178332979714504 x^{14}+
 1465606225295080443499814135 x^{16}-334001055654838940923381631142 x^{18}+61696895553372946658371608398260 x^{20}-9371065799739996444209867470746342
x^{22}+
 1184503190700523363931367496108174258 x^{24}-125873847120266594700589074184590986574 x^{26}+11344834830006609241007414340813175962903 x^{28}- $ \\ $873842489680814103936875801650855210599904
x^{30}+ $ \\ $
 57907359971798417848233686641575233706022820 x^{32}- $ \\ $3320812268179294358204334604806491772747303602 x^{34}+ $ \\ $165657978475671321362400699182472698198694568952
x^{36}- $ \\ $
 7221545224672198500014257691817957619188402208714 x^{38}+ $ \\ $276228270124730321172471466072121221008101667866506 x^{40}- $ \\ $9304694183914723348006039786741894531472168340301972
x^{42}+ $ \\ $
 276907880953847085086437287897372785475019321356677065 x^{44}- $ \\ $7301576915110758180634769803322294003932348734570932318 x^{46}+ $ \\ $171023510685604924139689623520714942817835675358492548115
x^{48}- $ \\ $
 3566447900281137024307343060578168612205432568968847462682 x^{50}+ $ \\ $66347446909327938830992487845965683100434068427687071237261 x^{52}- $ \\ $1103015887452521380635163175757711698032859550816066654144668
x^{54}+ $ \\ $
 16412373368163821280781570897031453307942086956673492968028772 x^{56}-  $ \\ $ 218859448891740400621199528307947000606528891519988011464042062 x^{58}+  $ \\ $
 2618496631038043496963760995802467321362483006344261374878222375 x^{60}-  $ \\ $ 28134736548930259326557059719903729906254185679109933567244714668 x^{62}+  $ \\ $
 271692816610864608627941120206874649343769324621546506638650315757 x^{64}-  $ \\ $ 2359573015932305442438512344717322830043626761790797714857545826226
x^{66}+  $ \\ $
 18438487574361840147231151293489327092847730619973130347501491544800 x^{68}-129693313979105951923356528018350389884972343260608426027988043534730
x^{70}+
 821348184337626735131004749997957901358062638811657029110805738888567 x^{72}-4684149040659253218080356892686712665895839636325024621901530451155146
x^{74}+
 24058549763509608253643215800785297074909236387032764850969481068523107 x^{76}-111290133750663185113119266316739726324307850446684482518585078820392412
x^{78}+
 463642331187130942872032557996634341912284570942109214236126043294324903 x^{80}-1739494005668525928603987080304993788760050437434022486206499352918312636
x^{82}+
 5876783952491307312500617998657940146305566006583368180407883657137544712 x^{84}- $ \\ $17876873443500286928830963161375413691980437605242812862398760595868799372
x^{86}+
 48959399868922071101775253850186455279797897834511746377256280093323265341 x^{88}- $ \\ $120707614560141400371154017037364619693063066905917611685399626092063367954
x^{90}+ $ \\ $
 267889639134957233391123221383595764297350957130066266269670609241761561628 x^{92}- $ \\ $535153340560911923747416719989046744462199539639218602298156637400325127050
x^{94}+ $ \\ $
 962257798369553642176404310911523528936752524200954022898009542694688066802 x^{96}- $ \\ $1557389676242939690964718764820921304429144257410135379168619027928784877626
x^{98}+ $ \\ $
 2268867499669128738095447565151916396237775176281576049585318798924089803177 x^{100}- $ \\ $ \ds 2975435327786708654457225199791039823504126016407255329525621664495980526412
x^{102}+ $ \\ $
 3512796801961120156367485640538130806660750734606120794614561406863230964953 x^{104}- $ \\ $3733820885516339916939972994788301323114685586021102231112592184576654482250
x^{106}+ $ \\ $
 3573473324143548611591038150435840683680223412926436828745623086816913445526 x^{108}- $ \\ $3079647912702619595172957557057038155388131953918495662022922646795054627874
x^{110}+ $ \\ $
 2390100555657890569842817178175813395058661542528666718340761457188820895825 x^{112}- $ \\ $1670541510225799156648502053116314501707523241628852928156742789632568038398
x^{114}+
 1051555212789828473896482704143053990908213964758242869285324939457517514741 x^{116}-596115883824954436499688752528621606054604169572033126577473154283291889508
x^{118}+
 304313862971460245626129595444719249388134313341162595367372101629392066320 x^{120}-139877736704528344586003947229262673250084069541282964739477455803641233300
x^{122}+
 57879623708654705450983440739182443279867141818597326383494355698585913998 x^{124}- $ } \\  \noindent
  \parbox[t]{5.8in}{$ 21554342757642145440453183250793493101312439796467779036241626625039162688
x^{126}+
 7221395160062704788603662388730656469065539266895520455783818977029641494 x^{128}-2175652316238552701649676552668784158856814678588799693298260718458059374
x^{130}+ $ \\$
 589115020596673091498039722192528620912783838470353405649601454579492541 x^{132}-143272830066040175042410165236815448427153455240621936927209345769202740
x^{134}+
 31270679215976194988613696643856615165810693180943541664002081667052890 x^{136}-6119515563398381264354386635125111650469391712590450777844115312384382
x^{138}+
 1072587998756312977701031692943554845304123415010941155182397176584088 x^{140}-168166977608434260695843529900004682752475667639647969010939937484786
x^{142}+
 23551207523233927325171285179391331040088851514246334934791180951790 x^{144}-2941235801219027721156632256494176557526041731755142360144887869024
x^{146}+
 326937581247772078045617049193150296196128135707331902355426907181 x^{148}-32275463289519537535752963423926680721082658666566950325209921590
x^{150}+
 2822733858020670971524873241108544944320003350642471675840970427 x^{152}-218080241052350094444950750248167634075885214761472487703696344 x^{154}+
 14834988596438495074307143267693640743929737502075137472093912 x^{156}-885203618742010986193286520653698283660735056432402221646040 x^{158}+ $ \\ $
 46130750512598204049027343233885057192922197555890722002980 x^{160}- $ \\ $2088968303602671457805327560112447728375335704774745736550 x^{162}+ $ \\ $81715229105323435143424062634493516702733025349169811700
x^{164}- $ \\ $
 2742152230163637736794448099795976508315245771547802250 x^{166}+ $ \\ $78294331746974078627014438546475178604560005418196125 x^{168}- $ \\ $1883458617179795955823811992656799415438680271348750
x^{170}+ $ \\ $
 37724174593224404947378346603725468691501908587500 x^{172}- $ \\ $620050207121892953400022652194264801552782768750 x^{174}+ $ \\ $8214413735975041444755179568141897103549556250
x^{176}- $ \\ $
 85752762699520608679205773115282855229250000 x^{178}+ $ \\ $685248778001013593988941840285985417500000 x^{180}- $ \\ $4035126800571803976550788405666750000000
x^{182}+ 16631007480485006982183011180000000000 x^{184}-  $ \\ $ 44595164631827946644543120000000000 x^{186}+69561806631999823257600000000000
x^{188}-   $ \\ $ 52374089642845696000000000000 x^{190}+14512627712000000000000000 x^{192})) \left. \right/ $ \\ $ \ds
 ( (-1+x) (1+x) (-1+2 x) (1+2 x) (-1+11 x) (1+11 x) ( 1-3 x+x^2) ( 1+3 x+x^2) ( 1-13 x+11 x^2) ( 1-8 x+11
x^2) ( 1+8 x+11 x^2) ( 1+13 x+11 x^2) ( 1-12 x^2+8 x^3) ( -1+12 x^2+8 x^3)
 ( -1+104 x-2661 x^2+24090 x^3-91993 x^4+158236 x^5-121128 x^6+44736 x^7-7808 x^8+512 x^9) ( 1+104 x+2661 x^2+24090 x^3+91993 x^4+158236
x^5+121128 x^6+44736 x^7+7808 x^8+512 x^9)
 ( -1+5 x+45 x^2-305 x^3+65 x^4+2175 x^5-1800 x^6-4525 x^7+1975 x^8+2500 x^9-875 x^{10}-375 x^{11}+125 x^{12})( -1-5 x+45 x^2+305
x^3+65 x^4-2175 x^5-1800 x^6+4525 x^7+1975 x^8-2500 x^9-875 x^{10}+375 x^{11}+125 x^{12})
 ( 1-22 x-214 x^2+4230 x^3-15001 x^4-2783 x^5+60858 x^6-12548 x^7-46005 x^8+16340 x^9+10377 x^{10}-5171 x^{11}-329 x^{12}+413 x^{13}-27 x^{14}-9
x^{15}+x^{16})
 ( 1+22 x-214 x^2-4230 x^3-15001 x^4+2783 x^5+60858 x^6+12548 x^7-46005 x^8-16340 x^9+10377 x^{10}+5171 x^{11}-329 x^{12}-413 x^{13}-27 x^{14}+9
x^{15}+x^{16})
 ( 1-82 x+2898 x^2-59227 x^3+789803 x^4-7334275 x^5+49339066 x^6-246611784 x^7+930982131 x^8-2681675595 x^9+5927069364 x^{10}-10071216252
x^{11}+13140900495 x^{12}-13114381556 x^{13}+9943609371 x^{14}-
5673467839 x^{15}+2404556811 x^{16}-744137999 x^{17}+164387022 x^{18}-25156085 x^{19}+2562619 x^{20}-164923 x^{21}+6260 x^{22}-125 x^{23}+x^{24})
 ( 1+82 x+2898 x^2+59227 x^3+789803 x^4+7334275 x^5+49339066 x^6+246611784 x^7+930982131 x^8+2681675595 x^9+5927069364 x^{10}+10071216252
x^{11}+13140900495 x^{12}+13114381556 x^{13}+9943609371 x^{14}+
5673467839 x^{15}+2404556811 x^{16}+744137999 x^{17}+164387022 x^{18}+25156085 x^{19}+2562619 x^{20}+164923 x^{21}+6260 x^{22}+125 x^{23}+x^{24})
 ( 1-104 x+4410 x^2-104155 x^3+1561045 x^4-15952817 x^5+116275468 x^6-623060570 x^7+2507306175 x^8-7693787635 x^9+18199448138 x^{10}-33436393552
x^{11}+47933912805 x^{12}-53730741510 x^{13}+
47072037995 x^{14}-32132854819 x^{15}+16992690151 x^{16}-6899340865 x^{17}+2123302560 x^{18}-486469225 x^{19}+80870375 x^{20}-9387875
x^{21}+715250 x^{22}-31875 x^{23}+625 x^{24})
 ( 1+104 x+4410 x^2+104155 x^3+1561045 x^4+15952817 x^5+116275468 x^6+623060570 x^7+2507306175 x^8+7693787635 x^9+18199448138 x^{10}+33436393552
x^{11}+47933912805 x^{12}+53730741510 x^{13}+
47072037995 x^{14}+32132854819 x^{15}+16992690151 x^{16}+6899340865 x^{17}+2123302560 x^{18}+486469225 x^{19}+80870375 x^{20}+9387875
x^{21}+715250 x^{22}+31875 x^{23}+625 x^{24})))$ } \\  \noindent  \parbox[t]{5.8in}{$ \ds
 -( ( 20 x^2 ( 594-1804554 x^2+1861067073 x^4-980843342138 x^6+312271486068241 x^8-65745220858591243 x^{10}+9687425374674323038 x^{12}-1038831237350709728110
x^{14}+83458218228361655487358 x^{16}-
 5138793668486066172756627 x^{18}+247053756207980542996136827 x^{20}-9419200761604590794195183553 x^{22}+288563111422071718470665921122 x^{24}-7183020595820022918356293143191
x^{26}+
 146652249608919689717899944278935 x^{28}-2475119608958112116008084298092578 x^{30}+34758125569891291721733111044716670 x^{32}-408304515747043069333312170291563122
x^{34}+
 4029449548707334703065632689075530706 x^{36}-33520664098879279530846935348951084554 x^{38}+235672643031643784431539474426952002775 x^{40}-1402984966041861133179528890471348738574
x^{42}+
 7080937105693023633976720410573018906241 x^{44}-30320071126392282095132323456378758467169 x^{46}+110172038060616666751347384033916095236232
x^{48}- $ \\ $339654301314797407777153865619599548204754 x^{50}+ $ \\ $
 887985454364155231342962388893289072575681 x^{52}- $ \\ $1967193073973489077747887034156386489102131 x^{54}+ $ \\ $3689432161691671920569060944667608096392527
x^{56}- $ \\ $
 5852119578251451581944352085122158040704060 x^{58}+ $ \\ $7843136224187563187226124828706969966528880 x^{60}- $ \\ $8874030077853675225145693452723432706700712
x^{62}+ $ \\ $
 8470662807268393313340334637925976926072906 x^{64}- $ \\ $6818435626776523787742731932422831662587020 x^{66}+ $ \\ $4627235068497006366393774486518345216647229
x^{68}- $ \\ $
 2647281955658943924493368557219644884136624 x^{70}+ $ \\ $1276826796481948530558612094007072047328791 x^{72}- $ \\ $519184186877237814267298561395368846269221
x^{74}+
 177956576580743269622848780662126557731099 x^{76}-51397351204492036816439862004619872030622 x^{78}+12498597071276042509524782837533790404985
x^{80}-2555674675066458205582629945069159959078 x^{82}+
 438542422244921221766078573463986639711 x^{84}-62973228671754616635759571412982109611 x^{86}+7538270982475572478351765733687307547 x^{88}-748421432953615982992344042701918850
x^{90}+
 61219012082088331470745014081798155 x^{92}-4090101472901437592294988836303100 x^{94}+220699636373609972767521846072825 x^{96}-9477025744173262424135480148125
x^{98}+
 317515492537202322240501006125 x^{100}-8077133667540703264127475000 x^{102}+149998872943707676049412500 x^{104}-1913445101529246015000000 x^{106}+15080477594669686400000
x^{108}-
58396676207056000000 x^{110}+44735868320000000 x^{112})) \left. \right/ $ \\ $ \ds (-1+x) (1+x) (-1+2 x) (1+2 x) (-1+11 x) (1+11 x) ( -1+9 x-6 x^2+x^3) ( 1+9 x+6 x^2+x^3) ( 1-12 x^2+8 x^3)
( -1+12 x^2+8 x^3) ( 1-59 x+402 x^2-863 x^3+572 x^4) ( 1+59 x+402 x^2+863 x^3+572 x^4)
 ( 1-18 x+93 x^2-195 x^3+164 x^4-37 x^5+x^6) ( 1+18 x+93 x^2+195 x^3+164 x^4+37 x^5+x^6) ( 1-33 x+375 x^2-1922 x^3+5192
x^4-7883 x^5+6797 x^6-3252 x^7+834 x^8-105 x^9+5 x^{10})
 ( 1+33 x+375 x^2+1922 x^3+5192 x^4+7883 x^5+6797 x^6+3252 x^7+834 x^8+105 x^9+5 x^{10}) ( -1+5 x+45 x^2-305 x^3+65 x^4+2175 x^5-1800
x^6-4525 x^7+1975 x^8+2500 x^9-875 x^{10}-375 x^{11}+125 x^{12})
 ( -1-5 x+45 x^2+305 x^3+65 x^4-2175 x^5-1800 x^6+4525 x^7+1975 x^8-2500 x^9-875 x^{10}+375 x^{11}+125 x^{12})
(1-22 x-214 x^2+4230 x^3-15001 x^4-2783 x^5+60858 x^6-12548 x^7-46005 x^8+16340 x^9+10377 x^{10}-5171 x^{11}-329 x^{12}+413 x^{13}-27 x^{14}-9
x^{15}+x^{16})
(1+22 x-214 x^2-4230 x^3-15001 x^4+2783 x^5+60858 x^6+12548 x^7-46005 x^8-16340 x^9+10377 x^{10}+5171 x^{11}-329 x^{12}-413
x^{13}-27 x^{14}+9 x^{15}+x^{16}))$  \\$
-((40 x^2 (-66+36093 x^2-5516439 x^4+388236010 x^6-15014461407 x^8+345537603621 x^{10}-4913151745670 x^{12}+43979525538890 x^{14}-251882028246762
x^{16}+944748415928748 x^{18}-2385367130666279 x^{20}+
4170446377769406 x^{22}-5178850852983665 x^{24}+4663663884891106 x^{26}-3092886616570548 x^{28}+1525543183091576 x^{30}-561928143366458 x^{32}+154231144536537
x^{34}-31238675921445 x^{36}+
4578594704312 x^{38}-469127758800 x^{40}+31622131840 x^{42}-1247296000 x^{44}+21504000 x^{46})) \left. \right/ $ \\ $
((-1+4 x) (1+4 x)(-1+34 x-160 x^2+248 x^3-152 x^4+32 x^5)(1+34 x+160 x^2+248 x^3+152 x^4+32 x^5)(-1+18 x-91
x^2+110 x^3+175 x^4-515 x^5+455 x^6-175 x^7+25 x^8)
(-1-18 x-91 x^2-110 x^3+175 x^4+515 x^5+455 x^6+175 x^7+25 x^8)(-1+8 x+14 x^2-194 x^3+302 x^4-65 x^5-120 x^6+62
x^7+4 x^8-7 x^9+x^{10}) (-1-8 x+14 x^2+194 x^3+302 x^4+65 x^5-120 x^6-62 x^7+4 x^8+7 x^9+x^{10}))) $ \
$ \ds  - \frac{20 x^2 (-13+121 x^2)}{(-1+x) (1+x) (-1+11 x) (1+11 x)} + \frac{8 x^2}{1-4 x^2}. $
   \bc $ \mbox{--------------------------------------    } $  \ec  } \noindent
\newpage
\parbox[t]{5.8in}{ ${\cal F}^{TG}_{10,1}(x) = {\cal F}^{TG}_{10,9}(x) = $ \\
 $ \ds
\noindent - 2 x (-126+4714 x+437565 x^2-12611526 x^3-567388479 x^4+13753278349 x^5+385520884470 x^6-8284187953578 x^7-160641003614150
x^8+3215643247348036 x^9+43980670674459483 x^{10}-870993551869846297 x^{11}-8074506245787202467 x^{12}+169625772037234381120 x^{13}+995921851945576442660
x^{14}-24047780344636955674382 x^{15}-82143064794651585210552 x^{16}+2527651631607405145222893 x^{17}+4381806202685160800918991 x^{18}-202495312238122666408663282
x^{19}-123707903557328312369562631 x^{20}+12753493024361174683538336762 x^{21}-  $ \\ $ 1682890068179646582089711519 x^{22}-649696714907241335201343229620
x^{23}+ $ \\ $ 407750580971314579666991272521 x^{24}+27368646178933360013158281840695 x^{25}- $ \\ $  26519445629173125277129099646898 x^{26}-967211735916993519624158100682906
x^{27}+  $ \\ $ 1176901464597591561699677739388018 x^{28}+28900838281550657791369813850233161 x^{29}-40616088351630299146161711679315317 x^{30}-732835110694018892941454065351244804
x^{31}+1135512287366392931640220881030442719 x^{32}+15799142494128738264306882481440048382 x^{33}-26105491550343884048588120518419200419 x^{34}-290073832704770629997335680514994110268
x^{35}+496620724524413814467519288195583692037 x^{36}+4544322241429391928765420772990545048589 x^{37}-7843709988408787645582149802807257247835 x^{38}-60873696920504343519414091014698850003696
x^{39}+ $ \\ $103095577116858005240460087581066904041397 x^{40}+ $ \\ $698595604170265734661315942081013616465696 x^{41}- $ \\ $1129708193863219306235648147833588263568512
x^{42}- $ \\ $6877982050309429179957619985827908643818456 x^{43}+ $ \\ $10334041184834306381779332507166494619734746 x^{44}+ $ \\ $58132302604340536673323040208135718269388369
x^{45}- $ \\ $78976324149437543259900553640138054268911236 x^{46}- $ \\ $421750002060523489070341947802324679311694583 x^{47}+ $ \\ $504402788450555133787241504130514062822892126
x^{48}+ $ \\ $2624761587323542739562610535631750438166150761 x^{49}- $ \\ $2691906469918894092320492985940558738727910243 x^{50}- $ \\ $13999895897935552439860249455745960564552686465
x^{51}+ $ \\ $11999089612286543467874011680943055552808716669 x^{52}+ $ \\ $63943232236245797051918900395814124786677206522 x^{53}- $ \\ $44640872183812406197139214510725225101293543907
x^{54}- $ \\ $249971776271555347747289472912416721033538634114 x^{55}+ $ \\ $138483655156632952275567060779112531466726829591 x^{56}+ $ \\ $836463558039531139582619885408440950298551362918
x^{57}- $ \\ $357773093966464250496510754685314065020434327459 x^{58}- $ \\ $2397572132773692449972745475889858361918040101546 x^{59}+ $ \\ $768528992524023637411936421724095555722647022253
x^{60}+ $ \\ $5894374669015964035425375364530645589183432463532 x^{61}- $ \\ $1369878675319019487740616996254165186121553118035 x^{62}- $ \\ $12451823157746928283548944649431479142417501938179
x^{63}+ $ \\ $2022017238374635811740832236539463905035910815113 x^{64}+ $ \\ $22651018398777680925240151182656344585747949218000 x^{65}- $ \\ $2470836229919363567724020370226514132619460448548
x^{66}- $ \\ $35562588638577680510810869956837080084202151152784 x^{67}+ $ \\ $2518748272826072507089151832184552533824420185028 x^{68}+ $ \\ $48297231429038403721493579182376418837576996542856
x^{69}- $ \\ $2214314443233722111215154885122804984730629745017 x^{70}- $ \\ $56853662633089501387197039794860584577625202506807 x^{71}+ $ \\ $1840060686834467071509838088596287620109406716253
x^{72}+ $ \\ $58108956893985244792626631820819874852827595734482 x^{73}- $ \\ $1663686318314317841079372547640439377944773040445 x^{74}- $ \\ $51633961702169960176764777680735921351860558195322
x^{75}+ $ \\ $1702269096045947570656987398162066715282497575553 x^{76}+ $ \\ $39919921295002750153428727622724287368812740977576 x^{77}- $ \\ $1748460580007129897131836443795052033963692270273
x^{78}- $ \\ $26862395467212628355493055696360534462358413570308 x^{79}+ $ \\ $1602868306377267328085956018975174526702132429113 x^{80}+ $} \\
\parbox[t]{5.8in}{ $15729895582837114540582097103218373920855680131592
x^{81}- $ \\ $1246947738060531925950689457494415226204395086642 x^{82}- $ \\ $8010417440869034806227486695246210894546864191445 x^{83}+ $ \\ $812807805724756266170556008275489454254614708238
x^{84}+ $ \\ $3543809487121956695769622373636215496253184082939 x^{85}- $ \\ $444047531170570566328567191327647026125783327570 x^{86}- $ \\ $1360034665691746377203076235209263222629638097106
x^{87}+ $ \\ $204017601948080197825246313532836477553968615368 x^{88}+ $ \\ $452005466074584079311602817561736138432313847316 x^{89}- $ \\ $79112922645690820922042018468673398611689947423
x^{90}- $ \\ $129838744234419019077702710246793796130819268894 x^{91}+ $ \\ $25968473803066199452124961642296567223184176771 x^{92}+ $ \\ $32167056962134644928083456897538846715468365081
x^{93}- $ \\ $7232427115622104105998830078306131616319083994 x^{94}- $ \\ $6857463982703248151332637844915957945040702486 x^{95}+ $ \\ $1712456199002200066896315799138249491193357028
x^{96}+ $ \\ $1254571553584855464912237849307490250766083294 x^{97}- $ \\ $345327340100126832187138513088378059038657851 x^{98}- $ \\ $196263443310593359268470079121057609170887801
x^{99}+ $ \\ $59400289741450268891445972905820058752203347 x^{100}+ $ \\ $26104619635872942573075637798466637583572162 x^{101}- $ \\ $8722094424514298052035003744342160705370027
x^{102}- $ \\ $2922823833139123594746827968542006899881092 x^{103}+ $ \\ $1091950273483654811578742537601320664026133 x^{104}+ $ \\ $270567900193658604720850900713961997988772
x^{105}- $ \\ $115986992642117906114003026051797096310821 x^{106}- $ \\ $20027321510484134803187989302890303887708 x^{107}+ $ \\ $10345495768987909273727671318275530307593
x^{108}+ $ \\ $1106579110698492290202506674951204294666 x^{109}- $ \\ $761896202381901597750883899954791545308 x^{110}-37490123025815057443063363473317568345 x^{111}+45245548871473318675276948248599232790
x^{112}-69845834029567070116227562026057635 x^{113}-2102011492986728979703912229378590670 x^{114}+97711651716729205271894427916703800 x^{115}+73555151848859315408266337151110000
x^{116}-6454962259090450989492008598241600 x^{117}-1844456054559257765597850962737200 x^{118}+238967514825930526075828819460000 x^{119}+30739514002710801839359729248000
x^{120}-5591717031117774150418451640000 x^{121}-293032539128210872669883680000 x^{122}+82386001607040137727041600000 x^{123}+871059440237948873552000000
x^{124}-722384265447172393584000000 x^{125}+ $ \\ $ 8234562669686880000000000 x^{126}+3391177756974129920000000 x^{127}-58940876047165440000000 x^{128}-  $ \\ $ 7380395951001600000000
x^{129}+54189660569600000000 x^{130}+5954215936000000000 x^{131}+47775744000000000 x^{132}) /  $ \\ $ \ds ((-1+x) (1+x) (1+2 x) (-1+4 x) (1+4
x) (1-8 x+4 x^2) (1+5 x+5 x^2) (1+7 x+8 x^2) (1+18 x+63 x^2+50 x^3+13 x^4+x^5) (-1+34 x-160 x^2+248
x^3-152 x^4+32 x^5) (1+34 x+160 x^2+248 x^3+152 x^4+32 x^5) (-1+18 x-91 x^2+110 x^3+175 x^4-515 x^5+455 x^6-175 x^7+25 x^8)
(-1-18 x-91 x^2-110 x^3+175 x^4+515 x^5+455 x^6+175 x^7+25 x^8) (-1+121 x-3901 x^2+49003 x^3-293115 x^4+896942 x^5-1381585 x^6+970100
x^7-264304 x^8+15552 x^9) (-1+8 x+14 x^2-194 x^3+302 x^4-65 x^5-120 x^6+62 x^7+4 x^8-7 x^9+x^{10}) (-1-8 x+14 x^2+194 x^3+302
x^4+65 x^5-120 x^6-62 x^7+4 x^8+7 x^9+x^{10}) (1-36 x+504 x^2-3603 x^3+14416 x^4-33263 x^5+44489 x^6-34794 x^7+16069 x^8-4345 x^9+653
x^{10}-47 x^{11}+x^{12}) (1+53 x+951 x^2+7589 x^3+32774 x^4+84127 x^5+133734 x^6+133056 x^7+81936 x^8+30215 x^9+6265 x^{10}+650 x^{11}+25
x^{12}) (1+92 x+3048 x^2+50453 x^3+483576 x^4+2928888 x^5+11832412 x^6+32949126 x^7+64377394 x^8+88750186 x^9+85833700 x^{10}+57206921
x^{11}+25441290 x^{12}+7182529 x^{13}+1203428 x^{14}+111956 x^{15}+5340 x^{16}+120 x^{17}+x^{18}) (1-66 x+1833 x^2-28651 x^3+283174 x^4-1881547
x^5+8719340 x^6-28859906 x^7+69342834 x^8-122339404 x^9+159752386 x^{10}-155169581 x^{11}+112318552 x^{12}-60476870 x^{13}+24066800 x^{14}-6990095
x^{15}+1450676 x^{16}-207780 x^{17}+19365 x^{18}-1050 x^{19}+25 x^{20}))+
   $} \\
   \parbox[t]{5.8in}{ $
   \noindent 10 x (-42+313864 x^2-851142342 x^4+1235764748147 x^6-1130051840878650 x^8+711776772859760335 x^{10}-326808647822865467172 x^{12}+113746719527560037690911
x^{14}-30882824693724401421476206 x^{16}+6685785526149426988979143964 x^{18}-1174395530723402676670266857496 x^{20}+169787858338662268550962070725389
x^{22}-20447624321970200692815309411305688 x^{24}+2072486682581050038132580005319863544 x^{26}-178378165031459557413137310844961534002 x^{28}+13140613813827635294551010663030721400270
x^{30}
- $ \\ $ 834361909358337388189829635232605077772794 x^{32}+ $ \\ $ 45949346183864680072918093762551215377347315 x^{34}- $ \\ $ 2207189611066353743874973361342771591837328318
x^{36}+ $ \\ $ 92950077601265825709297627456440947651076556442 x^{38}- $ \\ $ 3447599576641764810918654136816230774514796386460 x^{40}+ $ \\ $ 113099555040864112998132343225638149527092703409572
x^{42}- $ \\ $ 3294037009849222735481930981652981527994459178101348 x^{44}+ $ \\ $ 85467355920785875887506546322683528920065446095040666 x^{46}- $ \\ $ 1981501993103321464977498652187042497986559330536679380
x^{48}+ $ \\ $ 41159223063649420115408008431119917827653388745530863803 x^{50}- $ \\ $ 767733968495092808641072045451177305261207106665326586854 x^{52}+ $ \\ $ 12884088271253271927947296246117730762555055135866016240025
x^{54}- $ \\ $ 194832687457259468087278107720230470490843494153683260959844 x^{56}+ $ \\ $ 2657929857601800051153891612245796291758739851730136878366341 x^{58}- $ \\ $ 32738515584454881345855021030693888247111599798817043104592162
x^{60}+ $ \\ $ 364278662366944196978569695940572497001506478922098037569576304 x^{62}- $ \\ $ 3662481313316366767319493197768343024435170684728101420568864792 x^{64}+ $ \\ $ 33273067189236408214698788940655843200761627850306821734694498691
x^{66}- $ \\ $ 273096217009350901800040911285975831728643715709406174552697244274 x^{68}+ $ \\ $ 2024481149545022365184511724303612665903492826755559296230430391829
x^{70}- $ \\ $ 13549226399353760094452922786305674279247489848279385355134230614622 x^{72}+ $ \\ $ 81831153372179192707289350315822852529937431121635370233968701727658
x^{74}- $ \\ $ 445772438867631837266182716247949117396455096916351783398792708217488 x^{76}+ $ \\ $ 2189182799263529530654515168823671592260568498940750798325264389021197
x^{78}- $ \\ $ 9687562414044150105829596300708578291192808627674188126416367089884512 x^{80}+ $ \\ $ 38611110700755436269236964409736594630710832673889663424913234790832251
x^{82}- $ \\ $ 138546144639359591068732290848007552523448617267006332701427139013129812 x^{84}+ $ \\ $ 447405696394706376198597483305419137272274034181274240871531788595314525
x^{86}- $ \\ $ 1299865841576273613761701443266421230567765761742941402339374338262038198 x^{88}+ $ \\ $ 3396856726106366392645924319574043893374131641404506686507999438374175222
x^{90}- $ \\ $ 7982839265228523478281247823351084560339331794647678554957075955052145470 x^{92}+ $ \\ $ 16868926548951469899928723924287621948430107464872844263693427932028062077
x^{94}- $ \\ $ 32051201701686736796035211969036638398713261706142549038799559577412636628 x^{96}+ $ \\ $ 54756067413833880108800038019197116428476534398815345424142499733199689076
x^{98}- $ \\ $ 84116428856431987450497875214510152632789374971245335234373903935407700272 x^{100}+ $ \\ $ 116209593690478590774299229200101703480209878294271640801103917286808641080
x^{102}- $ \\ $ 144406286502293100272056216024615887671310821635766905142830755770851881228 x^{104}+ $ \\ $ 161434830593758860883462763006696567086788248705659441374774310160696992061
x^{106}- $ \\ $ 162393125192600024300745473933059205037158521537869945356220247541005527076 x^{108}+ $ \\ $ 147025126734291550558084933415921346658856781537816043080259487976085360647
x^{110}- $ \\ $ 119828120546339826383406669606868572885516437140154031617109391756956593132 x^{112}+ $ \\ $ 87932121515729300014343507030655810476533180306989529451900217898082341476
x^{114}- $ \\ $ 58105601281303824810400477131287154123082930771268693941708015656151466080 x^{116}+ $ \\ $ 34578401109947820251526098874007332074505087242622980402791400190538573085
x^{118}- $ \\ $ 18531527768350768903955133894128040095396248937133246986327603756194477178 x^{120}+ $ \\ $ 8943381048386568229770163431879018794275728311067420377392501633845078913
x^{122}- $ \\ $ 3885950097942957253543669426596472161078784818414557879048781683259698910 x^{124}+ $ \\ $ 1519739400721375416874073631095812738914433828454132253328903861353696004
x^{126}- $ \\ $ 534730218820996724487414090898744856830627416881002519575707424200421110 x^{128}+ $ \\ $ 169181449268700732022768576581979072653789357328542536510795271695542822
x^{130}- $} \\
\parbox[t]{5.8in}{
 $ 48097129127070040774771899300916907929305076085635659101947088719076332 x^{132}+ $ \\ $ 12276176832594508346366575366484654498329064751834013738455362036811212
x^{134}- $ \\ $ 2810235064489231077029533098439043286398257023301653096148996475543742 x^{136}+ $ \\ $ 576287812517470580530577828240339641030413956992145114517601453147511
x^{138}- $ \\ $ 105718970529165839532622368262894087751056370743326307527966406064532 x^{140}+ $ \\ $ 17321782246578449507536703928824102708052660661470229487773139661698
x^{142}-2530276925018341797904936864100643494147857521233103374358475270804 x^{144}+328833913321337246618649319418020949345078318804252358692922274802
x^{146}-37930331383452395272392254586869353387290643030741766220069676836 x^{148}+3872723717303709475093538522228919041281720671027582887414515176
x^{150}-348904622825961341705930089455120449886760849754688780952903620 x^{152}+27636507074247403253215204934323013412422159426991326884328925 x^{154}-1916495090976212998851907979462263809055814145275746084290884
x^{156}+ $ \\ $ 10430950519870815402120276740507496451374763236449582800 x^{164}+ $ \\ $ 336343214263458003667277345626013161094409447476815750 x^{166}- $ \\ $ 8984504210731988900438530973514548043759658509323750
x^{168}+ $ \\ $ 194204573884366489143173291826736961135001039337500 x^{170}- $ \\ $ 3267692220850107774517111035703187810808592182500 x^{172}+ $ \\ $ 39541220457391557236964949469685704043566209375
x^{174}- $ \\ $ 266665955532110137591464036389141582138856250 x^{176}- $ \\ $ 880294790033078497219602036905620090000000 x^{178}+ $ \\ $ 47052457070586921926313171546035372500000
x^{180}-600945674879069863433016053011350000000 x^{182}+4329105872907071408304286126000000000 x^{184}-18614930530108161917762400000000000 x^{186}+44979276158405373945600000000000
x^{188}-52407753203528448000000000000 x^{190}+20855305687040000000000000 x^{192})/ $
\\ $\ds ((-1+x) (1+x) (-1+2 x) (1+2 x) (-1+11 x) (1+11
x) (1-3 x+x^2) (1+3 x+x^2) (1-13 x+11 x^2) (1-8 x+11 x^2) (1+8 x+11 x^2) (1+13 x+11 x^2)
(1-12 x^2+8 x^3) (-1+12 x^2+8 x^3) (-1+104 x-2661 x^2+24090 x^3-91993 x^4+158236 x^5-121128 x^6+44736 x^7-7808 x^8+512
x^9) (1+104 x+2661 x^2+24090 x^3+91993 x^4+158236 x^5+121128 x^6+44736 x^7+7808 x^8+512 x^9) (-1+5 x+45 x^2-305 x^3+65 x^4+2175
x^5-1800 x^6-4525 x^7+1975 x^8+2500 x^9-875 x^{10}-375 x^{11}+125 x^{12}) (-1-5 x+45 x^2+305 x^3+65 x^4-2175 x^5-1800 x^6+4525 x^7+1975
x^8-2500 x^9-875 x^{10}+375 x^{11}+125 x^{12}) (1-22 x-214 x^2+4230 x^3-15001 x^4-2783 x^5+60858 x^6-12548 x^7-46005 x^8+16340 x^9+10377
x^{10}-5171 x^{11}-329 x^{12}+413 x^{13}-27 x^{14}-9 x^{15}+x^{16}) (1+22 x-214 x^2-4230 x^3-15001 x^4+2783 x^5+60858 x^6+12548 x^7-46005
x^8-16340 x^9+10377 x^{10}+5171 x^{11}-329 x^{12}-413 x^{13}-27 x^{14}+9 x^{15}+x^{16}) (1-82 x+2898 x^2-59227 x^3+789803 x^4-7334275
x^5+49339066 x^6-246611784 x^7+930982131 x^8-2681675595 x^9+5927069364 x^{10}-10071216252 x^{11}+13140900495 x^{12}-13114381556 x^{13}+9943609371
x^{14}-5673467839 x^{15}+2404556811 x^{16}-744137999 x^{17}+164387022 x^{18}-25156085 x^{19}+2562619 x^{20}-164923 x^{21}+6260 x^{22}-125 x^{23}+x^{24})
(1+82 x+2898 x^2+59227 x^3+789803 x^4+7334275 x^5+49339066 x^6+246611784 x^7+930982131 x^8+2681675595 x^9+5927069364 x^{10}+10071216252 x^{11}+13140900495
x^{12}+13114381556 x^{13}+9943609371 x^{14}+5673467839 x^{15}+2404556811 x^{16}+744137999 x^{17}+164387022 x^{18}+25156085 x^{19}+2562619 x^{20}+164923
x^{21}+6260 x^{22}+125 x^{23}+x^{24}) (1-104 x+4410 x^2-104155 x^3+1561045 x^4-15952817 x^5+116275468 x^6-623060570 x^7+2507306175 x^8-7693787635
x^9+18199448138 x^{10}-33436393552 x^{11}+47933912805 x^{12}-53730741510 x^{13}+47072037995 x^{14}-32132854819 x^{15}+16992690151 x^{16}-6899340865
x^{17}+2123302560 x^{18}-486469225 x^{19}+80870375 x^{20}-9387875 x^{21}+715250 x^{22}-31875 x^{23}+625 x^{24}) (1+104 x+4410 x^2+104155
x^3+1561045 x^4+15952817 x^5+116275468 x^6+623060570 x^7+2507306175 x^8+7693787635 x^9+18199448138 x^{10}+33436393552 x^{11}+47933912805 x^{12}+53730741510
x^{13}+47072037995 x^{14}+32132854819 x^{15}+16992690151 x^{16}+6899340865 x^{17}+2123302560 x^{18}+486469225 x^{19}+80870375 x^{20}+9387875 x^{21}+715250
x^{22}+31875 x^{23}+625 x^{24}))+$ }
 \\
   \parbox[t]{5.8in}{ $ \ds
\noindent - 20 x (12-36233 x^2+35693525 x^4-18015088281 x^6+5544243396574 x^8-1140498928876223 x^{10}+165784297773487726 x^{12}-17659082811422467959
x^{14}+1413747980204751728411 x^{16}-86712812542489990236049 x^{18}+4139783491801797354374336 x^{20}-155976703325674603887820165 x^{22}+4695170776839340840967375255
x^{24}-114132018850569329682605350253 x^{26}+2260921906876186678981281116392 x^{28}-36771687605301104127340111100417 x^{30}+493809392736966628417744981828282
x^{32}-5496514036146995436911463800815519 x^{34}+50801602467190410672105857414795571 x^{36}-389638103580670814183312541635836105 x^{38}+2470500312785634497681191154896104823
x^{40}-12837177451800051118585034431008799804 x^{42}+53707861081585046034005118053286201463 x^{44}-174115790501291591107223590801974694251 x^{46}+393881431087274978644674738635632440874
x^{48}-351103829225154254734893983224668896264 x^{50}-1712727400393175292612084199252226403602 x^{52}+10533391603117541298754751466523385355409 x^{54}-34228559924503375918144129551284191775333
x^{56}+80288657332318482346781605324023689202959 x^{58}- 146569330004818274830798372339984793020723 x^{60}+ $ \\ $ 214485835521386635407406558034397997694437
x^{62}- $ \\ $ 255269964316562780616534033897514913410482 x^{64}+ $ \\ $ 249100865204895186426356477997558201184145 x^{66}- $ \\ $ 200326432092412817356607982959633536599604
x^{68}+ $ \\ $ 133231849164372820331579064825007668419343 x^{70}-  $ \\ $  73467374652032533560455497947111132742429 x^{72}+33652874976731492525893788655422058390791
x^{74}-12822449712227472973878582738869948148858 x^{76}+4066977863853705603986355769215115494264 x^{78}-1073928383588405271419554429292303079964
x^{80}+235941613641636492215234204957809084050 x^{82}-43061906258254333102247139321023972679 x^{84}+6511919119987911315675092886316998195 x^{86}-812746119034732679187931909889357613
x^{88}+83260609150888784743201085359158870 x^{90}-6948602517050959224143088698286450 x^{92}+467658564024256759618996160805700 x^{94}-25038227012225248707574504136100
x^{96}+1046682062822702699778633624000 x^{98}-33276179695900052675350996250 x^{100}+773608690503924003517654375 x^{102}-12328139666261194870390000
x^{104}+118258166454843810700000 x^{106}-440077648384537400000 x^{108}-2116229832156000000 x^{110}+17656673320000000 x^{112})/ $ \\ $((-1+x)
(1+x) (-1+2 x) (1+2 x) (-1+11 x) (1+11 x) (-1+9 x-6 x^2+x^3) (1+9 x+6 x^2+x^3) (1-12 x^2+8 x^3) (-1+12 x^2+8
x^3) (1-59 x+402 x^2-863 x^3+572 x^4) (1+59 x+402 x^2+863 x^3+572 x^4) (1-18 x+93 x^2-195 x^3+164 x^4-37 x^5+x^6)
(1+18 x+93 x^2+195 x^3+164 x^4+37 x^5+x^6) (1-33 x+375 x^2-1922 x^3+5192 x^4-7883 x^5+6797 x^6-3252 x^7+834 x^8-105 x^9+5 x^{10})
(1+33 x+375 x^2+1922 x^3+5192 x^4+7883 x^5+6797 x^6+3252 x^7+834 x^8+105 x^9+5 x^{10}) (-1+5 x+45 x^2-305 x^3+65 x^4+2175 x^5-1800
x^6-4525 x^7+1975 x^8+2500 x^9-875 x^{10}-375 x^{11}+125 x^{12}) (-1-5 x+45 x^2+305 x^3+65 x^4-2175 x^5-1800 x^6+4525 x^7+1975 x^8-2500
x^9-875 x^{10}+375 x^{11}+125 x^{12}) (1-22 x-214 x^2+4230 x^3-15001 x^4-2783 x^5+60858 x^6-12548 x^7-46005 x^8+16340 x^9+10377 x^{10}-5171
x^{11}-329 x^{12}+413 x^{13}-27 x^{14}-9 x^{15}+x^{16}) (1+22 x-214 x^2-4230 x^3-15001 x^4+2783 x^5+60858 x^6+12548 x^7-46005 x^8-16340
x^9+10377 x^{10}+5171 x^{11}-329 x^{12}-413 x^{13}-27 x^{14}+9 x^{15}+x^{16})) $ \\  \\$
 - 10 x (-9+4699 x^2-611640 x^4+39853698 x^6-1654541277 x^8+47171448683 x^{10}-907848961210 x^{12}+11283990656808 x^{14}-87630117751305
x^{16}+424309837572525 x^{18}-1316944054137175 x^{20}+2715321651741897 x^{22}-3844967038911191 x^{24}+3841989868583469 x^{26}-2764056440448407 x^{28}+1450624056331021
x^{30}-558782965472919 x^{32}+157793958793899 x^{34}-32350122050946 x^{36}+4715296752120 x^{38}-470462657040 x^{40}+30019059200 x^{42}-1072166400
x^{44}+15360000 x^{46})/ $ \\ $ ((-1+4 x) (1+4 x) (-1+34 x-160 x^2+248 x^3-152 x^4+32 x^5) (1+34 x+160 x^2+248 x^3+152
x^4+32 x^5) (-1+18 x-91 x^2+110 x^3+175 x^4-515 x^5+455 x^6-175 x^7+25 x^8) (-1-18 x-91 x^2-110 x^3+175 x^4+515 x^5+455 x^6+175
x^7+25 x^8) (-1+8 x+14 x^2-194 x^3+302 x^4-65 x^5-120 x^6+62 x^7+4 x^8-7 x^9+x^{10}) (-1-8 x+14 x^2+194 x^3+302 x^4+65 x^5-120
x^6-62 x^7+4 x^8+7 x^9+x^{10})) + $ \\ $ \ds \frac{20 x (1+11 x^2)}{(-1+x) (1+x) (-1+11 x) (1+11 x)} + \frac{4 x}{1-4 x^2}$
\bc $ \mbox{--------------------------------------    } $  \ec}
    \\ \noindent
\parbox[t]{5.8in}{ ${\cal F}^{TG}_{10,2}(x) = {\cal F}^{TG}_{10,8}(x) = $ \\
 $ \ds
2 x (46+1376 x-217625 x^2-3180944 x^3+345272944 x^4+2490792541 x^5-270947482370 x^6-869665513372 x^7+126629126161460 x^8+98949137339289
x^9-39376684889493468 x^{10}+42285392372031707 x^{11}+8584595987201580737 x^{12}-23283523171385601670 x^{13}-1335302299814544715435 x^{14}+5591333078113361960712
x^{15}+149360639878168136676722 x^{16}-824089151275096662900673 x^{17}-12157213128401490596058106 x^{18}+82301527251852610260951207 x^{19}+732909361120852753953328426
x^{20}-5895879797569636310659879862 x^{21}-33381771502384436510944678661 x^{22}+315427428241304314512677601615 x^{23}+1171670620109532384089041441404
x^{24}-12987244764595246253258802545575 x^{25}-32336082777167266312396782218377 x^{26}+420925951025249041794783348621726 x^{27}+718419992365757170664898344964382
x^{28}-10918053790902226856605466458871501 x^{29}-13271957917324789528048438726531448 x^{30}+229171334962491461628745738446414019 x^{31}+213443441025950269955959028000690656
x^{32}-3914234058789326860692196284312129082 x^{33}-3148734347714652982231344767589293571 x^{34}+54323429284989387948670246212459893123 x^{35}+43912818088464812257037761730749993098
x^{36}-605089443498881414084357682635175774439 x^{37}-569044389491347853877462049307892156000 x^{38}+5220878715557154021521940649621624708556 x^{39}+6508844235431769745634483161682747172238
x^{40}- $ \\ $ 31312949416002620684089448770078672506466 x^{41}- $ \\ $ 62035830840273142151188778864117775173393 x^{42}+ $ \\ $ 67512628467109223629597552547600738952341
x^{43}+ $ \\ $ 462907281695254776364394321916294822366064 x^{44}+ $ \\ $ 1150124058798841207872277950691887252424236 x^{45}- $ \\ $ 2410182015256556554130934737402442011296299
x^{46}- $ \\ $ 18754721304177397859154779848784596070275562 x^{47}+ $ \\ $ 4962104925741844239442989645754202672438839 x^{48}+ $ \\ $ 169758079278535377469176905380650808809742304
x^{49}+ $ \\ $ 54052188532350338377479734227653232530498138 x^{50}- $ \\ $ 1131008694484192945943456981819758756987583345 x^{51}- $ \\ $ 759034479578625186383320251391183487501152119
x^{52}+ $ \\ $ 5954762641859609723077421357174788441276243938 x^{53}+ $ \\ $ 5612173118486469325689348069142749587262366922 x^{54}- $ \\ $ 25489175486849636471886692127487291959461863406
x^{55}- $ \\ $ 30019686442850459092441957656457165179295004306 x^{56}+ $ \\ $ 89985461770519080059824895833497768232283164842 x^{57}+ $ \\ $ 125691788607120374541898678027073834967579401174
x^{58}- $ \\ $ 264366614752423292209687714467067498056940890129 x^{59}- $ \\ $ 425820964430437903944612588986299081415885400978 x^{60}+ $ \\ $ 650976669947374612296683257977166112910923720123
x^{61}+ $ \\ $ 1188703767239106054126143922355195992679894728960 x^{62}- $ \\ $ 1353296239898151274059150672954433907226753504806 x^{63}- $ \\ $ 2767281816129362135768674864514868778733773403653
x^{64}+ $ \\ $ 2395621671027669188498243311822840507575709395500 x^{65}+ $ \\ $ 5420566816035425459789297705449835170964008715468 x^{66}- $ \\ $ 3650625899243473324174186157879930401833690632566
x^{67}- $ \\ $ 8997976453880767229880718788619205251627196849578 x^{68}+ $ \\ $ 4854819742548640375441950937784538235836118141929 x^{69}+ $ \\ $ 12731557017044046684671074340866010238708468090132
x^{70}- $ \\ $ 5724208601734121665800860540698772826573127087233 x^{71}- $ \\ $ 15426783914798362981527390347385743704957570089333 x^{72}+ $ \\ $ 6077797748927983466026725091623157745811805483448
x^{73}+ $ \\ $ 16063353171802888804627345774351313761042115444640 x^{74}- $ \\ $ 5875682027878365971877426842763458962810581443698 x^{75}- $ \\ $ 14405586845732333037850316400535990454391828922158
x^{76}+ $ \\ $ 5186036756143243526701553286115820157105851667504 x^{77}+ $ \\ $ 11136286258416877782524292946881844145240727987288 x^{78}- $ \\ $ 4153802824303357314079316825108239450773342885167
x^{79}- $ \\ $ 7416572048805349048127510539219254423365083376368 x^{80}+ $ } \\
\parbox[t]{5.8in}{$2984265773587600318869087312497624561619586174703 x^{81}+ $ \\ $ 4245518590261815168573619073151459546774801293537
x^{82}- $ \\ $ 1898435914021288164803623857464026933895627941070 x^{83}- $ \\ $ 2079937923475563486330468292587554114548440743983 x^{84}+ $ \\ $ 1057210719172001069112356503595008899583148477781
x^{85}+ $ \\ $ 865944715115122099326420674211937384288933422460 x^{86}- $ \\ $ 510734652696903103332658964865624399826584743324 x^{87}- $ \\ $ 302901037843214641894844222443821209138136544428
x^{88}+ $ \\ $ 212530933938864837536292678753092572360096982854 x^{89}+ $ \\ $ 87309084084137500607540885183987263384304327603 x^{90}- $ \\ $ 75734576997170395238432074996691684149914395296
x^{91}- $ \\ $ 19974064830365800389464801498909415866432312556 x^{92}+ $ \\ $ 22985323372962285318373270166478668391125030269 x^{93}+ $ \\ $ 3304460544302352241560753768707271516305254934
x^{94}- $ \\ $ 5907731266157350861142552516140918155755433094 x^{95}- $ \\ $ 259528667464077236283453436569295472213918308 x^{96}+ $ \\ $ 1277479490924279792619375439640087922462799641
x^{97}- $ \\ $ 53203266964195556918208247846968114802813954 x^{98}- $ \\ $ 230529761578347467223073696774946655307690019 x^{99}+ $ \\ $ 27252543052661342103050406885836549509517623
x^{100}+ $ \\ $ 34351237308971591042846502793485554346765628 x^{101}- $ \\ $ 6565382661406655339356714288977195154418648 x^{102}- $ \\ $ 4165231666059949091206884605537074258388888
x^{103}+ $ \\ $ 1112683444078636053151050947679836406781002 x^{104}+ $ \\ $ 402029938059366565160971118106761505237548 x^{105}- $ \\ $ 143724121948666277674972862479485096687194
x^{106}- $ \\ $ 29743081168358793753045310712143140234077 x^{107}+ $ \\ $ 14496299929026094319092178814968444596222 x^{108}+ $ \\ $ 1553736431253179444912012128817607803139
x^{109}-1147387095010303616425607954772857449487 x^{110}-42648478160269384804251550283231418570 x^{111}+70922276437071877542468838023316116985 x^{112}-1055374653794311655436952413967693165
x^{113}-3382762184718203434466266747444119700 x^{114}+187508486538742325262154515889915800 x^{115}+122145945422660334353025075463146600 x^{116}-10807608789892071407961638630040400
x^{117}-3247293087244926802378830712639200 x^{118}+379503922816969678282015518328000 x^{119}+61099223165915795540454289944000 x^{120}-8736442405643690427276805160000
x^{121}-769610406856795845379585920000 x^{122}+130258585332331027154256000000 x^{123}+6013684165721471766720000000 x^{124}-1187971120530700538640000000
x^{125}-26851833413180858368000000 x^{126}+ $ \\ $ 5926564650703299840000000 x^{127}+74087465438689280000000 x^{128}-13022898156236800000000 x^{129}- $ \\ $ 186621766860800000000
x^{130}+8236163072000000000 x^{131}+143327232000000000 x^{132})) / $ \\
$((-1+x) (1+x) (1+2 x) (-1+4 x) (1+4 x) (1-8 x+4 x^2)
(1+5 x+5 x^2) (1+7 x+8 x^2) (1+18 x+63 x^2+50 x^3+13 x^4+x^5) (-1+34 x-160 x^2+248 x^3-152 x^4+32 x^5)
(1+34 x+160 x^2+248 x^3+152 x^4+32 x^5) (-1+18 x-91 x^2+110 x^3+175 x^4-515 x^5+455 x^6-175 x^7+25 x^8) (-1-18 x-91 x^2-110
x^3+175 x^4+515 x^5+455 x^6+175 x^7+25 x^8) (-1+121 x-3901 x^2+49003 x^3-293115 x^4+896942 x^5-1381585 x^6+970100 x^7-264304 x^8+15552
x^9) (-1+8 x+14 x^2-194 x^3+302 x^4-65 x^5-120 x^6+62 x^7+4 x^8-7 x^9+x^{10}) (-1-8 x+14 x^2+194 x^3+302 x^4+65 x^5-120 x^6-62
x^7+4 x^8+7 x^9+x^{10}) (1-36 x+504 x^2-3603 x^3+14416 x^4-33263 x^5+44489 x^6-34794 x^7+16069 x^8-4345 x^9+653 x^{10}-47 x^{11}+x^{12})
(1+53 x+951 x^2+7589 x^3+32774 x^4+84127 x^5+133734 x^6+133056 x^7+81936 x^8+30215 x^9+6265 x^{10}+650 x^{11}+25 x^{12}) (1+92 x+3048
x^2+50453 x^3+483576 x^4+2928888 x^5+11832412 x^6+32949126 x^7+64377394 x^8+88750186 x^9+85833700 x^{10}+57206921 x^{11}+25441290 x^{12}+7182529
x^{13}+1203428 x^{14}+111956 x^{15}+5340 x^{16}+120 x^{17}+x^{18}) (1-66 x+1833 x^2-28651 x^3+283174 x^4-1881547 x^5+8719340 x^6-28859906
x^7+69342834 x^8-122339404 x^9+159752386 x^{10}-155169581 x^{11}+112318552 x^{12}-60476870 x^{13}+24066800 x^{14}-6990095 x^{15}+1450676 x^{16}-207780
x^{17}+19365 x^{18}-1050 x^{19}+25 x^{20})) +$}
   \\
\parbox[t]{5.8in}{ $
10 x^2 (-1104+5489423 x^2-11166526727 x^4+12992482181651 x^6-9970482068563665 x^8+5463389849244066039 x^{10}-2247593052404730178653
x^{12}+717994844547678609637470 x^{14}-182334854645136459340349488 x^{16}+37440189828668690609959232069 x^{18}- $ \\ $6297027295578297540584449391469 x^{20}+876550965774781289080259532128824
x^{22}- $ \\ $ 101877845327018065269714137272473196 x^{24}+9962683575392788477859919262886347506 x^{26}- $ \\ $  825322123947922609252892814717057679516 x^{28}+ 58271170299859004098568603799993894174099
x^{30}- $ \\ $ 3525352970006782250450029168391184616533527 x^{32}+ $ \\ $ 183621735164227590804154775053000140402220648 x^{34}- $ \\ $ 8267983639168737809249500398007700850024626488
x^{36}+ $ \\ $ 322953831172839259761629034957305196473808326282 x^{38}- $ \\ $ 10974405200407601633684939620120463026087146971786 x^{40}+ $ \\ $ 325130800347790286980601863591328563457431826782290
x^{42}- $ \\ $ 8409462008138200207569142861794144623991160010655192 x^{44}+ $ \\ $ 189971039388045686773284792960863481711769015687733985 x^{46}- $ \\ $ 3744776736631884218707526296039909502689469086987196969
x^{48}+ $ \\ $ 64232689703339222018657101730261546132417343525391255545 x^{50}- $ \\ $ 952953187598434172676732640196393299241186881743478300249 x^{52}+ $ \\ $ 12083749961002920693284516648690254498358820989186450750753
x^{54}- $ \\ $ 127758882843339023010133437955070856652150374218565986826205 x^{56}+ $ \\ $ 1060436894984722647355390033262675461605698143869707636470258 x^{58}- $ \\ $ 5578379141164463376614546655685636069840073343987765045294554
x^{60}- $ \\ $ 10427656698488101417875342523887690886263909091754325463653911 x^{62}+ $ \\ $ 725442888058351328437001423058924356121735217718522933326895415 x^{64}- $ \\ $ 11283871457189743038866632267211015171196313817367509084441149037
x^{66}+ $ \\ $ 123953790298522306390293990368999496946345721526170198757538671151 x^{68}- $ \\ $ 1103440027186411439525322079826094067237820514889803835049239911308
x^{70}+ $ \\ $ 8341853473056395483062733123650716637104917836373184312279992156092 x^{72}- $ \\ $ 54735259583826639715075136217830108711357098803422894443015967897123
x^{74}+ $ \\ $ 315476015292533985446527594231872942098819557292288865080166151837941 x^{76}- $ \\ $ 1608913298121882475600130126947086520991155295288579927774622946385145
x^{78}+ $ \\ $ 7295210745830573214615062539664771995602228448456647080760360259761735 x^{80}- $ \\ $ 29505256306569733592036724868735223209887363780739351770305902928783111
x^{82}+ $ \\ $ 106688599723094166420838178302588828173607020290382466689712885648565229 x^{84}- $ \\ $ 345473772001966422860870668895843542278295816551743187357102775499577074
x^{86}+ $ \\ $ 1003038951507466297073599664538120582275966218493624000842531349973293020 x^{88}- $ \\ $ 2613457322529287088157561993775973185678172326968841477836792229068505179
x^{90}+ $ \\ $ 6115025411967736538022887659443140334121681950470623627424520141804910195 x^{92}- $ \\ $ 12855411954205705282780857733283133752777361686908028384141399028046170644
x^{94}+ $ \\ $ 24291163719594743808249036669579249075294296823518204448323548152336046430 x^{96}- $ \\ $ 41268652918348465205108886806108028081309596840096840922452243872342556226
x^{98}+ $ \\ $ 63054084658527857988642951409933325866248806355931062915906886536825961060 x^{100}- $ \\ $ 86660659739172570039038765406127146302167719816871622044646736639539796443
x^{102}+ $ \\ $ 107159169194901307945439981715053201179114757308066644634855855635929910401 x^{104}- $ \\ $ 119236102642969254186727449597031896009856206687964454218282254651890395457
x^{106}+ $ \\ $ 119404182756961980107198326922607471895150502977923449530589844389718579427 x^{108}- $ \\ $ 107625510919584127409765271065993521918724926108183003878761949009520584410
x^{110}+ $ \\ $ 87323557920550269807819385823470661631930223616892123471967000168610351334 x^{112}- $ \\ $ 63780493049113840339124144934440624517891039195968729241835444310634412105
x^{114}+ $ \\ $ 41935746697584809479042402513817088437682973866062658446966636772125324075 x^{116}- $ \\ $ 24819767393875362365976905518674291340772264932492866999736943427229759795
x^{118}+ $ \\ $ 13221434170975841307671978878796767714367069310416273974919498385525265347 x^{120}- $ \\ $ 6337949554158998135228007691556653964516511007021738736572075453308259606
x^{122}+ $ \\ $ 2733370741706075247824304531703266812227772175523565438228171653849775934 x^{124}- $ \\ $ 1060190194469032444391708630235555751446628004968337690280254496218389322
x^{126}+ $ \\ $ 369677874895650267165368737393457933933133905708985012065238074251974296 x^{128}- $ \\ $ 115823500925066771247257491935543452420478836067286994317996610282812354
x^{130}+ $ }
   \\
\parbox[t]{5.8in}{$ 32586639489960221090642892175330800805491477093055717864050705956177268 x^{132}- $ \\ $ 8227001374346177821969032220749017786837257001661541147736631270934417
x^{134}+ $ \\ $ 1862263703755046686915929765647974894660271606313311907212125161047017 x^{136}- $ \\ $ 377590353425249818266299858942934875413754617415496566482067809546802
x^{138}+ $ \\ $ 68501660497029754352400168967407895507258531456907942410368713926120 x^{140}- $ \\ $ 11105365183430858507629457683684843169630429304793560853163166939810
x^{142}+ $ \\ $ 1606531109473718534594530182665952079562954439432361598630717186236 x^{144}- $ \\ $ 207040173678939354726729332698936652911106381521251519357889105280
x^{146}+ $ \\ $ 23725381454208217615297576176963428232818874679802766445008594690 x^{148}- $ \\ $ 2412310961657048946296173910145662411783527155642611007337012021
x^{150}+ $ \\ $ 217095118910027467160288291483232288984125117244905095790365759 x^{152}- $ \\ $ 17244238890870092795654597417700220025987925425589658553559407 x^{154}+ $ \\ $ 1205077512777402031615916716203648867995217155679218833218371
x^{156}- $ \\ $ 73815890507916320676911789220378258814331426144792858248950 x^{158}+ $ \\ $ 3946234528864430181274400320996867933891840096143776755490 x^{160}- $ \\ $ 183208374687979113510601425933222783247296920338183962750
x^{162}+ $ \\ $ 7343470634529034263997971855872818683791273206571064850 x^{164}- $ \\ $ 252386024928336323757054707034042564283133498762815500 x^{166}+ $ \\ $ 7377219808182764935914860042284586198735853854686500
x^{168}- $ \\ $ 181608181279946828536236488212847960492136084478125 x^{170}+ $ \\ $ 3720870390260477293342366907407412832343772765625 x^{172}- $ \\ $ 62531478159246985678688894033140505718654706250
x^{174}+ $ \\ $ 846486577971149809174369696179591097754012500 x^{176}- $ \\ $ 9020287465658281265111985694080724339000000 x^{178}+ $ \\ $ 73447683765231069508597890169422765000000
x^{180}- $ \\ $ 439311193933486419544786460231500000000 x^{182}+1829042325676928704927548200000000000 x^{184}-4909393424488807413731360000000000 x^{186}+7559224240682295244800000000000
x^{188}-5500825133923328000000000000 x^{190}+1451262771200000000000000 x^{192}))/
$ \\ $
((-1+x) (1+x) (-1+2 x) (1+2 x) (-1+11 x) (1+11
x) (1-3 x+x^2) (1+3 x+x^2) (1-13 x+11 x^2) (1-8 x+11 x^2) (1+8 x+11 x^2) (1+13 x+11 x^2)
(1-12 x^2+8 x^3) (-1+12 x^2+8 x^3) (-1+104 x-2661 x^2+24090 x^3-91993 x^4+158236 x^5-121128 x^6+44736 x^7-7808 x^8+512
x^9) (1+104 x+2661 x^2+24090 x^3+91993 x^4+158236 x^5+121128 x^6+44736 x^7+7808 x^8+512 x^9) (-1+5 x+45 x^2-305 x^3+65 x^4+2175
x^5-1800 x^6-4525 x^7+1975 x^8+2500 x^9-875 x^{10}-375 x^{11}+125 x^{12}) (-1-5 x+45 x^2+305 x^3+65 x^4-2175 x^5-1800 x^6+4525 x^7+1975
x^8-2500 x^9-875 x^{10}+375 x^{11}+125 x^{12}) (1-22 x-214 x^2+4230 x^3-15001 x^4-2783 x^5+60858 x^6-12548 x^7-46005 x^8+16340 x^9+10377
x^{10}-5171 x^{11}-329 x^{12}+413 x^{13}-27 x^{14}-9 x^{15}+x^{16}) (1+22 x-214 x^2-4230 x^3-15001 x^4+2783 x^5+60858 x^6+12548 x^7-46005
x^8-16340 x^9+10377 x^{10}+5171 x^{11}-329 x^{12}-413 x^{13}-27 x^{14}+9 x^{15}+x^{16}) (1-82 x+2898 x^2-59227 x^3+789803 x^4-7334275
x^5+49339066 x^6-246611784 x^7+930982131 x^8-2681675595 x^9+5927069364 x^{10}-10071216252 x^{11}+13140900495 x^{12}-13114381556 x^{13}+9943609371
x^{14}-5673467839 x^{15}+2404556811 x^{16}-744137999 x^{17}+164387022 x^{18}-25156085 x^{19}+2562619 x^{20}-164923 x^{21}+6260 x^{22}-125 x^{23}+x^{24})
(1+82 x+2898 x^2+59227 x^3+789803 x^4+7334275 x^5+49339066 x^6+246611784 x^7+930982131 x^8+2681675595 x^9+5927069364 x^{10}+10071216252 x^{11}+13140900495
x^{12}+13114381556 x^{13}+9943609371 x^{14}+5673467839 x^{15}+2404556811 x^{16}+744137999 x^{17}+164387022 x^{18}+25156085 x^{19}+2562619 x^{20}+164923
x^{21}+6260 x^{22}+125 x^{23}+x^{24}) (1-104 x+4410 x^2-104155 x^3+1561045 x^4-15952817 x^5+116275468 x^6-623060570 x^7+2507306175 x^8-7693787635
x^9+18199448138 x^{10}-33436393552 x^{11}+47933912805 x^{12}-53730741510 x^{13}+47072037995 x^{14}-32132854819 x^{15}+16992690151 x^{16}-6899340865
x^{17}+2123302560 x^{18}-486469225 x^{19}+80870375 x^{20}-9387875 x^{21}+715250 x^{22}-31875 x^{23}+625 x^{24}) (1+104 x+4410 x^2+104155
x^3+1561045 x^4+15952817 x^5+116275468 x^6+623060570 x^7+2507306175 x^8+7693787635 x^9+18199448138 x^{10}+33436393552 x^{11}+47933912805 x^{12}+53730741510
x^{13}+47072037995 x^{14}+32132854819 x^{15}+16992690151 x^{16}+6899340865 x^{17}+2123302560 x^{18}+486469225 x^{19}+80870375 x^{20}+9387875 x^{21}+715250
x^{22}+31875 x^{23}+625 x^{24}))+ $}
 \\
\parbox[t]{5.8in}{ $
20 x^2 (-234+320088 x^2-136124684 x^4+19593775368 x^6+2051776259798 x^8-1228350029049507 x^{10}+222968262737347774 x^{12}-24260112070770702232
x^{14}+1822705736380052048069 x^{16}-100916244524214917441711 x^{18}+4280351276362169014606513 x^{20}-142714587242091466173878208 x^{22}+3809244824928988744064951017
x^{24}-82520285133755774770493062281 x^{26}+1467444186063988950248319979540 x^{28}-21645799220520736094505619118872 x^{30}+267673821756028722432428512055244
x^{32}-2806982839652108649030286648082110 x^{34}+25272503448488173051336420691992951 x^{36}-197780012316676527080052996403029466 x^{38}+1359169178626735798084611650934559940
x^{40}-8248031136427815283635899573046948609 x^{42}+44177863065405964773884157087947265203 x^{44}-207622415152659005897709259821867905936 x^{46}+848497628761386032955407146818233252549
x^{48}-2986630066500259391059753396810405403376 x^{50}+8978250680997615844943555979241325835152 x^{52}-22895948659860885033520654294356119238385
x^{54}+49283420613392374626429789336199914351832 x^{56}-89218926571231646024024810770340452641720 x^{58}+ $ \\ $ 135504944868856507823416830950748516957057
x^{60}- $ \\ $ 172389160906176314420710885275180212724527 x^{62}+ $ \\ $ 183547316219318903960320280829096116071264 x^{64}- $ \\ $ 163511137364714938845665108789571404106054
x^{66}+ $ \\ $ 121892330878948391165227140997734856300186 x^{68}- $ \\ $ 76072332280581439617597030145837199231769 x^{70}+ $ \\ $ 39768616467462574761163933983127050741442
x^{72}-17423630167343235325212146702079447097908 x^{74}+6399637265892835771340913358933746610354 x^{76}-1970513471030550875880140613668297005840
x^{78}+508394838137195907177822243919451418219 x^{80}-109789516243417400238915411635236675515 x^{82}+19810446202617140638361529402358493938 x^{84}-2978992033214528091220468262602926129
x^{86}+371971689389903353069545144071327698 x^{88}-38379338936238591743463114713161240 x^{90}+3251185572664588565055860085862595 x^{92}-224233918626783202070571619435700
x^{94}+12454497048355022000718356741550 x^{96}-549122212654098218823400038125 x^{98}+18852262128595778955123212625 x^{100}-490751406011489188851217500
x^{102}+9320772012592835172937500 x^{104}-121647973242146321200000 x^{106}+981927330967529600000 x^{108}-3885738425184000000 x^{110}+2771248480000000
x^{112}))/ $ \\ $ ((-1+x) (1+x) (-1+2 x) (1+2 x) (-1+11 x) (1+11 x) (-1+9 x-6 x^2+x^3) (1+9 x+6 x^2+x^3) (1-12
x^2+8 x^3) (-1+12 x^2+8 x^3) (1-59 x+402 x^2-863 x^3+572 x^4) (1+59 x+402 x^2+863 x^3+572 x^4) (1-18
x+93 x^2-195 x^3+164 x^4-37 x^5+x^6) (1+18 x+93 x^2+195 x^3+164 x^4+37 x^5+x^6) (1-33 x+375 x^2-1922 x^3+5192 x^4-7883 x^5+6797
x^6-3252 x^7+834 x^8-105 x^9+5 x^{10}) (1+33 x+375 x^2+1922 x^3+5192 x^4+7883 x^5+6797 x^6+3252 x^7+834 x^8+105 x^9+5 x^{10}) (-1+5
x+45 x^2-305 x^3+65 x^4+2175 x^5-1800 x^6-4525 x^7+1975 x^8+2500 x^9-875 x^{10}-375 x^{11}+125 x^{12}) (-1-5 x+45 x^2+305 x^3+65 x^4-2175
x^5-1800 x^6+4525 x^7+1975 x^8-2500 x^9-875 x^{10}+375 x^{11}+125 x^{12}) (1-22 x-214 x^2+4230 x^3-15001 x^4-2783 x^5+60858 x^6-12548
x^7-46005 x^8+16340 x^9+10377 x^{10}-5171 x^{11}-329 x^{12}+413 x^{13}-27 x^{14}-9 x^{15}+x^{16}) (1+22 x-214 x^2-4230 x^3-15001 x^4+2783
x^5+60858 x^6+12548 x^7-46005 x^8-16340 x^9+10377 x^{10}+5171 x^{11}-329 x^{12}-413 x^{13}-27 x^{14}+9 x^{15}+x^{16}))+$ \\
\\ $10 x^2 (153-29597 x^2+1087980 x^4+42425030 x^6-3931103503 x^8+114477354009 x^{10}-1757059345360 x^{12}+16138162540382 x^{14}-94374636014721
x^{16}+361278697104035 x^{18}-926515970968211 x^{20}+1633115810080701 x^{22}-2029670985218743 x^{24}+1818156169145257 x^{26}-1193705004547281 x^{28}+580696031421991
x^{30}-210302687884583 x^{32}+56589053389377 x^{34}-11202352029008 x^{36}+1598514878512 x^{38}-158589514560 x^{40}+10266827840 x^{42}-384000000 x^{44}+6144000
x^{46}))/ $ \\ $ \ds ((-1+4 x) (1+4 x) (-1+34 x-160 x^2+248 x^3-152 x^4+32 x^5) (1+34 x+160 x^2+248 x^3+152 x^4+32 x^5)
(-1+18 x-91 x^2+110 x^3+175 x^4-515 x^5+455 x^6-175 x^7+25 x^8) (-1-18 x-91 x^2-110 x^3+175 x^4+515 x^5+455 x^6+175 x^7+25 x^8)
(-1+8 x+14 x^2-194 x^3+302 x^4-65 x^5-120 x^6+62 x^7+4 x^8-7 x^9+x^{10}) (-1-8 x+14 x^2+194 x^3+302 x^4+65 x^5-120 x^6-62 x^7+4 x^8+7
x^9+x^{10}))+ $ \\ $ \ds
\frac{240 x^2}{(-1+x) (1+x) (-1+11 x) (1+11 x)} +
\frac{8 x^2}{1-4 x^2}
   $ \bc $ \mbox{--------------------------------------    } $  \ec }
  \\
\parbox[t]{5.8in}{ ${\cal F}^{TG}_{10,3}(x) = {\cal F}^{TG}_{10,7}(x) = $ \\
 $ \ds
-2 x (-36-1916 x+168735 x^2+5010324 x^3-274378534 x^4-4757396231 x^5+228843935150 x^6+2025044991592 x^7-111577291533350
x^8-351784868410599 x^9+34709854247997148 x^{10}-34224714438481007 x^{11}-7228139570023107567 x^{12}+31484087151192214710 x^{13}+1022544557324685448165
x^{14}-7726403586188404895132 x^{15}-96927246923228013595192 x^{16}+1110292692380516699256673 x^{17}+5787032606728018989068536 x^{18}-107024913198782087817627047
x^{19}-161936681754414137310937526 x^{20}+7363849217031780970200396412 x^{21}- $ \\ $ 5103265539648415668733523649 x^{22}-377007276918192621691983323335
x^{23}+ $ \\ $ 814424380614062430253250143406 x^{24}+14854797156536370812767804103645 x^{25}- $ \\ $ 47444804546241151730315447453123 x^{26}-464807852968619111518109495428186
x^{27}+1811873185531787270465347762924488 x^{28}+11930197928982070276439367794697501 x^{29}-50797582722824878493282804727691372 x^{30}-259986161076557036689640239414096519
x^{31}+1090776426980624605622189726780448964 x^{32}+4969547018363442145475560980703270892 x^{33}-18211073820660843827319672959217077929 x^{34}-85135169833414862520908525451264813903
x^{35}+234772295338935068809824990236485723882 x^{36}+1311251933592394296445067090264572153729 x^{37}-2233938586334662638636245659125876508990 x^{38}-17926309036444108577185692626715264379976
x^{39}+13057535908042630686034170038972661451122 x^{40}+ $ \\ $ 213503897268271369291205614772900281740876 x^{41}+ $ \\ $ 11942703203239479252622866715936665605983
x^{42}- $ \\ $ 2179643296990347355591253994135364500009211 x^{43}- $ \\ $ 1400332310309499241938435675019641359806534 x^{44}+ $ \\ $ 18867496983220826568169506882617368827379464
x^{45}+ $ \\ $ 20221436637104425292971526317275687933765789 x^{46}- $ \\ $ 137648914756743532843556306205509510025923348 x^{47}- $ \\ $ 191451061106249771522974055941451069353101089
x^{48}+ $ \\ $ 844165244475461400543956274507100420448139056 x^{49}+ $ \\ $ 1382404398775516625584271594630147301334133962 x^{50}- $ \\ $ 4350569921243651172165308859465882298992407885
x^{51}- $ \\ $ 7997835350801929281998017686821726079764334791 x^{52}+ $ \\ $ 18867944017272600388961468253894418273069526052 x^{53}+ $ \\ $ 37923903961575131564776043748056924014809725008
x^{54}- $ \\ $ 69051116005894445846938834684153259489987283484 x^{55}- $ \\ $ 149292783277158409938890178952977240380261269884 x^{56}+ $ \\ $ 214128943684090145833646407125392402737492559008
x^{57}+ $ \\ $ 492147658986179500815073230736569819493983577106 x^{58}- $ \\ $ 565734019821988649775796329605823969557792979851 x^{59}- $ \\ $ 1367589248484230250289518027124766412971441247652
x^{60}+ $ \\ $ 1282101525337831397831399640179867037217958276087 x^{61}+ $ \\ $ 3221262140039836175406543937752562410298848886090 x^{62}- $ \\ $ 2512109952571249121285404582729418334150525500754
x^{63}- $ \\ $ 6462631108446716063860559605318793409605106728837 x^{64}+ $ \\ $ 4292410502638103736617386264622176038468532872040 x^{65}+ $ \\ $ 11090608632794755540037476442617132770684107747052
x^{66}- $ \\ $ 6450967502022728552624652596886874199633565483404 x^{67}- $ \\ $ 16340394528514224460352930271543579137143790685532 x^{68}+ $ \\ $ 8590783629340927231223030228549252522418485800541
x^{69}+ $ \\ $ 20732941366425606820638218969191703723055871444598 x^{70}- $ \\ $ 10189967962701197804045503177647253969425181371737 x^{71}- $ \\ $ 22708605957959876839421247633531256215800348715717
x^{72}+ $ \\ $ 10789546426634996532849475119997603196871152735092 x^{73}+ $ \\ $ 21506936307411287369259389782946888449347349247870 x^{74}- $ \\ $ 10190492654169351518317593508843893787658278015062
x^{75}- $ \\ $ 17628471249633504173415860543025600400997152054412 x^{76}+ $ \\ $ 8559304260765623984619410294466164590804168593886 x^{77}+ $ \\ $ 12505920532346866707087441097482649399627136909022
x^{78}- $ \\ $ 6366255619740655357299366212575030653293867360313 x^{79}- $ \\ $ 7671314125552612910659128615463488597155255665462 x^{80}+ $}  \\
\parbox[t]{5.8in}{ $ 4173841459066035523630590859289795540296006308627
x^{81}+ $ \\ $ 4060129773656062830513291382887072135261415926063 x^{82}- $ \\ $ 2401528027750720332911111979568835391311111689570 x^{83}- $ \\ $ 1847110705100133652247168856585744119018940801527
x^{84}+ $ \\ $ 1207814202812711595768462298351485236383125314439 x^{85}+ $ \\ $ 717935783248206530641262116237831800033117205110 x^{86}- $ \\ $ 529015240726349040459348935526421277410621669676
x^{87}- $ \\ $ 236074515095465841918522486527477088250531080282 x^{88}+ $ \\ $ 201061589024348389785817914650665157471527960706 x^{89}+ $ \\ $ 64586932647694501290681726703266943669718879747
x^{90}- $ \\ $ 66061881749215181108610074962727799952933803044 x^{91}- $ \\ $ 14246289880628663983127902407151714357822730944 x^{92}+ $ \\ $ 18686159159803262350807232538694428815188675551
x^{93}+ $ \\ $ 2354969865021389684484798074355442693094877156 x^{94}- $ \\ $ 4528292840022064554946645760527744494477780546 x^{95}- $ \\ $ 222979309586946537733886101580663340672685812
x^{96}+ $ \\ $ 934762157241035392240145387464924568014652569 x^{97}- $ \\ $ 16310382302301099357939512280958189475626626 x^{98}- $ \\ $ 163240549167171917628614142832851854163662451
x^{99}+ $ \\ $ 12486324140443905088460189522686374507002677 x^{100}+ $ \\ $ 23916605194993724961756722842770694637641942 x^{101}- $ \\ $ 3113615636906831580264512497699747490684672
x^{102}- $ \\ $ 2910123122887909578140795582029463434826802 x^{103}+ $ \\ $ 526492210810517257755405625091807788838068 x^{104}+ $ \\ $ 290422425775152621330145950103904170459752
x^{105}- $ \\ $ 67233194380924146428815639872844417067586 x^{106}- $ \\ $ 23399408200317334916422711040101343373183 x^{107}+ $ \\ $ 6688880744417021037524526118998530451568
x^{108}+ $ \\ $ 1490969312450495634801382726014192521101 x^{109}-522493346107885097091398951520140132033 x^{110}-72995970232301946385277269709301064610 x^{111}+31965568239313812611014774131543921065
x^{112}+2624326919216007470130970868200137915 x^{113}-1519024845577441160073185013926343670 x^{114}-63350510078567124003738981822693800 x^{115}+55385339062745059619974689082766000
x^{116}+762820258695875590863590180838400 x^{117}-1524898406969209282514771091817200 x^{118}+7233652169272220117438357244000 x^{119}+31037166475816379824971933208000
x^{120}-503501276797022466754815080000 x^{121}-452200065221602358330196480000 x^{122}+10354100042319884708192000000 x^{123}+4456516208552562339248000000
x^{124}-118522096312353717232000000 x^{125}-26803219242657227520000000 x^{126}+737317094012558080000000 x^{127}+83077840933058560000000 x^{128}-1828120427008000000000
x^{129}-120046659174400000000 x^{130}+889987072000000000 x^{131}+47775744000000000 x^{132})/$ \\ $((-1+x) (1+x) (1+2 x) (-1+4 x) (1+4
x) (1-8 x+4 x^2) (1+5 x+5 x^2) (1+7 x+8 x^2) (1+18 x+63 x^2+50 x^3+13 x^4+x^5) (-1+34 x-160 x^2+248
x^3-152 x^4+32 x^5) (1+34 x+160 x^2+248 x^3+152 x^4+32 x^5) (-1+18 x-91 x^2+110 x^3+175 x^4-515 x^5+455 x^6-175 x^7+25 x^8)
(-1-18 x-91 x^2-110 x^3+175 x^4+515 x^5+455 x^6+175 x^7+25 x^8) (-1+121 x-3901 x^2+49003 x^3-293115 x^4+896942 x^5-1381585 x^6+970100
x^7-264304 x^8+15552 x^9) (-1+8 x+14 x^2-194 x^3+302 x^4-65 x^5-120 x^6+62 x^7+4 x^8-7 x^9+x^{10}) (-1-8 x+14 x^2+194 x^3+302
x^4+65 x^5-120 x^6-62 x^7+4 x^8+7 x^9+x^{10}) (1-36 x+504 x^2-3603 x^3+14416 x^4-33263 x^5+44489 x^6-34794 x^7+16069 x^8-4345 x^9+653
x^{10}-47 x^{11}+x^{12}) (1+53 x+951 x^2+7589 x^3+32774 x^4+84127 x^5+133734 x^6+133056 x^7+81936 x^8+30215 x^9+6265 x^{10}+650 x^{11}+25
x^{12}) (1+92 x+3048 x^2+50453 x^3+483576 x^4+2928888 x^5+11832412 x^6+32949126 x^7+64377394 x^8+88750186 x^9+85833700 x^{10}+57206921
x^{11}+25441290 x^{12}+7182529 x^{13}+1203428 x^{14}+111956 x^{15}+5340 x^{16}+120 x^{17}+x^{18}) (1-66 x+1833 x^2-28651 x^3+283174 x^4-1881547
x^5+8719340 x^6-28859906 x^7+69342834 x^8-122339404 x^9+159752386 x^{10}-155169581 x^{11}+112318552 x^{12}-60476870 x^{13}+24066800 x^{14}-6990095
x^{15}+1450676 x^{16}-207780 x^{17}+19365 x^{18}-1050 x^{19}+25 x^{20}))+
   $} \\
 \parbox[t]{5.8in}{ $
 -10 x (5+12116 x^2-79418156 x^4+137140867371 x^6-126508615717472 x^8+74767076102836781 x^{10}-30942949149634802119 x^{12}+9474224318337889138895
x^{14}-2232969550357170034098693 x^{16}+418666074569555119051522960 x^{18}-64398858149144925068715966395 x^{20}+ $ \\ $8380549398143979118631696745651 x^{22}-950353277648081037400875899749218
x^{24}+ $ \\ $ 96166253683043030683405714746310640 x^{26}- 8792671215727369494525660472248856129 x^{28}+ $ \\ $ 726431573860985445629172152338973849642 x^{30}- 53798911213912575702326451029551494505063
x^{32}+ $ \\ $ 3537041647440449141091748495572121759171627 x^{34}- $ \\ $ 204921187006042892771855403946720565775891226 x^{36}+ $ \\ $ 10420451672656398773177209637593960317776673940
x^{38}- $ \\ $ 464574261578436636974365554081388609823037887244 x^{40}+ $ \\ $ 18173213589595381981298097198477475680225585776658 x^{42}- $ \\ $ 624908252678769459968024882300955823830594046178503
x^{44}+ $ \\ $ 18932972835331332900494207397483318358405483346878536 x^{46}- $ \\ $ 506645827276974021061686954398768831988005446790414876 x^{48}+ $ \\ $ 12003413445861334707963534368389598629541759708354889745
x^{50}- $ \\ $ 252331412959443401022688614962893956206296850502898884612 x^{52}+ $ \\ $ 4715795006592029282544383550848335485940811875647240445403 x^{54}- $ \\ $ 78487285463193038364584002227312381535862216674549935342749
x^{56}+ $ \\ $ 1165046290521685133214209718327497791570656140136198298109117 x^{58}- $ \\ $ 15442812909805651340672989975528735982914073474512388259682979 x^{60}+ $ \\ $ 182977750311047710034345040083038068561084101369340495943426800
x^{62}- $ \\ $ 1939663773210336010648372584343908539911860679341498590418276668 x^{64}+ $ \\ $ 18407972063428115196727823647040468291087661753116450203992748767
x^{66}- $ \\ $ 156483275436745088163493327124237157527383915894910147614127604443 x^{68}+ $ \\ $ 1192032663588453188636719530285268875420447867964754432031081460355
x^{70}- $ \\ $ 8139450714520556617631586123457434234058573567496087323586845586213 x^{72}+ $ \\ $ 49828495446843377107558048069230075854878576225379936862269833932662
x^{74}- $ \\ $ 273522688760435447683941721355414268996739648431776542650301532681392 x^{76}+ $ \\ $ 1346401325180100284214171590687485399563700464728372399915910200599243
x^{78}- $ \\ $ 5943421964008175511708926578056651344116897808780805082502953702381898 x^{80}+ $ \\ $ 23527945857702709573936161343767789953265247293219975356432250361804601
x^{82}- $ \\ $ 83525073888832706915791705619488426178753343416477542876432340041041581 x^{84}+ $ \\ $ 265910823567429568614445517131522099821535138339159760419321110141507911
x^{86}- $ \\ $ 759185171579093056810065374765768522799159734093044958218155861957088189 x^{88}+ $ \\ $ 1943868869176729791208920065518878574471329118608259392191255594937396422
x^{90}- $ \\ $ 4463951590575857864845925749111540779398612798411265031683371489478392313 x^{92}+ $ \\ $ 9194798929902132275679830302034860240699544487890965937510047811852078371
x^{94}- $ \\ $ 16989576360092651399188099561997843627512750816477144463253853993691771492 x^{96}+ $ \\ $ 28164087047319318264900705315803002488133099370226233533165596637417849460
x^{98}- $ \\ $ 41892978838686483195497319334134016970652660371262642296182535869955627951 x^{100}+ $ \\ $ 55921476431718526318975650434431995277386842605498694329110096282384030776
x^{102}- $ \\ $ 66997973808755369796112943793634994220136455402164648700188296147308129474 x^{104}+ $ \\ $ 72049808293548540767760266431693622548580954761803915204469275049003110499
x^{106}- $ \\ $ 69553194411943777834266005659860482482633803800107573186654319196098070075 x^{108}+ $ \\ $ 60271979747886817678692147417517488277285764394829973545715540572955695623
x^{110}- $ \\ $ 46881143548794248642431957083066699715374137915368058857257154576986182241 x^{112}+ $ \\ $ 32726318466471098605114887198050551084612284665705165762957145792468286678
x^{114}- $ \\ $ 20497314121136849362189361591842619709961214856986553953580334352499902524 x^{116}+ $ \\ $ 11513946116651347159013570955103452008331834498101419443461571533756318473
x^{118}- $ \\ $ 5797402717328986514251230556474300092287212890897992439415700302276897215 x^{120}+ $ \\ $ 2614500970163823323223492592644809292916890668890970542261118441971434577
x^{122}- $ \\ $ 1054964580013359327689484487031375380100209082989886573323001904092263934 x^{124}+ $ \\ $ 380342789334010891582609761777148294258923192432266575791633625238020242
x^{126}- $ \\ $ 122289032452474187177677777755412177221482682549858300520133740887344744 x^{128}+ $ \\ $ 34975957438918481379966994161919059670729272785641699683433127096289836
x^{130}- $} \\
 \parbox[t]{5.8in}{  $ 8867216335858408485217966795135537628625490781859865415911624430361503 x^{132}+ $ \\ $ 1982624968515707637882074749906573421341207210557318049774437204700534
x^{134}- $ \\ $ 387989956867150841539685396282870602794296777489383115051049176531255 x^{136}+ $ \\ $ 65643439581663042595026003087583729005006126850300437025068914527619
x^{138}- $ \\ $ 9392978584589560782463227527721090762503409838765585137809104047858 x^{140}+ $ \\ $ 1085079304345741842178680500123055802538537940244207945025150841960
x^{142}- $ \\ $ 88495250760041851373048611707004737239783763224475009416367975128 x^{144}+ $ \\ $ 1790052550388233104048219091304253202990729749372583270202309066
x^{146}+ $ \\ $ 1003651541025064484844332148522101800594356584488486112086238211 x^{148}- $ \\ $ 237744796217808129585073596645411533868824824815971527447233894
x^{150}+ $ \\ $ 34955479540841192462625004541662671964556133838389234629589100 x^{152}- $ \\ $ 3978498491577539079609191812924069506324893217877858023989683 x^{154}+ $ \\ $ 372170923315028263281435005279062208055390030676997320391049
x^{156}- $ \\ $ 29285215996028369935028045814603066957693117403020564430595 x^{158}+ $ \\ $ 1957223619862996298504154794316123978047628015601370232650 x^{160}- $ \\ $ 111459186075598287557357052611641746335074322313046302550
x^{162}+ $ \\ $ 5405339264174080415093851046473072210522644367389902550 x^{164}- $ \\ $ 222505700416139768735317491851352491491658997664218500 x^{166}+ $ \\ $ 7731495850912541731260091645150134812023365828796875
x^{168}- $ \\ $ 225017823393816165756080531458991265025745174253750 x^{170}+ $ \\ $ 5429584089654686978951872937358530273323065666875 x^{172}- $ \\ $ 107204318171259230142700949209717874664888165625
x^{174}+ $ \\ $ 1703161891355082735952840513091010492878731250 x^{176}- $ \\ $ 21305578284273495288972444199134108364500000 x^{178}+ $ \\ $ 204007873783900670205110921677009827500000
x^{180}- $ \\ $ 1440110752865636403817658922898850000000 x^{182}+7121003927421030307978111698000000000 x^{184}-22957124325133939279933280000000000 x^{186}+43377645661355179680000000000000
x^{188}-40364378214778112000000000000 x^{190}+13684748267520000000000000 x^{192})/$ \\ $((-1+x) (1+x) (-1+2 x) (1+2 x) (-1+11 x) (1+11
x) (1-3 x+x^2) (1+3 x+x^2) (1-13 x+11 x^2) (1-8 x+11 x^2) (1+8 x+11 x^2) (1+13 x+11 x^2)
(1-12 x^2+8 x^3) (-1+12 x^2+8 x^3) (-1+104 x-2661 x^2+24090 x^3-91993 x^4+158236 x^5-121128 x^6+44736 x^7-7808 x^8+512
x^9) (1+104 x+2661 x^2+24090 x^3+91993 x^4+158236 x^5+121128 x^6+44736 x^7+7808 x^8+512 x^9) (-1+5 x+45 x^2-305 x^3+65 x^4+2175
x^5-1800 x^6-4525 x^7+1975 x^8+2500 x^9-875 x^{10}-375 x^{11}+125 x^{12}) (-1-5 x+45 x^2+305 x^3+65 x^4-2175 x^5-1800 x^6+4525 x^7+1975
x^8-2500 x^9-875 x^{10}+375 x^{11}+125 x^{12}) (1-22 x-214 x^2+4230 x^3-15001 x^4-2783 x^5+60858 x^6-12548 x^7-46005 x^8+16340 x^9+10377
x^{10}-5171 x^{11}-329 x^{12}+413 x^{13}-27 x^{14}-9 x^{15}+x^{16}) (1+22 x-214 x^2-4230 x^3-15001 x^4+2783 x^5+60858 x^6+12548 x^7-46005
x^8-16340 x^9+10377 x^{10}+5171 x^{11}-329 x^{12}-413 x^{13}-27 x^{14}+9 x^{15}+x^{16}) (1-82 x+2898 x^2-59227 x^3+789803 x^4-7334275
x^5+49339066 x^6-246611784 x^7+930982131 x^8-2681675595 x^9+5927069364 x^{10}-10071216252 x^{11}+13140900495 x^{12}-13114381556 x^{13}+9943609371
x^{14}-5673467839 x^{15}+2404556811 x^{16}-744137999 x^{17}+164387022 x^{18}-25156085 x^{19}+2562619 x^{20}-164923 x^{21}+6260 x^{22}-125 x^{23}+x^{24})
(1+82 x+2898 x^2+59227 x^3+789803 x^4+7334275 x^5+49339066 x^6+246611784 x^7+930982131 x^8+2681675595 x^9+5927069364 x^{10}+10071216252 x^{11}+13140900495
x^{12}+13114381556 x^{13}+9943609371 x^{14}+5673467839 x^{15}+2404556811 x^{16}+744137999 x^{17}+164387022 x^{18}+25156085 x^{19}+2562619 x^{20}+164923
x^{21}+6260 x^{22}+125 x^{23}+x^{24}) (1-104 x+4410 x^2-104155 x^3+1561045 x^4-15952817 x^5+116275468 x^6-623060570 x^7+2507306175 x^8-7693787635
x^9+18199448138 x^{10}-33436393552 x^{11}+47933912805 x^{12}-53730741510 x^{13}+47072037995 x^{14}-32132854819 x^{15}+16992690151 x^{16}-6899340865
x^{17}+2123302560 x^{18}-486469225 x^{19}+80870375 x^{20}-9387875 x^{21}+715250 x^{22}-31875 x^{23}+625 x^{24}) (1+104 x+4410 x^2+104155
x^3+1561045 x^4+15952817 x^5+116275468 x^6+623060570 x^7+2507306175 x^8+7693787635 x^9+18199448138 x^{10}+33436393552 x^{11}+47933912805 x^{12}+53730741510
x^{13}+47072037995 x^{14}+32132854819 x^{15}+16992690151 x^{16}+6899340865 x^{17}+2123302560 x^{18}+486469225 x^{19}+80870375 x^{20}+9387875 x^{21}+715250
x^{22}+31875 x^{23}+625 x^{24})) +$}
 \\
\parbox[t]{5.8in}{$
20 x (-1-7462 x^2+10821110 x^4-5831400069 x^6+1697820496893 x^8-308858374732578 x^{10}+37702475169808132 x^{12}-3185614084338056754
x^{14}+185590028685218837998 x^{16}-6977471751331290769044 x^{18}+116319318777276366674463 x^{20}+3999608763470504622094343 x^{22}-377737416256394840648610318
x^{24}+15519679525165384897331544882 x^{26}- $ \\ $ 433927200141827426983587685042 x^{28}+9125056924927835556655466937503 x^{30}-$ \\ $ 150536536854634405519175680765654
x^{32}+ 1993544783725378630249399433052553 x^{34}- 21499362635105896408621603402843080 x^{36}+ 190641781155325395202533674826755142 x^{38}- 1399217823032215716335407078488612543
x^{40}+ 8539127744177678958060920648882825538 x^{42}- 43463945090054981464218242987508598032 x^{44}+ 184864845779386486382692454012059277145 x^{46}- 657688834829591094722077056371086082332
x^{48}+1957766585510085808050179364694581126812 x^{50}-4875232794264694656151162666299958638316 x^{52}+10151155868176949128615346944290379140933
x^{54}-17663636894646983285752078518863185757343 x^{56}+25675826594048285268183840759942792750132 x^{58}-31179991896285055472961990580261975987806
x^{60}+31654893677412544844443219632667434433441 x^{62}-26901116475608043194315960173902454931844 x^{64}+19158139606098240680001128673664551962516
x^{66}-11424361390006327607950339936180762768054 x^{68}+5668976093562867051564563126606570777766 x^{70}-2299354022466441346099809259718404773236
x^{72}+729434661055564132571856270866151396569 x^{74}-159439615475006473968083099727921925278 x^{76}+10344118514096746041938408544272297229 x^{78}+9561257817018618562109881539028346644
x^{80}-5369756299516475821266336994021720743 x^{82}+1718144171830579878909467633407256974 x^{84}-399219702188962880466123482662543144 x^{86}+71662605680566682337258252423059683
x^{88}-10163082704123961106407234076358225 x^{90}+1147519550739261609993649422486375 x^{92}-103161083863250469590135854626400 x^{94}+7344296866782507412058917789350
x^{96}-409753388119012653712344902750 x^{98}+17626639128811501846661608750 x^{100}-570757985824669654350612500 x^{102}+13421821873595697379427500
x^{104}-216707453609405809250000 x^{106}+2180021700889543400000 x^{108}-11153092053400000000 x^{110}+14806523160000000 x^{112})/ $ \\ $ ((-1+x)
(1+x) (-1+2 x) (1+2 x) (-1+11 x) (1+11 x) (-1+9 x-6 x^2+x^3) (1+9 x+6 x^2+x^3) (1-12 x^2+8 x^3) (-1+12 x^2+8
x^3) (1-59 x+402 x^2-863 x^3+572 x^4) (1+59 x+402 x^2+863 x^3+572 x^4) (1-18 x+93 x^2-195 x^3+164 x^4-37 x^5+x^6)
(1+18 x+93 x^2+195 x^3+164 x^4+37 x^5+x^6) (1-33 x+375 x^2-1922 x^3+5192 x^4-7883 x^5+6797 x^6-3252 x^7+834 x^8-105 x^9+5 x^{10})
(1+33 x+375 x^2+1922 x^3+5192 x^4+7883 x^5+6797 x^6+3252 x^7+834 x^8+105 x^9+5 x^{10}) (-1+5 x+45 x^2-305 x^3+65 x^4+2175 x^5-1800
x^6-4525 x^7+1975 x^8+2500 x^9-875 x^{10}-375 x^{11}+125 x^{12}) (-1-5 x+45 x^2+305 x^3+65 x^4-2175 x^5-1800 x^6+4525 x^7+1975 x^8-2500
x^9-875 x^{10}+375 x^{11}+125 x^{12}) (1-22 x-214 x^2+4230 x^3-15001 x^4-2783 x^5+60858 x^6-12548 x^7-46005 x^8+16340 x^9+10377 x^{10}-5171
x^{11}-329 x^{12}+413 x^{13}-27 x^{14}-9 x^{15}+x^{16}) (1+22 x-214 x^2-4230 x^3-15001 x^4+2783 x^5+60858 x^6+12548 x^7-46005 x^8-16340
x^9+10377 x^{10}+5171 x^{11}-329 x^{12}-413 x^{13}-27 x^{14}+9 x^{15}+x^{16}))- $ \\  \\
$ 10 x (-5+843 x^2-119678 x^4+10867994 x^6-458965303 x^8+8922095485 x^{10}-61636790888 x^{12}-401648853762 x^{14}+9116648473545
x^{16}-57595864199635 x^{18}+187629020347989 x^{20}-369284813892781 x^{22}+476512235911457 x^{24}-423463849180949 x^{26}+266671095305693 x^{28}-120427555936855
x^{30}+38823243083667 x^{32}-8699482508427 x^{34}+1260261232950 x^{36}-91964916840 x^{38}-2608957200 x^{40}+1217168000 x^{42}-102336000 x^{44}+3072000
x^{46})/ $ \\ $ ((-1+4 x) (1+4 x) (-1+34 x-160 x^2+248 x^3-152 x^4+32 x^5) (1+34 x+160 x^2+248 x^3+152 x^4+32 x^5)
(-1+18 x-91 x^2+110 x^3+175 x^4-515 x^5+455 x^6-175 x^7+25 x^8) (-1-18 x-91 x^2-110 x^3+175 x^4+515 x^5+455 x^6+175 x^7+25 x^8)
(-1+8 x+14 x^2-194 x^3+302 x^4-65 x^5-120 x^6+62 x^7+4 x^8-7 x^9+x^{10}) (-1-8 x+14 x^2+194 x^3+302 x^4+65 x^5-120 x^6-62 x^7+4 x^8+7
x^9+x^{10}))+$ \\ \\
$ \ds \frac{20 x (1+11 x^2)}{(-1+x) (1+x) (-1+11 x) (1+11 x)}+ \frac{4 x}{1-4 x^2}  $\bc $ \mbox{--------------------------------------    } $  \ec } \\
  \noindent
\parbox[t]{5.8in}{ ${\cal F}^{TG}_{10,4}(x) = {\cal F}^{TG}_{10,6}(x) = $ \\
 $ \ds
2 x (76-1394 x-238925 x^2+1712406 x^3+326860329 x^4-862982509 x^5-248411745750 x^6+598148325518 x^7+112441361628130 x^8-425819187134586
x^9-31635562487737893 x^{10}+169069561363732747 x^{11}+5783691389458258347 x^{12}-38091262241469735960 x^{13}-730642369070520564440 x^{14}+5327668007688979504782
x^{15}+70288034693143851128732 x^{16}-509546186173135832171883 x^{17}-5771491712553417562186101 x^{18}+37476694985740526040381412 x^{19}+426520647639540583956084441
x^{20}-2394527171196445018361375072 x^{21}-27402020733527879719989268251 x^{22}+141136895183253856958991863440 x^{23}+1450862846642585378107875741569
x^{24}-7444174317678966677583089116935 x^{25}-61683086368986780465239430705782 x^{26}+332198399812433110497570224897246 x^{27}+2101497319878504023127929342494062
x^{28}-12153254849021617265465509221458271 x^{29}-57912172911174653508468879074566763 x^{30}+361836547950576663098760835207121634 x^{31}+1308355368877520625301526566676180541
x^{32}-8804053584283913860310509427431464812 x^{33}-24590770259974391554844046897898644771 x^{34}+176434024844834939178014726972824636048 x^{35}+390287084413688337049722688100400508143
x^{36}-2935186122993845950713753256757764268399 x^{37}-5306696728085749119312546111212846451945 x^{38}+40810342651099464351201530222713836042606
x^{39}+62606471606168801679124998218538268952943 x^{40}- $ \\ $ 476738700918722768304386203447296161619916 x^{41}- $ \\ $ 646848537132358507328709256432653115721028
x^{42}+ $ \\ $ 4697350268576585848415782065392731630659506 x^{43}+ $ \\ $ 5878123961351749019488759929893251993494134 x^{44}- $ \\ $ 39143484028413651276167897193179764877714039
x^{45}- $ \\ $ 46933361068140037553087882726792512609720694 x^{46}+ $ \\ $ 276360871711679782276177858874509336982250173 x^{47}+ $ \\ $ 327592952368208058058088146654034419643522404
x^{48}- $ \\ $ 1655064448725583660687122054021505759823194011 x^{49}- $ \\ $ 1985215294592474362037047404736373394191841777 x^{50}+ $ \\ $ 8414819399798759720560063649514813552892229525
x^{51}+ $ \\ $ 10375260782719129180429583424945415362164619921 x^{52}- $ \\ $ 36350096990699504999959936062822761621670954072 x^{53}- $ \\ $ 46518462354728614824593740508832595146020775943
x^{54}+ $ \\ $ 133533393159687830120045743391792592660572254024 x^{55}+ $ \\ $ 178319885580967893892594122423718671080978292069 x^{56}- $ \\ $ 417635817334154005985422845341612823138229587688
x^{57}- $ \\ $ 583425601257625126505595124801107684320884357631 x^{58}+ $ \\ $ 1113684240256620739851068019122582421187927676336 x^{59}+ $ \\ $ 1628735690589529547953291272424292377586016996407
x^{60}- $ \\ $ 2536602330197671386302658849342221370873986788932 x^{61}- $ \\ $ 3882437915630040535920218874228791855792830975785 x^{62}+ $ \\ $ 4944971527505951275279982350658309133061036522519
x^{63}+ $ \\ $ 7913208682713515236429634503313428145676899492567 x^{64}- $ \\ $ 8269823681172962541385972174627064439473704021060 x^{65}- $ \\ $ 13815960884029107987010015833749181748861763839202
x^{66}+ $ \\ $ 11893897770309070649683730328407223225603955215904 x^{67}+ $ \\ $ 20704756341764314299375228398387445612641453212552 x^{68}- $ \\ $ 14748679350809389969308266947292317747796024252546
x^{69}- $ \\ $ 26688877241098479365608979092500557183344009914693 x^{70}+ $ \\ $ 15807268115015427684697139892919836865274053567057 x^{71}+ $ \\ $ 29652229323814775402516150246858150718351979024217
x^{72}- $ \\ $ 14675334568210423353994402813558829931174670818482 x^{73}- $ \\ $ 28451300686238204279380701744794806806882323641285 x^{74}+ $ \\ $ 11821034941255061396192636512044022909676194504702
x^{75}+ $ \\ $ 23618661404044406404798541598983651217489452136937 x^{76}- $ \\ $ 8267255590608315234052788237465396648388984326346 x^{77}- $ \\ $ 16992121061482092504276982101774320632220211093257
x^{78}+ $ \\ $ 5016389933986179051313427690359939827483300200858 x^{79}+ $ \\ $ 10611295309042114464879898796386974065677715729747 x^{80}-  $ } \\
\parbox[t]{5.8in}{ $
2633550703523053101113823181467305993344840449642
x^{81}- $ \\ $ 5761093827250612814897835052167120505360981310358 x^{82}+ $ \\ $ 1189419782155923876489992102023186439928475245605 x^{83}+ $ \\ $ 2723981960856611974747984431336012590583428956372
x^{84}- $ \\ $ 457493952713815687953531557425810062810172958509 x^{85}- $ \\ $ 1123967308448762999510235882985712485296419123970 x^{86}+ $ \\ $ 147269529489891609894138126193531231986872854316
x^{87}+ $ \\ $ 405755632936043957961653240435723137701647173442 x^{88}- $ \\ $ 38417552558907475986677891182403036871439639836 x^{89}- $ \\ $ 128563608983293213699169429875481149397754243467
x^{90}+ $ \\ $ 7563145550394346658488192781417045244836404514 x^{91}+ $ \\ $ 35884085865245716127244213136328680445926691459 x^{92}- $ \\ $ 882666413044803609991000460700186874654968951
x^{93}- $ \\ $ 8854096496047569190227693378871773065496863996 x^{94}- $ \\ $ 49266355189601144549968883561878755469658194 x^{95}+ $ \\ $ 1935483345476653378590186722481035716374714252
x^{96}+ $ \\ $ 54832009208749471094266164144249878293484376 x^{97}- $ \\ $ 374537294662737761444773263350566879864890579 x^{98}- $ \\ $ 15197877462053260192509989827651046875351759
x^{99}+ $ \\ $ 63827967357861685014207745625989528826307943 x^{100}+ $ \\ $ 2599514922475714022597387201846620111204908 x^{101}- $ \\ $ 9482984520097272492139509105890097184405693
x^{102}- $ \\ $ 279499137274335659814865699584689453690378 x^{103}+ $ \\ $ 1210630428324895124436659545625520837440287 x^{104}+ $ \\ $ 11414403382052752814491991871795511356328
x^{105}- $ \\ $ 130484285370157915640919297896303996709769 x^{106}+ $ \\ $ 2086752952950718685863717367609941712828 x^{107}+ $ \\ $ 11642555818084990833433691590209735778677
x^{108}- $ \\ $ 508083546585748888780571116785682580146 x^{109}-841974685513335569359381708398195541952 x^{110}+59956376536143910347860239305084744825 x^{111}+48245552099801569047183180378163613500
x^{112}-4695491493539315967387460794962554015 x^{113}-2137046734722587785930242557254552700 x^{114}+259349893739981348471330658408540600 x^{115}+71224314031485062607717954605759400
x^{116}-10193681207823594729100365539024400 x^{117}-1734972463112469632851268537303200 x^{118}+281724021719971135665504399056000 x^{119}+30032941647402517402033777104000
x^{120}-5341786411150276941306571480000 x^{121}-362732666908249569673396320000 x^{122}+67117115638337759921096000000 x^{123}+3078865576388639154720000000
x^{124}-534402916425858334864000000 x^{125}-18809994434686683648000000 x^{126}+2540231345489337600000000 x^{127}+77102666216724480000000 x^{128}-6426991539712000000000
x^{129}-160153693388800000000 x^{130}+5974777856000000000 x^{131}+143327232000000000 x^{132})/$ \\ $ ((-1+x) (1+x) (1+2 x) (-1+4 x) (1+4
x) (1-8 x+4 x^2) (1+5 x+5 x^2) (1+7 x+8 x^2) (1+18 x+63 x^2+50 x^3+13 x^4+x^5) (-1+34 x-160 x^2+248
x^3-152 x^4+32 x^5) (1+34 x+160 x^2+248 x^3+152 x^4+32 x^5) (-1+18 x-91 x^2+110 x^3+175 x^4-515 x^5+455 x^6-175 x^7+25 x^8)
(-1-18 x-91 x^2-110 x^3+175 x^4+515 x^5+455 x^6+175 x^7+25 x^8) (-1+121 x-3901 x^2+49003 x^3-293115 x^4+896942 x^5-1381585 x^6+970100
x^7-264304 x^8+15552 x^9) (-1+8 x+14 x^2-194 x^3+302 x^4-65 x^5-120 x^6+62 x^7+4 x^8-7 x^9+x^{10}) (-1-8 x+14 x^2+194 x^3+302
x^4+65 x^5-120 x^6-62 x^7+4 x^8+7 x^9+x^{10}) (1-36 x+504 x^2-3603 x^3+14416 x^4-33263 x^5+44489 x^6-34794 x^7+16069 x^8-4345 x^9+653
x^{10}-47 x^{11}+x^{12}) (1+53 x+951 x^2+7589 x^3+32774 x^4+84127 x^5+133734 x^6+133056 x^7+81936 x^8+30215 x^9+6265 x^{10}+650 x^{11}+25
x^{12}) (1+92 x+3048 x^2+50453 x^3+483576 x^4+2928888 x^5+11832412 x^6+32949126 x^7+64377394 x^8+88750186 x^9+85833700 x^{10}+57206921
x^{11}+25441290 x^{12}+7182529 x^{13}+1203428 x^{14}+111956 x^{15}+5340 x^{16}+120 x^{17}+x^{18}) (1-66 x+1833 x^2-28651 x^3+283174 x^4-1881547
x^5+8719340 x^6-28859906 x^7+69342834 x^8-122339404 x^9+159752386 x^{10}-155169581 x^{11}+112318552 x^{12}-60476870 x^{13}+24066800 x^{14}-6990095
x^{15}+1450676 x^{16}-207780 x^{17}+19365 x^{18}-1050 x^{19}+25 x^{20})) + $ }
\\
\parbox[t]{5.8in}{ $
10 x^2 (-517+676787 x^2+774317814 x^4-2419714399691 x^6+2564483665392264 x^8-1627769273521566449 x^{10}+715691075642214615863
x^{12}-233906029848566006481854 x^{14}+59313351407358200830787539 x^{16}-12013328363705595847784537069 x^{18}+1984349884196736665101768128147 x^{20}-271463193781361220249645795491528
x^{22}+31115900105736469009127132399005036 x^{24}-3014933108824719813854145977709201724 x^{26}+248640108435949101007074334182548536845 x^{28}-17546939243467420889827947318016441956491
x^{30}+ $ \\ $ 1064228669269128709654342038483829542056933 x^{32}- $ \\ $ 55663639427111526948746996248843075135705834 x^{34}+ $ \\ $ 2517670156066476423918462780149718178364412756
x^{36}- $ \\ $ 98674239084542259941612965935587414996253484214 x^{38}+ $ \\ $ 3355389702977370331088633569058611403520556514950 x^{40}- $ \\ $ 99028855474319474897458266938217942777364845505712
x^{42}+ $ \\ $ 2534182412216439254014009046352331358356861372119239 x^{44}- $ \\ $ 56064257828215371971529618834775015607207629223982885 x^{46}+ $ \\ $ 1065542465729307899206344732852213131176346924691696406
x^{48}- $ \\ $ 17179471589186357312301912827347484445450205448364639213 x^{50}+ $ \\ $ 228781292956823863871533607554331597212638823865025962512 x^{52}- $ \\ $ 2354397326393522209327470619321442417343364899214897758669
x^{54}+ $ \\ $ 14539969280002673502059996633145336114372231684718352090463 x^{56}+ $ \\ $ 63031428368363696051800695948107508580692826559424890331014 x^{58}- $ \\ $ 3609481034048246752102564165390454663110709570892611643192157
x^{60}+ $ \\ $ 66892295189246125747800725065410245287285736846372719868073431 x^{62}- $ \\ $ 890721530969234818961963195859532226310150733060254600073007594 x^{64}+ $ \\ $ 9615326987964775277396780233245607751411304632503726365754153107
x^{66}- $ \\ $ 87811561099147692718399662391819540560958837209147319780001788555 x^{68}+ $ \\ $ 692103978495189213254721898306828193738807454540671576516415098168
x^{70}- $ \\ $ 4758915738658339330332941039748917752113368762591419582664808924429 x^{72}+ $ \\ $ 28726958092288839503995360296910447771046246573265060286369714258737
x^{74}- $ \\ $ 152804853567875457564079247173840993454818797085033824662347599674526 x^{76}+ $ \\ $ 717737200775214988598847699882150518690757586134539908900291236725001
x^{78}- $ \\ $ 2979759548111045598827821158263085129987950557399984407134500477184602 x^{80}+ $ \\ $ 10933799801730220822313162735837648378557201697985073399843106380626761
x^{82}- $ \\ $ 35425400102155758904127878877252456577206157039628123555467147935227139 x^{84}+ $ \\ $ 101140301640890902315796938026303539825737895434982874835711332669449506
x^{86}- $ \\ $ 253560622552322329431496562726098658039458858257099400575486998786665135 x^{88}+ $ \\ $ 555002161803252600721471908466808338028784660227855535132110020881893243
x^{90}- $ \\ $ 1050405451878301255709545442953843166523830296707434832022116868411211567 x^{92}+ $ \\ $ 1688945935033053153294708365108480074990825589713020133901434589832402584
x^{94}- $ \\ $ 2223928698281681848620119804409356220774990171203421486842984677012308772 x^{96}+ $ \\ $ 2173206655352974054019773238297536609204426630698877852326153102985170408
x^{98}- $ \\ $ 948526227314467419228209777792585361771448273778582244132509903987249153 x^{100}- $ \\ $ 1819405291715479863854490605473807612283378137526808434984763180722069553
x^{102}+ $ \\ $ 5914900965923269044217147939270892463829271258748985705262973630118104872 x^{104}- $ \\ $ 10416690207964498527299758498674794734734889989423408235821808303584525753
x^{106}+ $ \\ $ 14010742070784830839481950415107870696148334621764385925222081862148541737 x^{108}- $ \\ $ 15609103795965286858771307668622258361517326240758088646934974906536593234
x^{110}+ $ \\ $ 14879449869511137596437078740188048716292007756804706137729405940739185701 x^{112}- $ \\ $ 12336935018847273215707626891668537571697475808797879475396011518789695457
x^{114}+ $ \\ $ 8981607212369400548331996086614247140399127470468225835655293202105694594 x^{116}- $ \\ $ 5775985174128641122660870490751668698711117708599584018189690236860073161
x^{118}+ $ \\ $ 3294327006800819725815703239752429419383940363053123768811695637334584765 x^{120}- $ \\ $ 1671064785123669497559692016090338632775473625887639662338486759447398868
x^{122}+ $ \\ $ 755397565263833758818117414273373006675509005695508881680921738616024592 x^{124}- $ \\ $ 304746405413598009974214883288266050125946011080100317086397887414187270
x^{126}+ $ \\ $ 109830824435795266668488877080951242412566422994203326560147320358159074 x^{128}- $ \\ $ 35385817942772152173244389799694846384093263005738301338599534955876512
x^{130}+  $ } \\
\parbox[t]{5.8in}{ $10195963470962056542321130168809353773125912823984231841369144407736431 x^{132}- $ \\ $ 2627772217304815569409509481934061109329211077260681625563310472475007
x^{134}+ $ \\ $ 605721231548127516355254034749868517133632187599075698752570592375403 x^{136}- $ \\ $ 124838923271262970341354668239667848405296684020390995127992207020396
x^{138}+ $ \\ $ 22992489661714434573877468047125611931742317752726782297564561184472 x^{140}- $ \\ $ 3781320349645477844159059688171641456843873628540698757620280649226
x^{142}+ $ \\ $ 554721178934607551355764084012846112562389887364988872441121202072 x^{144}- $ \\ $ 72495830600074202591519099804026189802278176043187127398527899758
x^{146}+ $ \\ $ 8426777421654188849524142099987110840932439370522550822675376121 x^{148}- $ \\ $ 869521307107100566353899390036614754720031483886138357021802307
x^{150}+ $ \\ $ 79463941688496938462022923368185041153855619840008314567766514 x^{152}- $ \\ $ 6414369882222424872323948455688066106533776281070574928462367 x^{154}+ $ \\ $ 455883144081545801239127682848756077407859558760079770147511
x^{156}- $ \\ $ 28422396251765368521950078941036612211682050372535075968090 x^{158}+ $ \\ $ 1547755864730391360029960397861110841405620564010689580640 x^{160}- $ \\ $ 73248084749928349307293551920039573229339009685655937550
x^{162}+ $ \\ $ 2994974405577175271483623843522625474634959200032957100 x^{164}- $ \\ $ 105076126893820753837645834960442509335769712953130500 x^{166}+ $ \\ $ 3137649227238363072797095736357465767042554172967125
x^{168}- $ \\ $ 78978181583617961146235033839521487328507712751875 x^{170}+ $ \\ $ 1656510163756245191787512797504180508693473846875 x^{172}- $ \\ $ 28548758185185238788280858041609023773725781250
x^{174}+ $ \\ $ 397404504169696759271639127292496611097512500 x^{176}- $ \\ $ 4373743081799402562302972385084301879000000 x^{178}+ $ \\ $ 37040848022637890072412679239613365000000
x^{180}- $ \\ $ 233018353643509938568427080679500000000 x^{182}+ $ \\ $ 1037991667026054240268223080000000000 x^{184}- 3055249507560388630652960000000000 x^{186}+5329343649197877196800000000000
x^{188}-4581595294645248000000000000 x^{190}+1451262771200000000000000 x^{192})/$ \\ $((-1+x) (1+x) (-1+2 x) (1+2 x) (-1+11 x) (1+11
x) (1-3 x+x^2) (1+3 x+x^2) (1-13 x+11 x^2) (1-8 x+11 x^2) (1+8 x+11 x^2) (1+13 x+11 x^2)
(1-12 x^2+8 x^3) (-1+12 x^2+8 x^3) (-1+104 x-2661 x^2+24090 x^3-91993 x^4+158236 x^5-121128 x^6+44736 x^7-7808 x^8+512
x^9) (1+104 x+2661 x^2+24090 x^3+91993 x^4+158236 x^5+121128 x^6+44736 x^7+7808 x^8+512 x^9) (-1+5 x+45 x^2-305 x^3+65 x^4+2175
x^5-1800 x^6-4525 x^7+1975 x^8+2500 x^9-875 x^{10}-375 x^{11}+125 x^{12}) (-1-5 x+45 x^2+305 x^3+65 x^4-2175 x^5-1800 x^6+4525 x^7+1975
x^8-2500 x^9-875 x^{10}+375 x^{11}+125 x^{12}) (1-22 x-214 x^2+4230 x^3-15001 x^4-2783 x^5+60858 x^6-12548 x^7-46005 x^8+16340 x^9+10377
x^{10}-5171 x^{11}-329 x^{12}+413 x^{13}-27 x^{14}-9 x^{15}+x^{16}) (1+22 x-214 x^2-4230 x^3-15001 x^4+2783 x^5+60858 x^6+12548 x^7-46005
x^8-16340 x^9+10377 x^{10}+5171 x^{11}-329 x^{12}-413 x^{13}-27 x^{14}+9 x^{15}+x^{16}) (1-82 x+2898 x^2-59227 x^3+789803 x^4-7334275
x^5+49339066 x^6-246611784 x^7+930982131 x^8-2681675595 x^9+5927069364 x^{10}-10071216252 x^{11}+13140900495 x^{12}-13114381556 x^{13}+9943609371
x^{14}-5673467839 x^{15}+2404556811 x^{16}-744137999 x^{17}+164387022 x^{18}-25156085 x^{19}+2562619 x^{20}-164923 x^{21}+6260 x^{22}-125 x^{23}+x^{24})
(1+82 x+2898 x^2+59227 x^3+789803 x^4+7334275 x^5+49339066 x^6+246611784 x^7+930982131 x^8+2681675595 x^9+5927069364 x^{10}+10071216252 x^{11}+13140900495
x^{12}+13114381556 x^{13}+9943609371 x^{14}+5673467839 x^{15}+2404556811 x^{16}+744137999 x^{17}+164387022 x^{18}+25156085 x^{19}+2562619 x^{20}+164923
x^{21}+6260 x^{22}+125 x^{23}+x^{24}) (1-104 x+4410 x^2-104155 x^3+1561045 x^4-15952817 x^5+116275468 x^6-623060570 x^7+2507306175 x^8-7693787635
x^9+18199448138 x^{10}-33436393552 x^{11}+47933912805 x^{12}-53730741510 x^{13}+47072037995 x^{14}-32132854819 x^{15}+16992690151 x^{16}-6899340865
x^{17}+2123302560 x^{18}-486469225 x^{19}+80870375 x^{20}-9387875 x^{21}+715250 x^{22}-31875 x^{23}+625 x^{24}) (1+104 x+4410 x^2+104155
x^3+1561045 x^4+15952817 x^5+116275468 x^6+623060570 x^7+2507306175 x^8+7693787635 x^9+18199448138 x^{10}+33436393552 x^{11}+47933912805 x^{12}+53730741510
x^{13}+47072037995 x^{14}+32132854819 x^{15}+16992690151 x^{16}+6899340865 x^{17}+2123302560 x^{18}+486469225 x^{19}+80870375 x^{20}+9387875 x^{21}+715250
x^{22}+31875 x^{23}+625 x^{24})) + $ }\\
\parbox[t]{5.8in}{ $
20 x^2 (-200+227916 x^2-143551403 x^4+59211315220 x^6-16142031622757 x^8+2969830066234826 x^{10}-382316360665224398 x^{12}+35837872496861650216
x^{14}-2540609624382851662222 x^{16}+140719723738727589263020 x^{18}-6240897926592956227508870 x^{20}+225110747310805357924668587 x^{22}-6660003556830916437940838306
x^{24}+162316996807451927964671963118 x^{26}-3267946859388856477882995178581 x^{28}+54493764903682874405492471627164 x^{30}-754789802562258606954295438398849
x^{32}+8710176876759175625256794912961703 x^{34}-83989518933713520276164333277158768 x^{36}+678505684233084242086224307311267018 x^{38}-4601764678039844186832082418971747224
x^{40}+26240719568124896297495730614578776603 x^{42}-125903131018368341627896237780865814549 x^{44}+508281831518302655864553118546026906843 x^{46}-1725160967673831404953340074839046824676
x^{48}+4914668126336719731191653826957898875950 x^{50}-11721438310525470355701848796246166224313 x^{52}+23318763826921005472092418225977577162806
x^{54}-38501353487038504567462207122619591221944 x^{56}+52383095686232443740810703571776810527083 x^{58}- $ \\ $ 58099509084138454101188067620615958952530
x^{60}+ $ \\ $ 51588786896859944335736004046101956105347 x^{62}- $ \\ $ 35381369534232697774012939619635963807182 x^{64}+ $ \\ $ 17086474080826697413137462259606256611900
x^{66}- $ \\ $ 3732276831071734850539593958403700936191 x^{68}- $ \\ $ 2428464355862983472585526141316516903263 x^{70}+ $ \\ $ 3402759123520154411139768160944204589088 x^{72}-2304421600133270558446524057938225180967
x^{74}+1110964046270048178462688556794423253747 x^{76}-415574896533491847214751261049715583078 x^{78}+124686048231164214848295403120327799844 x^{80}-30477610844808149564938260498044727792
x^{82}+6116007349863691739549892718238204983 x^{84}-1010800363747704400662293867440045524 x^{86}+137603507223062562818735957731904073 x^{88}-15393297838937611452367754688486335
x^{90}+1408354506935192832661126658966870 x^{92}-104616848382664188424922025675350 x^{94}+6244804624327547674469167153425 x^{96}-295326823427544973498457436875
x^{98}+10851627741004024944060042625 x^{100}-301485415137407502901657500 x^{102}+6087466710598138314937500 x^{104}-84022542585257538800000 x^{106}+713242643294985600000
x^{108}-2979936064864000000 x^{110}+2771248480000000 x^{112})/$ \\ $ ((-1+x) (1+x) (-1+2 x) (1+2 x) (-1+11 x) (1+11 x) (-1+9 x-6
x^2+x^3) (1+9 x+6 x^2+x^3) (1-12 x^2+8 x^3) (-1+12 x^2+8 x^3) (1-59 x+402 x^2-863 x^3+572 x^4)
(1+59 x+402 x^2+863 x^3+572 x^4) (1-18 x+93 x^2-195 x^3+164 x^4-37 x^5+x^6) (1+18 x+93 x^2+195 x^3+164 x^4+37 x^5+x^6)
(1-33 x+375 x^2-1922 x^3+5192 x^4-7883 x^5+6797 x^6-3252 x^7+834 x^8-105 x^9+5 x^{10}) (1+33 x+375 x^2+1922 x^3+5192 x^4+7883 x^5+6797
x^6+3252 x^7+834 x^8+105 x^9+5 x^{10}) (-1+5 x+45 x^2-305 x^3+65 x^4+2175 x^5-1800 x^6-4525 x^7+1975 x^8+2500 x^9-875 x^{10}-375 x^{11}+125
x^{12}) (-1-5 x+45 x^2+305 x^3+65 x^4-2175 x^5-1800 x^6+4525 x^7+1975 x^8-2500 x^9-875 x^{10}+375 x^{11}+125 x^{12}) (1-22
x-214 x^2+4230 x^3-15001 x^4-2783 x^5+60858 x^6-12548 x^7-46005 x^8+16340 x^9+10377 x^{10}-5171 x^{11}-329 x^{12}+413 x^{13}-27 x^{14}-9 x^{15}+x^{16})
(1+22 x-214 x^2-4230 x^3-15001 x^4+2783 x^5+60858 x^6+12548 x^7-46005 x^8-16340 x^9+10377 x^{10}+5171 x^{11}-329 x^{12}-413 x^{13}-27 x^{14}+9
x^{15}+x^{16}))$ \\ \\
$10 x^2 (141-20317 x^2+1769412 x^4-106602114 x^6+4065572659 x^8-95245006503 x^{10}+1370793580650 x^{12}-12092519738638 x^{14}+65574984909135
x^{16}-226307435369999 x^{18}+515543092515601 x^{20}-794633284937229 x^{22}+841939662220809 x^{24}-617434377175313 x^{26}+310358521797465 x^{28}-101785469424327
x^{30}+17785066908763 x^{32}+757211048985 x^{34}-1348937563312 x^{36}+382686538032 x^{38}-59462965440 x^{40}+5491199040 x^{42}-281600000 x^{44}+6144000
x^{46})/$ \\ $ ((-1+4 x) (1+4 x) (-1+34 x-160 x^2+248 x^3-152 x^4+32 x^5) (1+34 x+160 x^2+248 x^3+152 x^4+32 x^5)
(-1+18 x-91 x^2+110 x^3+175 x^4-515 x^5+455 x^6-175 x^7+25 x^8) (-1-18 x-91 x^2-110 x^3+175 x^4+515 x^5+455 x^6+175 x^7+25 x^8)
(-1+8 x+14 x^2-194 x^3+302 x^4-65 x^5-120 x^6+62 x^7+4 x^8-7 x^9+x^{10}) (-1-8 x+14 x^2+194 x^3+302 x^4+65 x^5-120 x^6-62 x^7+4 x^8+7
x^9+x^{10}))+ $ \\ \\ $\ds
\frac{240 x^2}{(-1+x) (1+x) (-1+11 x) (1+11 x)} + \frac{8 x^2}{1-4 x^2}$
 \bc $ \mbox{--------------------------------------    } $  \ec } \\
\parbox[t]{5.8in}{ ${\cal F}^{TG}_{10,5}(x) = $ \\
 $ \ds
-2 x (-51-2531 x+346545 x^2+6281569 x^3-626107154 x^4-6720852191 x^5+557692878360 x^6+3924628100007 x^7-303896690325215
x^8-1203141163993754 x^9+110385064013034818 x^{10}+129738487704458173 x^{11}-27731403108413092267 x^{12}+34160159071697768285 x^{13}+4912263037693206222860
x^{14}-15954779864093857882007 x^{15}-623281892861255956436402 x^{16}+3097145780008994257521933 x^{17}+57632783639522368138374201 x^{18}-377934991774510513575939182
x^{19}-3952158907346024732403779791 x^{20}+32460308118041603852559150442 x^{21}+203786575810957015978655619471 x^{22}-2077813633463470867265288726055
x^{23}-7948961501060129247409085221049 x^{24}+102802217882332666430474733354935 x^{25}+232622429206461224171013982875562 x^{26}-4034402051064428138538406384326721
x^{27}-4900682055953589557666106815753572 x^{28}+128060859092710534391335093622529141 x^{29}+63051543270494126746830832945671583 x^{30}-3338490507604724111102033951338547814
x^{31}+26788511188693114060249776800215459 x^{32}+72356721471260820357813632619322970952 x^{33}-25133888076030218523548415691876946199 x^{34}-1316599417966012079962043281074495545073
x^{35}+742495978080795122963703947179610357632 x^{36}+20269576331788360890752199000142043677074 x^{37}-13862824906131506810766699016256900960915
x^{38}-265607493352496242359856362940910711915736 x^{39}+193104389804418555043447324391809895693127 x^{40}+ $ \\ $ 2975216590177302257886759045720008372793261
x^{41}- $ \\ $ 2111817918355102371185068791194960721452202 x^{42}- $ \\ $ 28570645331934163623948593487911673599500336 x^{43}+ $ \\ $ 18487694709094661432907887261184552837124826
x^{44}+ $ \\ $ 235581468114998519645838461855224488305073809 x^{45}- $ \\ $ 130156394689861484366232451128499870990468851 x^{46}- $ \\ $ 1668989256110220098195946316866546797554529153
x^{47}+ $ \\ $ 731417562602339275971261625262436986942402381 x^{48}+ $ \\ $ 10158855568490050874446218127772793907793908546 x^{49}- $ \\ $ 3198253249786663176802047271684130146636834688
x^{50}- $ \\ $ 53108779145386981783564979650152383616566550465 x^{51}+ $ \\ $ 10128033086120794207566571571046105705526451654 x^{52}+ $ \\ $ 238382059576456471861910521774034716007311766422
x^{53}- $ \\ $ 17221109780482571116786192851239191587398123757 x^{54}- $ \\ $ 918635161824208359247552665162981432872785417134 x^{55}- $ \\ $ 33195366068704491758994589883688270853002940664
x^{56}+ $ \\ $ 3040577629649310776598270380203292039919688800658 x^{57}+ $ \\ $ 400196059547950418222174462590664705515435620391 x^{58}- $ \\ $ 8652537823786096071652322353673649787183678251881
x^{59}- $ \\ $ 1839752907314622571026734469523168083989405111807 x^{60}+ $ \\ $ 21202716809862237156822487730685535276278991654367 x^{61}+ $ \\ $ 5903523420339799260371399837261110215036855209930
x^{62}- $ \\ $ 44834552043335708954505705972482347114055469684559 x^{63}- $ \\ $ 14744682581002030836899333703692875295653926228962 x^{64}+ $ \\ $ 82014457953943649674965051178817460289536021215555
x^{65}+ $ \\ $ 29863060596899002266660546446251805386332824925422 x^{66}- $ \\ $ 130137466567016964377506206069967536159477001537059 x^{67}- $ \\ $ 50053627002770222166627376101885876440432458356957
x^{68}+ $ \\ $ 179612493443689814186527356422746942236647246072246 x^{69}+ $ \\ $ 70235891648446344107801479915893386405663885576028 x^{70}- $ \\ $ 216173911972664059606906981717451566757458755155447
x^{71}- $ \\ $ 83076701121466780717363057720697280259747943692692 x^{72}+ $ \\ $ 227380374071409807308921118593306719565945503226892 x^{73}+ $ \\ $ 83137577281686798107796684139729872757623656114015
x^{74}- $ \\ $ 209365843407986573894358246160972068370712087311012 x^{75}- $ \\ $ 70461316925701078941006771908989977332975544991382 x^{76}+ $ \\ $ 168926823594511366682206960322103657949050819207126
x^{77}+ $ \\ $ 50480700674821294557379787438558260881230116262837 x^{78}- $ \\ $ 119467885941242496991002033696268682672607206171073 x^{79}- $ \\ $ 30396677030912213570439085714643100882979411593462
x^{80}+ $ } \\ \parbox[t]{5.8in}{  $ 74020463837949215696205173175300374984605142344152 x^{81}+ $ \\ $ 15197881227364908605675007638412617172581757209863 x^{82}- $ \\ $ 40128862392750652022136897009811599005073308003030
x^{83}- $ \\ $ 6152878805987423805276211517166965970667420375242 x^{84}+ $ \\ $ 18997536394461918549664063346870105159076338828474 x^{85}+ $ \\ $ 1900388850847166181951675437482252296274290069450
x^{86}- $ \\ $ 7832191839678246919734677017867495833467141242296 x^{87}- $ \\ $ 365399651689524059728343699524403419774914588397 x^{88}+ $ \\ $ 2802193964815597934939014943380747326366064960001
x^{89}- $ \\ $ 16835332987879145552417242317924585488588957143 x^{90}- $ \\ $ 866308406412822132399069201881506931370442990979 x^{91}+ $ \\ $ 49888904500015272960502242118186671756572892776
x^{92}+ $ \\ $ 230212846577582644227486528608136047856559439351 x^{93}- $ \\ $ 24906031147133352566492345138371357002782063214 x^{94}- $ \\ $ 52250179636338263131907344309120893625441147691
x^{95}+ $ \\ $ 8277699145310025749463885476605039347697929953 x^{96}+ $ \\ $ 10048398217541402857728024267494891772275023944 x^{97}- $ \\ $ 2105984461825530959671151172262247224745299426
x^{98}- $ \\ $ 1620900014269096181173400651490200441434643381 x^{99}+ $ \\ $ 427865890582649419864944018613027239279086942 x^{100}+ $ \\ $ 216372606893901948914468519952550532201294652
x^{101}- $ \\ $ 70480627023716470610346842862468964051941617 x^{102}- $ \\ $ 23445919611481704609414636889112228678067462 x^{103}+ $ \\ $ 9450269499593860084876826613561231502792428
x^{104}+ $ \\ $ 1999939004828987307604697849230729911169432 x^{105}- $ \\ $ 1029122034653812682572137628192624648134651 x^{106}- $ \\ $ 126594218202456922523147996987918621162003
x^{107}+ $ \\ $ 90440776236817384538157776269938860389718 x^{108}+ $ \\ $ 5052589819186147732132437936795903137926 x^{109}-6351553952451503074249343839812369826573
x^{110}-22664819101711128318447084542557984040 x^{111}+351830496887428666123569104723733957290 x^{112}-13435814835754352840124128090090275660 x^{113}-15117410702244283319581825228861246020
x^{114}+1094813438728454165994918549559584000 x^{115}+493311516515768181121201093952770000 x^{116}-50036189966890691796258300544529600 x^{117}-11896752290285081874041038063833200
x^{118}+1512543300338047970724655575772000 x^{119}+204555449167210073548457376448000 x^{120}-30942312932846690051739107360000 x^{121}-2391993570434141288613123280000
x^{122}+418664918044319757870120800000 x^{123}+18006370964198364023608000000 x^{124}-3538686190756713370448000000 x^{125}-85470368056527605120000000
x^{126}+  $ \\ $ 16891388123502675200000000 x^{127}+293197981277634560000000 x^{128}-38145542748876800000000 x^{129}- $ \\ $ 730539005542400000000 x^{130}+25023176704000000000
x^{131}+525533184000000000 x^{132})/$ \\ $((-1+x) (1+x) (1+2 x) (-1+4 x) (1+4 x) (1-8 x+4 x^2) (1+5 x+5 x^2) (1+7
x+8 x^2) (1+18 x+63 x^2+50 x^3+13 x^4+x^5) (-1+34 x-160 x^2+248 x^3-152 x^4+32 x^5) (1+34 x+160 x^2+248 x^3+152
x^4+32 x^5) (-1+18 x-91 x^2+110 x^3+175 x^4-515 x^5+455 x^6-175 x^7+25 x^8) (-1-18 x-91 x^2-110 x^3+175 x^4+515 x^5+455 x^6+175
x^7+25 x^8) (-1+121 x-3901 x^2+49003 x^3-293115 x^4+896942 x^5-1381585 x^6+970100 x^7-264304 x^8+15552 x^9) (-1+8 x+14 x^2-194
x^3+302 x^4-65 x^5-120 x^6+62 x^7+4 x^8-7 x^9+x^{10}) (-1-8 x+14 x^2+194 x^3+302 x^4+65 x^5-120 x^6-62 x^7+4 x^8+7 x^9+x^{10}) (1-36
x+504 x^2-3603 x^3+14416 x^4-33263 x^5+44489 x^6-34794 x^7+16069 x^8-4345 x^9+653 x^{10}-47 x^{11}+x^{12}) (1+53 x+951 x^2+7589 x^3+32774
x^4+84127 x^5+133734 x^6+133056 x^7+81936 x^8+30215 x^9+6265 x^{10}+650 x^{11}+25 x^{12}) (1+92 x+3048 x^2+50453 x^3+483576 x^4+2928888
x^5+11832412 x^6+32949126 x^7+64377394 x^8+88750186 x^9+85833700 x^{10}+57206921 x^{11}+25441290 x^{12}+7182529 x^{13}+1203428 x^{14}+111956 x^{15}+5340
x^{16}+120 x^{17}+x^{18}) (1-66 x+1833 x^2-28651 x^3+283174 x^4-1881547 x^5+8719340 x^6-28859906 x^7+69342834 x^8-122339404 x^9+159752386
x^{10}-155169581 x^{11}+112318552 x^{12}-60476870 x^{13}+24066800 x^{14}-6990095 x^{15}+1450676 x^{16}-207780 x^{17}+19365 x^{18}-1050 x^{19}+25
x^{20}))
   $ } \\
   \parbox[t]{5.8in}{$ \ds
   -20 x (2+5659 x^2-14484607 x^4+1623166223 x^6+19612453184950 x^8-25320534081702041 x^{10}+17091892743121647834 x^{12}-7655837232974703848604
x^{14}+2492277605569640395833581 x^{16}-620099827786773705171611869 x^{18}+121885949780878818966486500169 x^{20}-19384309001900158601675819590310
x^{22}+2540186001434970800443884399355721 x^{24}-278235840580135809493641738709226336 x^{26}+25765384988114400513290639323578465316 x^{28}-2035553668877659915082716579609632578913
x^{30}+ $ \\ $ 138196312427301076159723996151208851675079 x^{32}- $ \\ $ 8109230896031867918266410848377817689640924 x^{34}+ $ \\ $ 413161160550946548074386128305818536822648499
x^{36}- $ \\ $ 18344103322200841094903263044635009064943799154 x^{38}+ $ \\ $ 711816904534000358494914943635275294850195464896 x^{40}- $ \\ $ 24195185649666128001631263025120708673126477249290
x^{42}+ $ \\ $ 721678280050024769218969825405647420558142476056042 x^{44}- $ \\ $ 18912930409545443103034473780503465060418810796183341 x^{46}+ $ \\ $ 435799785917531445698829869799646400409340297457680583
x^{48}- $ \\ $ 8830009741753543163707469435200664569445824112203405647 x^{50}+ $ \\ $ 157198076496313230975646610274296281597109849037468640802 x^{52}- $ \\ $ 2454070354971455070662305063780743901748713735246305422475
x^{54}+ $ \\ $ 33466538141352151499532582791125094139455921443680209782438 x^{56}- $ \\ $ 395888331472760042299157128136911216502274641086632905513340 x^{58}+ $ \\ $ 4009750831589993497127268562617739272447604308340131540444385
x^{60}- $ \\ $ 33870841213004933871120926511210614898252679704449832033744953 x^{62}+ $ \\ $ 223949855657319688350878425976362530338378085000502022678755621 x^{64}- $ \\ $ 921462875839575071588591369463043652525707048884121111528853893
x^{66}- $ \\ $ 1826548520005609845984665795401945905775154499722805743624353842 x^{68}+ $ \\ $ 83758280051599460787091479473116770318079299843494501711550277584
x^{70}- $ \\ $ 1027285931910734858099310457561751998235691256560572904464524259553 x^{72}+ $ \\ $ 8967259351220872011077153767668367595881327758325251306097559852159
x^{74}- $ \\ $ 63440246962823249673909165117689191400102893653392948841586588496283 x^{76}+ $ \\ $ 380384157549626783130849524646317509227080629068121575223467875015497
x^{78}- $ \\ $ 1973537580869792683021248989352069132562510658842050140032004708090884 x^{80}+ $ \\ $ 8960729858891577905811008572630611446854911253463679456211189028525753
x^{82}- $ \\ $ 35848741085864395698132749891629191767962513604165732992920437683415716 x^{84}+ $ \\ $ 126925097405552838886235621891701363852352059062130769008765810973825210
x^{86}- $ \\ $ 398901092418377304241414075486889280864766490675190442602637215961593763 x^{88}+ $ \\ $ 1115199542594761854173624748910849443271658042502312675256644724618461325
x^{90}- $ \\ $ 2777758819899025882861031133251130097920398385094901405394695927895270173 x^{92}+ $ \\ $ 6171887585950640536233308481583175532091448408576743626905910740115368920
x^{94}- $ \\ $ 12244596800927822083598399274736287470567787812362881446247930723254806073 x^{96}+ $ \\ $ 21708232524582362102676634194012784749853509218702418772935430042276040782
x^{98}- $ \\ $ 34415982695758774633026344840956657057525423661143543963251323713314712788 x^{100}+ $ \\ $ 48822557794320119578058943135929592376719383814702632017318845864117288957
x^{102}- $ \\ $ 62008403932388029294187629828840491036045386323342349790841473402349574011 x^{104}+ $ \\ $ 70546739103255702476981855502474574737702669144117262038726367295574409817
x^{106}- $ \\ $ 71930122894272438404689062583654889655419194490842537760804817565863844620 x^{108}+ $ \\ $ 65757961617815612723189915581298198503277898335401690937775529211300226786
x^{110}- $ \\ $ 53922515483442644598178634993747162301831435258703657746483489403389384211 x^{112}+ $ \\ $ 39677049034669816206409587991332018712017565179790995825177198211744995613
x^{114}- $ \\ $ 26205908355603413844990105210820784914073418978699979154361605255412437523 x^{116}+ $ \\ $ 15540677047625759529182462703868548821497238359929159199774257668121505601
x^{118}- $ \\ $ 8276473889891497546477408617044943555339968243037682846407369535837002808 x^{120}+ $ \\ $ 3959033754799320453893164823280523335270054179615956238265142793264080238
x^{122}- $ \\ $ 1701105518674175574551493647766236493409926136158119733621370896727046445 x^{124}+ $ \\ $ 656542990989592417348625741487879377319426258958781570243093578803916784
x^{126}- $ \\ $ 227580103777549506068084641880256490019659457087714061253774872963800848 x^{128}+ $ \\ $ 70834884842512314536778258043118786548266175539893877638807925806556672
x^{130}- $} \\
  \parbox[t]{5.8in}{
 $ 19790330194890916651534968542711581210369213927725727755747964390504682 x^{132}+ $ \\ $ 4960742544385660758050409066252175516085005999980936006227765109072555
x^{134}- $ \\ $ 1114965408252275546691711366971472827850362369177808569751457148704719 x^{136}+ $ \\ $ 224524125026871302328965147855176330086721432772721002344692680641782
x^{138}- $ \\ $ 40470954750979642753538174797600193341865930262582203299880325752531 x^{140}+ $ \\ $ 6522484527368778953449973284479582692866856074217892707007561374018
x^{142}- $ \\ $ 938623824921590972669146517068697112311254760684888308613488469202 x^{144}+ $ \\ $ 120419567092177406591038675265056473432666809078090186644227888884
x^{146}- $ \\ $ 13747766293669092745250547824255227917032348835795909324767440930 x^{148}+ $ \\ $ 1393691287672287850081556830889435116855420688327049013385687843
x^{150}- $ \\ $ 125145452427006640167902103013205097232068982381428059224080893 x^{152}+ $ \\ $ 9924495724138899418787178708185784388161228130689225107867703 x^{154}- $ \\ $ 692717654611336590835950467233870863314381823699368983990018
x^{156}+ $ \\ $ 42383548518450024305013765537116171849144635477753961710730 x^{158}- $ \\ $ 2262213418454639456479373460003381347256055513386909318425 x^{160}+ $ \\ $ 104722440678757379322071746661202509907487401907861926350
x^{162}- $ \\ $ 4174807247219310825397505396455503773664774258173953600 x^{164}+ $ \\ $ 142070795170315335896284229860216232653265539388845500 x^{166}- $ \\ $ 4081335488728972215084901266311623117500995827927500
x^{168}+ $ \\ $ 97542596542325886160901390409841083590895138256875 x^{170}- $ \\ $ 1901115576372506933834356905953362581345663705625 x^{172}+ $ \\ $ 29342359742088281927166003690771020937903218750
x^{174}- $ \\ $ 341588812305336986413471889238776928586228125 x^{176}+ $ \\ $ 2709785456016506931219420585932164975125000 x^{178}- $ \\ $ 10151850504867281490674504538905811250000
x^{180}- $ \\ $ 52202370660679192022575858760375000000 x^{182}+972298069735283950652233953000000000 x^{184}-5928705592980063547455520000000000 x^{186}+17784983723555471913600000000000
x^{188}-24251168563898496000000000000 x^{190}+10543094200320000000000000 x^{192})/$ \\ $((-1+x) (1+x) (-1+2 x) (1+2 x) (-1+11 x) (1+11
x) (1-3 x+x^2) (1+3 x+x^2) (1-13 x+11 x^2) (1-8 x+11 x^2) (1+8 x+11 x^2) (1+13 x+11 x^2)
(1-12 x^2+8 x^3) (-1+12 x^2+8 x^3) (-1+104 x-2661 x^2+24090 x^3-91993 x^4+158236 x^5-121128 x^6+44736 x^7-7808 x^8+512
x^9) (1+104 x+2661 x^2+24090 x^3+91993 x^4+158236 x^5+121128 x^6+44736 x^7+7808 x^8+512 x^9) (-1+5 x+45 x^2-305 x^3+65 x^4+2175
x^5-1800 x^6-4525 x^7+1975 x^8+2500 x^9-875 x^{10}-375 x^{11}+125 x^{12}) (-1-5 x+45 x^2+305 x^3+65 x^4-2175 x^5-1800 x^6+4525 x^7+1975
x^8-2500 x^9-875 x^{10}+375 x^{11}+125 x^{12}) (1-22 x-214 x^2+4230 x^3-15001 x^4-2783 x^5+60858 x^6-12548 x^7-46005 x^8+16340 x^9+10377
x^{10}-5171 x^{11}-329 x^{12}+413 x^{13}-27 x^{14}-9 x^{15}+x^{16}) (1+22 x-214 x^2-4230 x^3-15001 x^4+2783 x^5+60858 x^6+12548 x^7-46005
x^8-16340 x^9+10377 x^{10}+5171 x^{11}-329 x^{12}-413 x^{13}-27 x^{14}+9 x^{15}+x^{16}) (1-82 x+2898 x^2-59227 x^3+789803 x^4-7334275
x^5+49339066 x^6-246611784 x^7+930982131 x^8-2681675595 x^9+5927069364 x^{10}-10071216252 x^{11}+13140900495 x^{12}-13114381556 x^{13}+9943609371
x^{14}-5673467839 x^{15}+2404556811 x^{16}-744137999 x^{17}+164387022 x^{18}-25156085 x^{19}+2562619 x^{20}-164923 x^{21}+6260 x^{22}-125 x^{23}+x^{24})
(1+82 x+2898 x^2+59227 x^3+789803 x^4+7334275 x^5+49339066 x^6+246611784 x^7+930982131 x^8+2681675595 x^9+5927069364 x^{10}+10071216252 x^{11}+13140900495
x^{12}+13114381556 x^{13}+9943609371 x^{14}+5673467839 x^{15}+2404556811 x^{16}+744137999 x^{17}+164387022 x^{18}+25156085 x^{19}+2562619 x^{20}+164923
x^{21}+6260 x^{22}+125 x^{23}+x^{24}) (1-104 x+4410 x^2-104155 x^3+1561045 x^4-15952817 x^5+116275468 x^6-623060570 x^7+2507306175 x^8-7693787635
x^9+18199448138 x^{10}-33436393552 x^{11}+47933912805 x^{12}-53730741510 x^{13}+47072037995 x^{14}-32132854819 x^{15}+16992690151 x^{16}-6899340865
x^{17}+2123302560 x^{18}-486469225 x^{19}+80870375 x^{20}-9387875 x^{21}+715250 x^{22}-31875 x^{23}+625 x^{24}) (1+104 x+4410 x^2+104155
x^3+1561045 x^4+15952817 x^5+116275468 x^6+623060570 x^7+2507306175 x^8+7693787635 x^9+18199448138 x^{10}+33436393552 x^{11}+47933912805 x^{12}+53730741510
x^{13}+47072037995 x^{14}+32132854819 x^{15}+16992690151 x^{16}+6899340865 x^{17}+2123302560 x^{18}+486469225 x^{19}+80870375 x^{20}+9387875 x^{21}+715250
x^{22}+31875 x^{23}+625 x^{24})) +
   $ } \\
    \parbox[t]{5.8in}{$ \ds
20 x (-6+14439 x^2-17489877 x^4+10330737178 x^6-3512457670994 x^8+773557628972348 x^{10}-118933118809202657 x^{12}+13402180795603772620
x^{14}-1143041824743722405247 x^{16}+75376256735760762870682 x^{18}-3900527153583558378830705 x^{20}+160213298841044278614369934 x^{22}-5275838401562286712534526324
x^{24}+140573690871732227718948395990 x^{26}-3056564285047841048567283322908 x^{28}+54652048844038774281403989770261 x^{30}-808942158731857921215515356108750
x^{32}+9967824187962503734284150480450070 x^{34}-102713732787270231411347953780078248 x^{36}+888250051092651001524860349740227677 x^{38}-6463166125218437400006981086907240411
x^{40}+39637978009471883546636305098439043760 x^{42}-205093971569105754927653524287955415546 x^{44}+895558787922679427854424612410673676432 x^{46}-3299173638099044870816627640048666130399
x^{48}+10245759838971553289794103842671422631885 x^{50}-26791934850374230320125306109278203915108 x^{52}+58906881739110815316562407273388595784955
x^{54}-108732542853795198593390914574744834434822 x^{56}+168239497159962925035017673451545751094463 x^{58}-217918473564650058361147962867208485475654
x^{60}+ $ \\ $ 236057112081995121270869493572868735267898 x^{62}- $ \\ $ 213713512024068312432499578346861834335878 x^{64}+ $ \\ $ 161683023301512813895399355837769391093165
x^{66}- $ \\ $ 102240860333144345801548115647872150591455 x^{68}+ $ \\ $ 54074149795443535376039142140102406234500 x^{70}- $ \\ $ 23943838634669362984469311992864393210682
x^{72}+8888700912481564685277221760614738560233 x^{74}-2771661455629275371014154547049255338390 x^{76}+727812657222702963167793555355730851828 x^{78}-161515258511255951877862805319023890021
x^{80}+30434257302491488042949030650375486214 x^{82}-4897312431601779726470344199863622918 x^{84}+676980041265615812866269278365661863 x^{86}-80737687927639046123515583011360482
x^{88}+8307477110059593202867794291904440 x^{90}-732682938604928561415945411279575 x^{92}+54623727635868006340062845596200 x^{94}-3372640551227980849360364652650
x^{96}+168032171870938912863676561625 x^{98}-6543238298083022807536718750 x^{100}+191057651766096466508155625 x^{102}- $ \\ $ 3932802250774791278810000 x^{104}+50816150733187918900000
x^{106}-287935206776374600000 x^{108}- $ \\ $ 1266111061732000000 x^{110}+20349020120000000 x^{112})/$ \\ $((-1+x) (1+x) (-1+2 x) (1+2 x) (-1+11
x) (1+11 x) (-1+9 x-6 x^2+x^3) (1+9 x+6 x^2+x^3) (1-12 x^2+8 x^3) (-1+12 x^2+8 x^3) (1-59 x+402
x^2-863 x^3+572 x^4) (1+59 x+402 x^2+863 x^3+572 x^4) (1-18 x+93 x^2-195 x^3+164 x^4-37 x^5+x^6) (1+18 x+93 x^2+195
x^3+164 x^4+37 x^5+x^6) (1-33 x+375 x^2-1922 x^3+5192 x^4-7883 x^5+6797 x^6-3252 x^7+834 x^8-105 x^9+5 x^{10}) (1+33 x+375
x^2+1922 x^3+5192 x^4+7883 x^5+6797 x^6+3252 x^7+834 x^8+105 x^9+5 x^{10}) (-1+5 x+45 x^2-305 x^3+65 x^4+2175 x^5-1800 x^6-4525 x^7+1975
x^8+2500 x^9-875 x^{10}-375 x^{11}+125 x^{12}) (-1-5 x+45 x^2+305 x^3+65 x^4-2175 x^5-1800 x^6+4525 x^7+1975 x^8-2500 x^9-875 x^{10}+375
x^{11}+125 x^{12}) (1-22 x-214 x^2+4230 x^3-15001 x^4-2783 x^5+60858 x^6-12548 x^7-46005 x^8+16340 x^9+10377 x^{10}-5171 x^{11}-329 x^{12}+413
x^{13}-27 x^{14}-9 x^{15}+x^{16}) (1+22 x-214 x^2-4230 x^3-15001 x^4+2783 x^5+60858 x^6+12548 x^7-46005 x^8-16340 x^9+10377 x^{10}+5171
x^{11}-329 x^{12}-413 x^{13}-27 x^{14}+9 x^{15}+x^{16}))-$ \\ \\ $ 20 x (-1-1084 x^2+181865 x^4-12397664 x^6+428673097 x^8-7910965080 x^{10}+75885327537 x^{12}-318421450140 x^{14}+6099467709
x^{16}+4470956945320 x^{18}-14829354664352 x^{20}+17943936244558 x^{22}-69992659920 x^{24}-23219650776430 x^{26}+28103414048922 x^{28}-16901971862574
x^{30}+5633156135784 x^{32}-787531780908 x^{34}-143356906659 x^{36}+93059619780 x^{38}-20683860360 x^{40}+2506426400 x^{42}-164601600 x^{44}+4608000
x^{46})/$ \\ $ ((-1+4 x) (1+4 x) (-1+34 x-160 x^2+248 x^3-152 x^4+32 x^5) (1+34 x+160 x^2+248 x^3+152 x^4+32 x^5)
(-1+18 x-91 x^2+110 x^3+175 x^4-515 x^5+455 x^6-175 x^7+25 x^8) (-1-18 x-91 x^2-110 x^3+175 x^4+515 x^5+455 x^6+175 x^7+25 x^8)
(-1+8 x+14 x^2-194 x^3+302 x^4-65 x^5-120 x^6+62 x^7+4 x^8-7 x^9+x^{10}) (-1-8 x+14 x^2+194 x^3+302 x^4+65 x^5-120 x^6-62 x^7+4 x^8+7
x^9+x^{10}))+ $ \\ \\ $ \ds
\frac{20 x (1+11 x^2)}{(-1+x) (1+x) (-1+11 x) (1+11 x)} + \frac{4 x}{1-4 x^2}   $
   \bc $ \mbox{--------------------------------------    } $  \ec} \\
\parbox[t]{5.8in}{ ${\cal F}^{TG(0)}_{10}(x) = $ \\ $ \ds  2 x+61098 x^2+308906 x^3+129930354 x^4+2720543472 x^5+566597492178 x^6+18309402983496 x^7+2938658810744994 x^8+116306062547543522 x^9+16178049740086515288
x^{10}+727292183495663026276 x^{11}+91590416197626458526954 x^{12}+4527128914550355698742952 x^{13}+527348632328247714223605228 x^{14}+28141085352317875290219968916
x^{15}+3073522709863650496508014495554 x^{16}+174855430602649847342302412927922 x^{17}+18089455293271238139140477649903078 x^{18}+1086333235497369435947818692121108606
x^{19}+107351220767799919507571884996197429624 x^{20}+6748862185028459643138579976924602126462 x^{21}+641627043758745484852708307999495553720888
x^{22}+ $ \\ $ 41926930850273492461849597562327593977489946 x^{23}+ $ \\ $3858618823943394269588240386686529017468650634 x^{24}+ $ \\ $260467820671447474586174871336823947674324860272
x^{25}  + \ldots $
 \bc $ \mbox{--------------------------------------    } $  \ec }  \noindent
 \\
\parbox[t]{5.8in}{ ${\cal F}^{TG}_{10,1}(x) = {\cal F}^{TG}_{10,9}(x) = $ \\
 $ \ds
1026 x+4180 x^2+2051592 x^3+33170592 x^4+8069846256 x^5+229045565290 x^6+40115886558012 x^7+1469214784114372 x^8+216667625165910036 x^9+9215230392687922900
x^{10}+1213567945794537846582 x^{11}+57414193981276811657742 x^{12}+6936778336725145347923466 x^{13}+356991726966366474382847320 x^{14}+40203601288654536608812622292
x^{15}+2218365210472271764159039832272 x^{16}+235526781691019196909383559047676 x^{17}+13782502354487762058634737982846660 x^{18}+1392191406729627260322579497951111622
x^{19}+85624673105717668604962699191914961602 x^{20}+8292493355622907262095650027980986669206 x^{21}+531939763587101129312734951569991372135300 x^{22}
+ $ \\ $ 49721834351877446556402146041368458726569482
x^{23}+ $ \\ $ 3304637053185180117384485662627188290697276482 x^{24}+ $ \\ $299841100528605106367555529876308008440746424306 x^{25}+ \ldots   $
    \bc $ \mbox{--------------------------------------    } $  \ec } \noindent \\
\parbox[t]{5.8in}{ ${\cal F}^{TG}_{10,2}(x) = {\cal F}^{TG}_{10,8}(x) = $ \\
 $ \ds
92 x+25218 x^2+471806 x^3+110135304 x^4+2969143592 x^5+544738148268 x^6+18733804343896 x^7+2906989357241034 x^8+117074541545912732 x^9+16126806424226575248
x^{10}+728717355076106649356 x^{11}+91503347671376455645554 x^{12}+4529796151476687283889602 x^{13}+527196799037450133481167468 x^{14}+28146094047791596217660650356
x^{15}+3073253455025125282943152429464 x^{16}+174864847843926862804318201550242 x^{17}+18088971857265650432893838691621708 x^{18}+1086350949490490741420195237197194316
x^{19}+107350344291239934770621956346509474854 x^{20}+6748895510691261530775791790843301575812 x^{21}+641625442244345393621045340731115667345458
x^{22}+ $ \\ $41926993550049337813771587624911437784823556 x^{23}+ $ \\ $3858615879222609487264258597095126866880206724 x^{24}+ $ \\ $260467938638858687027343013376015523986867164442
x^{25} + \ldots   $ \bc $ \mbox{--------------------------------------    } $  \ec } \\
\parbox[t]{5.8in}{ ${\cal F}^{TG}_{10,3}(x) = {\cal F}^{TG}_{10,7}(x) = $ \\
 $ \ds
216 x+7720 x^2+1359522 x^3+41326692 x^4+7224846606 x^5+244339052710 x^6+38872336271472 x^7+1497584405532272 x^8+214663386349476186 x^9+9268212962888129850
x^{10}+1210186806711127350552 x^{11}+57513573632708543746002 x^{12}+6930923246121002208627546 x^{13}+357178488377886900023024160 x^{14}+40193282071127277218839333842
x^{15}+2218716446757381095251745042372 x^{16}+235508351491935853471352363887656 x^{17}+13783163096442342022306579828853140 x^{18}+1392158139647409982133567021284313112
x^{19}+85625916211959364159888583860255390652 x^{20}+8292432789495082787047256034690725994066 x^{21}+531942102426568280941249248546221193046250 x^{22}+ $ \\ $ 49721723315230444674700433893953054270083442
x^{23}+ $ \\ $ 3304641453646815273059373950532468477928520782 x^{24}+ $ \\ $ 299840895821602146325432803809692855748240378556 x^{25}+ \ldots
     $ \bc $ \mbox{--------------------------------------    } $  \ec }

   \noindent
\parbox[t]{5.8in}{ ${\cal F}^{TG}_{10,4}(x) = {\cal F}^{TG}_{10,6}(x) = $ \\
 $ \ds
152 x+16248 x^2+624686 x^3+93066714 x^4+3290759442 x^5+516152502498 x^6+19364483660936 x^7+2859302099969304 x^8+118279367322734672 x^9+16045905490331081598
x^{10}+730996931066814837346 x^{11}+91363659130860668507844 x^{12}+4534093898798572924420242 x^{13}+526951852599307829987691948 x^{14}+28154186172061642466142676106
x^{15}+3072818241097287491246726066634 x^{16}+174880077107369071228724241445472 x^{17}+18088189923998590847682978221505078 x^{18}+1086379605861572895343307027610902516
x^{19}+107348926301965169450150269026412094304 x^{20}+6748949429080552776832015805855025487462 x^{21}+641622851054722146910701255125556753513288
x^{22}+$ \\ $ 41927094997964750458643481040889762138737366 x^{23}+ $ \\ $ 3858611114638599162347547786075041230010339874 x^{24}+$ \\ $ 260468129512499412268673475355192713041040718692
x^{25}+ \ldots
   $
\bc $ \mbox{--------------------------------------    } $  \ec } \\ \noindent
\parbox[t]{5.8in}{ ${\cal F}^{TG}_{10,5}(x) = $ \\
 $ \ds
306 x+10570 x^2+1211262 x^3+45545722 x^4+6849153816 x^5+252828385030 x^6+38186186511492 x^7+1514324450538202 x^8+213473030558695386 x^9+9300371737211057120
x^{10}+1208126218216126646952 x^{11}+57574580612326159255702 x^{12}+6927322391997258224033496 x^{13}+357293629161430374920263960 x^{14}+40186915480643390114180558472
x^{15}+2218933329755178601233759183642 x^{16}+235496967927647912670879570947646 x^{17}+13783571327298277992609557533586710 x^{18}+1392137583848052715245251761256984922
x^{19}+85626684406466871629751060858520922032 x^{20}+8292395360373738613387480674093825919686 x^{21}+531943547850287795263052296745922204885900 x^{22}+ $ \\ $ 49721654692613685201224263885906336633095402
x^{23}+ $ \\ $ 3304644173242482973584919641678295907935508182 x^{24}+ $ \\ $ 299840769306885738263268533595957733529298131316 x^{25}+ \ldots
   $ \bc $ \mbox{--------------------------------------    } $  \ec } \\
   \parbox[t]{5.8in}{ ${\cal F}^{KB}_{10,0}(x) = {\cal F}^{KB}_{10,2}(x) = {\cal F}^{KB}_{10,4}(x) ={\cal F}^{KB}_{10,6}(x) ={\cal F}^{KB}_{10,8}(x)  = $ \\ \\
 $ \ds
-2 x (51-831 x-153498 x^2+3494142 x^3+35737210 x^4-978785019 x^5+2181225858 x^6+40625932692 x^7-169809854157 x^8-539171574110
x^9+3248646540326 x^{10}+1733979403674 x^{11}-24220044711174 x^{12}+8425813503548 x^{13}+71539494759240 x^{14}-35544483302528 x^{15}-100136791120288
x^{16}+44771106019392 x^{17}+71841831791616 x^{18}-24300161941120 x^{19}-26998071010176 x^{20}+5807215608832 x^{21}+5044113799168 x^{22}-464409133056
x^{23}-381171302400 x^{24}-9920315392 x^{25}+3439853568 x^{26}) / $ \\ $((1+x) (-1+4 x) (1+4 x) (1+18 x+63 x^2+50 x^3+13 x^4+x^5)
(-1+34 x-160 x^2+248 x^3-152 x^4+32 x^5) (1+34 x+160 x^2+248 x^3+152 x^4+32 x^5) (-1+121 x-3901 x^2+49003 x^3-293115 x^4+896942
x^5-1381585 x^6+970100 x^7-264304 x^8+15552 x^9)) - $ \\ \\
$4 x (-45+71778 x^2-17369136 x^4+2188805644 x^6-76349085123 x^8+1009862205436 x^{10}-6106521487546 x^{12}+16764122743590
x^{14}-19163012747207 x^{16}+12803644578280 x^{18}-3806262284672 x^{20}+530610371584 x^{22}-33432641536 x^{24}+733872128 x^{26}) / $ \\ $((-1+2
x) (1+2 x) (1-3 x+x^2) (1+3 x+x^2) (1-13 x+11 x^2) (1+13 x+11 x^2) (-1+104 x-2661 x^2+24090 x^3-91993
x^4+158236 x^5-121128 x^6+44736 x^7-7808 x^8+512 x^9) (1+104 x+2661 x^2+24090 x^3+91993 x^4+158236 x^5+121128 x^6+44736 x^7+7808 x^8+512
x^9))+$ \\ \\ $ \ds 4 x (-30+17253 x^2-1654579 x^4+77516864 x^6-668471522 x^8+969635693 x^{10}+811800760 x^{12}-988086913 x^{14}+106897934 x^{16}) / $ \\ $ ((-1+x)
(1+x) (-1+11 x) (1+11 x) (-1+9 x-6 x^2+x^3) (1+9 x+6 x^2+x^3) (1-59 x+402 x^2-863 x^3+572 x^4) (1+59 x+402
x^2+863 x^3+572 x^4))- $ \\ \\ $12 x (-5+316 x^2-10552 x^4+58592 x^6-55552 x^8+12288 x^{10}) / $ \\ $ ((-1+4 x) (1+4 x) (-1+34 x-160 x^2+248 x^3-152 x^4+32
x^5) (1+34 x+160 x^2+248 x^3+152 x^4+32 x^5)) + $ \\ \\ $ \ds
\frac{20 x (1+11 x^2)}{(-1+x) (1+x) (-1+11 x) (1+11 x)} -\frac{4 x}{(-1+2 x) (1+2 x)}$
  \bc $ \mbox{--------------------------------------    } $  \ec } \\
   \parbox[t]{5.8in}{ ${\cal F}^{KB}_{10,0}(x) = {\cal F}^{KB}_{10,2}(x) = {\cal F}^{KB}_{10,4}(x) ={\cal F}^{KB}_{10,6}(x) ={\cal F}^{KB}_{10,8}(x)  = $ \\
 $ \ds
486 x+8742 x^2+1555326 x^3+40278270 x^4+7448028216 x^5+241043236914 x^6+39200186604960 x^7+1490511381690590 x^8+215200380276792594 x^9+9254216356840123792
x^{10}+1211105179284264511656 x^{11}+57486654067588797857634 x^{12}+6932526989474806983345204 x^{13}+357127332503295826122106028 x^{14}+40196121488251463309195643576
x^{15}+2218619758316902730081288019678 x^{16}+235513434522519195050742272384802 x^{17}+13782980800175828110273507021038930 x^{18}+1392167325142261666969798765433806674
x^{19}+85625572900721260086625938902471959360 x^{20}+8292449521724293420373158586932502245554 x^{21}+531941456216739020475250879088873942922380 x^{22}+$ \\ $ 49721753998437283054794227398187548320126318
x^{23}+$ \\ $ 3304640237580311820148409462658325632228709634 x^{24} +$ \\ $ 299840952395743492141155881299266273849675986916 x^{25}+ \ldots
   $ \bc $ \mbox{--------------------------------------    } $  \ec} \\ \noindent
   \parbox[t]{5.8in}{ ${\cal F}^{KB}_{10,1}(x) = {\cal F}^{KB}_{10,3}(x) ={\cal F}^{KB}_{10,5}(x) ={\cal F}^{KB}_{10,7}(x)  ={\cal F}^{KB}_{10,9}(x) =$ \\ \\
 $ \ds
2 x (-81+4743 x+18000 x^2-2761290 x^3-4824904 x^4+619528503 x^5-2779817376 x^6-20740283688 x^7+138457535763 x^8+160243309118
x^9-2194640836178 x^{10}+1013957190534 x^{11}+13884435328872 x^{12}-14282346936560 x^{13}-35197752351024 x^{14}+37740167279168 x^{15}+45100692942928
x^{16}-40565360679360 x^{17}-31522265045184 x^{18}+20159309425792 x^{19}+12193598770560 x^{20}-4556866880512 x^{21}-2441348405248 x^{22}+361003155456
x^{23}+203458510848 x^{24}+4195090432 x^{25}-1911029760 x^{26})/ $ \\ $((1+x) (-1+4 x) (1+4 x) (1+18 x+63 x^2+50 x^3+13 x^4+x^5)
(-1+34 x-160 x^2+248 x^3-152 x^4+32 x^5) (1+34 x+160 x^2+248 x^3+152 x^4+32 x^5) (-1+121 x-3901 x^2+49003 x^3-293115 x^4+896942
x^5-1381585 x^6+970100 x^7-264304 x^8+15552 x^9)) -$ \\ \\ $8 x^2 (-1413+1562003 x^2-366965421 x^4+24987288639 x^6-585637285790 x^8+6113962085337 x^{10}-30911433420511 x^{12}+73580273791090
x^{14}-71815424338881 x^{16}+29488786410400 x^{18}-5860745339328 x^{20}+585340968960 x^{22}-27329642496 x^{24}+444071936 x^{26})/$ \\ $((-1+2
x) (1+2 x) (1-3 x+x^2) (1+3 x+x^2) (1-13 x+11 x^2) (1+13 x+11 x^2) (-1+104 x-2661 x^2+24090 x^3-91993
x^4+158236 x^5-121128 x^6+44736 x^7-7808 x^8+512 x^9) (1+104 x+2661 x^2+24090 x^3+91993 x^4+158236 x^5+121128 x^6+44736 x^7+7808 x^8+512
x^9))) - $ \\ \\ $\ds 4 x^2 (1434-580778 x^2+52255635 x^4-1208428990 x^6+7871937815 x^8-17833483893 x^{10}+14395777753 x^{12}-3147975204 x^{14}+178151688
x^{16})/$ \\ $((-1+x) (1+x) (-1+11 x) (1+11 x) (-1+9 x-6 x^2+x^3) (1+9 x+6 x^2+x^3) (1-59 x+402 x^2-863 x^3+572 x^4)
(1+59 x+402 x^2+863 x^3+572 x^4)) - $ \\ \\$\ds 24 x^2 (-71+3736 x^2-39920 x^4+82624 x^6-48640 x^8+8192 x^{10})/$ \\ $((-1+4 x) (1+4 x) (-1+34 x-160 x^2+248 x^3-152
x^4+32 x^5) (1+34 x+160 x^2+248 x^3+152 x^4+32 x^5))
- $ \\ \\ $ \ds \frac{4 x^2 (-61+121 x^2)}{(-1+x) (1+x) (-1+11 x) (1+11 x)} - \frac{8 x^2}{(-1+2 x) (1+2 x)} $
  \bc $ \mbox{--------------------------------------    } $  \ec  }  \\
   \parbox[t]{5.8in}{ ${\cal F}^{KB}_{10,1}(x) = {\cal F}^{KB}_{10,3}(x) ={\cal F}^{KB}_{10,5}(x) ={\cal F}^{KB}_{10,7}(x)  ={\cal F}^{KB}_{10,9}(x) =$ \\
 $ \ds
162 x+26034 x^2+547254 x^3+105829506 x^4+3087139472 x^5+536544694458 x^6+18933459082308 x^7+2893320662340498 x^8+117429397122323214 x^9+16103929942562306184
x^{10}+729366115378484886684 x^{11}+91464255048323099506626 x^{12}+4530999924945274114423212 x^{13}+527128666585524788535089424 x^{14}+28148344110178105135600193324
x^{15}+3073132790947505523051236159154 x^{16}+174869068437238729059725316725722 x^{17}+18088755416815791020200633182335238 x^{18}+1086358879418462870325386565743428290
x^{19}+107349952098484950244056428392192616856 x^{20}+6748910421312067944031536548102249714502 x^{21}+641624725830162415536581332117994094315744
x^{22}+$ \\ $ 41927021596188944275645891032533653726633426 x^{23}+$ \\ $ 3858614562134145244469089962988617292496409666 x^{24}+$ \\ $ 260467991400549291223901340714858394942974085272
x^{25}+ \ldots
   $
     \bc $ \mbox{--------------------------------------    } $  \\ $ \mbox{--------------------------------------    } $  \ec }

\end{document}